\documentclass{amsart}

\usepackage{bbm,amsmath,amsfonts,amssymb,bm}
\usepackage[mathscr]{eucal}
\usepackage{xcolor}

\newtheorem{theorem}{Theorem}[section]
\newtheorem{lemma}[theorem]{Lemma}
\newtheorem{corollary}[theorem]{Corollary}
\newtheorem{proposition}[theorem]{Proposition}
\newtheorem{definition}[theorem]{Definition}
\newtheorem{example}[theorem]{Example}
\newtheorem{remark}[theorem]{Remark}

\numberwithin{equation}{section}

\def\bZ{{\mathbb Z}}

\def\bN{{\mathbb N}}

\def\bC{{\mathbb C}}
\def\bR{{\mathbb R}}

\def\bT{{\mathbb T}}
\def\bS{{\mathbb S}}

\def\bB{{\mathbb B}}
\def\bM{{\mathbb M}}

\def\cA{\mathcal{A}}

\def\cC{\mathcal{C}}

\def\cF{\mathcal{F}}

\def\cI{\mathcal{I}}
\def\cJ{\mathcal{J}}
\def\cK{\mathcal{K}}
\def\cL{\mathcal{L}}
\def\cM{\mathcal{M}}
\def\cN{\mathcal{N}}

\def\cP{\mathcal{P}}
\def\cQ{\mathcal{Q}}

\def\cS{\mathcal{S}}

\def\xx{{x}}
\def\yy{{y}}
\def\zz{{ z}}

\def\supp{\operatorname{supp}}
\def\eps{{\varepsilon}}
\def\ONE{{\mathbbm 1}}

\def\diam{\operatorname{diam}}
\def\ph{\varphi}

\def\R{\mathbb{R}}
\def\RR{\mathbb{R}}
\def\ZZ{{\mathbb Z}}

\def\cf{{c_\flat}}
\def\ct{{\tilde{c}}}

\def\KK{{\mathscr{K}}}
\def\BB{{\mathscr{B}}}

\def\Real{\operatorname{Re}}

\def\sB{\mathscr{B}}

\def\sL{\mathscr{L}}
\def\sE{\mathscr{E}}

\def\one{{\bm 1}}
\def\zero{{\bm 0}}

\def\balpha{{\bm \alpha}}
\def\bbeta{{\bm \beta}}
\def\bdelta{{\bm \delta}}
\def\bgamma{{\bm \gamma}}
\def\blambda{{\bm \lambda}}
\def\bnu{{\bm \nu}}
\def\bxi{{\bm \xi}}
\def\bsigma{{\bm \sigma}}
\def\btau{{\bm \tau}}
\def\beps{{\bm \varepsilon}}
\def\bell{{\bm \ell}}
\def\bkappa{{\bm \kappa}}
\def\bet{{\bm \eta}}

\def\ba{{\mathbf a}}

\def\dd{{\mathbf d}}

\def\jj{{\mathbf j}}
\def\kk{{\mathbf k}}

\def\nn{{\mathbf n}}
\def\mm{{\mathbf m}}

\def\ss{{\mathbf s}}

\def\tt{{\mathbf t}}
\def\uu{{\mathbf u}}
\def\vv{{\mathbf v}}
\def\xx{{\mathbf x}}
\def\yy{{\mathbf y}}
\def\zz{{\mathbf z}}

\def\sM{\mathscr{M}}

\def\tB{\widetilde{B}}
\def\tF{\widetilde{F}}

\def\BB{\mathscr{B}}
\def\FF{\mathscr{F}}
\def\tBB{\widetilde{\BB}}
\def\tFF{\widetilde{\FF}}

\def\m{{\mathfrak m}}

\def\DD{{\mathfrak D}}

\def\XX{X}

\begin{document}

\title[Distribution spaces on product spaces associated with operators]
{Spaces of distributions on product metric spaces associated with operators}

\author{A. G. Georgiadis}
\address{School of Computer Science and Statistics,
Trinity College of Dublin}
\email{georgiaa@tcd.ie}

\author{G. Kyriazis}
\address{Department of Mathematics and Statistics,
University of Cyprus, 1678 Nicosia, Cyprus}
\email{kyriazis@ucy.ac.cy}

\author{P. Petrushev}
\address{Department of Mathematics\\University of South Carolina\\
Columbia, SC 29208}
\email{pencho@math.sc.edu}

\subjclass[2010]{Primary 58J35, 46E35, 43A85, 42B30 ; Secondary 42B25, 42B15, 42C15, 42C40}

\keywords{Product spaces, Besov spaces, distributions, Heat kernel, metric spaces, self-adjoint operators,
Triebel-Lizorkin spaces}

\thanks{The third author has been supported by Grant KP-06-N62/4 of the Fund for Scientific Research of the Bulgarian Ministry of Education and Science.}

\begin{abstract}
We lay down the foundation of the theory of spaces of distributions
on the product $\XX_1\times \XX_2$ of doubling metric measure spaces $\XX_1$, $\XX_2$
in the presence of non-negative self-adjoint operators $L_1$, $L_2$,
whose heat kernels have Gaussian localization and the Markov property.
This theory includes the development of two-parameter functional calculus induced by $L_1, L_2$,
integral operators with highly localized kernels,
test functions and distributions associated to $L_1, L_2$,
spectral spaces accompanied by maximal Peetre and Nikolski type inequalities.
Hardy spaces are developed in this two-parameter product setup.
Two types of Besov and Triebel-Lizorkin spaces are introduced and studied:
ordinary spaces and spaces with dominating mixed smoothness, with emphasis on the latter.
Embedding results are obtained and spectral multiplies are developed.
\end{abstract}

\date{December 27, 2023}

\maketitle
\tableofcontents

\section{Introduction}\label{Introduction}

The purpose of this article is to 
is to develop the theory of spaces of distributions, in particular, Hardy, Besov and Triebel-Lizorkin spaces,
on the product $\XX_1\times \XX_2$ of doubling metric measure spaces $\XX_1$, $\XX_2$
in the presence of non-negative self-adjoint operators $L_1$, $L_2$,
whose heat kernels have Gaussian localization and the Markov property.

Our first step in this undertaking is to develop the two-parameter functional calculus induced by $L_1, L_2$,
which includes the development of the associated two-parameter spectral measure and spectral resolution.
The latter is in principle well-known but we will add some details to the picture.
An important part of the functional calculus is the development of integral operators with highly localized kernels. 
The development of distributions associated to $L_1, L_2$ is the next step,
followed by the development of spectral spaces accompanied by maximal Peetre and Nikolski type inequalities.

The Hardy spaces $H^p$, $0<p<\infty$, in the two-parameter setting of this article are the first type of spaces that we consider.
We show that the two-parameter Hardy spaces $H^p$, $0<p<\infty$, can be equivalently defined by means of several maximal operators 
and can be identified with the Lebesgue spaces $L^p$ for $1<p<\infty$.

Another major focus of our study is on the Besov and Triebel-Lizorkin spaces in the two-parameter setting.
The classical Besov and Triebel-Lizorkin spaces have been mainly developed for
the purposes of approximation theory and 
partial differential equations.
We refer the reader to the monographs \cite{Triebel-1, Triebel-2} of H. Triebel for the history
and a complete account of this theory.
The Littlewood-Paley approach to Besov and Triebel-Lizorkin spaces in the classical case on $\RR^n$
has been initiated by J.~Peetre \cite{Peetre}
and further implemented mainly by H. Triebel \cite{Triebel-1}, M. Frazier and B.~Javerth \cite{FJ1,FJ2,FJW}.
In this theory the Laplace operator $\Delta$ plays a prominent role.

It is natural to develop abstract Besov and Triebel-Lizorkin spaces,
where the role of $\Delta$ is played by a self-adjoint operator $L$ and
the spectral decomposition of the Laplacian (via the Fourier transform)
is replaced by the spectral decomposition associated to $L$.
This idea has been promoted by J. Peetre \cite[Chapter 10]{Peetre} and H.~Triebel \cite[Chapter 3]{Triebel-0}.
The problem, however, is that very little can be achieved if one operates in the general setting of
an abstract Hilbert space with the presence of a self-adjoint operator $L$.
Additional conditions need to be imposed in order to have a substantial theory. 

It turns out that Besov and Triebel-Lizorkin spaces with full set of parameters can be developed
in the setting of a doubling metric measure space accompanied by a non-negative self-adjoint operator $L$
whose heat kernel has Gaussian localization, H\"{o}lder continuity and the Markov property.
The basics of this theory have been developed in \cite{CKP,KP},
where distributions were introduced and Besov and Triebel-Lizorkin spaces of distributions 
with full set of parameters were defined and studied in complete analogy with the classical case on $\RR^n$.
Various aspects of the theory have been further developed in \cite{BBD,BD,DKKP,GKKP,GKKP2,GN,GN2,HH,LYY,ZY}.
The point is that many particular cases are covered by this setting.
Examples include the classical setting on $\RR^n$, the sphere, ball, and simplex in $\RR^n$ with weights
as well as Lie groups or homogeneous spaces with polynomial volume growth
and Riemannian manifolds with Ricci curvature bounded from below and satisfying the {\em volume doubling condition}.
It turns out that in the general setting of strictly local regular Dirichlet spaces
with a complete intrinsic metric it suffices to only verify
the local Poincar\'{e} inequality and the global doubling condition on the measure
and then the above general setting applies in full.
This theory allows to verify the conditions of our setting in a number of specific cases.
See \cite{CKP, KP} for details.

One of the main goals of this article is to develop the basics of the theory
of Hardy, Besov and Triebel-Lizorkin spaces on product domains $\XX_1\times \XX_2$ 
associated with non-negative self-adjoint operators $L_1,L_2$. 
The roots of this theory can be traced  back to the seventies in the work of  
S.-Y. A. Chang and R. Fefferman  \cite{CF,CF2,CF3,F1,F2,F3} on the Multiparameter Harmonic Analysis 
which basically extended the Calder\'{o}n-Zygmund Theory to product domains on $\R^n\times \R^m$. 
We refer the reader also to \cite{RobFS,GS,J} for these early developments.

We assume that $\XX_1, \XX_2$ are metric measure spaces with the doubling property
and the heat kernels associated to $L_1,L_2$ have Gaussian localization, H\"{o}lder continuity and the Markov property.
In developing the Hardy spaces $H^p$ in this two-parameter setting we show that they can be equivalently defined via several maximal operators
in the spirit of the classical case on $\RR^n$.

In the same setting we consider two kinds of Besov and Triebel-Lizorkin spaces:
ordinary spaces and spaces with dominating mixed smoothness.
The ordinary Besov and Triebel-Lizorkin spaces on $\RR^m\times \RR^n$
are simply the respective spaces on $\RR^{m+n}$ and depend on a single smoothness parameter.
The Besov and Triebel-Lizorkin spaces with dominating mixed smoothness on $\RR^m\times \RR^n$
depend on two-parameter $(s_1, s_2)$ smoothness.
For short we will call these spaces {\em mixed-smoothness B and F-spaces}.
The theory of the mixed smoothness B and F-spaces 
in the classical setting
on $\RR^m\times\RR^n$ has been mainly developed by  H.-J. Schmeisser and H. Triebel \cite{ST},
see also \cite{Schmeisser, V} and the references therein.
These spaces are closely related to approximation from hyperbolic cross trigonometric polynomials, splines, and wavelets,
see \cite{DKT,Schmeisser} and the references therein.
Homogeneous mixed smoothness B and F-spaces on $\RR^m\times \RR^n$ are introduced in \cite{GKP} and used
in nonlinear $m$-term approximation from hyperbolic cross wavelets.

It should be pointed out that the theory of Besov and Triebel-Lizorkin spaces with dominating mixed smoothness
on product domains other than $\RR^m\times\RR^n$ and $[0, 2\pi)^m\times [0, 2\pi)^n$ (the torus) is virtually not existent.
Ordinary Besov and Triebel-Lizorkin spaces on product domains are introduced and studied in \cite{IPX}.
More specifically, in \cite{IPX} ordinary B and F-spaces are considered on $[-1,1]^d$ associated with the Jacobi operator
and on $\bB^{d_1}\times [-1,1]^{d_2}$ associated with operators with polynomial eigenfunctions
on the unit ball $\bB^{d_1}$ in $\RR^{d_1}$ and $[-1,1]^{d_2}$.
Ordinary B and F-spaces on products of other domains such as the sphere, simplex, torus, and $\RR^d$
associated with operators are also discussed in \cite{IPX}.
All these cases are naturally covered by our theory.
Furthermore, our theory covers mixed-smoothness B and F-spaces on products of domains such as
the sphere, ball, simplex as well as Riemannian manifolds, Lee groups and more general Dirichlet spaces.

For notational convenience
we consider Besov and Triebel-Lizorkin spaces on the product $\XX_1\times \XX_2$ of only two domains $\XX_1, \XX_2$
associated with two operators $L_1, L_2$.
The extension of this theory to Besov and Triebel-Lizorkin spaces
on product domains $\XX_1\times\cdots\times \XX_n$
associated with operators $L_1, \dots, L_n$
is almost automatic by changing the notation.
The main focus of this article is on the development of the mixed-smoothness Besov and Triebel-Lizorkin spaces,
while the ordinary Besov and Triebel-Lizorkin spaces will be sparingly developed for comparison without many proofs.

The main contributions of this article can be summarised as follows:
\begin{itemize}
\item
The underlying two-parameter smooth functional calculus on
the product domain $\XX_1\times \XX_2$ associated with the operators $L_1, L_2$ is developed,
which includes the development of the two-parameter spectral measure induced by $L_1, L_2$,
and the development of integral operators of the form $F(\sqrt{L_1}, \sqrt{L_2})$
with highly localized kernels.

\item
Distributions associated to the operators $L_1, L_2$ are developed,
some of their basic properties are presented,
and a fundamental Carder\'{o}n reproducing formula is established.

\item
Spectral spaces, a substitute for band limited functions, associated to $L_1, L_2$ are introduced
and Peetre and Nikolksi type estimates are established.

\item
We utilize our development of integral operators and distributions
to develop Hardy spaces in the two-parameter setting of this article;
we show that the two-parameter Hardy spaces $H^p$, $p>0$, can be equivalently defined by means of several different maximal operators.

\item
Besov and Triebel-Lizorkin spaces with dominating mixed smoothness
on $\XX_1\times \XX_2$ associated with $L_1, L_2$  are introduced and explored.
In particular, embeddings between test functions, mixed-smoothness B and F-spaces, and distributions
associated with operators are obtained.

\item
Spectral multipliers are introduced and their boundedness on test functions and distributions
as well as on mixed-smoothness B and F-spaces are obtained.

\item
The basics of ordinary (one parameter) Besov and Triebel-Lizorkin spaces
on $\XX_1\times \XX_2$ associated with $L_1, L_2$ are also developed.
The main purpose of these results is to put in perspective our development of
the mixed-smoothness B and F-spaces.

\end{itemize}

Basic components of the theory of Besov and Triebel-Lizorkin spaces on product domains associated with operators such as
frame decomposition,
interpolation, and
atomic and molecular decompositions of B-F-spaces
are left for future development.

An outline of this article is as follows:
In Section~\ref{sec:setting}, we introduce our setting and present some basic facts about the coordinate spaces $\XX_1, \XX_2$.
In Section~\ref{sec:product-m-spaces}, we introduce the product spaces and present some of their properties.
In Section~\ref{sec:func-calc}, we develop the underlying two-parameter smooth functional calculus on
the product domain $\XX_1\times \XX_2$ associated with the operators $L_1, L_2$.
Distributions associated to the operators $L_1, L_2$ are developed in Section~\ref{sec:distributions}.
Spectral spaces 
are introduced in Section~\ref{spectral-spaces}
and Peetre and Nikolksi type estimates are established.
Section~\ref{sec:Hardy} is devoted to the development of Hardy spaces in the two-parameter setting of this article.
Mixed-smoothness Besov and Triebel-Lizorkin spaces are introduced and explored in Section~\ref{sec:B-F-spaces}.
Section~\ref{sec:multipliers} is devoted to the study of spectral multipliers.
Bounded spectral multipliers on test functions and distributions
as well as on mixed-smoothness B and F-spaces are obtained.
In Section~\ref{sec:reg-B-F-spaces}, we develop the basics of ordinary B and F-spaces
associated with $L_1,L_2$.
Examples of coordinate spaces covered by our theory are given in the appendix. 
%
We also place the lengthy proof of a proposition from Section~\ref{sec:func-calc} in the appendix.

{\em Notation}:
Throughout this article $\bN$ stands for the set of all integers and $\bN_0:=\bN\cup\{0\}$;
$\RR_+:= (0, \infty)$ and $\ONE_E$ stands for the characteristic function of the set $E$.
Also, $\|\cdot\|_p$ stands for the $L^p$ norm in the respective space.
As usual $\cC^k(\RR)$ denotes the set of all functions on $\RR$ with continuous derivatives of order up to $k$
and $\cS(\RR)$ stands for the Schwartz class on $\RR$.
Positive constants are denoted by $c, \tilde{c}, c_1, \dots$ and they may vary at every occurrence.
The notation $a\sim b$ stands for $c^{-1}\le a/b\le c$.
We also use the standard notation
$a\wedge b:= \min\{a, b\}$ and
$a\vee b:= \max\{a, b\}$.

As a general rule letters in bold like $\xx$, $\yy$, $\zz$, $\tt$, $\bdelta$, $\blambda$, $\bbeta$, $\bgamma$, etc.
stand for vectors or multi-indices.
In particular, we will use the notation $\one:=(1,1)$ and $\zero:=(0,0)$.
For a vector, say, $\blambda=(\lambda_1, \lambda_2)$ we denote $|\blambda|:= (\lambda_1^2+\lambda_2^2)^{1/2}$,
while for a multi-index, say, $\bdelta=(\delta_1, \delta_2)$ we denote $|\bdelta|:=\delta_1+\delta_2$.
For multi-indices $\bdelta=(\delta_1,\delta_2)$, $\bsigma=(\sigma_1,\sigma_2)$ 
all operations should be understood coordinate-wise,
in particular,
$\bdelta\pm\bsigma:=(\delta_1\pm\sigma_1,\delta_2\pm\sigma_2)$,
$\bdelta\bsigma:= (\delta_1\sigma_1, \delta_2\sigma_2)$,
$\bdelta/\bsigma:=(\delta_1/\sigma_1,\delta_2/\sigma_2)$,
especially, $\one/\bdelta=\bdelta^{-1}:=\big(\frac{1}{\delta_1},\frac{1}{\delta_2}\big)$,
$\bdelta^{\bsigma}:=\delta_1^{\sigma_1}\delta_2^{\sigma_2}$,
and
$\bdelta\cdot\bsigma:=\delta_1\sigma_1+\delta_2\sigma_2$.
Also,
$\bsigma \ge \bdelta$ stands for $\sigma_i\ge \delta_i$, $i=1,2$.

\section{Setting and geometry of the coordinate spaces}\label{sec:setting}

As alluded to in the introduction we are interested in studying
function spaces defined on the product $\XX=\XX_1\times \XX_2$
of spaces $\XX_1$, $\XX_2$ of homogeneous type associated with shelf-adjoint operators $L_1, L_2$.
In what follows we present the main conditions on the coordinate spaces $\XX_1$, $\XX_2$
as well as some of their basic  properties.

\subsection{The setting}\label{sec:genset}

We will operate in the following general setting, where for $i=1,2$ we stipulate:

I. $(\XX_i,\rho_i,\mu_i)$ is a metric measure space such that
$(\XX_i, \rho_i)$ is locally compact with distance $\rho_i(\cdot, \cdot)$
and $\mu_i$ is a~positive Radon measure.
We assume the following {\em doubling volume condition}:
There exists a~constant $c_{i0}>1$ such that
\begin{equation}\label{eq:doubling-0}
0 < \mu_i(B_i(x_i,2r))\le c_{i0}\mu_i(B_i(x_i,r))<\infty
\quad\hbox{for $x_i \in \XX_i$ and $r>0$,}
\end{equation}
where
$B_i(x_i,r):=\big\{y_i\in \XX_i: \rho_i(x_i,y_i)<r\big\}$.
We will use the abbreviated notation
\begin{equation}\label{def-Vi}
V_i(x_i,r):= \mu_i(B_i(x_i,r)).
\end{equation}

\medskip

II. The space $(\XX_i,\rho_i,\mu_i)$ is accompanied by
a densely defined non-negative self-adjoint \textit{operator} $L_i$, 
mapping real-valued to real-valued functions,
such that the associated semigroup $P_{i,t}=e^{-tL_i}$ consists of integral operators with
(heat) kernel $p_{i,t}(x_i,y_i)$ obeying the conditions:

$(a)$ {\em Gaussian upper bound:} There exist constants $c_{i1},\;c_{i2}>0$ such that
\begin{equation}\label{Gauss-local}
|p_{i,t}(x_i,y_i)|
\le \frac{ c_{i1}\exp\big(-\frac{\rho_i^2(x_i,y_i)}{c_{i2}t}\big)}{\big[V_i(x_i,\sqrt t) V_i(y_i,\sqrt t)\big]^{1/2}}
\quad\hbox{for} \;\;x_i,y_i\in \XX_i,\,t>0.
\end{equation}

$(b)$ {\em H\"{o}lder continuity:} There exists a constant $\alpha_i>0$ such that
\begin{equation}\label{lip}
\big|  p_{i,t}(x_i,y_i) - p_{i,t}(x_i,y_i ')\big|
\le c_{i1}\Big(\frac{\rho_i(y_i,y_i ')}{\sqrt t}\Big)^{\alpha_i}
\frac{\exp\big(-\frac{\rho_i^2(x_i,y_i)}{c_{i2}t} \big)}{\big[V_i(x_i,\sqrt t) V_i(y_i,\sqrt t)\big]^{1/2}}
\end{equation}
for $x_i, y_i, y_i '\in \XX_i$ and $t>0$, whenever $\rho_i(y_i,y_i ')\le \sqrt{t}$.

$(c)$ {\em Markov property:}
\begin{equation}\label{hol3}
\int_{\XX_i} p_{i,t}(x_i,y_i) d\mu_i(y_i)= 1
\quad\hbox{for $t >0$.}
\end{equation}

There is a variety of spaces $(\XX_i,\rho_i,\mu_i)$ associated with operators $L_i$
that are covered by this setting. Several examples are given in Section~\ref{sec:examples}.

\smallskip

In a few instances we will assume that the following additional condition is satisfied:

$(d)$ {\em Non-collapsing condition:}
There exists a~constant $c_{i3}>0$ such that
\begin{equation}\label{non-collapsing}
\inf_{x_i\in \XX_i} V_i(x_i,1)\ge c_{i3}\end{equation}
Note that this condition is always satisfied in the case of compact spaces (see \cite{CKP}).
However, we do not consider \eqref{non-collapsing} a part of our general setting.

Observe that $e^{-tL_i}$ extends to a holomorphic semigroup on $L^2(\XX_i)$ and
from \eqref{Gauss-local} and \eqref{eq:doubling-0} it follows that $e^{-zL_i}$, $z\in\bC_+$,
is an integral operator with kernel $p_{i,z}(x_i,y_i)$ satisfying
\begin{equation}\label{Gauss-local-z}
|p_{i,z}(x_i,y_i)|
\le \frac{ c_{i1}\exp\big(-c_{i2}\Real \frac{\rho_i^2(x_i,y_i)}{z}\big)}{\big[V_i(x_i,\sqrt{\Real z}) V_i(y_i,\sqrt{\Real z})\big]^{1/2}}
\quad\hbox{for} \;\;x_i,y_i\in \XX_i,\; \Real z >0.
\end{equation}
Furthermore, by analytic continuation it follows from \eqref{hol3} that
\begin{equation}\label{markov-2}
\int_{\XX_i} p_{i,z}(x_i,y_i) d\mu_i(y_i)= 1,
\quad\hbox{$\Real z >0$.}
\end{equation}
For more details, see (1.5) and (1.8) in \cite{CKP}.

From the  doubling volume property \eqref{eq:doubling-0} it readily follows that 
there exist constants $\hat{c}_{i0}>0$ and $0<d_i\le \log_2 c_{i0}$ such that
\begin{equation}\label{doubling-d}
V_i(x_i,\lambda r)\le \hat{c}_{i0}\lambda^{d_i}V_i(x_i,r)
\quad\hbox{for every}\;\; x_i\in \XX_i, r>0, \text{ and }  \lambda\ge 1.
\end{equation}
Of course $d_i$ can be a non-integer number; it serves as a dimension of the doubling space.

\subsection{Auxiliary inequalities}

We next present several inequalities that will be needed in what follows;
their simple proofs can be found in \cite{CKP, KP}.

The volumes of balls with different centers $x_i,y_i\in \XX_i, i=1,2$, and the same radius $r>0$
can be compared  using  the inequality
\begin{equation}\label{changecenter}
V_i(x_i,r)^{\gamma_i}\le c\big(1+r^{-1}\rho_i(x_i,y_i)\big)^{d_i |\gamma_i|}V_i(y_i,r)^{\gamma_i}, \quad  \gamma_i\in\bR.
\end{equation}

If we assume the non-collapsing condition \eqref{non-collapsing},
then using \eqref{doubling-d} it follows that
\begin{equation}\label{non-collaps-2}
V_i(x_i,r)\ge (c_{i3}/c_{i})r^{d_i},\quad x_i\in \XX_i,\ 0<r\le1.
\end{equation}

If $\tau_i>d_i$, then there exists a constant $c=c(\tau_i)>0$ such that
\begin{equation}\label{int-ineq-1}
\int_{\XX_i}\big(1+r^{-1}\rho_i(x_i,y_i)\big)^{-\tau_i}d\mu_i(y_i)\le cV_i(x_i,r),
\quad r>0,
\end{equation}
and 
\begin{equation}\label{int-ineq-2}
\int_{\XX_i}\frac{d\mu_i(y_i)}{\big(1+r^{-1}\rho_i(x_i,y_i)\big)^{\tau_i}\big(1+r^{-1}\rho_i(y_i,z_i)\big)^{\tau_i}}
\le \frac{cV_i(x_i,r)}{\big(1+r^{-1}\rho_i(x_i,z_i)\big)^{\tau_i-d_i}}.
\end{equation}

If $\tau_i>2d_i$, then there exists a constant $c=c(\tau_i)>0$ such that
\begin{equation}\label{int-ineq-3}
\int_{\XX_i}\frac{d\mu_i(y_i)}{V_i(y_i,r)\big(1+r^{-1}\rho_i(x_i,y_i)\big)^{\tau_i}\big(1+r^{-1}\rho_i(y_i,z_i)\big)^{\tau_i}}
\le \frac{c}{\big(1+r^{-1}\rho_i(x_i,z_i)\big)^{\tau_i}},
\end{equation}
and if in addition $r\le1$, then
\begin{equation}\label{int-ineq-4}
\int_{\XX_i}\frac{d\mu_i(y_i)}{V_i(y_i,r)\big(1+r^{-1}\rho_i(x_i,y_i)\big)^{\tau_i}\big(1+\rho_i(y_i,z_i)\big)^{\tau_i}}
\le \frac{c}{\big(1+\rho_i(x_i,z_i)\big)^{\tau_i}}.
\end{equation}

\subsection{Functional calculus}

Since $L_i$ is a non-negative self-adjoint operator, by Spectral theory
there exists a unique spectral resolution $E^i_\lambda$ associated to $L_i$,
consisting of orthoprojections in $L^2(\XX_i, d\mu_i)$ such that
\begin{equation}\label{eq:spec_resl}
L_i=\int_0^\infty \lambda d E^i_{\lambda}.
\end{equation}
Moreover, $L_i$ maps real-valued to real-valued functions and  for any real-valued, measurable and
bounded function $F_i$ on $[0,\infty)$ the operator (spectral multiplier) $F_i(L_i)$, defined by
\begin{equation}\label{eq:spec_multiplier}
F_i(L_i):=\int_0^\infty F_i(\lambda)dE^i_{\lambda}
\end{equation}
is bounded on $L^2(\XX_i, d\mu_i)$, self-adjoint, and also maps real-valued functions to real-valued functions.
Furthermore, if $F_i(L_i)$ is an integral operator,
then its kernel $\KK_{F_i(L_i)}(x_i, y_i)$
is real-valued and
$\KK_{F_i(L_i)}(y_i, x_i)=\KK_{F_i(L_i)}(x_i, y_i)$.
In particular, for the heat kernel $p_{i,t}(x_i, y_i)$
we have $p_{i,t}(x_i, y_i)\in \R$ and $p_{i,t}(y_i,x_i) = p_{i,t}(x_i, y_i)$.

The following {\em Davies-Gaffney estimate} follows from our basic assumptions I - II (see \cite{CS, KP}):
\begin{equation}\label{davies-gaffney}
|\langle P_{i,t} f_1, f_2\rangle| \le \exp\Big\{-\frac{c^\star_i r^2}{t}\Big\} \|f_1\|_2\|f_2\|_2,
\quad t>0,
\end{equation}
for all open sets $U_j \subset \XX_i$ and $f_j\in L^2(\XX_i)$ with $\supp f_j\subset U_j$, $j=1, 2$,
where $r:=\rho(U_1, U_2)$ and $c^\star_i=c_{i2}^{-1}$ with $c_{i2}>0$ the constant from (\ref{Gauss-local}).

In turn, the Davies-Gaffney estimate implies (is equivalent to, see \cite{CS})
the {\em finite speed propagation property}, which will play an important role in what follows:
\begin{equation}\label{finite-speed}
\big\langle \cos(t\sqrt{L_i})f_1, f_2 \big\rangle=0,
\quad 0< \ct_i t<r,
\quad \ct_i:=\frac{1}{2\sqrt{c^\star_i}},
\end{equation}
for all open sets $U_j \subset \XX_i$, $f_j\in L^2(\XX_i)$, $\supp f_j\subset U_j$,
$j=1, 2$, where $r:=\rho(U_1, U_2)$.

We next present the main localization result in the functional calculus induced by $L_i$ (see \cite[Theorem 3.4]{KP}).

\begin{theorem}\label{thm:S-local-kernels1d}
For $i=1,2$ let $F_i\in \cC^{k_i}(\bR)$, $k_i>d_i$,
be real-valued and even and let
\begin{equation}\label{eq:S-local-kernels1d}
|F_i^{(\nu)}(\lambda)|\le c^\diamond (1+\lambda)^{-r_i}
\  \text{for}\ \lambda\ge 0\  \text{and}\  0\le \nu\le k_i,\  \text{where}\ r_i \ge k_i+d_i+1.
\end{equation}
Then $F_i(\delta \sqrt L_i)$, $\delta>0$, is an integral operator with kernel
$\KK_{F_i(\delta \sqrt L_i)}(x_i, y_i)$ satisfying
\begin{equation}\label{local-ker-1d}
\big|\KK_{F_i(\delta \sqrt L_i)}(x_i, y_i)\big|
\le c' V_i(y_i,\delta)^{-1}\big(1+\delta^{-1}\rho_i(x_i,y_i)\big)^{-k_i+d_i/2}
\end{equation}
and
\begin{equation}\label{local-ker-smth}
\big|\KK_{F_i(\delta \sqrt L_i)}(x_i, y_i)-\KK_{F_i(\delta \sqrt L_i)}(x_i, y_i')\big|
\le \frac{c'' V_i(y_i,\delta)^{-1}\big(\frac{\rho_i(y_i,y_i')}{\delta}\big)^{\alpha_i}}
{\big(1+\delta^{-1}\rho_i(x_i,y_i)\big)^{k_i-d_i/2}},
\end{equation}
if $\rho_i(y_i,y_i')\le \delta$,
where $\alpha_i$  is the constant from \eqref{lip}
and  the constant $c'>0$ depends on the structural constants from \eqref{eq:doubling-0}, \eqref{Gauss-local}, \eqref{lip}, and $c^\diamond$, $k_i$;
the constant $c''$ depends in addition on $\alpha$. 
Furthermore,
$$
\int_{\XX_i} \KK_{F_i(\delta \sqrt L_i)}(x_i, y_i)d\mu_i(y_i)=F_i(0), \quad x_i\in \XX_i.
$$
\end{theorem}
Note that instead of the  decay condition (\ref{eq:S-local-kernels1d})
one can assume that $F_i$ is an even $\cC^{k_i}(\bR)$ function with compact support.

\section{Product metric measure spaces}\label{sec:product-m-spaces}

In this section we recall the definition of product measure spaces and set up our notation.

We assume that $(\XX_1,\rho_1,\mu_1)$ and $(\XX_2,\rho_2,\mu_2)$ are two measure metric spaces
associated with non-negative self-adjoint operators $L_1$ and $L_2$, respectively,
which satisfy conditions I -- II of Section \ref{sec:genset}.
The \textit{product space} $\XX=\XX_1\times \XX_2$ is defined by
\begin{equation*}\label{space}
\XX=\XX_1\times \XX_2:=\big\{\xx=(x_1,x_2):\;x_i\in \XX_i,\;i=1,2\big\}.
\end{equation*}
We consider $\XX$ equipped with the \textit{metric}
\begin{equation}\label{metric}
\rho(\xx,\yy):=\max\big(\rho_1(x_1,y_1),\rho_2(x_2,y_2)\big),
\quad\xx,\yy\in \XX,
\end{equation}
and  the \textit{product measure} $\mu:=\mu_1\times\mu_2$, that is,
\begin{equation}\label{measure}
d\mu(\xx)=d\mu_1(x_1)d\mu_2 (x_2),\quad\xx\in \XX.
\end{equation}

Clearly, the \textit{ball} $B(\xx,r) \subset \XX$ centered at $\xx=(x_1,x_2)\in \XX$ of radius $r>0$
is given by
\begin{equation}\label{balls}
B(\xx,r)=B_1(x_1,r)\times B_2(x_2,r)
\end{equation}
and
\begin{equation}\label{ballvolume}
\mu(B(\xx,r))=\mu_1(B_1(x_1,r))\mu_2(B_2(x_2,r)),
\end{equation}
where $B_i(x_i,r)$ is the respective balls in $\XX_i$.
Therefore, by \eqref{eq:doubling-0} it follows that $(\XX,\rho,\mu)$ has the doubling volume property
\begin{equation}\label{doublingB}
\mu(B(\xx,2r)) \le c_{10}c_{20}\mu(B(\xx,r))
\end{equation}
and from \eqref{doubling-d}
\begin{equation}\label{doubl-2}
\mu(B(\xx,\lambda r)) \le c\lambda^{d_1+d_2}\mu(B(\xx,r)),
\quad \forall r>0, \; \lambda \ge 1.
\end{equation}
The following dimension-type vector 
\begin{equation}\label{dimension multi-index}
\dd:=(d_1,d_2).
\end{equation}
will appear frequently in what follows.

For any $\xx=(x_1, x_2)\in \XX$ and $\bdelta \in(0,\infty)^2$
we define the {\it rectangle} $Q(\xx,\bdelta)$ by
\begin{equation}\label{rectangles}
Q(\xx,\bdelta):=B_1(x_1,\delta_1)\times B_2(x_2,\delta_2).
\end{equation}
To simplify our notation we set for $\xx\in \XX$, $\bdelta \in(0,\infty)^2$,
\begin{equation}\label{def-V}
V(\xx, \bdelta):= \mu(Q(\xx,\bdelta))
=\mu_1(B_1(x_1,\delta_1))\mu_2(B_2(x_2,\delta_2))
= V_1(x_1, \delta_1)V_2(x_2, \delta_2).
\end{equation}

Also, for $\xx, \yy\in\XX$ and $\bdelta, \bsigma \in(0,\infty)^2$, we define
\begin{equation}\label{kernelsD*}
\DD^*_{\bdelta,\bsigma}(\xx,\yy):=\prod_{i=1,2}\big(1+\delta_{i}^{-1}\rho_i (x_i,y_i)\big)^{-\sigma_i}
\quad\hbox{and}\quad
\end{equation}
\begin{equation}\label{kernelsD}
\DD_{\bdelta,\bsigma}(\xx,\yy)
:= \frac{\DD^*_{\bdelta,\bsigma}(\xx,\yy)}{[V(\xx, \bdelta)V(\yy, \bdelta)]^{1/2}}
= \frac{\prod_{i=1,2}\big(1+\delta_{i}^{-1}\rho_i (x_i,y_i)\big)^{-\sigma_i}}
{[V_1(x_1, \delta_1)V_2(x_2, \delta_2)V_1(y_1, \delta_1)V_2(y_2, \delta_2)]^{1/2}}.
\end{equation}

\subsection{Geometric properties of product spaces}

By \eqref{doubling-d} and \eqref{def-V} it readily follows that
\begin{equation}\label{rect-doubling}
V(\xx,\blambda\bdelta) \le c\blambda^\dd V(\xx,\bdelta)),\quad \bdelta\in (0,\infty)^2, \;\blambda\in [1,\infty)^{2}.
\end{equation}
Recall the notation $\blambda\bdelta:= (\lambda_1\delta_1, \lambda_2\delta_2)$
and $\blambda^\dd:=\lambda_1^{d_1}\lambda_2^{d_2}$.

To compare the volumes of products of balls with different centers we use \eqref{changecenter} and \eqref{def-V}
to obtain
\begin{equation}\label{rect-doubling2}
\begin{aligned}
V(\xx,\bdelta) &\le cV(\yy,\bdelta)\prod_{i=1,2}\big(1+\delta_i^{-1}\rho_i(x_i,y_i)\big)^{d_i}
\\
&=cV(\yy,\bdelta)\DD^*_{\bdelta, \dd}(\xx,\yy)^{-1}.
\end{aligned}
\end{equation}
Using (\ref{rect-doubling2}) it easily follows that
\begin{equation}\label{D-D*}
\DD_{\bdelta,\bsigma}(\xx,\yy)\le cV(\xx,\bdelta)^{-1}\DD^*_{\bdelta,\bsigma-\dd/2}(\xx,\yy).
\end{equation}

We will frequently use the following notation
\begin{equation}\label{rectangles-gamma}
V(\xx,\bdelta)^\bgamma := V_1(x_1,\delta_1)^{\gamma_1} V_2(x_2,\delta_2)^{\gamma_2},
\end{equation}
where $\xx\in X$, $\bdelta\in (0,\infty)^2$, $\bgamma\in \RR^2$.
As a variation of this we will use for example the notation
\begin{equation}\label{def-Vjsd}
V(\xx,2^{-\jj})^{-\ss/\dd} := V_1(x_1,2^{-j_1})^{-s_1/d_1} V_2(x_2,2^{-j_2})^{-s_2/d_2},
\end{equation}
where by definition $2^{-\jj}:= (2^{-j_1},2^{-j_2})$, $\ss/\dd:=(s_1/d_1,s_2/d_2)$.

From \eqref{changecenter} and \eqref{def-V} we get the following useful inequality
when changing centers of balls:
\begin{equation}\label{V-gamma-xy}
V(\xx,\bdelta)^\bgamma \le cV(\yy,\bdelta)^\bgamma\prod_{i=1,2}\big(1+\delta_i^{-1}\rho_i(x_i,y_i)\big)^{d_i|\gamma_i|},
\quad \bgamma\in \RR^2.
\end{equation}

If we assume that the non-collapsing condition \eqref{non-collapsing} is satisfied,
then \eqref{non-collaps-2} is valid, which in turn implies
\begin{equation}\label{rect-non-coll}
V(\xx,\bdelta)\ge c\bdelta^{\dd},\quad \bdelta\in(0,1]^{2}.
\end{equation}

\subsection{Integral estimates}

The following integral estimates follow readily from \eqref{int-ineq-1}--\eqref{int-ineq-4}:
If $\bdelta,\bsigma\in(0,\infty)^2$ and $\bsigma>\dd$,
i.e. $\sigma_i>d_i$, $i=1,2$,
then
\begin{equation}\label{tech-1}
\int_\XX \DD^*_{\bdelta,\bsigma}(\xx,\yy) d\mu(\yy) \le c V(\xx,\bdelta),
\quad \forall \xx\in \XX,\quad\text{and}
\end{equation}
\begin{equation}\label{tech-2}
\int_\XX \DD^*_{\bdelta,\bsigma}(\xx,\yy)\DD^*_{\bdelta,\bsigma}(\yy,\zz)d\mu(\yy)
\le c V(\xx,\bdelta)\DD^*_{\bdelta,\bsigma-\dd}(\xx,\zz),
\quad \forall \xx,\zz\in \XX.
\end{equation}
Furthermore, if $\bsigma>2\dd$,
then
\begin{equation}\label{tech-3}
\int_\XX \DD_{\bdelta,\bsigma}(\xx,\yy)\DD_{\bdelta,\bsigma}(\yy,\zz)d\mu(\yy)
\le c\DD_{\bdelta,\bsigma}(\xx,\zz),
\quad \forall \xx,\zz\in \XX,
\end{equation}
and if in addition $\bdelta\in(0,1]^2$, then
\begin{equation}\label{tech-4}
\int_\XX V(\yy,\bdelta)^{-1}\DD^*_{\bdelta,\bsigma}(\xx,\yy)\DD^*_{\one,\bsigma}(\yy,\zz)d\mu(\yy)
\le c\DD^*_{\one,\bsigma}(\xx,\zz),
\quad \forall \xx,\zz\in \XX.
\end{equation}

\section{Two-parameter functional calculus on product spaces}\label{sec:func-calc}

The present section lays down some of the ground work that will be needed in developing our theory,
namely, the development of the two-parameter functional calculus induced by the operators $L_1$, $L_2$ on the doubling product measure spaces $X=X_1\times X_2$.
This theory includes the development of the associated two-parameter spectral measure and spectral resolution
as well as the development of associated integral operators with highly localized kernels. 
The theory of the two-parameter spectral measure and spectral resolution in our setting is in principe well-know (see e.g. \cite{Schmudgen} or \cite{Prugov}).
However, since we do not find it in the form we needed in the literature we will present it in sufficient detail.
The development of integral operators with highly localized kernels in our setting is based on the respective one-parameter integral operators from \cite{CKP, KP}.
 
Just as in Section~\ref{sec:genset} we assume that $L_1$, $L_2$ are two
densely defined self-adjoint operator on $L^2(\XX_1, d\mu_1)$ and $L^2(\XX_2, d\mu_2)$, respectively.
We also assume that the spaces $L^2(\XX_1, d\mu_1)$ and $L^2(\XX_2, d\mu_2)$ are separable.
As is shown in \cite{KP} this is the case in our setting because there exist countable frames for the respective spaces.
%
At this point we do not assume that $L_1$, $L_2$ are non-negative operators.

We first recall the definition and basic properties of tensor product operators of the form
$L_1^{\nu_1}\otimes L_2^{\nu_2}$ and then
present the basics of the two-parameter functional calculus induced by $L_1$, $L_2$.
We will use the standard notation $f\otimes g$ for the tensor product of any two functions $f,g$
defined on $\XX_1$, $\XX_2$,
that is,
\begin{equation}\label{prod-f-g}
f\otimes g(\xx):= f(x_1)g(x_2), \quad \xx=(x_1,x_2) \in \XX_1\times \XX_2.
\end{equation}

\subsection{Tensor product operators of the form $L_1^{\nu_1}\otimes L_2^{\nu_2}$}\label{subsec:L1-times-L2}

For $\nu_1, \nu_2 \in \bN_0$ we denote by $D(L_1^{\nu_1})$ and $D(L_2^{\nu_2})$
the domains of the operators $L_1^{\nu_1}$ and $L_2^{\nu_2}$.
Since $L_i$ is self-adjoint, then $L_i^{\nu_i}$ is also self-adjoint for any $\nu_i\in\bN$, $i=1,2$ 
(see e.g. \cite[Theorem~VIII.6]{RS}).
Denote by $D(L_1^{\nu_1})\otimes D(L_2^{\nu_2})$ the set of all finite linear combinations of functions
of the form $f\otimes g$, where $f\in D(L_1^{\nu_1})$ and $g\in D(L_2^{\nu_2})$.
As $D(L_i^{\nu_i})$ is dense in $L^2(\XX_i, d\mu_i)$, $i=1,2$,
then $D(L_1^{\nu_1})\otimes D(L_2^{\nu_2})$ is dense in $L^2(\XX_1\times \XX_2, d\mu_1\otimes d\mu_2)$.
The operator $L_1^{\nu_1}\otimes L_2^{\nu_2}$ is defined on $D(L_1^{\nu_1})\otimes D(L_2^{\nu_2})$ by
\begin{equation}\label{def:L1xL2}
\big(L_1^{\nu_1}\otimes L_2^{\nu_2}\big)\Big(\sum_{j=1}^N f_j\otimes g_j\Big)
:= \sum_{j=1}^N L_1^{\nu_1}f_j\otimes L_2^{\nu_2}g_j
\end{equation}
for any
$\sum_{j=1}^N f_j\otimes g_j \in D(L_1^{\nu_1})\otimes D(L_2^{\nu_2})$, $N<\infty$.
As is well-known $L_1^{\nu_1}\otimes L_2^{\nu_2}$ is well defined
and the following claim is valid, see e.g. \cite[Section VIII.10]{RS}.

\begin{proposition}\label{prop:L1xL2}
For any $\nu_1, \nu_2 \in \bN_0$ the operator $L_1^{\nu_1}\otimes L_2^{\nu_2}$ is essentially self-adjoint
on $D(L_1^{\nu_1})\otimes D(L_2^{\nu_2})$.
\end{proposition}

We will denote the closure of $L_1^{\nu_1}\otimes L_2^{\nu_2}$ again by $L_1^{\nu_1}\otimes L_2^{\nu_2}$.

\subsection{Two-parameter spectral measure induced by $L_1$ and $L_2$}\label{subsec:2-parameter}

We will be dealing with operators $F(L_1,L_2)$ on $L^2(\XX, d\mu)$
generated by functions $F(\lambda_1, \lambda_2)$.

\smallskip

\noindent
{\bf The space \boldmath $\sL^2(\XX, d\mu)$.}
Denote by
\begin{equation*}
\sL^2(\XX, d\mu)=\sL^2(\XX_1\times \XX_2, d\mu_1\otimes d\mu_2)
\end{equation*}
the subspace of $L^2(\XX, d\mu)$ consisting of all functions of the form
\begin{equation*}
\sum_{j=1}^N f_j\otimes g_j,
\quad f_j\in L^2(\XX_1, d\mu_1),\; g_j\in L^2(\XX_2, d\mu_2),
\end{equation*}
for some $N<\infty$ that may vary.

In what follows we will denote by $\langle \cdot, \cdot\rangle$ and $\|\cdot\|$
the inner product and norm in $L^2(\XX, d\mu)$ or $L^2(\XX_i, d\mu_i)$.
It will be clear from the context which $L^2$-space is considered.

For $i=1,2$, let $E_i$ be the spectral measure induced by the self-adjoint operator $L_i$.
The operator $E_i$ is an orthogonal projection operator in $L^2(\XX_i, d\mu_i)$.
More precisely, if we denote by $\BB(\R)$ the Borel $\sigma$-algebra on $\R$,
then $E_i(S)$ is an orthogonal projector for any Borel set $S\in \BB(\R)$
obeying the following conditions:

(a) $E_i(\R)=I$, the identity operator, and

(b) for any sequence $\{S_n\}_{n\ge 1}$ of disjoint sets in $\BB(\R)$ one has
\begin{equation}\label{additive}
E_i\big(\cup_{n\ge 1}S_n\big) = \sum_{n\ge 1}E_i(S_n),
\end{equation}
where the convergence is in the strong sense, i.e.
\begin{equation*}
\Big\|E_i\big(\cup_{n\ge 1}S_n\big)f - \sum_{n=1}^NE_i(S_n)f\Big\| \to 0
\quad\hbox{as $\;\;N\to\infty$}, \quad \forall f\in L^2(\XX_i, d\mu_i).
\end{equation*}

We next define the product spectral measure $E= E_1\otimes E_2$
on the Borel $\sigma$-algebra $\BB(\R^2)$ on $\R^2$ in three steps.

{\bf Step 1.} Denote
\begin{equation}\label{def:J}
\cJ(\R):=\{(a,b]: a<b\}, \quad -\infty\le a<b \le \infty.
\end{equation}
For any $J_1, J_2\in\cJ(\R)$ we define the operator
\begin{equation}\label{E0-J1-J2}
E_0(J_1\times J_2) = E_1(J_1)\otimes E_2(J_2)
\quad\hbox{on}\quad \sL^2(\XX,d\mu)
\end{equation}
as follows:
For any
$\sum_{j=1}^N f_j\otimes g_j$ in $\sL^2(\XX,d\mu)$ we set
\begin{equation}\label{def:E0}
E_0(J_1\times J_2)\Big(\sum_{j=1}^N f_j\otimes g_j\Big)
:= \sum_{j=1}^N E_1(J_1)f_j\otimes  E_2(J_2)g_j.
\end{equation}

\begin{lemma}\label{lem:E0}
For any $J_1, J_2\in\cJ(\R)$ the linear operator $E_0(J_1\times J_2)$ is bounded
as an operator from $\sL^2(\XX, d\mu)$ into $L^2(\XX, d\mu)$
and $\|E_0(J_1\times J_2)\|=1$.
\end{lemma}

\begin{proof}
The argument here is quite standard.
We give it for completeness.

Let $\{u_j\}_{j\ge 1}$ and $\{v_j\}_{j\ge 1}$ be orthonormal bases
for the (separable Hilbert) spaces $L^2(\XX_1, d\mu_1)$ and $L^2(\XX_2, d\mu_2)$, respectively.
As is well known, e.g. \cite[page 51]{RS},
$\{u_j\otimes v_k\}_{j,k\ge 1}$ is an orthonormal basis for
the space $L^2(\XX, d\mu)$.
Therefore, the set of all finite sums $\sum_{j,k} c_{jk}u_j\otimes v_k$
that is a subspace of $\sL^2(\XX, d\mu)$ is dense in $L^2(\XX, d\mu)$.
For any such sum we have
\begin{align*}
&\Big\|(E_1(J_1)\otimes I)\sum_{j,k} c_{jk}u_j\otimes v_k\Big\|^2
= \Big\|\sum_k \Big(E_1(J_1)\sum_j c_{jk}u_j\Big)\otimes v_k\Big\|^2
\\
& = \sum_k \Big\|E_1(J_1)\sum_j c_{jk}u_j\Big\|^2 \le \sum_k \Big\|\sum_j c_{jk}u_j\Big\|^2
= \sum_{j,k} |c_{jk}|^2
= \Big\|\sum_{j,k} c_{jk}u_j\otimes v_k\Big\|^2
\end{align*}
and hence
$\|E_1(J_1)\otimes I\| \le 1$,
where $I$ is the identity operator in $L^2(X_2, d\mu_2)$.
Above we used that
$\|E_1(J_1)\|=1$.
Similarly as above we get $\|I\otimes E_2(J_2)\| \le 1$.
Evidently,
\begin{equation*}
E_0(J_1\times J_2)=(E_1(J_1)\otimes I)(I\otimes E_2(J_2))
\end{equation*}
and from above it follows that
\begin{equation}\label{E0-E1-E2}
\|E_0(J_1\times J_2)\| \le \|E_1(J_1)\otimes I\|\|I\otimes E_2(J_2)\| \le 1.
\end{equation}

For the inequality in the other direction, choose $f_i\in L^2(\XX_i, d\mu_i)$, $i=1,2$, so that
$\|f_i\|_{L^2}=1$ and
$E_i(J_i)f_i=f_i$.
Then
\begin{equation*}
E_0(J_1\times J_2) (f_1\otimes f_2) = E_1(J_1)f_1 \otimes E_2(J_2)f_2 = f_1\otimes f_2
\end{equation*}
and hence
$\|E_0(J_1\times J_2)\| \ge 1$.
This along with \eqref{E0-E1-E2} completes the proof.
\end{proof}

Because $\sL^2(\XX, d\mu)$ is dense in $L^2(\XX, d\mu)$
and for any $J_1, J_2\in\cJ(\R)$ the linear operator $E_0(J_1\times J_2)$ is bounded ($\|E_0(J_1\times J_2)\|=1$)
there exists a unique extension of $E_0(J_1\times J_2)$ to $L^2(\XX, d\mu)$.
We denote this extension by $\overline{E}_0(J_1\times J_2)$.

\begin{proposition}\label{prop:E0}
For any $J_1, J_2\in\cJ(\R)$ the linear operator $\overline{E}_0(J_1\times J_2)$ is self-adjoint and bounded
on $L^2(\XX, d\mu)$ with norm $\|\overline{E}_0(J_1\times J_2)\|=1$.
Moreover, the operator $\overline{E}_0(J_1\times J_2)$ is idempotent,
i.e. $\overline{E}_0^2(J_1\times J_2) = \overline{E}_0(J_1\times J_2)$.
Therefore, $\overline{E}_0(J_1\times J_2)$ is a projector.
In addition, for $J_1, J_2, I_1, I_2\in\cJ(\R)$
\begin{equation}\label{E0E0}
\overline{E}_0(J_1\times J_2)\overline{E}_0(I_1\times I_2)=0
\quad\hbox{if}\quad
(J_1\times J_2)\cap(I_1\times I_2) =\emptyset.
\end{equation}
\end{proposition}

\begin{proof}
Fix $J_1, J_2\in\cJ(\R)$.
Since $E_1(J_1)$ is a projector there exists a closed subspace $H_1$ of $L^2(\XX_1, d\mu_1)$
such that $E_1(J_1)$ projects onto $H_1$.
We will use the notation $E_{H_1}:=E_1(J_1)$.
Clearly, there exists an orthonormal basis $\{u_j\}_{j\ge 1}$ for the space $L^2(\XX_1, d\mu_1)$
such that for each $j\in \bN$ we have either
$E_{H_1}u_j = u_j$
or
$E_{H_1}u_j = 0$.
Thus, $\{u_j\}_{j\ge 1}$ is the union of orthonormal bases
for the subspaces $H_1$ and $H_1^\perp$.
Denote by $\bN_1$ the set of all $j\in\bN$ such that $E_{H_1}u_j = u_j$,
i.e. $\{u_j\}_{j\in\bN_1}$ is a basis for the subspace $H_1$.

Similarly, there exists a closed subspace $H_2$ of $L^2(\XX_2, d\mu_2)$
such that $E_2(J_2)$ projects onto $H_2$.
Denote $E_{H_2}:=E_2(J_2)$.
Also, assume that $\{v_j\}_{j\ge 1}$ is an orthonormal basis
for the spaces $L^2(\XX_2, d\mu_2)$
such that for each $j\in \bN$ we have either
$E_{H_2}v_j = v_j$
or
$E_{H_2}v_j = 0$.
Denote by $\bN_2$ the set of all $j\in\bN$ such that $E_{H_2}v_j = v_j$.

As was alluded to above $\{u_j\otimes v_k\}_{j,k\ge 1}$ is an orthonormal basis for $L^2(\XX, d\mu)$.
Then for any function $h\in L^2(\XX, d\mu)$ we have
\begin{equation*}
h = \sum_{j, k \ge 1} c_{jk} u_j\otimes v_k,
\quad c_{jk}= \langle h, u_j\otimes v_k \rangle.
\end{equation*}
With the notation from the proof of Lemma~\ref{lem:E0} we obtain
\begin{align*}
\overline{E}_0(J_1\times J_2)h
&= \lim_{N,K\to\infty}E_0(J_1\times J_2)\Big(\sum_{j=1}^N\sum_{k=1}^K c_{jk} u_j\otimes v_k\Big)
\\
&= \lim_{N,K\to\infty}\sum_{j=1}^N\sum_{k=1}^K c_{jk} E_1(J_1)u_j\otimes E_2(J_2)v_k
\\
&= \sum_{j,k\ge 1} c_{jk} E_{H_1}u_j\otimes E_{H_2}v_k
= \sum_{j\in\bN_1, k\in\bN_2} c_{jk} E_{H_1}u_j\otimes E_{H_2}v_k
\\
&= \sum_{j\in\bN_1, k\in\bN_2} c_{jk} u_j\otimes v_k,
\end{align*}
where the convergence is in $L^2(\XX, d\mu)$; in the last three series the convergence is unconditional.
From above it follows that the operator $\overline{E}_0(J_1\times J_2)$ is symmetric
and
$\|\overline{E}_0(J_1\times J_2)h\|^2 = \sum_{j\in\bN_1, k\in\bN_2} |c_{jk}|^2 \le \|h\|^2$.
Hence $\overline{E}_0(J_1\times J_2)$ is self-adjoint.
Also, it follows readily that
$\overline{E}_0^2(J_1\times J_2)=\overline{E}_0(J_1\times J_2)$
and identity \eqref{E0E0} is valid.
\end{proof}

\smallskip

{\bf Step 2.}
Let $\BB_0(\R^2)$ be the collection of all sets $R$ of the form
\begin{equation}\label{R-JJ}
R = \cup_{j=1}^N J_{1j}\times J_{2j},
\quad J_{ij}\in \cJ(\R), \;  (J_{1j}\times J_{2j})\cap (J_{1k}\times J_{2k}) = \emptyset, j\ne k,
\end{equation}
for some $N<\infty$.
It is readily seen that $\BB_0(\R^2)$ is an algebra (it is closed under finite unions and complements)
that generates the Borel $\sigma$-algebra $\BB(\R^2)$ on $\R^2$.

We now extend the definition of $\overline{E}_0(R)$ to all sets $R$ from $\BB_0(\R^2)$.
For any set $R$ as in \eqref{R-JJ} we define
\begin{equation}\label{def:E0bar}
\overline{E}_0(R) := \sum_{j=1}^N \overline{E}_0(J_{1j}\times J_{2j})
\quad\hbox{on}\quad L^2(\XX,d\mu).
\end{equation}

\begin{proposition}\label{prop:prealgebra}
The operator $\overline{E}_0(R)$ is well defined.
Furthermore, $\overline{E}_0(R)$ is a projection-valued set function
that is a {\em premeasure} in $\BB_0(\R^2)$,
i.e.
$\overline{E}_0(\RR^2) = I$, the identity operator,
and
\begin{equation}\label{additive-B0}
\overline{E}_0\big(\cup_{j\ge 1}R_j\big) = \sum_{j\ge 1}\overline{E}_0(R_j)
\end{equation}
for any sequence $\{R_j\}_{j\ge 1}$ of disjoint sets from $\BB_0(\R^2)$
such that
$\cup_{j\ge 1}R_j \in \BB_0(\R^2)$.
\end{proposition}

To streamline our presentation we divert the lengthy proof of this proposition to the appendix.

\smallskip
 
{\bf Step 3.}
The following theorem completes the construction of the two-parameter spectral measure $E$.

\begin{theorem}\label{thm:salgebra}
There is a unique spectral measure $E(R)$ on the Borel $\sigma$-algebra $\BB(\R^2)$
such that
\begin{equation}\label{E-E0}
E(R) = \overline{E}_0(R), \quad \forall R \in \BB_0(\R^2).
\end{equation}
Recall that by definition $E(S)$ is a spectral measure on $\BB(\R^2)$
if it obeys the following conditions:
$E(\RR^2) = I$ 
and
\begin{equation}\label{additive-B}
E\big(\cup_{j\ge 1}S_j\big) = \sum_{j\ge 1}E(S_j)
\end{equation}
for any sequence $\{S_j\}_{j\ge 1}$ of disjoint sets from $\BB(\R^2)$.

\end{theorem}

This theorem is standard and follows from Proposition~\ref{prop:prealgebra},
see e.g. \cite[Lemma 4.9]{Schmudgen}.

From Theorem~\ref{thm:salgebra} it readily follows that
\begin{equation}\label{prop-E}
E(\emptyset)=0
\quad\hbox{and}\quad
E(S)E(R)= E(S\cap R), \; \forall S, R \in \BB(\R^2).
\end{equation}
Also, as a projector operator $E(S)$ for $S\in \BB(\R^2)$ is a self-adjoint bounded operator on $L^2(\XX, d\mu)$
with $\|E(S)\|=1$.

\subsection{Definition of operators $F(L_1, L_2)$}

Let $F:\R^2 \to \bC$ be a measurable and bounded function.
We next define the operator $F(L_1, L_2)$.

Let $E$ be the two-parameter spectral measure defined above.
As is well known and easy to see for any $f\in L^2(\XX, d\mu)$
\begin{equation}\label{def-nu}
\nu_f(S) := \langle E(S)f, f\rangle, \quad S\in \BB(\R^2).
\end{equation}
is a finite measure on $\BB(\R^2)$.
Clearly,
\begin{equation}\label{nu-f-S}
\nu_f(S) = \langle E(S)^2f,f\rangle = \langle E(S)f, E(S)f\rangle = \|E(S)f\|^2
\quad\hbox{and}\quad
\nu_f(\R^2)=\|f\|^2.
\end{equation}
From \eqref{def-nu} it follows by polarization that for any $f, g\in L^2(\XX, d\mu)$
\begin{equation}\label{def-nu-fg}
\nu_{f,g}(S) := \langle E(S)f,g\rangle, \quad S\in \BB(\R^2),
\end{equation}
is a linear combination of four positive measures on $\BB(\R^2)$.
Therefore,
\begin{equation}\label{def-form}
(f, g):= \int_{\R^2} F d\nu_{f,g}, \quad f, g\in L^2(\XX, d\mu),
\end{equation}
is a bilinear form on $L^2(\XX, d\mu)$.

We next show that the bilinear form $(f, g)$ is bounded.
By a standard construction in Measure theory (see e.g. \cite{Folland}) there exist a sequence $\{\phi_k\}_{k\ge 1}$
of simple functions
\begin{equation*}
\phi_k = \sum_{j=1}^{N_k} c_{kj} \ONE_{S_{kj}}, \quad S_{kj}\cap S_{k\ell} = \emptyset, \; j\ne \ell,
\end{equation*}
such that
\begin{equation*}
|\phi_1| \le |\phi_2| \le \cdots \le |F| \le \|F\|_\infty
\end{equation*}
and for any $f, g\in L^2(\XX, d\mu)$
\begin{equation*}
\lim_{k\to\infty} \int_{\R^2} \phi_k(\blambda) d\nu_{f, g}(\blambda)
= \int_{\R^2} F(\blambda) d\nu_{f, g}(\blambda).
\end{equation*}
We have
\begin{align*}
\Big|\int_{\R^2}& \phi_k(\blambda) d\nu_{f, g}(\blambda)\Big|
=  \Big|\sum_{j=1}^{N_k} c_{kj} \nu_{f,g}(S_{kj})\Big|
\le \sum_{j=1}^{N_k} |c_{kj}||\langle E(S_{kj})f, g \rangle|
\\
&= \sum_{j=1}^{N_k} |c_{kj}||\langle E(S_{kj})f, E(S_{kj})g \rangle|
\le \|F\|_\infty\sum_{j=1}^{N_k} \|E(S_{kj})f\| \|E(S_{kj})g \|
\\
&\le \|F\|_\infty\Big(\sum_{j=1}^{N_k} \|E(S_{kj})f\|^2\Big)^{1/2}\Big(\sum_{j=1}^{N_k} \|E(S_{kj})g\|^2\Big)^{1/2}
\\
&= \|F\|_\infty\Big(\sum_{j=1}^{N_k} \langle E(S_{kj})f, f\rangle \Big)^{1/2}
\Big(\sum_{j=1}^{N_k} \langle E(S_{kj})g, g\rangle \Big)^{1/2}
\\
&= \|F\|_\infty\Big|\Big\langle \Big(\sum_{j=1}^{N_k} E(S_{kj})\Big)f, f\Big\rangle \Big|^{1/2}
\Big|\Big\langle \Big(\sum_{j=1}^{N_k} E(S_{kj})\Big)g, g\Big\rangle \Big|^{1/2}
\le \|F\|_\infty\|f\|\|g\|.
\end{align*}
Above we used twice the Cauchy-Schwarz inequality and the fact that
$E(S_{kj})$ and $\sum_{j=1}^{N_k} E(S_{kj})$ are projectors.
From above it follows that
$|(f, g)| \le \|F\|_\infty\|f\|\|g\|$.
Therefore, by the Riesz theorem there exists a bounded linear operator, denoted by $F(L_1, L_2)$, on $L^2(\XX, d\mu)$
such that
\begin{equation*}
\langle F(L_1, L_2)f, g \rangle = (f,g).
\end{equation*}
Thus, we define $F(L_1, L_2)$ by
\begin{equation}\label{def-FL1L2}
\langle F(L_1, L_2)f, g \rangle
:= \int_{\R^2} F d \langle Ef,g\rangle, \quad \forall f, g\in L^2(\XX, d\mu).
\end{equation}
Clearly,
$\|F(L_1, L_2)\|\le \|F\|_\infty$.

\begin{proposition}\label{prop:self-adjoint}
If $F(\lambda_1, \lambda_2)$ is a real-valued bounded Borel measurable function on $\R^2$,
then the operator $F(L_1, L_2)$ is self-adjoint.
\end{proposition}
\begin{proof}
If $f, g\in L^2(\XX, d\mu)$, then from \eqref{def-FL1L2} we get
\begin{align*}
\langle F(L_1, L_2)f, g \rangle
&= \int_{\R^2} F d \langle Ef,g\rangle
= \int_{\R^2} F d \langle f,Eg\rangle
= \int_{\R^2} F d \overline{\langle Eg, f\rangle}
\\
&= \overline{\int_{\R^2} F d \langle Eg, f\rangle}
= \overline{\langle F(L_1, L_2)g, f \rangle}
= \langle f, F(L_1, L_2)g\rangle,
\end{align*}
where we used that the operator $E(S)$ is self-adjoint, and $\overline{z}$ stands for the complex conjugate of $z\in \bC$.
Therefore, the operator $F(L_1, L_2)$ is symmetric and bounded and hence it is self-adjoint.
\end{proof}

Some properties of the operators $F(L_1, L_2)$ are given in the following

\begin{theorem}\label{thm:F-G}
Let $F(\lambda_1, \lambda_2)$ and $G(\lambda_1, \lambda_2)$ be two bounded Borel measurable functions on $\R^2$.
Then for any $a,b\in \bC$
\begin{equation}\label{aF-bG}
(aF+ bG)(L_1, L_2) = aF(L_1, L_2) + bG(L_1, L_2)
\end{equation}
and
\begin{equation}\label{F-G}
(FG)(L_1,L_2) = F(L_1,L_2)G(L_1,L_2),
\end{equation}
and, therefore,
$F(L_1,L_2)$ and $G(L_1,L_2)$ commute.
\end{theorem}

For the proof of this theorem we need the following well known

\begin{lemma}\label{lem:mu-1}
Let $(\XX,d\mu)$ be a measure space.
Assume that $f$ and $g$ are $\mu$-measurable complex-valued functions on $\XX$ and
$fg$ and $g$ are $\mu$-integrable.
Let the complex measure $\mu_1$ be defined by
\begin{equation*}
\mu_1(S):=\int_S g(\blambda)d\mu(\blambda)
\quad\hbox{for any measurable set} \;\; S\subset \XX.
\end{equation*}
Then $f$ is $\mu_1$-integrable and
\begin{equation*}
\int_\XX f(\blambda)d\mu_1(\blambda)=\int_\XX f(\blambda)g(\blambda)d\mu(\blambda).
\end{equation*}
\end{lemma}

\begin{proof}[Proof of Theorem~\ref{thm:F-G}]
Identity \eqref{aF-bG} follows readily by the definition of $F(L_1,L_2)$ and $G(L_1,L_2)$.

We now prove \eqref{F-G}.
Let $f, g\in L^2(\XX, d\mu)$. Then using \eqref{def-FL1L2} we have
\begin{equation*}
\langle F(L_1, L_2)G(L_1, L_2) f, g\rangle
= \int_{\R^2} F d \langle E G(L_1,L_2)f, g\rangle.
\end{equation*}
Using again \eqref{def-FL1L2} we have for any $S\in \BB(\R^2)$
\begin{equation}\label{EG}
\langle E(S) G(L_1,L_2)f, g\rangle
= \langle G(L_1,L_2)f, E(S)g\rangle
= \int_{\R^2} G d\langle E f, E(S)g\rangle.
\end{equation}
But, for each $R\in \BB(\R^2)$
\begin{equation*}
\langle E(R) f, E(S)g\rangle = \langle E(S)E(R)f, g\rangle = \langle E(S\cap R)f, g\rangle,
\end{equation*}
where we used \eqref{prop-E}.
Hence the measure $\nu_{f,g}:=\langle E(R) f, E(S)g\rangle$, $R\in \BB(\R^2)$, where $S\in \BB(\R^2)$ is fixed,
vanishes on sets outside $S$.
From this and \eqref{EG} it follows that
\begin{align*}
\langle E(S) G(L_1,L_2)f, g\rangle
= \int_{\R^2} G d\langle E f, E(S)g\rangle
= \int_S G d \langle E f, g\rangle.
\end{align*}
Now, appealing to Lemma~\ref{lem:mu-1} we get
\begin{align*}
\int_{\R^2} F d \langle E G(L_1,L_2)f, g\rangle
= \int_{\R^2} FG d \langle E f, g\rangle
= \langle (FG)(L_1,L_2)f,g\rangle,
\end{align*}
and \eqref{F-G} follows.
\end{proof}

\subsection{Spectral resolution (decomposition of the identity)}\label{subsec:spectr-func}

It is sometimes more convenient to work with a spectral resolution rather then a spectral measure.
For any $\blambda=(\lambda_1, \lambda_2)\in\R^2$ we define
the {\em spectral function} $E_\blambda$ by
\begin{equation}\label{def-E-lam}
E_\blambda=E_{(\lambda_1, \lambda_2)} :=E(I_{\lambda_1, \lambda_2}),
\quad
I_{\lambda_1, \lambda_2}:=(-\infty,\lambda_1]\times(-\infty,\lambda_2].
\end{equation}

It is easy to see that the projection-valued function $E_\blambda=E_{(\lambda_1,\lambda_2)}$
has the following properties:

$(i)$ ${\rm Range}\,(E_\blambda) \subset {\rm Range}\,(E_{\blambda'})\quad\hbox{if}\quad \blambda \le \blambda'$,
where for $\blambda=(\lambda_1,\lambda_2)$, $\blambda'=(\lambda'_1,\lambda'_2)$
$\blambda\le \blambda'$ means that $\lambda_1\le \lambda_1'$ and $\lambda_2\le \lambda_2'$.

$(ii)$ $E_\blambda = \lim_{\blambda'\to\blambda, \blambda'\ge \blambda}E_{\blambda'}$,

$(iii)$ $E_{-\infty}= \lim_{\blambda\to -\infty} E_\blambda =0$ and
$E_{\infty}= \lim_{\blambda\to \infty} E_\blambda = I$.

Above $\blambda\to -\infty$ stands for $\lambda_1\to -\infty$ and $\lambda_2\to -\infty$,
and the convergence is in the strong sense.

As is well known properties (i) - (iii) completely characterise the spectral measure $E(S)$,
see e.g. \cite[Theorem 5.6]{Prugov}.

Using the notation from above it follows from \eqref{nu-f-S} and \eqref{def-nu-fg}
that for any $f,g\in L^2(\XX, d\mu)$
\begin{equation*}
\int_{\RR^2}1\, d\|E_{(\lambda_1,\lambda_2)}f\|^2 = \|f\|^2
\quad\hbox{and}\quad
\int_{\RR^2} 1\, d\langle E_{(\lambda_1, \lambda_2)}f, g \rangle = \langle f, g \rangle.
\end{equation*}
Furthermore, from \eqref{def-FL1L2} it follows that for any bounded Borel measurable function
$F:\RR^2\to \bC$
\begin{equation}\label{F-L1-L2}
\langle F(L_1, L_2)f, g \rangle
= \int_{\R^2} F(\lambda_1, \lambda_2) d \langle E_{(\lambda_1, \lambda_2)}f,g\rangle, \quad \forall f, g\in L^2(\XX, d\mu).
\end{equation}

\subsection{Tensor-product integral operators}\label{subsec:tensor-product}

Let $F_1$ and $F_2$ be two bounded Borel measurable functions on $\R$.
Set $F(\lambda_1, \lambda_2):= F_1(\lambda_1) F_2(\lambda_2)$.
We define the operator $F_1(L_1)\otimes F_2(L_2)$ by
\begin{equation}\label{def:F1-F2}
F_1(L_1)\otimes F_2(L_2) : = F(L_1, L_2).
\end{equation}

It is easy to see that if $f_i\in L^2(\XX_i, d\mu_i)$, $i=1,2$,
then with $F_1$ and $F_2$ as above we have
\begin{equation}\label{Fi-fi}
\big[F_1(L_1)\otimes F_2(L_2)\big](f_1\otimes f_2) = \big[F_1(L_1)f_1\big]\otimes \big[F_2(L_2)f_2\big]
\quad\hbox{(see \eqref{def-FL1L2})}.
\end{equation}

Tensor-product integral operators will play an important role.
Let the operators $F_i(L_i)$, $i=1,2$, be as above
and assume that $F_i(L_i)$ is an integral operator of the form
\begin{equation}\label{Int-oper-Fi}
F_i(L_i)f(x_i):=\int_{\XX_i}\KK_{F_i(L_i)}(x_i,y_i)f(y_i) d\mu_i(y_i),
\end{equation}
where
\begin{equation}\label{Kern-Fi}
|\KK_{F_i(L_i)}(x_i,y_i)|
\le c\big[V_i(x_i,\delta_i) V(y_i,\delta_i)\big]^{-1/2}
\big(1+\delta_i^{-1}\rho_i(x_i,y_i)\big)^{-\sigma_i}
\end{equation}
with $\delta_i>0$ and $\sigma_i > 2d_i$.
Then it follows from \eqref{changecenter} and \eqref{int-ineq-1} that
\begin{equation*}
\sup_{x_i\in \XX_i}\|\KK_{F_i(L_i)}(x_i,\cdot)\|_{L^1} \le c
\quad\hbox{and}\quad
\sup_{y_i\in \XX_i}\|\KK_{F_i(L_i)}(\cdot, y_i)\|_{L^1} \le c.
\end{equation*}
Now, by the Schur lemma or interpolation it follows that
\begin{equation}\label{oper-H}
\|F_i(L_i)f\|_p \le c'\|f\|_p, \quad \forall f\in L^p(\XX_i, d\mu_i), \; 1\le p\le \infty.
\end{equation}

\begin{proposition}\label{prop:prod-ker}
Under the assumptions from above
$F_1(L_1)\otimes F_2(L_2)$
is an integral operator with kernel
\begin{equation}\label{prod-ker}
\KK_{F_1(L_1)\otimes F_2(L_2)}(\xx,\yy) = \KK_{F_1(L_1)}(x_1,y_1)\KK_{F_2(L_2)}(x_2,y_2)
\end{equation}
and hence
\begin{equation}\label{est-kern}
|\KK_{F_1(L_1)\otimes F_2(L_2)}(\xx,\yy)|
\le c\frac{\prod_{i=1,2} \big(1-\frac{\rho_i(x_i,y_i)}{\delta_i}\big)^{-\sigma_i}}
{\big[V(\xx,\bdelta) V(\yy,\bdelta)\big]^{1/2}}
= c \DD_{\bdelta, \bsigma}(\xx,\yy).
\end{equation}
\end{proposition}

\begin{proof}
If $f_i\in L^2(\XX_i, d\mu_i)$, $i=1,2$, then by \eqref{Fi-fi}
\begin{align*}
\big[F_1(L_1)&\otimes F_2(L_2)\big](f_1\otimes f_2)(\xx)
= \big[F_1(L_1)f_1(x_1)\big]\big[F_2(L_2)f_2(x_2)\big]
\\
&= \int_{\XX_1}\KK_{F_1(L_1)}(x_1,y_1)f_1(y_1) d\mu_1(y_1)
\int_{\XX_2}\KK_{F_2(L_2)}(x_2,y_2)f_2(y_2) d\mu_1(y_2)
\\
&= \int_{\XX_1\times \XX_2}\KK_{F_1(L_1)}(x_1,y_1)
\KK_{F_2(L_2)}(x_2,y_2)(f_1\otimes f_2)(y_1,y_2) d\mu(y_1,y_2),
\end{align*}
where the last equality is easily justified applying Fubini's theorem.
From above it readily follows that for any function
$\sum_{j=1}^N f_j\otimes g_j$ in $\sL^2(\XX,d\mu)$ we have
\begin{align*}
&\big[F_1(L_1)\otimes F_2(L_2)\big]\Big(\sum_{j=1}^N f_j\otimes g_j\Big)(\xx)
\\
&= \int_{\XX_1\times \XX_2}\KK_{F_1(L_1)}(x_1,y_1)
\KK_{F_2(L_2)}(x_2,y_2)\Big(\sum_{j=1}^N f_j\otimes g_j\Big)(y_1,y_2) d\mu(y_1,y_2).
\end{align*}
But, as we know $\sL^2(\XX,d\mu)$ is dense in $L^2(\XX,d\mu)$ and by a limiting argument it follows that
\begin{equation*}
\big[F_1(L_1)\otimes F_2(L_2)\big]f(\xx)
= \int_{\XX_1\times \XX_2}\KK_{F_1(L_1)}(x_1,y_1)
\KK_{F_2(L_2)}(x_2,y_2)f(\yy) d\mu(\yy)
\end{equation*}
for all $f\in L^2(\XX, d\mu)$,
which implies \eqref{prod-ker}.
Estimate \eqref{est-kern} follows at once from \eqref{prod-ker} and \eqref{Kern-Fi}.
\end{proof}

\smallskip

\noindent
{\bf Assumption:}
Henceforth we will assume that $L_1$ and $L_2$ are the two {\em non-negative} self-adjoint operators
from our general setting described in Section~\ref{sec:genset}.
The fact that the operators $L_1$ and $L_2$ are non-negative implies that all statements
in the two-parameter calculus presented above
are valid with $\BB(\R^2)$ replaced by $\BB([0, \infty)^2)$. 

\smallskip

\noindent
{\bf Product of the semigroups \boldmath $e^{-t_1L_1}\otimes e^{-t_2L_2}$.}
The product
\begin{equation*}
P_{1,t_1}\otimes P_{2,t_2}:=e^{-t_1L_1}\otimes e^{-t_2L_2}, \quad t_1, t_2 \ge 0,
\end{equation*}
will play an important role in what follows.
In light of Proposition~\ref{prop:prod-ker} the operator $e^{-t_1L_1}\otimes e^{-t_2L_2}$ is an integral operator.
We will denote its kernel by $p_\tt(\xx,\yy) = p_{(t_1,t_2)}(\xx,\yy)$.
Applying Proposition~\ref{prop:prod-ker} we get
\begin{equation}\label{prodhk}
p_{(t_1,t_2)}(\xx,\yy)=p_{1,t_1}(x_1,y_1)p_{2,t_2}(x_2,y_2), \quad \xx,\yy\in \XX.
\end{equation}

Observe that the above theory works also for complex valued functions $F_1, F_2$
and, in particular, $e^{-z_1L_1}\otimes e^{-z_2L_2}$, $z_1, z_2\in \bC_+:=\{z\in \bC: \Real z >0\}$, is an integral operator
with kernel $p_{(z_1,z_2)}(\xx,\yy)$ represented in the form
\begin{equation}\label{prodhk-z}
p_{(z_1,z_2)}(\xx,\yy)=p_{1,z_1}(x_1,y_1)p_{2,z_2}(x_2,y_2), \quad \xx,\yy\in \XX, \; z_1, z_2\in \bC_+.
\end{equation}

Identity \eqref{prodhk} leads to the following

\begin{proposition}\label{prop:Gauss-local}
For any $\bsigma=(\sigma_1,\sigma_2)>\zero$ and $\tt=(t_1,t_2)>\zero$ we have:

$(a)$ {\em Gaussian upper bound:}
\begin{equation}\label{prodGauss-local}
|p_{(t_1, t_2)}(\xx,\yy)|
\le  c\frac{\prod_{i=1,2} \exp\big(-\frac{\rho_i^2(x_i,y_i)}{c_{i2}t_i}\big)}
{\big[V(\xx,\sqrt{\tt}) V(\yy,\sqrt{\tt})\big]^{1/2}}
\le c \DD_{\sqrt{\tt}, \bsigma}(\xx,\yy),
\quad \xx,\yy\in \XX.
\end{equation}

$(b)$ {\em H\"older continuity:}
If $\xx,\yy\in \XX$ with $\rho_i(y_i,y_i ')\le \sqrt{t_i}$, $i=1,2$, then
\begin{align}\label{prodlip}
&|p_{(t_1,t_2)}(\xx,\yy)-p_{(t_1,t_2)}(\xx,\yy')| \nonumber
\\
&\le c \sum_{i=1,2}\Big(\frac{\rho_i(y_i,y_i ')}{\sqrt{t_i}}\Big)^{\alpha_i}
\frac{\prod_{j=1,2}\exp\big(-\frac{\rho_j^2(x_j,y_j)}{2c_{j2} t_j} \big)}
{\big[ V(\xx,\sqrt{\tt}) V(\yy,\sqrt{\tt})\big]^{1/2}}
\\
&\le c \sum_{i=1,2}\Big(\frac{\rho_i(y_i,y_i ')}{\sqrt{t_i}}\Big)^{\alpha_i}\DD_{\sqrt{\tt}, \bsigma}(\xx,\yy). \nonumber
\end{align}

Above $\sqrt{\tt}:= (\sqrt{t_1}, \sqrt{t_2})$ and
$c_{12},c_{22}$ are the structural constants from $(\ref{Gauss-local})$ and $\eqref{lip}$
and the constant $c>0$ depends only on the constants $c_{i0}, c_{i1}, c_{i2}$
from $\eqref{eq:doubling-0},\eqref{Gauss-local}, \eqref{lip}$.
\end{proposition}

\begin{proof}
Inequality (\ref{prodGauss-local}) follows readily from (\ref{Gauss-local}) and \eqref{prodhk}
taking into account that for each $j=1,2$ and $\sigma_j>0$
\begin{equation}\label{exp-power}
\exp\Big(-\frac{c \rho_j^2(x_j,y_j)}{t_j}\Big)\le c'\biggl(1+\frac{\rho_j(x_j,y_j)}{\sqrt{t_j}}\biggr)^{-\sigma_j}.
\end{equation}

We next establish estimate (\ref{prodlip}).
Assume that $\rho_i(y_i,y_i ')\le \sqrt{t_i}$, $i=1,2$.
Then using \eqref{lip} and \eqref{changecenter} we obtain
\begin{align*}
|p_{(t_1,t_2)}&(\xx,\yy)-p_{(t_1,t_2)}(\xx,\yy')|
\\
&=|p_{1,t_1}(x_1,y_1)p_{2,t_2}(x_2,y_2)-p_{1,t_1}(x_1,y_1')p_{2,t_2}(x_2,y_2')|
\\
&\le |p_{1,t_1}(x_1,y_1)-p_{1,t_1}(x_1,y_1')||p_{2,t_2}(x_2,y_2)|
\\
& \qquad\hspace{1.5in} +|p_{1,t_1}(x_1,y_1')||p_{2,t_2}(x_2,y_2)-p_{2,t_2}(x_2,y_2')|
\\
&\le c  \Big(\frac{\rho_1(y_1,y_1')}{\sqrt{t_1}}\Big)^{\alpha_1}
\frac{\exp\big(-\frac{\rho_1^2(x_1,y_1)}{c_{12}t_1} \big)}
{\big[V_1(x_1,\sqrt{t_1}) V_1(y_1,\sqrt{t_1})\big]^{1/2}}
\frac{\exp\big(-\frac{\rho_2^2(x_2,y_2)}{c_{22}t_2} \big)}
{\big[V_2(x_2,\sqrt{t_2}) V_2(y_2,\sqrt{t_2})\big]^{1/2}}
\\
&+c\Big(\frac{\rho_2(y_2,y_2')}{\sqrt{t_2}}\Big)^{\alpha_2}
\frac{\exp\big(-\frac{\rho_2^2(x_2,y_2)}{c_{22}t_2 }\big)}
{\big[V_2(x_2,\sqrt{t_2}) V_2(y_2,\sqrt{t_2})\big]^{1/2}}
\frac{\exp\big(-\frac{\rho_1^2(x_1,y_1')}{c_{12}t_1} \big)}
{\big[V_1(x_1,\sqrt{t_1}) V_1(y_1',\sqrt{t_1})\big]^{1/2}}
\\
&\le c\sum_{i=1,2}\Big(\frac{\rho_i(y_i,y_i ')}{\sqrt{t_i}}\Big)^{\alpha_i}
\frac{\prod_{j=1,2}\exp\big(-\frac{\rho_j^2(x_j,y_j)}{2c_{j2}t_j} \big)}
{\big[V(\xx,\sqrt{\tt}) V(\yy,\sqrt{\tt})\big]^{1/2}}.
\end{align*}
This coupled with \eqref{exp-power} yields \eqref{prodlip}.
\end{proof}

\subsection{Integral operators}\label{subsec:prod-oper}

We will constantly work with integral operators.
We start with a basic statement on the boundedness on $L^p$ of integral operators.

\begin{proposition}\label{prop:bound-int-oper}
Let $H$ be an integral operator with a measurable kernel $\KK_H(\xx,\yy)$, i.e.
\begin{equation*}
Hf(\xx) = \int_X \KK_H(\xx,\yy)f(\yy) d\mu(\yy)
\end{equation*}
and assume that 
$\KK_H(\xx,\yy)=\KK_H(\yy,\xx)$ and for some $\bdelta>\zero$ and $\bsigma>3\dd/2$
\begin{equation}\label{K-H-D}
|\KK_H(\xx,\yy)| \le c\DD_{\bdelta, \bsigma}(\xx,\yy), \quad \xx,\yy\in X.
\end{equation}
Then we have:

$(i)$ $H$ is a self-adjoint bounded operator on $L^2(X, d\mu)$.

$(ii)$ For $1\le p\le\infty$
\begin{equation}\label{bound-Lp}
\|Hf\|_p \le c_p\|f\|_p,
\quad \forall f\in L^p(X, d\mu).
\end{equation} 
\end{proposition}

\begin{proof}
From \eqref{K-H-D} and \eqref{D-D*} it follows  that
\begin{equation*}
|\KK_H(\xx,\yy)| \le c V(\xx,\bdelta)^{-1}\DD^*_{\bdelta, \bsigma-\dd/2}(\xx,\yy)
\end{equation*} 
and hence using inequality \eqref{tech-1} and that $\bsigma>3\dd/2$ we get
\begin{equation*}
\int_X |\KK_H(\xx,\yy)| d\mu(\yy) = \int_X |\KK_H(\xx,\yy)| d\mu(\xx) \le c<\infty.
\end{equation*}
Therefore, the operator $H$ is bounded from $L^1$ to $L^1$ and from $L^\infty$ to $L^\infty$.
Then \eqref{bound-Lp} follows by interpolation.

Furthermore, since the operator $H$ is symmetric and bounded on $L^2$ it is self-adjoint. 
\end{proof}

The following proposition will be instrumental in dealing with products of integral and non-integral operators.


\begin{proposition}\label{prop:prod-oper}
In the general setting of a doubling metric measure space $(\XX, \rho, \mu)$ as above,
let $U, W: L^2(\XX, d\mu) \to L^2(\XX, d\mu)$ be integral operators
with kernels $\KK_U(\xx, \yy)$ and $\KK_W(\xx, \yy)$.
Assume that for some $\bdelta=(\delta_1, \delta_2)$, $ \delta_i > 0$,
and $\bsigma =(\sigma_1, \sigma_2)$, $\sigma_i \ge d_i+1$, we have
\begin{equation}\label{local-UV}
|\KK_U(\xx,\yy)| \le c_1\DD_{\bdelta, \bsigma}(\xx,\yy)
\quad\hbox{and}\quad
|\KK_W(\xx,\yy)| \le c_1\DD_{\bdelta,\bsigma}(\xx,\yy).
\end{equation}
Let $R: L^2(\XX, d\mu)\to L^2(\XX, d\mu)$ be a bounded operator, not necessarily an integral operator.
Consider the operator
\begin{equation*}
G:= V(\cdot, \bdelta)^{\frac{1}{2}} URW V(\cdot, \bdelta)^{\frac{1}{2}},
\quad\hbox{i.e.}\quad
Gf(\xx):= V(\xx, \bdelta)^{\frac{1}{2}} URW[V(\cdot, \bdelta)^{\frac{1}{2}} f(\cdot)](\xx).
\end{equation*}
Then the operator $G: L^1(\XX, d\mu)\to L^\infty(\XX, d\mu)$ is bounded and hence an integral operator
with kernel $\KK_G(\xx,\yy)$ satisfying
$\|\KK_G\|_\infty = \|G\|_{1\to \infty} \le c\|R\|_{2\to 2}$.

Consequently, $U R  W $ is an integral operator
with kernel $\KK_{U  R  W} (\xx,\yy)$ of the form
\begin{equation}\label{ker-URW}
\KK_{U  R  W} (\xx,\yy)
=  \frac{\KK_G(\xx,\yy)}{\big[V(\xx, \bdelta) V(\yy, \bdelta)\big]^{1/2}}
\end{equation}
and hence
\begin{equation}\label{local-URV}
|\KK_{U  R  W} (\xx,\yy)|
\le  \frac{c\|R \|_{2 \to 2}}{\big[V(\xx, \bdelta) V(\yy, \bdelta)\big]^{1/2}}
\quad\hbox{for a.a. $\xx,\yy\in \XX$}.
\end{equation}
More precisely, the action of the operator $URW$ can be described as follows:
For any $f \in L^1(\XX, V(\cdot, \bdelta)^{-1/2}d\mu)$
\begin{equation}\label{URW-bounded}
\big\|V(\cdot, \bdelta)^{1/2} URWf(\cdot)\big\|_\infty \le c\|f\|_{L^1(\XX, V(\cdot, \bdelta)^{-1/2}d\mu)}.
\end{equation}

\end{proposition}

The proof of this proposition relies on the following well-known result (\cite{DS}, Theorem~6, p. 503).


\begin{proposition}\label{prop:kernel-oper}
An operator $T: L^1(\XX, d\mu) \to L^{\infty}(\XX, d\mu)$ is bounded if and only if
$T$ is an integral operator with kernel $K\in L^\infty(\XX\times \XX)$, i.e.
$$
\hbox{
$Tf(\xx)=\int_\XX K(\xx, \yy)f(\yy)d\mu(\yy)$ a.e. on $\XX$,
}
$$
and if this is the case
$\|T\|_{1\to \infty} = \|K\|_{L^\infty}$.
Moreover, the boundedness of $T$ can be expressed in the bilinear form
$|\langle Tf, g \rangle| \le c\|f\|_{L^1}\|g\|_{L^1}$, $\forall f, g\in L^1$.
\end{proposition}

\begin{proof}[Proof of Proposition~\ref{prop:prod-oper}]
Consider the operators
\begin{equation}\label{S-T}
S:= V(\cdot, \bdelta)^{1/2}U
\quad\hbox{and}\quad
T:= W V(\cdot, \bdelta)^{1/2}.
\end{equation}
More precisely,
\begin{align*}
Sf(\xx):= V(\xx, \bdelta)^{1/2}Uf(\xx)
&= V(\xx, \bdelta)^{1/2} \int_\XX \KK_U(\xx,\yy)f(\yy) d\mu(\yy)
\\
&= \int_\XX V(\xx, \bdelta)^{1/2}\KK_U(\xx,\yy)f(\yy) d\mu(\yy).
\end{align*}
Hence, $S$ is an integral operator with kernel $\KK_S(\xx, \yy)$ given by
\begin{equation*}
\KK_S(\xx, \yy) = V(\xx, \bdelta)^{1/2}\KK_U(\xx,\yy).
\end{equation*}
Observe that from \eqref{local-UV}, \eqref{kernelsD}, and \eqref{rect-doubling2} it follows that
\begin{equation*}
|\KK_S(\xx, \yy)| \le V(\yy, \bdelta)^{-1/2} \DD_{\bdelta, \bsigma}^*(\xx, \yy)
\le cV(\xx, \bdelta)^{-1/2}\DD_{\bdelta, \bsigma-\dd/2}^*(\xx, \yy).
\end{equation*}
We now use \eqref{tech-1} and that $\sigma_i \ge d_i+1$ to obtain
\begin{align*}
\|\KK_S(\xx, \cdot)\|_2^2
&\le cV(\xx, \bdelta)^{-1} \int_\XX \DD_{\bdelta, 2\bsigma-\dd}^*(\xx, \yy)d\mu(\yy)
\le c.
\end{align*}
In turn, using Cauchy-Schwarz, this implies that for any $f\in L^2(\XX, d\mu)$
\begin{equation*}
\|Sf\|_\infty
\le \Big\|\int_\XX \KK_S(\cdot,\yy)f(\yy)d\mu(\yy) \Big\|_\infty
\le \|\KK_S(\xx, \cdot)\|_2\|f\|_2 \le c\|f\|_2
\end{equation*}
and hence
$S: L^2(\XX, d\mu) \to L^\infty(\XX, d\mu)$ is a bounded operator ($\|S\|_{2\to \infty} \le c$).

We now focus on the operator $T$. We have
\begin{align*}
Tf(\xx):= W[V(\cdot, \bdelta)^{1/2} f](\xx)
= \int_\XX \KK_W(\xx,\yy)V(\yy, \bdelta)^{1/2} f(\yy) d\mu(\yy)
\end{align*}
and hence, $T$ is an integral operator with kernel $\KK_T(\xx, \yy)$ given by
\begin{equation*}
\KK_T(\xx, \yy) = \KK_W(\xx,\yy)V(\yy, \bdelta)^{1/2}.
\end{equation*}
As above from \eqref{local-UV}, \eqref{kernelsD}, and \eqref{rect-doubling2} it follows that
\begin{equation*}
|\KK_T(\xx, \yy)| \le V(\xx, \bdelta)^{-1/2} \DD_{\bdelta, \bsigma}^*(\xx, \yy)
\le cV(\yy, \bdelta)^{-1/2} \DD_{\bdelta, \bsigma-\dd/2}^*(\xx, \yy).
\end{equation*}
Using \eqref{tech-1} and that $\sigma_i \ge d+1$ we obtain
\begin{align*}
\|\KK_T(\cdot, \yy)\|_2^2
\le c'V(\yy,\bdelta)^{-1} \int_\XX \DD_{\bdelta, \bsigma-\dd/2}^*(\xx, \yy)d\mu(\xx)
\le c.
\end{align*}
Therefore, using Minkowski's inequality, for any $f\in L^1(\XX, d\mu)$
\begin{equation*}
\|Tf\|_2
= \Big\|\int_\XX \KK_T(\cdot, \yy)f(\yy) d\mu(\yy)\Big\|_2
\le \int_\XX \|\KK_T(\cdot, \yy)\|_2|f(\yy)|d\mu(\yy)
\le c\|f\|_1,
\end{equation*}
implying that
$T: L^1(\XX, d\mu) \to L^2(\XX, d\mu)$ is a bounded operator ($\|T\|_{1\to 2} \le c$).

Using the above and the assumption that $\|R\|_{2\to 2} <\infty$ we get
\begin{equation*}
\|G\|_{1\to \infty}=\|SRT\|_{1\to \infty} \le \|S\|_{2\to\infty}\|R\|_{2\to 2} \|T\|_{1\to 2}
\le c\|R\|_{2\to 2}.
\end{equation*}
Applying Proposition~\ref{prop:kernel-oper} we conclude that
the operator $G$ is an integral operator
with kernel $\KK_{G}(\xx,\yy)$ satisfying
$\|\KK_{G}\|_\infty = \|G\|_{1\to \infty}\le c\|R\|_{2\to 2}$.

For convenience we introduce the abbreviated notation
\begin{equation*}
Q:= URW.
\end{equation*}
By definition
\begin{equation*}
G = V(\cdot,\bdelta)^{1/2} URW V(\cdot,\bdelta)^{1/2}
= V(\cdot,\bdelta)^{1/2} Q V(\cdot,\bdelta)^{1/2}
\end{equation*}
and hence for any $f\in L^1(\XX, d\mu)$
\begin{equation*}
Gf(\xx)= V(\xx,\bdelta)^{1/2} Q [V(\cdot,\bdelta)^{1/2} f(\cdot)](\xx),
\end{equation*}
implying
\begin{equation*}
Q [V(\cdot,\bdelta)^{1/2} f(\cdot)](\xx) = V(\xx,\bdelta)^{-1/2} Gf(\xx).
\end{equation*}
Substituting $g(\xx)= V(\xx,\bdelta)^{1/2} f(\xx)$ we get
\begin{align*}
Qg(\xx) &= V(\xx,\bdelta)^{-1/2} G\big[V(\cdot,\bdelta)^{-1/2} g(\cdot)\big](\xx)
\\
& = V(\xx,\bdelta)^{-1/2} \int_\XX\KK_G(\xx, \yy) V(\yy,\bdelta)^{-1/2} g(\yy) d\mu(\yy).
\end{align*}
Therefore, for any $g\in L^1(\XX, V(\cdot,\bdelta)^{-1/2} d\mu)$ we have
\begin{equation*}
Qg(\xx) = \int_\XX \frac{\KK_G(\xx, \yy)}{\big[V(\xx,\bdelta) V(\yy,\bdelta)\big]^{1/2}}g(\yy) d\mu(\yy).
\end{equation*}
Thus $Q=URW$ is an integral operator with kernel $\KK_{URW}(\xx,\yy)$
satisfying \eqref{ker-URW} and \eqref{local-URV}.
The proof of Proposition~\ref{prop:prod-oper} is complete.
\end{proof}

\subsection{Integral operators $F(\sqrt{L_1}, \sqrt{L_2})$ with localized kernels}\label{subsec:local-kern}

Our next goal is to establish localization estimates for the kernels of operators
$F(\sqrt{L_1}, \sqrt{L_2})$ generated by smooth functions $F(\lambda_1, \lambda_2)$.
We first consider the case of operators $F(\sqrt{L_1}, \sqrt{L_2})$ generated by
compactly supported functions $F$.

\begin{proposition}\label{prop:rough-kernels}
Let $F$ be a bounded measurable function on $[0,\infty)^2$ with the property
$\supp F \subset [0,\tau_1]\times [0,\tau_2]$ for some $\tau_1,\tau_2 >0$. 
Then the operator
$F(\sqrt L):=F(\sqrt L_1,\sqrt L_2)$
is an integral operator with kernel $\KK_{F(\sqrt L)}(\xx,\yy)$
satisfying
\begin{equation}\label{eq:rough1}
| \KK_{F(\sqrt L)}(\xx,\yy)|
\le \frac{\cf\|F\|_\infty}
{\big[V(\xx, \btau^{-1})V(\yy, \btau^{-1})\big]^{1/2}},
\quad \forall \xx,\yy\in \XX, \quad \btau=(\tau_1, \tau_2).
\end{equation}
More specifically, the action of $F(\sqrt L)$ can be described as follows:
For any $f\in L^1(\XX, V(\cdot, \btau^{-1})^{-1/2}d\mu)$ we have
\begin{equation}\label{FL-bound}
\|V(\cdot, \btau^{-1})^{1/2} F(\sqrt L)f(\cdot)\|_\infty \le c\|F||_\infty \|f\|_{L^1(\XX, V(\cdot, \btau^{-1})^{-1/2} d\mu)}
\end{equation}
and, of course, $F(\sqrt L)$ is bounded from $L^2$ to $L^2$.

Furthermore, for any $\xx, \yy, \xx', \yy'\in \XX$ with
$\rho_i(x_i,x_i')\le \tau_i^{-1}$ and $\rho_i(y_i,y_i')\le \tau_i^{-1}$, $i=1,2$,
\begin{align}\label{eq:smoothrough}
|\KK_{F(\sqrt L)}(\xx,\yy)-\KK_{F(\sqrt L)}(\xx',\yy')|
\le \cf \sum_{i=1,2} \frac{\big[\big( \tau_i \rho_i(x_i, x'_i)\big)^{\alpha_i}
+\big( \tau_i \rho_i(y_i, y'_i)\big)^{\alpha_i}\big]\|F\|_\infty}
{\big[V(\xx, \btau^{-1})V(\yy, \btau^{-1})\big]^{1/2}}.
\end{align}
The constant $\cf>0$ above depends only on the constants $c_{i0}, c_{i1}, c_{i2}$
from $\eqref{eq:doubling-0}$, $\eqref{Gauss-local}$, $\eqref{lip}$.
\end{proposition}

\begin{proof}
Given $\tt=(t_1,t_2)$, $t_1,t_2>0$, we have, using Theorem~\ref{thm:F-G},
\begin{equation*}
F(\sqrt L) = e^{-t_1^2L_1-t_2^2L_2}
\big[e^{2t_1^2L_1+2t_2^2L_2}F(\sqrt L)\big]
e^{-t_1^2L_1-t_2^2L_2}.
\end{equation*}
We choose $t_i:= \tau_i^{-1}$, $i=1,2$.
Clearly,
\begin{align*}
\big\|e^{2t_1^2L_1+2t_2^2L_2}F(\sqrt L)\big\|_{L^2\to L^2}
&\le \sup_{(\lambda_1,\lambda_2)\in [0,\tau_1]\times[0,\tau_2]}|e^{2t_1^2\lambda_1^2+2t_2^2\lambda_2^2}
F(\sqrt{\lambda_1},\sqrt{\lambda_2})|
\\
&= e^{2t_1^2\tau_1^2+2t_2^2\tau_2^2}\|F\|_\infty
\le c\|F\|_\infty.
\end{align*}
On the other hand, from \eqref{prodGauss-local}
$|p_{\tt^2}(\xx, \yy)| \le c(\bsigma) \DD_{\tt, \bsigma}(x,y)$
for any $\bsigma>\zero$.
We choose $\bsigma:=(3/2)\dd +\one$. 
Now, applying Proposition~\ref{prop:prod-oper} (see \eqref{URW-bounded}) we conclude that
$F(\sqrt L)$
is an integral operator with kernel $\KK_{F(\sqrt L)}(\xx,\yy)$
satisfying \eqref{eq:rough1} and \eqref{FL-bound}.

To prove (\ref{eq:smoothrough}) we set
$G(\lambda_1,\lambda_2):=F(\lambda_1,\lambda_2)e^{t_1^2\lambda_1^2+t_2^2\lambda_2^2}$.
Then we have the representation
$F(\sqrt L)=G(\sqrt L)e^{-t_1^2L_1-t_2^2L_2}$.

Let $\xx, \yy, \xx', \yy'\in \XX$ with
$\rho_i(x_i,x_i')\le \tau_i^{-1}$ and $\rho_i(y_i,y_i')\le \tau_i^{-1}$, $i=1,2$.
We use estimate \eqref{prodlip} with $\sigma_i = 3d_i/2+1$ and $t_i=\tau_i^{-1}$, $i=1,2$,
along with \eqref{eq:rough1}, applied to $G(\sqrt L)$,  to obtain
\begin{align}\label{est-KK1}
|\KK_{F(\sqrt L)}&(\xx,\yy)-\KK_{F(\sqrt L)}(\xx,\yy')| \nonumber
\\
&=\Big|\int_\XX  \KK_{G(\sqrt{L})}(\xx,\zz)\big[p_{\tt^2}(\zz,\yy)-p_{\tt^2}(\zz,\yy')\big]d \mu(\zz)\Big| \nonumber
\\
&\le c \sum_{i=1,2}\Big(\frac{\rho_i(y_i,y_i')}{t_i}\Big)^{\alpha_i}
\frac{\|G\|_\infty}{\big[V(\xx, \btau^{-1})\big]^{1/2}}
\int_\XX \frac{\DD_{\tt,\bsigma}(\zz,\yy)}{\big[V(\zz, \btau^{-1})\big]^{1/2}} d\mu(\zz)
\\
&\le c \sum_{i=1,2}\Big(\frac{\rho_i(y_i,y_i ')}{t_i}\Big)^{\alpha_i}
\frac{\|F\|_\infty}{\big[V(\xx, \btau^{-1})V(\yy, \btau^{-1})\big]^{1/2}}
\int_\XX \frac{\DD_{\tt,\bsigma-\dd/2}^*(\zz,\yy)}{V(\yy, \tt)} d\mu(\zz) \nonumber
\\
&\le c \sum_{i=1,2}(\tau_i\rho_i(y_i,y_i'))^{\alpha_i}
\frac{\|F\|_\infty}{\big[V(\xx, \btau^{-1})V(\yy, \btau^{-1})\big]^{1/2}}, \nonumber
\end{align}
where for the second inequality we used \eqref{D-D*},
and for the last (\ref{tech-1}) along with the fact that $\bsigma>(3/2)\dd$. 
Because of the symmetry we also have from above
\begin{align}\label{est-KK2}
|\KK_{F(\sqrt L)}(\xx,\yy')-\KK_{F(\sqrt L)}(\xx',\yy')|
&\le c \sum_{i=1,2}(\tau_i\rho_i(x_i,x_i'))^{\alpha_i}
\frac{\|F\|_\infty}{\big[V(\xx, \btau^{-1})V(\yy', \btau^{-1})\big]^{1/2}} \nonumber
\\
&\le c' \sum_{i=1,2}(\tau_i\rho_i(x_i,x_i'))^{\alpha_i}
\frac{\|F\|_\infty}{\big[V(\xx, \btau^{-1})V(\yy, \btau^{-1})\big]^{1/2}}.
\end{align}
Here for the last inequality we used that
$V(\yy, \btau^{-1}) \le c V(\yy', \btau^{-1})$
which follows by \eqref{rect-doubling2} and the assumption $\rho_i(y_i,y_i')\le \tau_i^{-1}$, $i=1,2$.
Estimate \eqref{eq:smoothrough} follows readily from \eqref{est-KK1} and \eqref{est-KK2}.
The proof of the proposition is complete.
\end{proof}

\noindent
{\bf Decomposition of unity.}
We will frequently use a decomposition of unity argument that is presented in the following

\begin{lemma}\label{lem:dec-unity}
$(a)$
There exist functions
$\ph_0,\ph\subset \cC^{\infty}(\RR^2)$ with the properties:
\begin{equation}
\begin{aligned}\label{dec-unity}
&\hbox{$(i)$ $\;\ph_0, \ph$ are real-valued and }
\\
& \qquad \ph_0(\pm\lambda_1, \pm\lambda_2)=\ph_0(\lambda_1, \lambda_2),\;
\ph(\pm\lambda_1, \pm\lambda_2)=\ph(\lambda_1, \lambda_2),
\\
&\hbox{$(ii)$ $\;\supp\ph_0 \subset [-2,2]^2$, $\ph_0=1$ on $[-1, 1]^2$,}
\\
&\hbox{$(iii)$ $\;\supp\ph \subset [-2,2]^2 \setminus (-1,1)^2$, and}
\\
&\hbox{$(iv)$ $\;\ph_0(\blambda)+\sum_{j\ge 1}\ph(2^{-j}\blambda)=1$, $\blambda\in\RR^2$.}
\end{aligned}
\end{equation}

$(b)$ 
Let $F$ be a bounded Borel measurable function on $[0,\infty)^2$.
Denote $\ph_j(\blambda):=\ph(2^{-j}\blambda)$, $j\ge 1$, where $\ph$ is from above,
and $F_j:= F\cdot \ph_j$, $j\ge 0$.
Then
\begin{equation}\label{dec-unity-2}
F(\sqrt{L}) = \sum_{j\ge 0} F_j(\sqrt{L}),
\end{equation}
where the convergence is in the strong $L^2$-sense. 
\end{lemma}

\begin{proof}
(a) Choose a function $\phi\in C^{\infty}(\RR)$ with the properties: $\phi$ is real-valued and even, 
$\supp \phi\subset [-2, 2]$, $0\le \phi \le 1$, and
$\phi(\lambda)=1$ for $\lambda \in [-1, 1]$.
Then set $\ph_0(\lambda_1,\lambda_2) := \phi(\lambda_1)\phi(\lambda_2)$.
Define $\ph(\blambda):=\ph_0(\blambda)-\ph_0(2\blambda)$.
Clearly, the constructed $\ph_0,\ph$ have the claimed properties.

(b) 
We define
\begin{equation*}
\Phi_n(\blambda):= \ph_0(\blambda) + \sum_{j=1}^n \ph(2^{-j}\blambda), \quad n\in \bN.
\end{equation*} 
From \eqref{dec-unity} it follows that $\supp \Phi_n \subset [-2^{n+1}, 2^{n+1}]^2$ and 
$\Phi_n(\blambda)=1$ for $\lambda\in [-2^{n-1}, 2^{n-1}]^2$
and hence
$
F(\sqrt{\blambda}) - \sum_{j=0}^n F_j(\sqrt{\blambda}) = F(\sqrt{\blambda})(1-\Phi_n(\sqrt{\blambda})),
$
$\blambda\ge \zero$.
Therefore, for any $f\in L^2(X, d\mu)$
\begin{align*}
\Big\|F(\sqrt{L})f - \sum_{j=0}^n F_j(\sqrt{L})f\Big\|_2^2
&= \|F(\sqrt{L})(1-\Phi_n(\sqrt{L}))f\|_2^2
\\
&=\Big\|\int_{[0,\infty)^2\setminus [0,4^{n-1})^2} F(\sqrt{\lambda})(1-\Phi_n(\sqrt{\lambda})) dE_\blambda f\Big\|_2^2
\\
&\le \int_{[0,\infty)^2\setminus [0,4^{n-1})^2}  |F(\sqrt{\blambda})(1-\Phi_n(\sqrt{\blambda}))|^2 d \|E_\blambda f\|_2^2
\\
&\le \|F\|_\infty^2\int_{[0,\infty)^2\setminus [0,4^{n-1})^2}  d \|E_\blambda f\|_2^2
\to 0 \quad\hbox{as} \;\; n\to \infty.
\end{align*}
Above we used that $\Phi_n(\sqrt{\blambda})=1$ for $\blambda \in [0, 4^{n-1}]^2$.
\end{proof}

We next extend the result of Proposition~\ref{prop:rough-kernels} to the case when $F$ in not compactly supported.

\begin{proposition}\label{prop:kern-cont}
Let $F$ be a Borel measurable real-valued function on $[0,\infty)^2$ with the property
\begin{equation}\label{decay}
|F(\blambda)|\le c (1+|\blambda|)^{-m}
\quad \hbox{ for some $m > d_1+d_2+\alpha_1\vee\alpha_2$.}
\end{equation}
Then $F(\bdelta\sqrt L)$ is an integral operator whose kernel
$\KK_{F(\bdelta\sqrt L)}(\xx,\yy)$ satisfies
\begin{equation}\label{KF-bound}
|\KK_{F(\bdelta\sqrt L)}(\xx,\yy)| \le c \big[V(\xx,\bdelta)V(\yy,\bdelta)\big]^{-1/2},
\quad\forall \xx,\yy\in X.
\end{equation}
More specifically, the action of $F(\bdelta\sqrt L)$ can be described as follows:
For any function $f\in L^1(\XX, V(\cdot, \bdelta)^{-1/2}d\mu)$ we have
\begin{equation*}
\|V(\cdot, \bdelta)^{1/2} F(\sqrt L)f(\cdot)\|_\infty \le c\|f\|_{L^1(\XX, V(\cdot, \bdelta)^{-1/2} d\mu)}
\end{equation*}
and, of course, $F(\bdelta\sqrt L)$ is bounded from $L^2$ to $L^2$.

Furthermore, $\KK_{F(\bdelta\sqrt L)}(\xx,\yy)$ is continuous on $\XX$.
More precisely,
\begin{equation}\label{KF-cont-2}
|\KK_{F(\bdelta\sqrt L)}(\xx,\yy)-\KK_{F(\bdelta\sqrt L)}(\xx',\yy')|
\le c \frac{
\sum_{i=1,2}\big[(\rho_i(x_i, x'_i)/\delta_i)^{\alpha_i} + (\rho_i(y_i, y'_i)/\delta_i)^{\alpha_i}\big]
}
{\big[V(\xx,\bdelta)V(\yy,\bdelta)\big]^{1/2}},
\end{equation}
if $\rho_i(x_i, x'_i)\le \delta_i$ and $\rho(y_i, y'_i)\le \delta_i$, $i=1,2$.
\end{proposition}

\begin{proof}
Choose functions
$\ph_0,\ph\subset \cC^{\infty}(\R^2)$ with properties \eqref{dec-unity}.
Set $\ph_j(\blambda):=\ph(2^{-j}\blambda)$, $j\ge 1$.
Then
$
\sum_{j\ge 0} \ph_j(\blambda)=1
$
for $\blambda\in\RR^2$
and hence
$
F(\bdelta\blambda)=\sum_{j\ge 0}F_j(\blambda), 
$
where
$
F_j(\blambda)=F(\bdelta\blambda)\ph_j(\bdelta\blambda). 
$
Therefore,
\begin{equation}\label{F-sum-Fj}
F(\bdelta\sqrt{L})=\sum_{j\ge 0}F_j(\sqrt{L}),
\end{equation}
where the convergence is strong in $L^2$, see Lemma~\ref{lem:dec-unity}.
Clearly
\begin{equation*}
\supp F_0 \subset [-2/\delta_1, 2/\delta_1]\times[-2/\delta_2, 2/\delta_2],
\quad
\supp F_j \subset I_j\setminus J_j, \quad j\ge 1,
\end{equation*}
where
\begin{align*}
I_j &:= [-2^{j+1}/\delta_1, 2^{j+1}/\delta_1]\times[-2^{j+1}/\delta_2, 2^{j+1}/\delta_2],
\\
J_j &:= (-2^{j-1}/\delta_1, 2^{j-1}/\delta_1)\times(-2^{j-1}/\delta_2, 2^{j-1}/\delta_2).
\end{align*}
From above and the decay condition \eqref{decay} on $F$
it follows that $\|F_j\|_\infty\le c 2^{-jm}$.
We now apply Proposition~\ref{prop:rough-kernels} to $F_j$  to conclude that  
$F_j(\sqrt{L})$ is an integral operator with kernel satisfying
\begin{equation}\label{KerFj}
| \KK_{F_j(\sqrt L)}(\xx,\yy)|
\le \frac{c2^{-jm}}
{\big[V(\xx, 2^{-j-1}\bdelta)V(\yy, 2^{-j-1}\bdelta)\big]^{1/2}}
\le \frac{c2^{-jm}2^{(d_1+d_2)(j+1)}}
{\big[V(\xx, \bdelta)V(\yy, \bdelta)\big]^{1/2}},
\end{equation}
where we used that
$V(\cdot, \bdelta) \le  c2^{(d_1+d_2)(j+1)} V(\cdot, 2^{-j-1}\bdelta)$ in light of \eqref{rect-doubling}.

Denote
\begin{equation}\label{Kdelta}
\KK_{\bdelta}(\xx,\yy):=\sum_{j\ge 0}\KK_{F_j(\sqrt{L})}(\xx,\yy).
\end{equation}
From \eqref{KerFj} it follows that
\begin{align}\label{Ker-bound}
|\KK_{\bdelta}(\xx,\yy)|
\le \sum_{j\ge 0}|\KK_{F_j(\sqrt{L})}(\xx,\yy)|
&\le \frac{c}
{\big[V(\xx, \bdelta)V(\yy,\bdelta)\big]^{1/2}} \sum_{j\ge 0}2^{-j(m-d_1-d_2)} \nonumber
\\
&\le \frac{c}{\big[V(\xx, \bdelta)V(\yy, \bdelta)\big]^{1/2}},
\end{align}
where we used that $m>d_1+d_2$.

We next show that the kernel $\KK_\bdelta(\xx,\yy)$ satisfies \eqref{KF-cont-2} and hence it is continuous.
To this end we estimate each
$|\KK_{F_j(\sqrt{L})}(\xx,\yy)-\KK_{F_j(\sqrt{L})}(\xx,\yy')|$
assuming that 
$\rho_i(y_i, y_i') \le \delta_i$, $i=1,2$.
Two cases present themselves here: 

{\em Case 1:} Let $\rho_1(y_1, y_1') \le \delta_1 2^{-j-1}$ and $\rho_2(y_2, y_2') \le \delta_2 2^{-j-1}$.
We apply \eqref{eq:smoothrough} to $F_j(\sqrt{L})$, $j\ge 0$,
using that 
$\supp F_j \subset [-2^{j+1}/\delta_1, 2^{j+1}/\delta_1]\times[-2^{j+1}/\delta_2, 2^{j+1}/\delta_2]$
and 
$\|F_j\|_\infty \le c2^{-jm}$
to obtain
\begin{align}\label{KerFj-KerFj}
|\KK_{F_j(\sqrt{L})}(\xx,\yy)-\KK_{F_j(\sqrt{L})}(\xx,\yy')|
&\le c\sum_{i=1,2}\frac{\big(2^{j+1}\delta_i^{-1}\rho_i(y_i, y'_i)\big)^{\alpha_i}\|F_j\|_\infty}
{\big[V(\xx,2^{-(j+1)}\bdelta)V(\yy, 2^{-(j+1)}\bdelta)\big]^{1/2}} \nonumber
\\
&\le  c\sum_{i=1,2}\frac{((\rho_i(y_i, y'_i)/\delta_i)^{\alpha_i}}
{\big[V(\xx,\bdelta)V(\yy,\bdelta)\big]^{1/2}}
2^{j(d_1+d_2+\alpha_1\vee \alpha_2-m)},
\end{align}
where we used \eqref{rect-doubling}. 

{\em Case 2:}
$\rho_1(y_1, y_1') > \delta_1 2^{-j-1}$ or $\rho_2(y_2, y_2') > \delta_2 2^{-j-1}$.
Using \eqref{KerFj} we obtain
\begin{align*}
|\KK_{F_j(\sqrt{L})}(\xx,\yy)-\KK_{F_j(\sqrt{L})}(\xx,\yy')|
&\le |\KK_{F_j(\sqrt{L})}(\xx,\yy)|+|\KK_{F_j(\sqrt{L})}(\xx,\yy')|
\\
&\le \frac{c2^{j(d_1+d_2-m)}}{\big[V(\xx, \bdelta)V(\yy, \bdelta)\big]^{1/2}}
+ \frac{c2^{j(d_1+d_2-m)}}{\big[V(\xx, \bdelta)V(\yy', \bdelta)\big]^{1/2}}
\\
&\le c\sum_{i=1,2}\frac{((\rho_i(y_i, y'_i)/\delta_i)^{\alpha_i}}
{\big[V(\xx,\bdelta)V(\yy,\bdelta)\big]^{1/2}}
2^{j(d_1+d_2+\alpha_1\vee \alpha_2-m)},
\end{align*}
where we used that 
$V(\yy, \bdelta) \le c2^{d_1+d_2} V(\yy', \bdelta)$
which follows from \eqref{rect-doubling2}.
Therefore, estimate \eqref{KerFj-KerFj} is valid whenever $\rho_i(y_i, y_i') \le \delta_i$, $i=1,2$.

Summing up inequalities \eqref{KerFj-KerFj} we arrive at 
\begin{align*}
|\KK_{\bdelta}(\xx,\yy) - \KK_{\bdelta}(\xx,\yy')| 
&\le \sum_{j\ge 0}|\KK_{F_j(\sqrt{L})}(\xx,\yy)- \KK_{F_j(\sqrt{L})}(\xx,\yy')| \nonumber
\\
& \le c\sum_{i=1,2}\frac{((\rho_i(y_i, y'_i)/\delta_i)^{\alpha_i}}
{\big[V(\xx,\bdelta)V(\yy,\bdelta)\big]^{1/2}},
\end{align*}
where we used the fact that $m > d_1+d_2+\alpha_1\vee \alpha_2$.
Because of the symmetry of $\KK_{\bdelta}(\xx,\yy)$ the above implies the estimate 
\begin{equation}\label{K-cont}
|\KK_{\bdelta}(\xx,\yy) - \KK_{\bdelta}(\xx',\yy')| 
\le c\sum_{i=1,2}\frac{((\rho_i(x_i, x'_i)/\delta_i)^{\alpha_i}+((\rho_i(y_i, y'_i)/\delta_i)^{\alpha_i}}
{\big[V(\xx,\bdelta)V(\yy,\bdelta)\big]^{1/2}},
\end{equation}
if $\rho_i(x_i, x_i') \le \delta_i$ and $\rho_i(y_i, y_i') \le \delta_i$, $i=1,2$.

It remains to show that $\KK_{\bdelta}(\xx,\yy)$ can be identified as the kernel of the operator $F(\bdelta\sqrt{L})$.
Denote by $L^2_0(X)$ the set of all compactly supported functions in $L^2(X, d\mu)$.
It is straightforward to show that
$L^2_0(X)$ is dense in $L^1(X, V(\cdot, \bdelta)^{-1/2}d\mu)$.
%

Let $H_\bdelta$ be the operator with kernel $\KK_\bdelta(\xx,\yy)$, that is,
\begin{equation*}
H_\bdelta f(\xx):= \int_X \KK_\bdelta(\xx,\yy)f(\yy)d\mu(\yy).
\end{equation*}
From \eqref{Ker-bound} it readily follow that for any $L^1(X, V(\cdot, \bdelta)^{-1/2}d\mu)$ we have
\begin{equation}\label{H-bound}
\|V(\cdot, \bdelta)^{1/2}H_\bdelta f(\cdot)\|_\infty \le c \|f\|_{L^1(X, V(\cdot, \bdelta)^{-1/2}d\mu)}. 
\end{equation}

Let $f\in L^2_0(X)$. Then from \eqref{F-sum-Fj} it follows that
\begin{equation*}
F(\bdelta\sqrt{L})f = \sum_{j\ge 0} F_j(\sqrt{L})f \quad \hbox{in} \quad L^2.
\end{equation*}
On the other hand, since $f\in L^2_0(X)\subset L^1(X, V(\cdot, \bdelta)^{-1/2}d\mu)$
it follows from \eqref{Ker-bound} and \eqref{K-cont} that
\begin{equation*}
\sum_{j\ge 0} F_j(\sqrt{L})f(\xx) 
= \sum_{j\ge 0} \int_X\KK_{F_j(\sqrt{L})}(\xx,\yy)f(y) d\mu(\yy)
= \int_X \KK_\bdelta(\xx,\yy)f(\yy) d\mu(\yy),
\end{equation*}
where the series converges point-wise for every $\xx\in X$ to a continuous function on~$X$.
Therefore, 
\begin{equation*}
F(\bdelta\sqrt{L})f(\xx)
=\sum_{j\ge 0} F_j(\sqrt{L})f(\xx) 
= \int_X \KK_\bdelta(\xx,\yy)f(\yy) d\mu(\yy)
= H_\bdelta f(\xx),
\;\; \forall f\in L^2_0(X).
\end{equation*}
Using that 
$L^2_0(X)$ is dense in $L^1(X, V(\cdot, \bdelta)^{-1/2}d\mu)$ 
it follows that $H_\bdelta$ is the unique extension of the operator $F(\bdelta\sqrt{L})$ from $L^2_0(X)$ to $L^1(X, V(\cdot, \bdelta)^{-1/2}d\mu)$.
We define 
\begin{equation*}
F(\bdelta\sqrt{L})f(\xx)
:= \int_X \KK_\bdelta(\xx,\yy)f(y) d\mu(\yy)
= H_\bdelta f(\xx),
\quad \forall f\in L^1(X, V(\cdot, \bdelta)^{-1/2}d\mu).
\end{equation*}
Thus $\KK_\bdelta(\xx,\yy)$ is the kernel of the operator $F(\bdelta\sqrt{L})$.

Finally, note that \eqref{KF-bound} follows from \eqref{Ker-bound} and \eqref{KF-cont-2} from \eqref{K-cont}.
\end{proof}

The Fourier transform of an integrable function $F:\R^2\to \bC$ is standardly defined by
$$
\widehat F(\bxi):=\int_{\R^2} F(\xx)e^{-i \xx\cdot\bxi} dx,\quad \bxi=(\xi_1,\xi_2), \xx=(x_1,x_2).
$$


The finite speed propagation property of the operators $L_1, L_2$ (see \eqref{finite-speed})
leads to the following localization result for
the kernels of operators of the form $F(t\sqrt{L})$ whenever $F$ is band-limited.

\begin{proposition}\label{prop:finite-sp}
Let $F\in L^1(\R^2)$, $\widehat{F}\in L^1(\R^2)$, and
$|F(\blambda)|\le c(1+|\blambda|)^{-m}$, $\blambda\in\RR^2$, for some $m\ge d_1+d_2+1$.

$(a)$
If $F(-\lambda_1,  \lambda_2)=F(\lambda_1,\lambda_2)$ for $(\lambda_1,\lambda_2)\in \R^2$,
and
$\supp \widehat F \subset [-A, A]\times \R$ for some $A>0$,
then for any $\tt=(t_1,t_2)>\zero$, and $\xx,\yy\in \XX$, we have
\begin{equation}\label{finite-speed-2}
\KK_{F(\tt\sqrt{L})}(\xx, \yy) = 0
\quad\hbox{if}\quad
\rho_1(x_1, y_1) > \ct_1 t_1A.
\end{equation}

$(b)$
If $F(\lambda_1, -\lambda_2)=F(\lambda_1,\lambda_2)$ for $(\lambda_1,\lambda_2)\in \R^2$,
and
$\supp \widehat F \subset \R\times [-A, A]$ for some $A>0$,
then for any $\tt=(t_1,t_2)>\zero$, and $\xx,\yy\in \XX$,
\begin{equation}\label{finite-speed-3}
\KK_{F(\tt\sqrt{L})}(\xx, \yy) = 0
\quad\hbox{if}\quad
\rho_2(x_2, y_2) > \ct_2 t_2A.
\end{equation}
Above the constants $\ct_1$, $\ct_2$ are from \eqref{finite-speed}.
Recall that $F(\tt\sqrt{L}) := F(t_1\sqrt{L_1}, t_2\sqrt{L_2})$.
\end{proposition}

\begin{proof}
We will prove only part (a) of this proposition. The proof of part (b) is the same (symmetric).

From the Fourier inversion formula, using that $F,\widehat{F}\in L^1(\R^2)$, we have
\begin{equation}\label{Fourier}
F(t_1\sqrt{\lambda_1}, t_2\sqrt{\lambda_2})
=\frac{1}{(2\pi)^2}\int_{\R^2}\widehat F(\bxi)e^{i(t_1\xi_1\sqrt{\lambda_1}+t_2\xi_2\sqrt{\lambda_2})}d\bxi,
\quad d\bxi=d\xi_1 d\xi_2.
\end{equation}
Assume that $\varphi_1\otimes\varphi_2, \psi_1\otimes\psi_2\in L^2(\XX)\cap  L^1(\XX)$.
Then using \eqref{Fourier} we get
\begin{align*}
&\big\langle F(\tt\sqrt L) \varphi_1\otimes\varphi_2, \psi_1\otimes\psi_2 \big\rangle
\\
&=\int_{[0,\infty)^2}F(t_1\sqrt{\lambda_1}, t_2\sqrt{\lambda_2})d\langle E_{(\lambda_1,\lambda_2)}
\varphi_1\otimes\varphi_2,  \psi_1\otimes\psi_2 \rangle
\\
&=\int_{[0,\infty)^2}
\Big(\frac{1}{(2\pi)^2}\int_{\R^2}\widehat F(\bxi)e^{i(t_1\xi_1\sqrt{\lambda_1}+t_2\xi_2\sqrt{\lambda_2})}d\bxi\Big)
d\langle E_{(\lambda_1,\lambda_2)}\varphi_1\otimes\varphi_2,  \psi_1\otimes\psi_2 \rangle
\\
&=\int_{[0,\infty)^2}\Big(\frac{1}{2\pi^2}\int_{[0,A]\times\R}
\widehat F(\bxi)\cos(t_1\xi_1\sqrt{\lambda_1})e^{it_2\xi_2\sqrt{\lambda_2}}d\bxi\Big)
d\langle E_{(\lambda_1,\lambda_2)}\varphi_1\otimes\varphi_2,  \psi_1\otimes\psi_2 \rangle
\\
&=\frac{1}{2\pi^2}\int_{[0,A]\times\R}\widehat F(\bxi)
\Big(\int_{[0,\infty)^2}\cos(t_1\xi_1\sqrt{\lambda_1})e^{it_2\xi_2\sqrt{\lambda_2}}
d\langle E_{(\lambda_1,\lambda_2)}\varphi_1\otimes\varphi_2,  \psi_1\otimes\psi_2 \rangle\Big) d\bxi
\\
&=\frac{1}{2\pi^2}\int_{[0,A]\times \R}\widehat F(\bxi)\big\langle \cos(t_1\xi_1\sqrt{L_1})
\varphi_1\otimes e^{it_2\xi_2\sqrt{L_2}} \varphi_2, \psi_1\otimes\psi_2\big\rangle d\bxi
\\
&=\frac{1}{2\pi^2}\int_{[0,A]\times\R}\widehat F(\bxi)\big\langle \cos(t_1\xi_1\sqrt{L_1})
\varphi_1,\psi_1\big\rangle \big\langle e^{it_2\xi_2\sqrt{L_2}} \varphi_2, \psi_2\big\rangle d\bxi,
\end{align*}
where we used that $\widehat F(-\xi_1, \xi_2)=\widehat F(\xi_1, \xi_2)$,
which follows from the assumption $F(-\lambda_1, \lambda_2)=F(\lambda_1, \lambda_2)$.
To justify the above change of the order of integration we use Fubini's theorem and the fact that
\begin{align*}
&\int_{[0,A]\times \R}|\widehat F(\bxi)|\biggl(\int_{[0,\infty)^2}|\cos(t_1\xi_1\sqrt{\lambda_1})e^{it_2\xi_2\sqrt{\lambda_2}}||d\langle E_{(\lambda_1,\lambda_2)}\varphi_1\otimes\varphi_2,  \psi_1\otimes\psi_2 \rangle|\biggr)d\bxi
\\
&\le\int_{\R^2}|\widehat F(\bxi)|\biggl(\int_{[0,\infty)^2}|d\langle E_{(\lambda_1,\lambda_2)}\varphi_1\otimes\varphi_2,  \psi_1\otimes\psi_2 \rangle|\biggr)d\bxi
\\
&\le \int_{\R^2}|\widehat F(\bxi)|d\bxi\biggl(\int_{[0,\infty)^2}d|| E_{(\lambda_1,\lambda_2)}\varphi_1\otimes\varphi_2||^2\biggr)^{\frac{1}{2}}
\biggl(\int_{[0,\infty)^2}d|| E_{(\lambda_1,\lambda_2)}\psi_1\otimes\psi_2||^2\biggr)^{\frac{1}{2}},
\\
&= \|\widehat F\|_{L^1}\| \varphi_1\otimes\varphi_2\|_{L^2}\|\psi_1\otimes\psi_2|\|_{L^2}<\infty.
\end{align*}
Therefore,
\begin{align}\label{eq:Gsmooth1}
\big\langle F(\tt\sqrt L) &\varphi_1\otimes\varphi_2, \psi_1\otimes\psi_2 \big\rangle \nonumber
\\
&=\frac{1}{2\pi^2}\int_{[0,A]\times\R}\widehat F(\bxi)\big\langle \cos(t_1\xi_1\sqrt{L_1})
\varphi_1,\psi_1\big\rangle \big\langle e^{it_2\xi_2\sqrt{L_2}} \varphi_2, \psi_2\big\rangle d\bxi.
\end{align}

For $i=1,2,$ we let
$\varphi_i^\epsilon=\mu_i(B(y_i,\epsilon))^{-1}\ONE_{B(y_i,\epsilon)}$
and
$\psi_i^\epsilon=\mu_i(B(x_i,\epsilon))^{-1}\ONE_{B(x_i,\epsilon)}$
be defined on $\XX_i$, where $\epsilon>0$ is such that
$ \rho_1(x_1,y_1)-2\epsilon>\ct_1 A t_1$.
Recall that $\ONE_E$ stands for the characteristic function of the set $E$.
Then applying the finite speed propagation formula (\ref{finite-speed}) we get
$$
\big\langle \cos(t_1\xi_1\sqrt{L_1}) \varphi_1^{\epsilon},\psi_1^{\epsilon}\big\rangle=0, \quad \xi_1\in [0,A].
$$
Using this in (\ref{eq:Gsmooth1}) it follows that
\begin{align*}
&\big\langle F(\tt\sqrt L) \varphi^\epsilon_1\otimes\varphi_2^\epsilon, \psi_1^\epsilon\otimes\psi_2^\epsilon \big\rangle=0.
\end{align*}
Finally, taking into account that the kernel $\KK_{F(t_1\sqrt{L_1}, t_2\sqrt{L_2})}$
is continuous on $\XX$, using Proposition~\ref{prop:kern-cont},
we conclude that for all $\xx,\yy\in \XX$ with $\rho_1(x_1, y_1) > \ct_1 t_1 A$ we have
$$
\KK_{F(\tt\sqrt L)}(\xx,\yy)=\lim_{\epsilon\rightarrow 0}
\langle F(\tt\sqrt L) \varphi_1^\epsilon\otimes\varphi_2^\epsilon, \psi_1^\epsilon\otimes\psi_2^\epsilon \big\rangle=0,
$$
which completes the proof.
\end{proof}

We now come to our principle localization estimate for the kernels of
operators of the form
$F(\bdelta\sqrt{L}):= F(\delta_1\sqrt{L_1}, \delta_2\sqrt{L_2})$,
where $F(\lambda_1, \lambda_2)$ is smooth and satisfies the condition
$F(\pm \lambda_1, \pm\lambda_2)=F(\lambda_1,\lambda_2)$,
i.e. $F(\lambda_1,\lambda_2)$ is even in $\lambda_1$ and $\lambda_2$.
For example,
$F(\lambda_1,\lambda_2)= f(\lambda_1^2+\lambda_2^2)$ with $f$ a univariate function
or
$F(\lambda_1,\lambda_2)=f_1(\lambda_1)f_2(\lambda_2)$ with $f_1, f_2$ even
are such functions.

\begin{theorem}\label{th:boxsupport}
Let $F\in \cC^{k_1+k_2}(\R^2)$, $k_i\in \bN$, $k_i>3d_i/2$, $i=1,2$, 
$k_1+k_2\ge 4$, be real-valued,
$\supp F\subset[-R,R]^2$ for some $R\ge 1$
and
$F(\pm \lambda_1, \pm\lambda_2)=F(\lambda_1,\lambda_2)$
for all $(\lambda_1,\lambda_2)\in [-R,R]^2$.
Then $F(\bdelta \sqrt L):=F(\delta_1\sqrt{L_1}, \delta_2\sqrt{L_2}),
\bdelta=(\delta_1,\delta_2)>\zero$,
is an integral operator bounded from $L^p(X,d\mu)$ to $L^p(X,d\mu)$, $1\le p\le \infty$,
with kernel
$\KK_{F(\bdelta \sqrt L)}(\xx,\yy)$
satisfying
\begin{equation}\label{local-1}
|\KK_{F(\bdelta \sqrt L)}(\xx,\yy)|\le c_\kk \DD_{\bdelta,\kk}(\xx, \yy) \quad \text{and}
\end{equation}
\begin{equation}\label{local-2}
|\KK_{F(\bdelta \sqrt L)}(\xx,\yy)- \KK_{F(\bdelta \sqrt L)}(\xx,\yy')|
\le c_{\kk'} \sum_{i=1,2}\Big(\frac{\rho_i(y_i, y'_i)}{\delta_i}\Big)^{\alpha_i}\DD_{\bdelta,\kk}(\xx,\yy)
\end{equation}
if $\rho_i(y_i,y_i')\le \delta_i$, $i=1,2$, where $\kk=(k_1,k_2)$.
Moreover,
$$
c_\kk=cR^{d_1+d_2} \big(\|F\|_{L^\infty}+R^{k_1+k_2}\|F\|_{W^{k_1+k_2}_\infty}\big)
$$
and $c_{\kk'}=R^{\alpha_1\vee\alpha_2} c_\kk$,
where $c$ depends on the constants $c_{i0},c_{i1}, c_{i2}$ from $(\ref{eq:doubling-0})-(\ref{lip})$ and $\kk$.
In addition,
\begin{equation}\label{int-K-1}
\int_\XX \KK_{F(\bdelta \sqrt L)}(\xx,\yy) d\mu(\yy)=F(\zero), \quad \forall \xx\in\XX.
\end{equation}
\end{theorem}

For the proof of this theorem we need some preparation.
We will apply some well known approximation techniques using polynomial reproducing kernel operators.
We will use the standard notation $\Phi_h(\cdot):=h^{-2}\Phi(h^{-1}\cdot)$, $h>0$,
for any function $\Phi$ on $\R^2$.

\begin{lemma}\label{lem:polyn-repr}
Let $\phi \in S(\R)$ be an even function such that
$\supp\widehat\phi\subset [-1,1]$, $\widehat \phi(0)=1$ and
$\widehat \phi^{(k)}(0)=0$ for all $k\in \bN$.
If $\Phi(x_1,x_2)=\phi(x_1)\phi(x_2), (x_1,x_2)\in \bR^2$,
then for any  $\bbeta=(\beta_1,\beta_2) \in \bN_0^2$
\begin{equation}\label{eq:pol-rep}
\Phi_h*\xx^\bbeta=\int_{\bR^2}\Phi_h(\xx-\yy)\yy^\bbeta d\yy = \xx^\bbeta, \quad h>0.
\end{equation}
Furthermore, for any $\bbeta\in \bN_0^2\setminus\{(0,0)\}$
\begin{equation}\label{Phi-h-y}
\int_{\bR^2}\Phi_h(\yy)\yy^\bbeta d\yy=0.
\end{equation}
Recall that $\xx^\bbeta:=x_1^{\beta_1}x_2^{\beta_2}$.
\end{lemma}
\begin{proof}
From standard properties of the Fourier transform we get
\begin{equation}\label{varphi-y}
\int_{\R}\phi(y) y^k\, dy= i^k\widehat\phi^{(k)}(0)=i^k\delta_{k0},
\quad \forall k\in\bN_0.
\end{equation}
From this it follows that for any $h>0$ and $k \in \bN_0$ we have
\begin{align*}
h^{-1}\int_{\bR}\phi\big(\frac{x-y}h\big)y^k dy
&=h^k\int_{\bR}\phi(y)(x/h-y)^k\, dy
\\
&=h^k\sum_{m=0}^k\binom{k}{m}(x/h)^{k-m} (-1)^m\int_{\bR }\phi(y)y^mdy
=x^k.
\end{align*}
Now, it easily follows that for any $h>0$ and $\bbeta \in \bN_0^2$
\begin{align*}
\int_{\bR^2}\Phi_h(\xx-\yy)\yy^\bbeta d\yy
= h^{-2}\prod_{i=1,2}\int_{\bR}\phi\big(\frac{x_i-y_i}{h}\big)y_i^{\beta_i}dy_i
=x_1^{\beta_1} x_2^{\beta_2},
 \end{align*}
which confirms \eqref{eq:pol-rep}.
Identity \eqref{Phi-h-y} follows at once from \eqref{varphi-y}.
\end{proof}

The next approximation-type lemma will be instrumental in proving our basic Theorem \ref{th:boxsupport}.

\begin{lemma}\label{lem:approximation}
Assume $F\in \cC^{k_1+k_2}(\bR^2)$, $k_i\in\bN$, $i=1,2$, $\supp F\subset [-1,1]^2$,
and let $\Phi$ be the function from Lemma~\ref{lem:polyn-repr}.
Then for any $h>0$ we have
\begin{equation}\label{eq:approx1}
\|F-F*\Phi_h\|_{L^\infty}\le c h^{k_1+k_2}\|F\|_{W^{k_1+k_2}_\infty}.
\end{equation}
Furthermore, if  $|\xx|>3$, then
\begin{equation}\label{eq:approx2}
|F(\xx)-F*\Phi_h(\xx)|=|F*\Phi_h(\xx)|\le c \|F\|_{L^\infty}h^{k_1+k_2}(1+|\xx|)^{-k_1-k_2}.
\end{equation}
\end{lemma}

\begin{proof}
By Taylor's formula we have, for $\xx, \yy\in \R^2$,
\begin{equation}\label{eq:Taylor}
F(\yy)=P_{k_1+k_2,\xx}(\yy)+R_{k_1+k_2, \xx}(\yy),
\end{equation}
where
$$
P_{k_1+k_2,\xx}(\yy)=\sum_{|\bbeta|< k_1+k_2}\frac{\partial^{\bbeta}F(\xx)}{\bbeta!}(\yy-\xx)^{\bbeta}
$$
and
\begin{equation}\label{eq:approx33}
 R_{k_1+k_2, \xx}(\yy)=\sum_{|\bbeta|=k_1+k_2}\frac{k_1+k_2}{\bbeta!}(\yy-\xx)^{\bbeta}
 \int_0^1(1-t)^{k_1+k_2-1}\partial^{\bbeta}F(\xx+t(\yy-\xx))dt.
\end{equation}
Clearly, \eqref{Phi-h-y} implies that if $|\bbeta|>0$, then
$
\int_{\R^2}(\yy-\xx)^\bbeta\Phi_h(\xx-\yy)d\yy=0.
$
Using this and \eqref{eq:Taylor} we obtain for $\xx\in \bR^2$
\begin{align*}
|F(\xx)&-F\ast \Phi_h (\xx)|=\Big|\int_{\R^2}[F(\xx)-F(\yy)]\Phi_h(\xx-\yy)d\yy\Big|
\\
&=\Big|\int_{\bR^2} [F(\xx)-P_{k_1+k_2,\xx}(\yy)]\Phi_h(\xx-\yy)d\yy
- \int_{\bR^2} R_{k_1+k_2,\xx}(\yy)\Phi_h(\xx-\yy)d\yy\Big|
\\
&\le \int_{\bR^2}|R_{k_1+k_2,\xx}(\yy)\Phi_h(\xx-\yy)|d\yy.
\end{align*}
We now use \eqref{eq:approx33} to obtain
\begin{align*}
|F(\xx)-F\ast \Phi_h (\xx)|
&\le c h^{k_1+k_2}\|F\|_{W^{k_1+k_2}_{\infty}(\R^2)}\int_{\bR^2}\Bigl(\frac{|\xx-\yy|}h\Bigr)^{k_1+k_2}|\Phi_h(\xx-\yy)|d\yy
\\
&=ch^{k_1+k_2}\|F\|_{W^{k_1+k_2}_{\infty}(\R^2)}\int_{\R^2}|\yy|^{k_1+k_2}|\Phi(\yy)|d\yy
\\
&\le ch^{k_1+k_2}\|F\|_{W^{k_1+k_2}_{\infty}(\R^2)},
\end{align*}
where for the last inequality we used that $\Phi\in\cS(\R^2)$.
This confirms (\ref{eq:approx1}).

We next prove estimate \eqref{eq:approx2}.
Assume  $|\xx|>3$. 
Because $\supp F\subset [-1,1]^2$
we have $|F(\xx)-F*\Phi_h (\xx)|=|F*\Phi_h (\xx)|$.
Choose $K := k_1+k_2+2$.
From the assumption $\Phi\in S(\bR^2)$ it follows that
$|\Phi_h(\xx)| \le ch^{-2}\big(1+|\xx|/h\big)^{-K}$ with $c>0$ depending on $K$.
Using this we get
\begin{equation}\label{eq:approx3}
\begin{aligned}
|F(\xx)-F*\Phi_h (\xx)|&\le \int_{[-1,1]^2}|F(\yy)||\Phi_h(\xx-\yy)|d\yy
\\
&\le \|F\|_{L^\infty} \int_{[-1,1]^2}|\Phi_h(\xx-\yy)|d\yy
\\
&\le c\|F\|_{L^\infty} h^{-2} \int_{|\yy|\le \sqrt 2} \Bigl(1+\frac{|\xx-\yy|}{h}\Bigr)^{-K}d\yy.
\end{aligned}
\end{equation}
From $|\yy|\le \sqrt 2$ it follows that
$|\xx-\yy|\ge |\xx|-|\yy|\ge|\xx|-\sqrt 2$.
We use this and the fact that $|\xx|\ge 3$ to obtain
\begin{align*}
h^{-2} \int_{|\yy|\le \sqrt 2} \Bigl(1+\frac{|\xx-\yy|}{h}\Bigr)^{-K}d\yy
&\le \int_{|\uu|\ge \frac{|\xx|-\sqrt 2}{h}} (1+|\uu|)^{-K}d\uu
\\
&\le c\Bigl(\frac{|\xx|-\sqrt 2}h \Bigr)^{-K+2}
\le ch^{K-2}(1+|\xx|)^{-K+2},
\end{align*}
where we applied the substitution $\uu=(\xx-\yy)/h$.
Finally, combining this with (\ref{eq:approx3}) and using that $K = k_1+k_2+2$ we obtain
$$
|F(\xx)-F\ast \Phi_h (\xx)|\le c \|F\|_{L^\infty} h^{k_1+k_2}(1+|\xx|)^{-k_1-k_2}.
$$
The proof of the lemma is complete.
\end{proof}

\begin{proof}[Proof of Theorem~\ref{th:boxsupport}]
We first show that it suffices to prove the theorem in the case $R=1$.
Indeed, assume that the theorem holds for $R=1$ and
let $F(\lambda_1,\lambda_2)$ satisfy its assumptions for some $R>1$.
We set
$G(\lambda_1,\lambda_2):=F(R\lambda_1,R\lambda_2)$ and note that $\supp G\subset[-1,1]^2$.
Applying the result to $G$ we get
\begin{equation}\label{eq:boxsupport2}
\begin{aligned}
|\KK_{F(\bdelta \sqrt L)}(\xx,\yy)|
&= |\KK_{G(R^{-1}\bdelta \sqrt L)}(\xx,\yy)|
\le c_k(G)\DD_{R^{-1}\bdelta,\kk}(\xx,\yy)
\\
&\le c \Big(\|G\|_\infty+\sum_{|\bbeta|=k_1+k_2} \|\partial^\bbeta G\|_{\infty}\Big) R^{d_1+d_2}\DD_{\bdelta,\kk}(\xx,\yy)
\\
&\le c \big(\|F\|_\infty+R^{k_1+k_2} \|F\|_{W^{k_1+k_2}_\infty}\big) R^{d_1+d_2}\DD_{\bdelta,\kk}(\xx,\yy).
\end{aligned}
\end{equation}
Similarly, if $\rho_i(y_i,y_i')\le \delta_i/R$, $i=1,2$, then
\begin{align*}
|\KK_{F(\bdelta \sqrt L)}(\xx,\yy)-\KK_{F(\bdelta \sqrt L)}(\xx,\yy')|
&= |\KK_{G(R^{-1}\bdelta \sqrt L)}(\xx,\yy)-\KK_{G(R^{-1}\bdelta \sqrt L)}(\xx,\yy')|
\\
&\le c_\kk(G) \sum_{i=1,2}\Big(\frac{\rho_i(y_i, y'_i)}{\delta_i/R}\Big)^{\alpha_i}\DD_{R^{-1}\bdelta,\kk}(\xx,\yy)
\\
&\le c_\kk' \sum_{i=1,2}\Big(\frac{\rho_i(y_i, y'_i)}{\delta_i}\Big)^{\alpha_i}\DD_{\bdelta,\kk}(\xx,\yy).
\end{align*}
If $\delta_i/R\le \rho_i(y_i,y_i')\le \delta_i$, then  \eqref{local-2} follows directly from (\ref{eq:boxsupport2}).

We next prove the theorem in the case $R=1$.
Applying Proposition~\ref{prop:rough-kernels} with $F(\blambda)$ replaced by $F(\bdelta\blambda)$, 
$\supp F(\bdelta\blambda) \subset [-1/\delta_1, 1/\delta_2]\times [-1/\delta_2, 1/\delta_2]$,
we conclude that $F(\bdelta\sqrt{L})$ is an integral operator whose kernel $\KK_{F(\bdelta\sqrt{L})}(\xx,\yy)$ is continuous and
\begin{equation}\label{K-F-est}
|\KK_{F(\bdelta\sqrt{L})}(\xx,\yy)| \le \frac{c\|F\|_\infty}{[V(\xx,\bdelta)V(\yy,\bdelta)]^{1/2}}.
\end{equation}
Our goal is to show that $\KK_{F(\bdelta\sqrt{L})}(\xx,\yy)$ satisfies \eqref{local-1} and \eqref{local-2}.

Assume $0<h\le 2(\ct_1\vee \ct_2)$ with the constants $\ct_1, \ct_2$ from \eqref{finite-speed}.
Let $\Phi:=\phi\otimes\phi$ be the function from Lemmas~\ref{lem:polyn-repr}, \ref{lem:approximation}.
Choose $\ph_0, \ph\in C^{\infty}(\RR^2)$ with the properties from \eqref{dec-unity}.
Set
$\ph_j(\blambda):=\ph(2^{-j}\blambda)$, $j\ge 1$.
Then
$\sum_{j \ge 0}\ph_j(\blambda)=1$ for $\blambda \in \RR^2$ and hence
\begin{equation*}
F(\bdelta\blambda)- F*\Phi_h(\bdelta\blambda) = \sum_{j \ge 0}[F(\bdelta\blambda)- F*\Phi_h(\bdelta\blambda)]\ph_j(\bdelta\blambda),\quad \bdelta>\zero.
\end{equation*}
Set $F_j(\blambda):= [F(\blambda)- F*\Phi_h(\blambda)]\ph_j(\blambda)$.
From above it follows that
\begin{equation}\label{phi-decomp}
F(\bdelta \sqrt{L}) - F*\Phi_h(\bdelta\sqrt{L})
= \sum_{j \ge 0} F_j(\bdelta\sqrt{L}),
\end{equation}
where the convergence is in the strong sense in $L^2$, see Lemma~\ref{lem:dec-unity}.

From conditions \eqref{dec-unity},(i)--(ii) on $\supp \phi_0$ and $\supp \phi$ it follows that
\begin{equation*}
\supp F_0(\bdelta \cdot\cdot)\cap[0,\infty)^2\subset [0, 2\delta_1^{-1}]\times[0, 2\delta_2^{-1}]
\quad\hbox{and}\quad
\end{equation*}
\begin{equation*}
\supp F_j(\bdelta\cdot\cdot)\cap[0,\infty)^2\subset [0, 2^{j+1}\delta_1^{-1}]\times[0, 2^{j+1}\delta_2^{-1}].
\end{equation*}
We now apply Proposition~\ref{prop:rough-kernels} to conclude that
\begin{equation}\label{est-KKj}
|\KK_{F_j(\bdelta\sqrt L)}(\uu,\vv)|
\le \frac{c\|F_j\|_\infty}{\big[V(\uu,2^{-j-1}\bdelta)V(\vv,2^{-j-1}\bdelta)\big]^{1/2}}
\le \frac{c2^{j(d_1+d_2)}\|F_j\|_\infty}{\big[V(\uu,\bdelta)V(\vv,\bdelta)\big]^{1/2}},
\end{equation}
where for the last inequality we used \eqref{rect-doubling}.

We next estimate $\|F_j\|_\infty$.
For $j=0,1,2$, using (\ref{eq:approx1}), we get
\begin{equation}\label{F012}
\|F_j\|_{\infty}\le \|\phi_j\|_{\infty}\|F-F*\Phi_h\|_{\infty}
\le c h^{k_1+k_2}\|F\|_{W^{k_1+k_2}_\infty}.
\end{equation}
Further, for $j\ge 3$ we use \eqref{eq:approx2} to obtain
\begin{align}\label{Fj}
\|F_j\|_{\infty}&\le \|\phi_j\|_{\infty}\|F-F*\Phi_h\|_{L^\infty(B(0,2^{j+1})\setminus B(0,2^{j-1}))}
\\
&\le c\|F\|_{\infty}h^{k_1+k_2} 2^{-j(k_1+ k_2)}. \nonumber
\end{align}
Combining \eqref{est-KKj}--\eqref{Fj} and using that $k_1+k_2>d_1+d_2$ we obtain
\begin{equation}\label{KFj}
|\KK_{F_j(\bdelta\sqrt L)}(\uu,\vv)|
\le \frac{c(\|F_j\|_\infty + \|F\|_{W^{k_1+k_2}_\infty}) h^{k_1+k_2}2^{j(d_1+d_2-k_1-k_2)}}
{\big[V(\uu,\bdelta)V(\vv,\bdelta)\big]^{1/2}}.
\end{equation}

Denote
$
\KK_{h,\bdelta}(\uu,\vv) := \sum_{j\ge 0} \KK_{F_j(\bdelta \sqrt{L})}(\uu, \vv).
$
Summing up estimates \eqref{KFj} taking into account that $k_1+k_2>d_1+d_2$ we obtain
\begin{equation}\label{KFdelta}
|\KK_{h,\bdelta}(\uu,\vv)|
\le \frac{c^\star(\|F_j\|_\infty + \|F\|_{W^{k_1+k_2}_\infty}) h^{k_1+k_2}}
{\big[V(\uu,\bdelta)V(\vv,\bdelta)\big]^{1/2}},
\quad \forall \uu,\vv\in X, \; 0<h\le 2(\ct_1\vee\ct_2),
\end{equation}
where the constant $c^\star >0$ is independent of $h$.

Since $F$ is compactly supported and $\Phi\in S(\bR^2)$,
then $F*\Phi_h\in S(\bR^2)$.
Moreover, $\widehat{F*\Phi_h}$ is even with respect to each variable because $\Phi_h$ and $F$ are even.
In addition,
\begin{equation*}
\widehat{F\ast \Phi_h}(\xi_1,\xi_2)
=(\widehat{F}\widehat{\Phi_h})(\xi_1,\xi_2)=\widehat F(\xi_1,\xi_2)\widehat\phi(h\xi_1)\widehat\phi(h\xi_2)
\end{equation*}
which implies that
$\supp \widehat{F*\Phi_h}\subset [-1/h, 1/h]^2$.

From Proposition~\ref{prop:kern-cont} it follows that the operator $F*\Phi_h(\bdelta\sqrt{L})$ is an integral operator 
whose kernel is continuous and
\begin{equation*}
|\KK_{F*\Phi_h(\bdelta\sqrt{L})}(\uu,\vv)| \le \frac{c}{[V(\uu,\bdelta)V(\vv,\bdelta)]^{1/2}}.
\end{equation*}

Let $f\in L^2_0(X)$. Recall that $L^2_0(X)$ is the set of all compactly supported $L^2$ functions.
From \eqref{phi-decomp} it follows that
\begin{equation*}
F(\bdelta \sqrt{L})f - F*\Phi_h(\bdelta\sqrt{L})f
= \sum_{j \ge 0} F_j(\bdelta\sqrt{L})f \quad \hbox{in} \;\; L^2.
\end{equation*}
On the other hand, since $ L^2_0(X)\subset L^1(X, V(\cdot, \bdelta)^{-1/2}d\mu)$ it follows from \eqref{KFdelta} that
\begin{equation*}
\sum_{j\ge 0} F_j(\bdelta\sqrt{L})f(\uu) = \int_X  \KK_{h,\bdelta}(\uu,\vv) f(\vv)d\mu(\vv), \quad \forall \uu\in X.
\end{equation*}
Therefore,
\begin{equation*}
F(\bdelta \sqrt{L})f(\uu) - F*\Phi_h(\bdelta\sqrt{L})f(\uu) = \int_X  \KK_{h,\bdelta}(\uu,\vv) f(\vv)d\mu(\vv), \quad \forall \uu\in X,
\end{equation*}
and hence for each $f\in L^2_0(X)$
\begin{align*}
\int_X\KK_{F(\bdelta \sqrt{L})}(\uu,\vv)f(\vv)d\mu(\vv) &- \int_X \KK_{F*\Phi_h(\bdelta\sqrt{L})}(\uu,\vv)f(\vv)d\mu(\vv) 
\\
&= \int_X  \KK_{h,\bdelta}(\uu,\vv) f(\vv)d\mu(\vv), \quad \forall \uu\in X.
\end{align*}
Using the fact that the kernels from above are continuous it follows from this identity that
\begin{equation}\label{K-K-K}
\KK_{F(\bdelta \sqrt{L})}(\uu,\vv) = \KK_{F*\Phi_h(\bdelta\sqrt{L})}(\uu,\vv) + \KK_{h,\bdelta}(\uu,\vv), \;\; \forall \uu,\vv\in X, 0<h\le 2(\ct_1\vee \ct_2).
\end{equation}

To prove \eqref{local-1} we will make use of Proposition~\ref{prop:finite-sp}.
Let $\xx,\yy\in \XX$.
Several cases naturally occur here:

{\em Case 1:} $\rho_i(x_i, y_i)\ge \delta_i$, $i=1,2$.

{\em Case 2:} $\rho_1(x_1, y_1)\ge \delta_1$ and $\rho_2(x_2, y_2)<\delta_2$.

{\em Case 3:} $\rho_1(x_1, y_1) < \delta_1$ and $\rho_2(x_2, y_2)\ge \delta_2$.

{\em Case 4:} $\rho_1(x_1, y_1)<\delta_1$ and $\rho_2(x_2, y_2)<\delta_2$.

We will consider in detail only Case 2. Then we choose $h>0$ so that
\begin{equation}\label{h-delta-rho}
\frac{\rho_1(x_1,y_1)}{2\delta_1}\le  \frac{\ct_1}{h}<\frac{\rho_1(x_1,y_1)}{\delta_1}.
\end{equation}
Hence $0<h\le 2\ct_1$.
As already alluded to above $F*\Phi_h\in S(\bR^2)$
and
$\supp \widehat{F*\Phi_h}\subset [-1/h, 1/h]^2$.
Therefore, in light of \eqref{h-delta-rho} the function
$F*\Phi_h$ satisfies the assumptions
of Proposition~\ref{prop:finite-sp} (a) with $A=1/h$
and hence 
\begin{equation*}
\KK_{F*\Phi_h(\bdelta\sqrt{L})}(\xx,\yy) =0.
\end{equation*}
From this, \eqref{KFdelta} with $\uu=\xx$, $\vv=\yy$, \eqref{K-K-K}, \eqref{h-delta-rho} and the fact that $\rho_2(x_2, y_2)<\delta_2$
it follows that
\begin{equation}\label{KF-main}
|\KK_{F(\bdelta\sqrt{L})}(\xx,\yy)|
\le \frac{c(\|F_j\|_\infty + \|F\|_{W^{k_1+k_2}_\infty})}
{\big[V(\xx,\bdelta)V(\yy,\bdelta)\big]^{1/2}\big(1+\frac{\rho_1(x_1,y_1)}{\delta_1}\big)^{k_1}\big(1+\frac{\rho_2(x_2,y_2)}{\delta_2}\big)^{k_2}},
\end{equation}
which confirms estimate \eqref{local-1}.
The proof of estimate \eqref{local-1} in Case~3 is symmetric and will be omitted.
In Case 1 we may assume that $\rho_1(x_1, y_1)/\delta_1 \ge \rho_2(x_2, y_2)/\delta_2$.
Then we proceed similarly as above using that
\begin{equation*}
h^{k_1+k_2}\le c\Big(1+\frac{\rho_1(x_1,y_1)}{\delta_1}\Big)^{-k_1}\Big(1+\frac{\rho_2(x_2,y_2)}{\delta_2}\Big)^{-k_2}
\end{equation*}
to obtain \eqref{KF-main}.  We omit the further details.
In Case~4 estimate \eqref{local-1} follows by Proposition~\ref{prop:rough-kernels}.

So far we have shown that $F(\bdelta \sqrt{L})$ is a bounded integral operator from $L^2_0(X)$ to $L^\infty(X, V(\cdot, \bdelta)^{1/2})$
with kernel $F(\bdelta \sqrt{L})(\xx,\yy)$ satisfying \eqref{local-1}.
Identifying $F(\bdelta \sqrt{L})$ in general with the operator with this kernel we conclude, using Proposition~\ref{prop:bound-int-oper} and the fact that $\kk>3\dd/2$,
that $F(\bdelta \sqrt{L})$ is a bounded linear operator from $L^p(X, d\mu)$ to $L^p(X, d\mu)$, $1\le p\le \infty$. 
%


To prove \eqref{local-2} we proceed similarly as in the proof of estimate (\ref{eq:smoothrough}).
Define
$
G(\lambda_1,\lambda_2):=F(\lambda_1,\lambda_2)e^{\delta_1^2\lambda_1^2+\delta_2^2\lambda_2^2}.
$
Then
$F(\bdelta\sqrt{L})= G(\bdelta\sqrt{L})e^{-\delta_1^2 L_1-\delta_2^2 L_2}$.
Using \eqref{local-1} and (\ref{prodlip}) we obtain
\begin{align*}
|\KK_{F(\bdelta\sqrt L)}&(\xx,\yy)-\KK_{F(\bdelta\sqrt L)}(\xx,\yy')|
\\
&=\Big|\int_\XX  \KK_{G(\bdelta\sqrt{L})}(\xx,\zz)
\big(p_{(\delta_1^2,\delta_2^2)}(\zz,\yy) - p_{(\delta_1^2,\delta_2^2)}(\zz,\yy')\big)d\zz\Big|
\\
&\le  cc_\kk\sum_{i=1,2}\big(\frac{\rho_i(y_i,y_i ')}{\delta_i}\big)^{\alpha_i}
\int_\XX \DD_{\bdelta,\kk}(\xx,\zz)\DD_{\bdelta,\kk}(\zz,\yy) d\mu(\zz)
\\
&\le cc_\kk\sum_{i=1,2}\Big(\frac{\rho_i(y_i,y_i ')}{\delta_i}\Big)^{\alpha_i}\DD_{\bdelta,\kk}(\xx,\yy),
\end{align*}
where for the last inequality we used  \eqref{tech-3}.
Thus estimate \eqref{local-2} is established.

\smallskip

To prove \eqref{int-K-1} we introduce the function
$H(\lambda_1, \lambda_2):= e^{|\lambda_1|+|\lambda_2|}F(\sqrt{|\lambda_1|}, \sqrt{|\lambda_2|})$.
Then
$F(\lambda_1, \lambda_2)=H(\lambda_1^2, \lambda_2^2)e^{-\lambda_1^2-\lambda_2^2}$
and hence
$F(\tt\sqrt{L})= H(\tt^2L)e^{-t_1^2L_1-t_2^2L_2}$.
We claim that
\begin{equation}\label{Fourier-H}
\widehat{H}(\xi_1,\xi_2) = O\Big(\prod_{i=1,2}(1+|\xi_i|)^{-3/2}\Big).
\end{equation}
To prove this we use that $\supp H \subset [-R^2,R^2]^2$ and $F(\pm\lambda_1, \pm\lambda_2) = F(\lambda_1, \lambda_2)$
to write
\begin{align*}
\widehat{H}(\xi_1,\xi_2)&:= \int_{[-a,a]^2}H(\lambda_1, \lambda_2)e^{-i(\lambda_1\xi_1+\lambda_2\xi_2)} d\lambda_1 d\lambda_2
\\
&= 4\int_{[0,a]^2}H(\lambda_1, \lambda_2) \cos\lambda_1\xi_1 \cos\lambda_2\xi_2 \, d\lambda_1 d\lambda_2,
\quad a:=R^2.
\end{align*}
Integrating by parts with respect to $\lambda_1$ and $\lambda_2$ we obtain
\begin{equation*} 
\widehat{H}(\xi_1,\xi_2) = \frac{1}{\xi_1\xi_2}\int_{[0,a]^2} 4\frac{\partial^2 H}{\partial \lambda_1\partial \lambda_2}(\lambda_1, \lambda_2)
\sin\lambda_1\xi_1 \sin\lambda_2\xi_2 \, d\lambda_1 d\lambda_2.
\end{equation*}
Evidently,
\begin{align*}
&4\frac{\partial^2 H}{\partial \lambda_1\partial \lambda_2}(\lambda_1, \lambda_2)
= 4e^{\lambda_1+\lambda_2}F(\sqrt{\lambda_1}, \sqrt{\lambda_2})
+ 2\lambda_1^{-1/2}e^{\lambda_1+\lambda_2}\frac{\partial F}{\partial \lambda_1}(\sqrt{\lambda_1}, \sqrt{\lambda_2})
\\
&\quad + 2\lambda_2^{-1/2}e^{\lambda_1+\lambda_2}\frac{\partial F}{\partial \lambda_2}(\sqrt{\lambda_1}, \sqrt{\lambda_2})
+ \lambda_1^{-1/2}\lambda_2^{-1/2}e^{\lambda_1+\lambda_2}\frac{\partial^2 F}{\partial \lambda_1 \partial \lambda_2}(\sqrt{\lambda_1}, \sqrt{\lambda_2})
\\
&\quad =: W_1(\lambda_1, \lambda_2)+W_2(\lambda_1, \lambda_2)+W_3(\lambda_1, \lambda_2)+W_4(\lambda_1, \lambda_2).
\end{align*}
We next show that
\begin{equation}\label{Fourier-W4}
\int_{[0,a]^2} W_4(\lambda_1, \lambda_2)\sin\lambda_1\xi_1 \sin\lambda_2\xi_2 \, d\lambda_1 d\lambda_2
= O\Big(\prod_{i=1,2}(1+|\xi_i|)^{-1/2}\Big).
\end{equation}
One similarly shows that for $i=1,2,3$
\begin{equation}\label{Fourier-Wi}
\int_{[0,a]^2} W_i(\lambda_1, \lambda_2)\sin\lambda_1\xi_1 \sin\lambda_2\xi_2 \, d\lambda_1 d\lambda_2
= O\Big(\prod_{i=1,2}(1+|\xi_i|)^{-1/2}\Big);
\end{equation}
we omit the details.

Observe that from the condition $F(\pm \lambda_1, \pm\lambda_2) = F(\lambda_1,\lambda_2)$ it follows that 
\begin{equation*}
\frac{\partial F}{\partial \lambda_1} (0, \lambda_2) = \frac{\partial F}{\partial \lambda_2} (\lambda_1,0) = 0 
\end{equation*} 
and hence
\begin{equation}\label{deriv-F}
\frac{\partial^{k+1} F}{\partial \lambda_1\partial^k\lambda_2} (0, \lambda_2) = \frac{\partial^{k+1} F}{\partial \lambda_1^k\partial \lambda_2} (\lambda_1,0) = 0,
\quad k\in\bN. 
\end{equation} 
We use the above to obtain
\begin{align*}
\frac{\partial^2 F}{\partial \lambda_1\partial\lambda_2} (\lambda_1, \lambda_2)
&= \frac{\partial^2 F}{\partial \lambda_1\partial\lambda_2} (0, \lambda_2) 
+ \lambda_1\int_0^1 \frac{\partial^3 F}{\partial \lambda_1^2\partial\lambda_2} (t\lambda_1, \lambda_2) dt
\\
& = \lambda_1\int_0^1 \frac{\partial^3 F}{\partial \lambda_1^2\partial\lambda_2} (t\lambda_1, \lambda_2) dt.
\end{align*}
Similarly, using again \eqref{deriv-F} we get
\begin{align*}
\frac{\partial^3 F}{\partial \lambda_1^2 \partial\lambda_2} (t\lambda_1, \lambda_2)
&= \frac{\partial^3 F}{\partial \lambda_1^2 \partial\lambda_2} (t\lambda_1, 0) 
+ \lambda_2\int_0^1 \frac{\partial^4 F}{\partial \lambda_1^2\partial\lambda_2^2} (t\lambda_1, s\lambda_2) ds
\\
& = \lambda_2\int_0^1 \frac{\partial^4 F}{\partial \lambda_1^2\partial\lambda_2^2} (t\lambda_1, s\lambda_2) ds.
\end{align*}
Combining the above identities we arrive at
\begin{equation*}
\frac{\partial^2 F}{\partial \lambda_1\partial\lambda_2} (\lambda_1, \lambda_2) 
= \lambda_1 \lambda_2\int_0^1 \int_0^1 \frac{\partial^4 F}{\partial \lambda_1^2\partial\lambda_2^2} (t\lambda_1, s\lambda_2) dt ds.
\end{equation*}
In turn this yields
\begin{equation}\label{est-W4}
|W_4(\lambda_1, \lambda_2)| \le c\Big\|\frac{\partial^4 F}{\partial \lambda_1^2\partial\lambda_2^2}\Big\|_\infty, \quad \blambda>\zero.
\end{equation}

Consider the case when $\xi_1, \xi_2> 1/a$; the other cases are easier to handle.
To prove \eqref{Fourier-W4} we split the set of integration as follows:
\begin{align*}
[0,a]^2 &= \Big[0, \frac{1}{\xi_1}\Big]\times \Big[0, \frac{1}{\xi_2}\Big] \cup \Big[0, \frac{1}{\xi_1}\Big]\times \Big[\frac{1}{\xi_2}, a\Big]
\cup \Big[\frac{1}{\xi_1}, a\Big]\times \Big[0, \frac{1}{\xi_2}\Big]\cup \Big[\frac{1}{\xi_1}, a\Big]\times \Big[\frac{1}{\xi_2}, a\Big]
\\
&=: A_1\cup A_2 \cup A_3 \cup A_4.
\end{align*}

From \eqref{est-W4} it readily follows that
\begin{equation}\label{Fourier-W4A1}
J_1:=\int_{A_1} W_4(\lambda_1, \lambda_2)\sin\lambda_1\xi_1 \sin\lambda_2\xi_2 \, d\lambda_1 d\lambda_2
= O\Big(\prod_{i=1,2}(1+|\xi_i|)^{-1}\Big).
\end{equation} 

We next show that
\begin{equation}\label{Fourier-W4A4}
J_4:=\int_{A_4} W_4(\lambda_1, \lambda_2)\sin\lambda_1\xi_1 \sin\lambda_2\xi_2 \, d\lambda_1 d\lambda_2
= O\Big(\prod_{i=1,2}(1+|\xi_i|)^{-1/2}\Big).
\end{equation}
The integrals over $A_2$ and $A_3$ are estimated similarly.

Integrating by parts with respect to $\lambda_1$ and $\lambda_2$ and using that $W_4(\lambda_1,a) = W_4(a, \lambda_2) = 0$
we get
\begin{align*}
J_4&= \frac{1}{\xi_1\xi_2}\Big(
W_4\big(\xi_1^{-1}, \xi_2^{-1}\big)\cos^2 1
+ \int_{1/\xi_1}^a \frac{\partial W_4}{\partial \lambda_1}\big(\lambda_1, \xi_2^{-1}\big)\cos \lambda_1\xi_1 \cos 1 \, d\lambda_1
\\
&+ \int_{1/\xi_2}^a \frac{\partial W_4}{\partial \lambda_2}\big(\xi_1^{-1}, \lambda_2\big)\cos 1\cos \lambda_2\xi_2 \, d\lambda_2
\\
&+ \int_{1/\xi_1}^a \int_{1/\xi_2}^a \frac{\partial^2 W_4}{\partial \lambda_1\partial \lambda_2}(\lambda_1, \lambda_2)
\cos \lambda_1\xi_1 \cos \lambda_2\xi_2 \, d\lambda_1 d\lambda_2
\Big)
\\
&=: \frac{1}{\xi_1\xi_2}(J_{41} + J_{42} + J_{43} + J_{44}).
\end{align*}
From \eqref{est-W4} it follows that 
$|J_{41}| \le c$.
We next estimate $ J_{44}$. The estimation of $J_{41}, J_{42}, J_{43}$ is similar.
One easily derives that
\begin{align*}
&\frac{\partial^2 W_4}{\partial \lambda_1\partial \lambda_2}(\lambda_1, \lambda_2)
\\
&=\Big(\frac{1}{4}\lambda_1^{-3/2}\lambda_2^{-3/2} - \frac{1}{2}\lambda_1^{-3/2}\lambda_2^{-1/2}  - \frac{1}{2}\lambda_1^{-1/2}\lambda_2^{-3/2}\Big)
e^{\lambda_1+\lambda_2}\frac{\partial^2 F}{\partial \lambda_1\partial\lambda_2}(\sqrt{\lambda_1}, \sqrt{\lambda_2})
\\
&\hspace{1in}\, +\Big(-\frac{1}{4}\lambda_1^{-3/2}\lambda_2^{-1} + \frac{1}{2}\lambda_1^{-1}\lambda_2^{-1/2}\Big)
e^{\lambda_1+\lambda_2}\frac{\partial^3 F}{\partial \lambda_1^2\partial\lambda_2}(\sqrt{\lambda_1}, \sqrt{\lambda_2})
\\
&\hspace{1in}\quad\;\, +\Big(\frac{1}{2}\lambda_1^{-1/2}\lambda_2^{-1} - \frac{1}{4}\lambda_1^{-1}\lambda_2^{-3/2}\Big)
e^{\lambda_1+\lambda_2}\frac{\partial^3 F}{\partial \lambda_1\partial\lambda_2^2}(\sqrt{\lambda_1}, \sqrt{\lambda_2})
\\
&\hspace{2in}\qquad\; + \frac{1}{4}\lambda_1^{-1}\lambda_2^{-1}
e^{\lambda_1+\lambda_2}\frac{\partial^4 F}{\partial \lambda_1^2\partial\lambda_2^2}(\sqrt{\lambda_1}, \sqrt{\lambda_2})
\end{align*}
and using that $F\in \cC^4(\R^2)$ and $F$ is compactly supported we get
$J_{44} = O(\xi_1^{1/2}\xi_2^{1/2})$.
This along with the similar estimates for $J_{42}, J_{43}$ and the estimate $|J_{41}| \le c$ imply \eqref{Fourier-W4A4}.
In turn, \eqref{Fourier-W4A1}, \eqref{Fourier-W4A4}, and the similar estimates for the integrals over $A_2$ and $A_3$
yield \eqref{Fourier-W4}.
Finally, \eqref{Fourier-W4} and \eqref{Fourier-Wi} imply \eqref{Fourier-H}.


From \eqref{Fourier-H} it follows that $\widehat{H}\in L^1(\R^2)$ and obviously $H\in L^1(\R^2)$.
Then the Fourier inversion formula gives
\begin{equation}\label{Fourier-inv}
H(\blambda) = \frac{1}{(2\pi)^2}\int_{\R^2} \widehat H(\bxi)e^{i\blambda\cdot\bxi} d\bxi.
\end{equation}
Assuming that $\varphi, \psi \in L^2(\XX)\cap L^1(\XX)$ we get using that
$F(\tt\sqrt{L})= H(\tt^2L)e^{-\tt^2L}$ and \eqref{Fourier-inv}
\begin{align*}
\langle F(\tt\sqrt{L})\varphi, &\psi \rangle
\\
&= \int_{[0,\infty)^2} H(t_1^2\lambda_1, t_2^2\lambda_2)e^{-t_1^2\lambda_1-t_2^2\lambda_2}
d\langle E_{(\lambda_1,\lambda_2)} \varphi, \psi\rangle
\\
&= \int_{[0,\infty)^2} \Big(\frac{1}{(2\pi)^2}\int_{\R^2}
\widehat H(\bxi)e^{i(t_1^2\lambda_1\xi_1+t_2^2\lambda_2\xi_2)} d\bxi\Big)e^{-t_1^2\lambda_1-t_2^2\lambda_2}
d\langle E_{(\lambda_1,\lambda_2)} \varphi, \psi\rangle
\\
&= \frac{1}{(2\pi)^2}\int_{\R^2} \widehat H(\bxi) \Big(\int_{[0,\infty)^2} e^{-t_1^2\lambda_1(1-i\xi_1)-t_2^2\lambda_2(1-i\xi_2)}
d\langle E_{(\lambda_1,\lambda_2)} \varphi, \psi\rangle\Big) d\bxi
\\
&= \frac{1}{(2\pi)^2}\int_{\R^2} \widehat H(\bxi)
\Big\langle e^{-t_1^2\lambda_1(1-i\xi_1)L_1-t_2^2\lambda_2(1-i\xi_2)L_2} \varphi, \psi\Big\rangle d\bxi.
\end{align*}
Above the shift of order of integration is justified by Fubini's theorem and the fact that
\begin{align*}
\int_{\R^2} |\widehat H(\bxi)|& \Big(\int_{[0,\infty)^2} \big|e^{-t_1^2\lambda_1(1-i\xi_1)-t_2^2\lambda_2(1-i\xi_2)}\big|
|d\langle E_{(\lambda_1,\lambda_2)} \varphi, \psi\rangle|\Big) d\bxi
\\
&\le \int_{\R^2} |\widehat H(\bxi)| \Big(\int_{[0,\infty)^2}
|d\langle E_{(\lambda_1,\lambda_2)} \varphi, \psi\rangle|\Big) d\bxi
\\
&\le \int_{\R^2} |\widehat H(\bxi)|d\bxi \Big(\int_{[0,\infty)^2} d\|E_{(\lambda_1,\lambda_2)} \varphi\|^2\Big)^{1/2}
\Big(\int_{[0,\infty)^2} d\|E_{(\lambda_1, \lambda_2)} \psi\|^2\Big)^{1/2}
\\
&\le \|\widehat H\|_{L^1}\|\varphi\|_{L^2}\|\psi\|_{L^2} <\infty.
\end{align*}
In going further we get from above
\begin{align*}
&\langle F(\tt\sqrt{L})\varphi, \overline{\psi} \rangle
\\
&= \frac{1}{(2\pi)^2}\int_{\R^2} \widehat H(\bxi)
\Big\langle e^{-t_1^2\lambda_1(1-i\xi_1)L_1-t_2^2\lambda_2(1-i\xi_2)L_2} \varphi, \overline{\psi}\Big\rangle d\xi \nonumber
\\
& =\frac{1}{(2\pi)^2}\int_{\R^2} \widehat H(\bxi)
\int_\XX\int_\XX p_{(t_1^2(1-i\xi_1),t_2^2(1-i\xi_2))}(\xx,\yy) \varphi(\yy)\psi(\xx) d\mu(\yy)d\mu(\xx) d\bxi\nonumber
\\
& =\int_\XX\int_\XX
\Big(\frac{1}{(2\pi)^2}\int_{\R^2} \widehat H(\bxi) p_{(t_1^2(1-i\xi_1),t_2^2(1-i\xi_2))}(\xx,\yy)d\bxi\Big)
\varphi(\yy)\psi(\xx) d\mu(\yy)d\mu(\xx) \nonumber
\end{align*}
To justify the shift of order of integration above we apply Fubini's theorem taking into account that
\begin{align*}
\int_{\R^2} \int_\XX\int_\XX&
|\widehat H(\bxi)| |p_{(t_1^2(1-i\xi_1),t_2^2(1-i\xi_2))}(\xx,\yy)| |\varphi(\yy)||\psi(\xx)| d\mu(\yy)d\mu(\xx) d\bxi
\\
&\le c\|\widehat H\|_{L^1}\|\varphi\|_{L^1} \|\psi\|_{L^1} <\infty.
\end{align*}
From above it follows that
\begin{equation*}
\KK_{F(\tt\sqrt{L})}(\xx,\yy)=\frac{1}{(2\pi)^2}\int_{\R^2} \widehat H(\xi)p_{(t_1^2(1-i\xi_1),t_2^2(1-i\xi_2))}(\xx,\yy)d\bxi,
\end{equation*}
implying
\begin{align*}
\int_\XX\KK_{F(\tt\sqrt{L})}(\xx,\yy) d\mu(\yy)&=\frac{1}{(2\pi)^2}\int_{\R^2} \widehat H(\bxi)
\int_\XX p_{(t_1^2(1-i\xi_1),t_2^2(1-i\xi_2))}(\xx,\yy)d\yy d\xi
\\
&=\frac{1}{(2\pi)^2}\int_{\R^2} \widehat H(\bxi) d\bxi = H(\zero)=F(\zero),
\end{align*}
where we used \eqref{markov-2}, \eqref{Gauss-local-z}, \eqref{prodhk-z}, and \eqref{Fourier-inv}.
The shift of the order of integration above is easily justified by applying Fubini's theorem.
The proof of \eqref{int-K-1} is complete.
\end{proof}

We next extend Theorem~\ref{th:boxsupport} from the case of compactly supported functions
to the case of smooth ones with sufficiently fast decay on $\bR^2$.

\begin{theorem}\label{thm:gen-local}
Let $F\in \cC^{k_1+k_2}(\bR^2)$, $k_i\in\bN$, $k_i\ge 3d_i/2$, $i=1,2$,
$k_1+k_2\ge 4$, 
be real-valued and
obey the conditions:
$F(\pm \lambda_1, \pm\lambda_2)=F(\lambda_1,\lambda_2)$ for $(\lambda_1,\lambda_2)\in \bR^2$
and for $\bbeta\in \bN_0^2$ such that $|\bbeta| \le k_1+k_2$
\begin{equation*}
|\partial^{\bbeta}F(\blambda)|\le C_\kk(1+|\blambda|)^{-r},\quad r>d_1+d_2+k_1+k_2.
\end{equation*}
Then $F(\bdelta \sqrt L)$, $\bdelta=(\delta_1,\delta_2)> \zero$, is an integral operator
with kernel $\KK_{F(\bdelta \sqrt L)}(\xx,\yy)$ 
satisfying 
\begin{equation}\label{K-local}
|\KK_{F(\bdelta \sqrt L)}(\xx,\yy)|\le c_\kk \DD_{\bdelta,\kk}(\xx, \yy) \quad \text{and}
\end{equation}
\begin{equation}\label{K-lip}
|\KK_{F(\bdelta \sqrt L)}(\xx,\yy)- \KK_{F(\bdelta \sqrt L)}(\xx,\yy')|
\le c_\kk' \sum_{i=1,2}\Big(\frac{\rho_i(y_i, y'_i)}{\delta_i}\Big)^{\alpha_i}\DD_{\bdelta,\kk}(\xx,\yy)
\end{equation}
if $\rho_i(y_i,y_i')\le \delta_i$, $i=1,2$, where $\kk=(k_1,k_2)$
and the constants $c_\kk,c_{\kk'}$ depend on $k_1$, $k_2$, $\dd$ and $C_\kk$.
Moreover,
\begin{equation}\label{int-K-11}
\int_\XX \KK_{F(\bdelta \sqrt L)}(\xx,\yy) d\mu(\yy)=F(\zero), \quad \forall \xx\in\XX.
\end{equation}
\end{theorem}

\begin{proof}
Let  $\ph_0,\ph\in \cC^{\infty}(\R^2)$ be functions satisfying conditions \eqref{dec-unity}.
Set $\ph_j(\blambda):=\ph(2^{-j}\blambda)$, $j\ge 1$, and $F_j(\blambda):=F(\blambda)\ph_j(\blambda)$, $j\ge 0$.
Then
\begin{equation}\label{F-Fj}
F(\bdelta\sqrt  L)=\sum_{j\ge 0}F(\bdelta\sqrt L) \ph_j(\bdelta\sqrt L)
=\sum_{j\ge 0}F_j(\bdelta\sqrt L),
\end{equation}
where the convergence is strong in $L^2$.
Observe that $F_j(\pm\lambda_1, \pm\lambda_2)=F_j(\lambda_1, \lambda_2)$
and for $j\ge 1$ we have
$\supp F_j\subset \Omega_j:=\{\blambda: 2^{j-1}\le |\blambda|\le 2^{j+1}\}$
and
\begin{equation*}
\|F_j\|_{L^\infty} \le \|F\|_{L^\infty(\Omega_j)}\le c 2^{-jr},
\quad
\|F_j\|_{W^{k_1+k_2}_\infty}\le c \|F\|_{W^{k_1+k_2}_\infty(\Omega_j)}\le c2^{-jr}.
\end{equation*}
We now apply Theorem~\ref{th:boxsupport} to $F_j$ with $R=2^{j+1}$ to obtain
\begin{align}\label{KF-KFj}
&|\KK_{F(\bdelta\sqrt L)}(\xx,\yy)|\le  \sum_{j\ge 0}|\KK_{F_j(\bdelta\sqrt L)}(\xx,\yy)| \nonumber
\\
&\le c\sum_{j\ge 0} 2^{j(d_1+d_2)}\big(\|F_j\|_{L^\infty}+2^{j(k_1+k_2)}\|F_j\|_{W^{k_1+k_2}_\infty}\big)\DD_{\bdelta,\kk}(\xx,\yy)
\\
&\le c\sum_{j\ge 0} 2^{j(d_1+d_2)}2^{-jr}2^{j(k_1+k_2)}\DD_{\bdelta,k}(\xx,\yy)
\le c \DD_{\bdelta,\kk}(\xx,\yy), \nonumber
\end{align}
where for the last inequality we used that $r>(d_1+d_2)+(k_1+k_2)$.
Therefore, $\KK_{F(\bdelta \sqrt L)}(\xx,\yy)$ satisfies \eqref{K-local}.

The proof that $\KK_{F(\bdelta \sqrt L)}(\xx,\yy)$ obeys \eqref{K-lip} is similar.
We omit it.

To prove \eqref{int-K-11} we use \eqref{F-Fj} and \eqref{KF-KFj} to write
\begin{equation*}
\KK_{F(\bdelta\sqrt  L)}(\xx,\yy)=\sum_{j\ge 0}\KK_{F_j(\bdelta\sqrt L)}(\xx, \yy)
\end{equation*}
and integrating we get
\begin{equation*}
\int_\XX \KK_{F(\bdelta\sqrt  L)}(\xx,\yy) d\mu(\yy)=\sum_{j\ge 0}\int_\XX \KK_{F_j(\bdelta\sqrt L)}(\xx, \yy) d\mu(\yy)
= \sum_{j\ge 0} F_j(\one) = F(\one).
\end{equation*}
Here we applied \eqref{int-K-1} to each of the kernels $\KK_{F_j(\bdelta\sqrt L)}(\xx, \yy)$ using the fact that $F_j$ is compactly supported.
This completes the proof of the theorem.
\end{proof}

We will frequently need the decay of kernels of operators of the form
\begin{equation*}
L^{\bnu}F(\bdelta\sqrt{L}):=(L_1^{\nu_1}\otimes L_2^{\nu_2})F(\bdelta\sqrt{L}),
\quad \bnu=(\nu_1,\nu_2)\in\bN_0^2.
\end{equation*}
In this regard we have the following corollary of Theorem \ref{thm:gen-local}.

\begin{corollary}\label{cor:gen-local}
Let $F\in \cC^{\infty}(\bR^2)$ be real-valued,
$F(\pm \lambda_1, \pm\lambda_2)=F(\lambda_1,\lambda_2)$ for $(\lambda_1,\lambda_2)\in \bR^2$,
and
\begin{equation*}
|\partial^{\bbeta}F(\blambda)|\le c_{r,\bbeta}(1+|\blambda|)^{-r},
\quad \forall r>0, \; \bbeta\in \bN_0^2.
\end{equation*}
Then for any $\bnu\in \bN_0^2$  and $\bdelta>\zero$,
$L^{\bnu}F(\bdelta \sqrt L)$ is an integral operator with
kernel $\KK_{L^{\bnu}F(\bdelta \sqrt L)}(\xx, \yy)$ satisfying the following inequalities:
For any $\bsigma\in \R^2_+$ there exists a constant $c_\bsigma>0$ such that
\begin{equation}\label{local-ker-nu}
\big|\KK_{L^{\bnu}F(\bdelta \sqrt L)}(\xx, \yy)\big|
\le c_\bsigma\bdelta^{-2\bnu}\DD_{\bdelta,\bsigma}(\xx,\yy)
\end{equation}
and
\begin{equation}\label{lip-ker-nu}
\big|\KK_{L^{\bnu}F(\bdelta \sqrt L)}(\xx, \yy)-\KK_{L^{\bnu}F(\bdelta \sqrt L)}(\xx, \yy')\big|
\le c_\bsigma\bdelta^{-2\bnu}\sum_{i=1,2}\Big(\frac{\rho_i(y_i, y'_i)}{\delta_i}\Big)^{\alpha_i}\DD_{\bdelta,\bsigma}(\xx,\yy)
\end{equation}
if $\rho_i(y_i,y'_i)\le \delta_i$, $i=1,2$.
\end{corollary}
\begin{proof}
Set  $G(\lambda_1,\lambda_2):=\lambda_1^{2\nu_1}\lambda_2^{2\nu_2}F(\lambda_1,\lambda_2)$
and observe that
\begin{equation*}
G(\bdelta\sqrt{L})=\delta_1^{2\nu_1}\delta_2^{2\nu_2}L^{\bnu}F(\bdelta \sqrt L).
\end{equation*}
It is easily seen  that $G$ satisfies the assumptions of Theorem~\ref{thm:gen-local}
and, therefore, applying that theorem to $G$ the result follows.
\end{proof}

\subsection{Maximal operators}\label{subsec:maxop}

The Hardy-Littlewood maximal operator is a handful tool in proving estimates in Harmonic analysis.
We will use it extensively in what follows.

In our setting the (strong) maximal operator $\cM_r$ with parameter $r>0$ is defined by
\begin{equation}\label{Max1}
\cM_r f(\xx):= \sup_{B_1\times B_2 \ni \xx} \Big(\frac{1}{\mu(B_1\times B_2)}\int_{B_1\times B_2}|f(\yy)|^r d\mu(\yy)\Big)^{1/r},
\quad \xx\in \XX,
\end{equation}
where the $\sup$ is taken over all balls $B_1\subset \XX_1$, $B_2\subset \XX_2$
such that $\xx\in B_1\times B_2$.

The maximal operator $\sM_1$ in the direction of $x_1\in \XX_1$ is defined by
\begin{equation*} 
\sM_1 f(\xx)=\sup_{B_1\ni x_1} \frac{1}{\mu_1(B_1)} \int_{B_1} |f(y_1,x_2)| d\mu_1(y_1),\;\;\xx\in \XX,
\end{equation*}
where the $\sup$ is taken over all  balls $B_1$ in $\XX_1$ containing $x_1$.
The maximal operator $\sM_2$ in the direction of $x_2\in \XX_2$ is defined similarly.
It is easy to see that this inequality is valid:
\begin{equation}\label{Max11}
\cM_r f(\xx)\le\left(\sM_2(\sM_1|f|^r)\right)^{1/r}(\xx), \quad \xx\in \XX.
\end{equation}
We will utilize the following version of the Fefferman-Stein vector-valued maximal inequality:
If $p\in(0,\infty),\;q\in(0,\infty]$, and $0<r<\min\{p,q\}$,
then for any sequence of functions $\{f_j\}$ on $\XX$
\begin{equation}\label{max}
\Big\|\Big(\sum_{j}|\cM_r f_j|^q\Big)^{1/q}\Big\|_{p}
\le c\Big\|\Big(\sum_{j}|f_j|^q\Big)^{1/q}\Big\|_{p}.
\end{equation}
From the fact that the Radon measure $\mu_i$ ($i=1,2$) obeys the doubling condition \eqref{eq:doubling-0}
it follows that the maximal inequality is valid in $(\XX_i, \mu_i)$ (see \cite{Stein} and also \cite{GLY})
and then \eqref{max} follows by iteration.

For future use we next prove the following

\begin{lemma}\label{lem:max}
Let $\bdelta, \btau\in (0,\infty)^2$, $\btau > \dd$, and $r>0$.
Then there exists a constant $c>0$ such that for any measurable function $f$ on $\XX$ we have
\begin{equation}\label{int-max}
\int_\XX\frac{|f(\yy)|^r}{V(\xx, \bdelta)\prod_{i=1,2}\big(1+\delta_i^{-1}\rho_i(x_i, y_i)\big)^{\tau_i}} d\mu(\yy)
\le c[\cM_rf(\xx)]^r, \quad \xx\in \XX.
\end{equation}
\end{lemma}
\begin{proof}
We introduce the abbreviated notation
\begin{equation*}
F(\xx,\yy):= \frac{|f(\yy)|^r}{V(\xx, \bdelta)\prod_{i=1,2}\big(1+\delta_i^{-1}\rho_i(x_i, y_i)\big)^{\tau_i}}
\end{equation*}
and
$B_{i,j}:= B_i(x_i, 2^j\delta_i)$, $i=1,2$.
Set
\begin{equation*}
W_{i,0}:= B_{i,0}
\quad\hbox{and}\quad
W_{i,j}:= B_{i,j}\setminus B_{i, j-1}, \quad j\ge 1.
\end{equation*}
Note that
$\XX= \cup_{j, k \ge 0} W_{1,j}\times W_{2,k}$
and
$W_{1,j}\times W_{2,k} \subset B_{1,j}\times B_{2,k}$.
Therefore,
\begin{equation}\label{int-F}
\int_\XX F(\xx,\yy) d\mu(\yy)
= \sum_{j,k\ge 0} \int_{W_{1,j}\times W_{2,k}} F(\xx,\yy) d\mu(\yy).
\end{equation}
Observe that
\begin{equation*}
F(\xx,\yy) \le c2^{-j\tau_1-k\tau_2}V(\xx, \bdelta)^{-1}|f(\yy)|^r, \quad \yy\in W_{1,j}\times W_{2,k},
\end{equation*}
and hence
\begin{align}\label{int-WW}
&\int_{W_{1,j}\times W_{2,k}} F(\xx,\yy) d\mu(\yy)
\le c2^{-j\tau_1-k\tau_2}V(\xx, \bdelta)^{-1}\int_{B_{1,j}\times B_{2,k}} |f(\yy)|^r d\mu(\yy) \nonumber
\\
&= c2^{-j\tau_1-k\tau_2}
\frac{\mu(B_{1,j}\times B_{2,k})}{V(\xx, \bdelta)}
\frac{1}{\mu(B_{1,j}\times B_{2,k})}\int_{B_{1,j}\times B_{2,k}} |f(\yy)|^r d\mu(\yy)
\\
&\le c2^{-j(\tau_1-d_1)-k(\tau_2-d_2)}\cM_rf(\xx)^r, \nonumber
\end{align}
where we used that in light of \eqref{doubling-d} and \eqref{def-V}
\begin{align*}
\mu(B_{1,j}\times B_{2,k})
&= \mu_1(B_1(x_1, 2^j\delta_1))\mu_2(B_2(x_2, 2^k\delta_2))
\\
&\le c2^{jd_1+kd_2} \mu_1(B_1(x_1,\delta_1))\mu_2(B_2(x_2,\delta_2))
= c2^{jd_1+kd_2}V(\xx, \bdelta).
\end{align*}
Finally, we insert \eqref{int-WW} in \eqref{int-F} and sum up using that $\tau_i>d_i$
to obtain \eqref{int-max}.
\end{proof}

\section{Distributions associated to operators}\label{sec:distributions}

Distributions associated to operators have been introduced in \cite{KP}. 
Here we develop distributions in the two-parameter setting of this article.

\subsection{Test functions and distributions}\label{subsec:test-dist}

We will operate in the setting described in Subsection~\ref{sec:genset}.
We define the class of test functions $\cS:= \cS(L_1,L_2)$
as the set of all function
\begin{equation*}
\phi\in \cap_{\bnu\in\bN_0^2} D(L^{\bnu}),
\quad L^{\bnu}:=L_1^{\nu_1}\otimes L_2^{\nu_2},\; \bnu=(\nu_1, \nu_2),
\end{equation*}
with $D(L^{\bnu})$ being the domain of the operator $L^{\bnu}$, such that
\begin{equation}\label{norm-S}
\cP_{m,k}(\phi)
:= \sup_{\xx\in \XX}\prod_{i=1,2}\big(1+\rho_i(x_i,x_{0i})\big)^{k} \max_{0\le \nu_i\le m} \big|L^{\bnu}\phi(\xx)\big| <\infty,
\quad \forall\, k,m\in\bN_0.
\end{equation}
Here the point $\xx_0=(x_{01},x_{02})\in \XX$ is selected arbitrarily and fixed once and for all.
Clearly, the definition of the class $\cS$ is independent of the choice of $\xx_0$.

It should be pointed out that the term $\prod_{i=1,2}\big(1+\rho_i(x_i,x_{0i})\big)^{k}$ in \eqref{norm-S}
can be removed if $\mu(\XX_i) <\infty$ for $i=1$ or $i=2$
because $\mu(\XX_i) <\infty$ if and only if $\diam(\XX_i) <\infty$, see \cite[Proposition~2.1]{CKP}.

\begin{proposition}\label{prop:Frechet}
The vector space $\cS:= \cS(L_1,L_2)$ with the natural topology induced
by the norms $\cP_{m,k}(\phi)$, $m,k\in\bN_0$,
is a Fr\'echet space.
\end{proposition}
\begin{proof}
Clearly, $\{\cP_{m,k}(\phi)\}_{m,k\ge 0}$ is a countable family of norms
and hence $\cS$ is a metrizable locally convex space.
Therefore, to prove that $\cS$ is Fr\'echet we only need to show that it is complete.
Assume that $\{\phi_n\}_{n\ge 1}$ is a Cauchy sequence in $\cS$,
i.e. $\lim_{n,j\rightarrow\infty}\cP_{m,k}(\phi_n-\phi_j)=0$ for all $m,k\ge 0$.
Fix $k\ge(d_1\vee d_2+1)/2$. Then for any $\bnu\in \bN_0^2$ with $0\le \nu_i\le m$ we have
\begin{align*}
\big\|L^{\bnu}(\phi_n-\phi_j)\big\|_{2}
&\le \cP_{m,k}(\phi_n-\phi_j)\int_{\XX_1\times \XX_2}\prod_{i=1,2}\big(1+\rho_i(x_i,x_{0i})\big)^{-2k} d\mu_1\times d\mu_2
\\
& \le cV_1(x_{01}, 1)V_2(x_{02}, 1)\cP_{m,k}(\phi_n-\phi_j),
\end{align*}
where we used \eqref{int-ineq-1}.
Therefore, $\{L^\bnu\phi_n\}_{n\ge 1}$ is a Cauchy sequence in $L^2(\XX, d\mu)$
and hence there exists
$\psi_\bnu\in L^2(\XX, d\mu)$ such that
$\lim_{n\to\infty}\|L^{\bnu}\phi_n -\psi_\bnu\|_{2}=0$.

Denote $\phi:= \psi_{(0,0)}$ and let $\bnu\in\bN_0^2$.
From
\begin{equation*}
\lim_{n\to\infty}\|\phi_n -\phi\|_{2}=0
\quad\hbox{and}\quad
\lim_{n\to\infty}\|L^{\bnu}\phi_n -\psi_\bnu\|_{2}=0
\end{equation*}
and the fact that the operator
$L^{\bnu}=L_1^{\nu_1}\otimes L_2^{\nu_2}$ is closed (as a self-adjoint operator, see Proposition~\ref{prop:L1xL2})
it follows that $\phi\in D(L^\bnu)$ and
\begin{equation}\label{Frechet-2}
\lim_{n\to\infty}\|L^{\bnu}\phi_n -L^{\bnu}\phi\|_{2}=0.
\end{equation}
Furthermore,
$\big\|L^{\bnu}\phi_n-L^{\bnu}\phi_j\big\|_\infty \le \cP_{m,0}(\phi_n-\phi_j)\to 0$ as $n,j\to\infty$.
From this, \eqref{Frechet-2} and the completeness of $L^\infty$ it follows that
\begin{equation*}
\lim_{n\to\infty}\|L^{\bnu}\phi_n -L^{\bnu}\phi\|_\infty=0,
\quad \forall \bnu\in\bN_0^2.
\end{equation*}
In turn, this and the fact that
$\lim_{n,j\rightarrow\infty}\cP_{m,k}(\phi_n-\phi_j)=0$ for all $m,k\ge 0$
imply that
$\lim_{n\rightarrow\infty}\cP_{m,k}(\phi_n-\phi)=0$ for all $m,k\ge 0$,
which completes the proof.
\end{proof}

\begin{definition}
The space $\cS'=\cS'(L_1, L_2)$ of distributions associated to $L_1$, $L_2$ is defined as the set of all
continuous linear functionals on $\cS=\cS(L_1, L_2)$.
The action of $f\in \cS'$ on $\overline{\phi}\in \cS$ will be denoted by
$\langle f, \phi\rangle:= f(\overline{\phi})$,
which is consistent with the inner product in $L^2(\XX, d\mu)$.
\end{definition}
Clearly, for any $f\in \cS'$ there exist constants $m,k\in \bN_0$ and $c>0$ such that
\begin{equation}\label{distribution-1}
|\langle f, \phi\rangle| \le c\cP_{m,k}(\phi), \quad \forall \phi\in\cS.
\end{equation}

We next show that there are sufficiently many test functions.

\begin{proposition}\label{prop:basicprop}
Let $\varphi \in \cS(\bR^2)$ be real-valued and
$\varphi(\pm\lambda_1,\pm\lambda_2)=\varphi(\lambda_1,\lambda_2)$
for $(\lambda_1, \lambda_2) \in\bR^2$.
Recall the notation $\varphi(\sqrt{L}):= \varphi(\sqrt{L_1}, \sqrt{L_2})$
and $\cS:=\cS(L_1, L_2)$.
Then:

$(i)$ $\varphi(\sqrt{L})$ is an integral operator whose kernel
$\KK_{\varphi(\sqrt{L})}(\xx, \cdot)\in \cS$, $\forall \xx\in \XX$,
and $\KK_{\varphi(\sqrt{L})}(\cdot, \yy)\in \cS$, $\forall \yy\in \XX$.

$(ii)$ For each $\phi\in\cS$ we have $\varphi(\sqrt{L})\phi\in\cS$. 
Furthermore, the spectral multiplier  $\varphi(\sqrt L)$ maps continuously $\cS$  into $\cS$.
\end{proposition}

\begin{proof}
(i) Since $\KK_{\varphi(\sqrt{L})}(\xx, \yy)=\KK_{\varphi(\sqrt{L})}(\yy,\xx)$ for $\xx,\yy\in \XX$ ($\varphi$ is real-valued)
it suffices to only show that $\KK_{\varphi(\sqrt{L})}(\xx, \cdot)\in \cS$.
To this end we need to show that
\begin{equation}\label{eq:basicprop1}
\cP_{m,k}(\KK_{\varphi(\sqrt{L})}(\xx, \cdot))<\infty,
\quad \forall m,k\in \bN_0.
\end{equation}
Fix $\xx\in \XX$.
From Corollary~\ref{cor:gen-local} it readily follows that
for any $m,k\in \bN_0$ there exists a constant $c>0$ such that for all $\bnu\in \bN_0^2$ with $\nu_i\le m$
\begin{equation*}
\sup_{\yy\in \XX}\prod_{i=1,2}(1+\rho(x_i,y_i)^k |\KK_{L^\bnu\varphi(\sqrt L)}(\xx,\yy)|\le c V(\xx,\one)^{-1}.
\end{equation*}
Therefore, to show that $\KK_{\varphi(\sqrt{L})}(\xx, \cdot)\in \cS$ it suffices to show that
\begin{equation}\label{eq:basicprop2}
\KK_{L^\bnu\varphi(\sqrt{L})}(\xx, \yy)=L^\bnu[\KK_{\varphi(\sqrt{L})}(\xx, \cdot)](\yy),\quad \xx,\yy\in \XX.
\end{equation}
For this we first need to verify that
$\KK_{\varphi(\sqrt{L})}(\xx, \cdot)\in D(L^\bnu)$.
Let $g\in D(L^\bnu)$.
As is well known (e.g. \cite[Chapter XI, Section 12, Theorem~3, (iv)]{Yosida})
$L^\bnu g \in D(\varphi(\sqrt{L}))$ if and only if $g\in D(\varphi(\sqrt{L})L^\bnu)$, and
\begin{equation}\label{func-calc}
\varphi(\sqrt{L})(L^\bnu g) = [\varphi(\sqrt{L})L^\bnu] g = [L^\bnu\varphi(\sqrt{L})] g.
\end{equation}
But $\varphi(\sqrt{L})$ is a bounded operator on $L^2(\XX, d\mu)$
and hence $L^\bnu g \in D(\varphi(\sqrt{L}))$,
which implies that $g\in D(\varphi(\sqrt{L})L^\bnu)$
and \eqref{func-calc} is valid.
Hence, for almost all $\xx\in \XX$
\begin{equation}\label{K-K}
\int_\XX \KK_{L^\bnu\varphi(\sqrt L)}(\xx,\yy)g(\yy) d\mu(\yy)
=\int_\XX \KK_{\varphi(\sqrt L)}(\xx,\yy)L^\bnu g(\yy) d\mu(\yy).
\end{equation}
But the kernels $\KK_{L^\bnu\varphi(\sqrt L)}(\xx,\yy)$ and $\KK_{\varphi(\sqrt L)}(\xx,\yy)$ are continuous
and hence \eqref{K-K} is valid for all $\xx\in \XX$.
Applying the Cauchy-Schwarz inequality we obtain from \eqref{K-K}
\begin{equation*}
\Big|\int_\XX L^\bnu g(\yy)\KK_{\varphi(\sqrt{L})}(\xx, \yy)d\mu(\yy)\Big|
\le \|\KK_{L^\bnu\varphi(\sqrt{L})}(\xx, \cdot)\|_2\|g\|_2
\le c\|g\|_2,
\end{equation*}
where we used Corollary~\ref{cor:gen-local} and \eqref{tech-1}.
Thus
\begin{equation*}
|\big\langle L^\bnu g, \KK_{\varphi(\sqrt{L})}(\xx, \cdot)\big\rangle|
\le c\|g\|_2, \quad \forall g\in D(L^\bnu),
\end{equation*}
which leads to the conclusion that
$\KK_{\varphi(\sqrt{L})}(\xx, \cdot)\in D((L^\bnu)^*)$.
But $D(L^\bnu)=D((L^\bnu)^*)$ because $L^\bnu$ is self-adjoint.
Hence $\KK_{\varphi(\sqrt{L})}(\xx, \cdot)\in D(L^\bnu)$.

We now show \eqref{eq:basicprop2}.
Appealing to \eqref{K-K} and the fact that $L^\bnu$ is self-adjoint
we obtain for all $g\in D(L^\bnu)$ and $\xx\in \XX$
\begin{align*}
\big\langle \KK_{L^\bnu\varphi(\sqrt L)}(\xx,\cdot), \overline{g} \big\rangle
&= \big\langle L^\bnu g, \KK_{\varphi(\sqrt L)}(\xx,\cdot) \big\rangle
=\big\langle g, (L^\bnu)^* [\KK_{\varphi(\sqrt L)}(\xx,\cdot)]\big\rangle
\\
&= \big\langle g, L^\bnu[\KK_{\varphi(\sqrt L)}(\xx,\cdot)]\big\rangle
= \big\langle L^\bnu[\KK_{\varphi(\sqrt L)}(\xx,\cdot)], \overline{g}\big\rangle,
\end{align*}
which implies (\ref{eq:basicprop2}).

\smallskip

(ii) Let $m,k\in\bN_0$ 
and assume that $\bnu\in\bN_0^2$ with $\bnu_i\le m$ and
$\bsigma\in(0,\infty)^2$ is such that $\sigma_i\ge k+2d_i+1$, $i=1,2$.
Denote $\tilde{d}:= \lfloor d_1\vee d_2 \rfloor +1$.
Using (\ref{local-ker-nu}) and (\ref{tech-2}) we get for $\phi\in\cS$
\begin{align*}
\big|L^{\bnu}\varphi(\sqrt L)&\phi(\xx)\big|=\big|\varphi(\sqrt L)L^{\bnu}\phi(\xx)\big|
\le \int_{\XX}|\KK_{\varphi(\sqrt L)}(\xx,\yy)||L^{\bnu}\phi(\yy)|d\mu(\yy)
\\
&\le c\int_{\XX}\DD_{\one,\bsigma}(\xx,\yy)|L^{\bnu}\phi(\yy)|d\mu(\yy)
\\
&\le c\cP_{m,k+\tilde{d}+1}(\phi)\int_{\XX}V(\xx,\one)^{-1}
\DD_{\one,\bar{k}+\dd+\one}^* (\xx,\yy)\DD_{\one,\bar{k}+\dd+\one}^* (\yy,\xx_0)d\mu(\yy)
\\
&\le c\cP_{m,k+\tilde{d}+1}(\phi)\DD_{\one,\bar{k}+\one}^* (\xx,\xx_0)
\le c\cP_{m,k+\tilde{d}+1}(\phi)\DD_{\one,\bar{k}}^*(\xx,\xx_0),
\end{align*}
where $\bar{k}=:(k,k)$ and $\one:= (1,1)$.
From above and \eqref{kernelsD*} it readily follows that
$\varphi(\sqrt{L})\phi\in\cS$ and
$\cP_{m,k}((\varphi\sqrt L)\phi)\le c\cP_{m,k+\tilde{d}+1}(\phi)$,
which leads to the conclusion that
$\varphi(\sqrt L)$ maps continuously $\cS$  into $\cS$.
\end{proof}

We use Proposition \ref{prop:basicprop} (ii) to make the following
\begin{definition}\label{def:phi-f}
Let $\varphi\in \cS(\bR^2)$ be real-valued and
$\varphi(\pm\lambda_1,\pm\lambda_2)=\varphi(\lambda_1,\lambda_2)$ for $(\lambda_1, \lambda_2)\in\bR^2$.
For any $f\in\cS'=\cS'(L_1, L_2)$ we define $\varphi(\sqrt{L})f$ by
\begin{equation}\label{distribution-2}
\langle \varphi(\sqrt{L})f, \phi\rangle = \langle f, \varphi(\sqrt{L})\phi\rangle,
\quad \forall \phi\in \cS.
\end{equation}
\end{definition}

\begin{proposition}\label{prop:f-phi}
Let $\varphi$ be as in Definition~\ref{def:phi-f}.
Then for any $f\in\cS'=\cS'(L_1, L_2)$
\begin{equation}\label{f-varphi}
\varphi(\sqrt{L})f(\xx) :=\langle f, \KK_{\varphi(\sqrt{L})}(\xx, \cdot) \rangle
\quad\hbox{for}\quad \xx\in X.
\end{equation}
\end{proposition}
\begin{proof}
Let $f\in\cS'$.
Just as in \eqref{distribution-1} there exist constants $m,k\in\bN_0$ and $c>0$ such that
\begin{equation}\label{distr}
|\langle f, \phi\rangle| \le c\cP_{m,k}(\phi),\quad \forall \phi\in\cS.
\end{equation}
Fix $\phi\in \cS$.
Clearly,
\begin{equation*}
\varphi(\sqrt L)\phi(\xx) = \int_\XX \KK_{\varphi(\sqrt L)}(\xx,\yy)\phi(\yy) d\mu(\yy), \quad \xx\in \XX.
\end{equation*}
We will regard the above integral as a Bochner integral over the Banach space
$$
V_{m,k}:=\{g\in \cap_{0\le\nu_i \le m} D(L^\bnu): \|g\|_{V_{m,k}}:=\cP_{m,k}(g)<\infty\}
$$
with $\cP_{m,k}$ defined in \eqref{norm-S}, see e.g. \cite{Yosida}, pp. 131--133.
Just as in the proof of Proposition~\ref{prop:Frechet} it follows that
the space $V_{m,k}$ is complete and hence it is a Banach space.
By the Hahn-Banach theorem and \eqref{distr} the continuous linear functional $f$
can be extended to $V_{m,k}$ with the same norm.
We denote this extension by $f$ again.

Denote $F_\yy(\xx):= \KK_{\varphi(\sqrt{L})}(\xx, \yy)\phi(\yy)$. Then
\begin{align*}
\|F_\yy\|_{V_{m,k}}
&= \sup_{\xx\in \XX}\prod_{i=1,2}\big(1+\rho_i(x_i,x_{0i})\big)^{k} \max_{0\le \bnu_i\le m}
\big|L^{\bnu}\KK_{\varphi(\sqrt{L})}(\xx, \yy)\phi(\yy)\big|
\\
&= \sup_{\xx\in \XX}\prod_{i=1,2}\big(1+\rho_i(x_i,x_{0i})\big)^{k} \max_{0\le \bnu_i\le m}
\big|\KK_{L^{\bnu}\varphi(\sqrt{L})}(\xx, \yy)\phi(\yy)\big|,
\end{align*}
where we used \eqref{eq:basicprop2}.
Clearly, the function $\varphi$ obeys the conditions of Corollary~\ref{cor:gen-local} and hence
there exists a constant $c_k>0$ such that
\begin{equation*}
|\KK_{L^{\bnu}\varphi(\sqrt{L})}(\xx, \yy)| \le c_kV(\yy,\one)^{-1}\prod_{i=1,2}\big(1+\rho_i(x_i,y_i)\big)^{-k},
\quad 0\le \nu_i \le m.
\end{equation*}
On the other hand, since $\phi\in \cS$ for any $N>0$ we have
\begin{equation*}
|\phi(\yy)|\le c_N\cP_{0,N}(\phi)\prod_{i=1,2}\big(1+\rho_i(y_i,x_{0i})\big)^{-N}.
\end{equation*}
We choose $N=k+2(d_1\vee d_2)+1$.
Putting the above estimates together we get
\begin{align*}
\|F_\yy\|_{V_{m,k}} \le \frac{c\cP_{0,N}(\phi)}{V(\yy,\one)}
\sup_{\xx\in \XX}\frac{\prod_{i=1,2}\big(1+\rho_i(x_i,x_{0i})\big)^{k}}
{\prod_{i=1,2}\big(1+\rho_i(x_i,y_i)\big)^{k} \prod_{i=1,2}\big(1+\rho_i(y_i,x_{0i})\big)^{N}}.
\end{align*}
Using the obvious inequality
$1+\rho_i(x_i,x_{0i}) \le (1+\rho_i(x_i,y_i))(1+\rho_i(y_i,x_{0i}))$
and \eqref{rect-doubling2} we obtain
\begin{align*}
\|F_\yy\|_{V_{m,k}} \le \frac{c\cP_{0,N}(\phi)}{V(\xx_0,\one)\prod_{i=1,2}\big(1+\rho_i(y_i,x_{0i})\big)^{d_i+1}}.
\end{align*}
From this and \eqref{tech-1} it follows that
\begin{equation*}
\int_\XX \|\KK_{\varphi(\sqrt{L})}(\cdot, \yy)\phi(\yy)\|_{V_{m,k}} d\mu(\yy) \le c\cP_{0,N}(\phi)<\infty.
\end{equation*}
Now, applying the theory of Bochner's integral we obtain
\begin{align*}
\langle f, \varphi(\sqrt{L})\phi\rangle
=\Big\langle f, \int_\XX \KK_{\varphi(\sqrt{L})}(\cdot,\yy)\phi(\yy)d\mu(\yy) \Big\rangle
= \int_\XX \big\langle f, \KK_{\varphi(\sqrt L)}(\cdot,\yy) \big\rangle \overline{\phi(\yy)} d\mu(\yy).
\end{align*}
This coupled with \eqref{distribution-2} implies \eqref{f-varphi}.
\end{proof}

\begin{remark}\label{rem:opercontinuity}
From Proposition \ref{prop:basicprop} (ii) it readily follows that the operator
$\varphi(\sqrt L)$ from above maps continuously $\cS'$  into $\cS'$.
\end{remark}

Our next step is to show that $\varphi(\sqrt{L})f$ is a slowly growing continuous function
for smooth $\varphi$ and $f\in \cS'$.

\begin{proposition}\label{prop:slow}
Let $\varphi\in \cS(\bR^2)$ be real-valued and
$\varphi(\pm\lambda_1,\pm\lambda_2)=\varphi(\lambda_1,\lambda_2)$ for $(\lambda_1,\lambda_2)\in\bR^2$.
Let $\bdelta=(\delta_1, \delta_2)>\zero$.
Then for any $f\in\cS'=\cS'(L_1, L_2)$, $\varphi(\delta\sqrt{L})f$ is a slowly growing continuous function.
More explicitly,
there exist constants $N,c>0$ $($depending on $f$$)$ such that
\begin{equation}\label{slow-1}
|\varphi(\bdelta\sqrt{L})f(\xx)| \le c V(x_0,\bdelta)^{-1}
\prod_{i=1,2}(\delta_i^N+\delta_i^{-N})\big(1+\rho_i(x_i,x_{0i})\big)^{N}
\quad\hbox{and}
\end{equation}
\begin{align}\label{cont-1}
|\varphi(\bdelta\sqrt{L})f(\xx)&-\varphi(\bdelta\sqrt{L})f(\xx')| \nonumber
\\
&\le c V(\xx_0,\bdelta)^{-1}\prod_{i=1,2}(\delta_i^N+\delta_i^{-N})\big(1+\rho_i(x_i,x_{0i})\big)^{N}
\sum_{i=1,2}\rho_i(x_i,x'_i)^{\alpha_i}
\end{align}
if $\rho_i(x_i,x'_i)\le \delta_i$.
Here $\alpha_1$, $\alpha_2$ are from \eqref{lip}.
\end{proposition}
\begin{proof}
We will only prove \eqref{cont-1}; the proof of \eqref{slow-1} is easier and will be omitted.
Let $f\in\cS'$.
Just as in \eqref{distribution-1} there exist $m,k\in\bN_0$ such that
\begin{equation*} 
|\langle f, \phi\rangle| \le c\cP_{m,k}(\phi),\quad \forall \phi\in\cS.
\end{equation*}
We use this, \eqref{distribution-2}, and \eqref{eq:basicprop2} to obtain
\begin{align*}
&|\varphi(\bdelta\sqrt{L})f(\xx)-\varphi(\bdelta\sqrt{L})f(\xx')|
\\
&= \big|\big\langle f, \KK_{\varphi(\bdelta\sqrt{L})}(\xx, \cdot) - \KK_{\varphi(\bdelta\sqrt{L})}(\xx', \cdot)\big\rangle\big|
\le c\cP_{m,k}\big(\KK_{\varphi(\bdelta\sqrt{L})}(\xx, \cdot) - \KK_{\varphi(\bdelta\sqrt{L})}(\xx', \cdot)\big)
\\
&=c\sup_{\yy\in \XX}\prod_{i=1,2}\big(1+\rho_i(y_i,x_{0i})\big)^{k}
\\
&\hspace{1.4in} \times\max_{0\le \nu_i\le m}
\big|L^{\bnu}[\KK_{\varphi(\bdelta\sqrt{L})}(\xx,\cdot)](\yy)-L^{\bnu}[\KK_{\varphi(\bdelta\sqrt{L})}(\xx',\cdot)](\yy)\big|
\\
&\le c\max_{0\le \nu_i\le m}\sup_{\yy\in \XX}\prod_{i=1,2}\big(1+\rho_i(y_i,x_{0i})\big)^{k}
\big|\KK_{L^\bnu\varphi(\bdelta\sqrt{L})}(\xx,\yy)-\KK_{L^\bnu\varphi(\bdelta\sqrt{L})}(\xx',\yy)\big|.
\end{align*}
The function $\varphi$ obeys the conditions of Corollary~\ref{cor:gen-local}
and hence if $\rho_i(x_i, x'_i)\le \delta_i$, then
\begin{align*}
&\big|\KK_{L^\bnu\varphi(\bdelta\sqrt{L})}(\xx,\yy)-\KK_{L^\bnu\varphi(\bdelta\sqrt{L})}(\xx',\yy)\big|
\\
&\le c V(\xx,\bdelta)^{-1} \bdelta^{-2\bnu}\prod_{i=1,2}\big(1+\delta_i^{-1}\rho_i(x_i,y_i)\big)^{-k}
\sum_{i=1,2}(\rho_i(x_i,x'_i)/\delta_i)^{\alpha_i}
\\
&\le c V(\xx,\bdelta)^{-1} \prod_{i=1,2}(1+\delta_i)^k(1+\delta_i^{-1})^{2\nu_i+\alpha_i}\big(1+\rho_i(x_i,y_i)\big)^{-k}
\sum_{i=1,2}\rho_i(x_i,x'_i)^{\alpha_i}.
\end{align*}
Now, we use that
$1+\rho_i(y_i,x_{0i}) \le (1+\rho_i(x_i,y_i))(1+\rho_i(x_i,x_{0i}))$
and
\begin{align*}
V(\xx_0,\bdelta)
&\le c\prod_{i=1,2}(1+\delta_i^{-1}\rho_i(x_i,x_{0i}))^{d_i}V(\xx, \bdelta)
\\
&\le c\prod_{i=1,2}(1+\delta_i^{-1})^{d_i}(1+\rho_i(x_i,x_{0i}))^{d_i}V(\xx, \bdelta),
\end{align*}
which follows from \eqref{rect-doubling2} to obtain
\begin{align*}
&\big|\KK_{L^\bnu\varphi(\bdelta\sqrt{L})}(\xx,\yy)-\KK_{L^\bnu\varphi(\bdelta\sqrt{L})}(\xx',\yy)\big|
\\
&\le c \frac{\prod_{i=1,2}(1+\delta_i)^k(1+\delta_i^{-1})^{2\nu_i+\alpha_i+d_i}\big(1+\rho_i(x_i,x_{0i})\big)^{k+d_i}}
{V(\xx_0,\bdelta)\prod_{i=1,2}\big(1+\rho_i(y_i,x_{0i})\big)^{k}}
\sum_{i=1,2}\rho_i(x_i,x'_i)^{\alpha_i}.
\end{align*}
Putting all of the above together we obtain \eqref{cont-1}
with $N$ sufficiently large.
\end{proof}

\subsection{Decomposition}

We next establish a basic convergence theorem and a \textit{Calder\'{o}n type reproducing} formula,
which will play an important role in this study.
We begin with the following

\begin{theorem}\label{thm:CRF}
Let $\varphi \in \cS(\R^2)$ be a real-valued function that satisfies the conditions:
$\varphi(\zero)=1$,
$\partial^\bbeta\varphi(\zero)=0$ for all $\bbeta\in\bN_0^2$ with $|\bbeta|\ge 1$, $\zero:=(0,0)$, and
$\varphi(\pm\lambda_1,\pm\lambda_2)=\varphi(\lambda_1,\lambda_2)$ for $(\lambda_1,\lambda_2)\in \R^2$.
Then for any $\phi\in \cS$
\begin{equation}\label{decomp-dist-1}
\phi=\lim_{\bdelta\to \zero} \varphi (\bdelta\sqrt L) \phi
\quad \mbox{ in }\; \cS
\end{equation}
and, therefore, for any $f\in \cS'$
\begin{equation}\label{decomp-dist-2}
f=\lim_{\bdelta\to \zero} \varphi (\bdelta\sqrt L) f
\quad \mbox{ in }\; \cS'.
\end{equation}
\end{theorem}
\begin{proof}
We need to show that for any $m,k\in\bN_0$ and $\phi\in\cS$
\begin{equation}\label{converge}
\begin{aligned}
\lim_{\bdelta\rightarrow \zero}&\cP_{m,k}(\phi-\varphi(\bdelta\sqrt{L})\phi)
\\
&=\lim_{\bdelta\rightarrow \zero}\sup_{\xx\in \XX}\prod_{i=1, 2}\big(1+\rho_i(x_i,x_{0i})\big)^{k}
\max_{0\le \nu_i\le m}\big|L^{\bnu}\big(\phi-\varphi(\bdelta\sqrt{L})\phi\big)(\xx)|=0.
\end{aligned}
\end{equation}
Choose $\nn= (n_1,n_2)\in \bN^2$, $n_i > k+5d_i/2$, and
$r\in \bN, r>(d_1+d_2+n_1+n_2)/2$.
We define
$\omega(\blambda):=(1-\varphi(\blambda))|\blambda|^{-2r}$.
Clearly, $I-\varphi(\bdelta\sqrt{L})= \omega(\bdelta\sqrt{L})(\delta_1^2L_1+\delta_2^2L_2)^r$.
Since $|\blambda|^{2r}=(\lambda_1^2+\lambda_2^2)^r$ is a radial function
from the definition of $\varphi$ it follows that $\omega$ satisfies the assumptions of Theorem \ref{thm:gen-local}.
Hence
\begin{equation}\label{eq:w_decay}
\KK_{\omega(\bdelta\sqrt L)}(\xx,\yy)\le c_\nn\DD_{\bdelta,\nn}(\xx,\yy),\quad  \xx,\yy\in \XX.
\end{equation}
From the representation
$(\delta_1^2L_1+\delta_2^2L_2)^{r}=\sum_{\ell=0}^{r}\binom{r}{\ell}\delta_1^{2\ell}\delta_2^{2(r-\ell)}L_1^{\ell}L_2^{r-\ell}$
and assuming $0\le \nu_i\le m$,
we get
\begin{equation}\label{eq:CRF}
\begin{aligned}
L^{\bnu}\big(\phi-\varphi(\bdelta\sqrt{L})\phi\big)(\xx)
&=L^{\bnu}\omega(\bdelta\sqrt L)(\delta_1^2 L_1+\delta_2^2 L_2)^{r}\phi(\xx)
\\
&=\sum_{\ell=0}^r \binom{r}{\ell}\delta_1^{2\ell}\delta_2^{2(r-\ell)} \omega(\delta\sqrt L) L^{(\nu_1+\ell,\nu_2+r-\ell)}\phi(\xx).
\end{aligned}
\end{equation}
Pick $N\ge (n_1-d_1/2)\vee (n_2-d_2/2)$ and assume $\zero<\bdelta\le \one$, $\one:=(1,1)$.
Then using \eqref{eq:w_decay}--\eqref{eq:CRF} we obtain
\begin{align*}
|L^{\bnu}&\big(\phi-\varphi(\bdelta\sqrt{L})\phi\big)(\xx)|
\le c\sum_{\ell=0}^r \delta_1^{2\ell}\delta_2^{2(r-\ell)}\int_\XX \DD_{\bdelta,\nn}(\xx,\yy)
|L^{(\nu_1+\ell,\nu_2+r-\ell)}\phi(\yy)| d\mu(\yy)
\\
&\le c \sum_{\ell=0}^r \delta_1^{2\ell}\delta_2^{2(r-\ell)}\cP_{m+r,N}(\phi)
\int_\XX \DD_{\bdelta,\nn}(\xx,\yy) \DD^*_{\one,\nn-\dd/2}(\yy,\xx_0)d\mu(\yy)
\\
&\le c\sum_{\ell=0}^r \delta_1^{2\ell}\delta_2^{2(r-\ell)}\cP_{m+r,N}(\phi)
\int_\XX \frac{1}{V(\yy,\bdelta)}\DD^*_{\bdelta,\nn-\dd/2}(\xx,\yy) \DD^*_{\one,\nn-\dd/2}(\yy,\xx_0)d\mu(\yy)
\\
&\le c \sum_{\ell=0}^r \delta_1^{2\ell}\delta_2^{2(r-\ell)}\cP_{m+r,N}(\phi)\DD^*_{\one,\nn-\dd/2}(\xx,\xx_0),
\end{align*}
where for the third inequality we used \eqref{D-D*}
and for the last inequality we used \eqref{tech-4} and that $n_i-d_i/2 > 2d_i$.
Finally, putting all together we obtain
\begin{align*}
\cP_{m,k}(\phi-\varphi(\bdelta\sqrt{L})\phi)
&\le c \sum_{\ell=0}^r \delta_1^{2\ell}\delta_2^{2(r-\ell)}\cP_{m+r,N}(\phi)
\sup_{\xx\in \XX}\prod_{i=1,2}\frac{\big(1+\rho_i(x_i,x_{0i})\big)^{k}}{\big(1+\rho_i(x_i,x_{0i})\big)^{n_i-d_i/2}}
\\
&\le c \cP_{m+r,N}(\phi)\sum_{\ell=0}^r \delta_1^{2\ell}\delta_2^{2(r-\ell)}\rightarrow 0 \;\text{ as }\; \delta_1,\delta_2\rightarrow 0,
\end{align*}
where we used that $n_i\ge k+d_i/2$.
The proof of \eqref{converge} is complete.
\end{proof}


We next derive a Calder\'{o}n type reproducing formula in the setting of this article.
(Compare with the result from Lemma~\ref{lem:dec-unity}.)
It relies on two pairs of functions $\varphi^i_0,\varphi^i \in\cC^\infty(\RR)$, $i=1,2$,  with the properties:
\begin{equation}\label{phi-phi0}
\begin{aligned}
&(i)\;\; \hbox{$\varphi^i_0,\varphi^i$ are real-valued and even,}
\\
&(ii)\ \supp\varphi^i_0\subset[-2,2], \ \supp\varphi^i\subset[2^{-1},2]\setminus[-1/2,1/2],
\end{aligned}
\end{equation}
and
\begin{equation}\label{1dCalderon}
\varphi^i_0(t)+\sum_{n\ge 1}\varphi^i(2^{-n} t)=1, \quad \forall t \in \RR.
\end{equation}
It is easy to construct functions $\varphi^i_0,\varphi^i \in\cC^\infty(\RR)$, $i=1,2$, with the above properties.
Indeed, let $\varphi_0^i\in \cC^{\infty}(\RR)$ ($i=1,2$) be an even real-valued function satisfying the conditions:
(a) $\varphi_0^i(t)=1$, $t\in[-1,1]$,
(b) $\supp\varphi_0^i\subset[-2,2]$,
(c) $0\le\varphi_0^i\le 1$,
and define $\varphi^i(t):=\varphi^i_0(t)-\varphi_0^i(2t)$.
Then it follows immediately that
$\varphi^i_0,\varphi^i$ satisfy conditions \eqref{phi-phi0}--\eqref{1dCalderon}.

Using the standard notation $\varphi^i_n(t):=\varphi^i(2^{-n}t), n\in \bN$,
for any $\jj=(j_1,j_2)\in \bN^2_0$ we set $\varphi_\jj:=\varphi^1_{j_1}\otimes\varphi^2_{j_2}$.
Then from \eqref{1dCalderon} it readily follows that
\begin{equation}\label{2dCalderon}
\sum_{\jj\in \bN_0^2}\varphi_\jj(\blambda)=1,\quad \forall \blambda \in \RR^2.
\end{equation}

\begin{corollary}\label{cor:Calderon}
Let $\varphi^1_0,\varphi^2_0, \varphi^1,\varphi^2 \in \cC^{\infty}(\RR)$ be functions
satisfying $(\ref{phi-phi0})$--$(\ref{1dCalderon})$
and let $\varphi_\jj$ be as above.
Then, for any $\phi\in \cS$
\begin{equation}\label{cald-1}
\phi=\sum_{\jj\in\bN_0^2}\varphi_\jj(\sqrt{L})\phi
\quad \mbox{ in }\; \cS
\end{equation}
and, therefore, for any $f\in \cS'$
\begin{equation}\label{cald-2}
f=\sum_{\jj\in\bN_0^2}\varphi_\jj(\sqrt{L})f
\quad \mbox{ in }\; \cS'.
\end{equation}
\end{corollary}
\begin{proof}
Let $f\in\cS$ (resp. $f\in\cS'$).
We set $\theta^i(t):=\varphi^i_0(t)+\varphi^i (2^{-1}t)$, $i=1,2$,
and observe that \eqref{phi-phi0}--\eqref{1dCalderon}
imply $\supp\theta^i\subset[-4,4]$ and $\theta^i(t)=1$ for $t\in[-2,2]$.
Moreover from the definition of $\theta^i$ and (\ref{1dCalderon}) we have
$$
\sum_{n=0}^{\ell+1}\varphi^i_n(t)=\theta^i\big(2^{-\ell}t\big),\quad \ell\in \bN_0.
$$
Set $\theta:=\theta^1\otimes\theta^2$. Then it is easily seen that for any $\kk:=(k_1,k_2)\in\bN^2$,
$$
\sum_{j_1=0}^{k_1+1}\sum_{j_2=0}^{k_2+1}\varphi_{\jj}(\blambda)
=(\theta^1\otimes\theta^2)(b_1^{-k_1}\lambda_1, b_2^{-k_2}\lambda_2)=\theta\big(b^{-\kk}\blambda\big).
$$
Finally, applying Theorem~\ref{thm:CRF} to $\theta$  we obtain
\begin{align*}
\sum_{\jj\in\bN_0^2}\varphi_\jj (\sqrt{L})f
&=\lim_{k_1,k_2\rightarrow\infty}\sum_{j_1=0}^{k_1+1}\sum_{j_2=0}^{k_2+1}
\big(\varphi^1_{j_1}(\sqrt{L_1})\otimes\varphi^2_{j_2}(\sqrt{L_2})\big)(f)
\\
&=\lim_{k_1,k_2\rightarrow\infty}\theta(b^{-\kk}\sqrt{L})f=f,
\end{align*}
which completes the proof.
\end{proof}

Our next step is to establish our principle convergence result.

\begin{theorem}\label{thm:converge}
Let $\varphi \in \cS(\R^2)$ be a real-valued function satisfying the conditions:
$\varphi(\zero)=1$
and
$\varphi(\pm\lambda_1,\pm\lambda_2)=\varphi(\lambda_1,\lambda_2)$ for $(\lambda_1,\lambda_2)\in \R^2$.
Then for any $\phi\in \cS$
\begin{equation}\label{conv-1}
\phi=\lim_{\bdelta\to \zero} \varphi (\bdelta\sqrt L) \phi
\quad \mbox{ in }\; \cS
\end{equation}
and, therefore, for any $f\in \cS'$
\begin{equation}\label{conv-2}
f=\lim_{\bdelta\to \zero} \varphi (\bdelta\sqrt L) f
\quad \mbox{ in }\; \cS'.
\end{equation}
Furthermore, \eqref{conv-2} holds in $L^p$ for any $f\in L^p$, $1\le p\le \infty$,
where $L^\infty= {\rm UCB}$ is the space of all uniformly continuous and bounded functions on $\XX$.
\end{theorem}

For the proof of this theorem we need the following

\begin{lemma}\label{lem:S}
Let $\bsigma >\zero$,
$k \ge (\sigma_1+d_1)\vee(\sigma_2+d_2)$, and
$m > \frac{d_1+\alpha_1}{2} \vee \frac{d_2+\alpha_2}{2}$
with $\balpha=(\alpha_1,\alpha_2)$ from $(\ref{lip})$.
Then there exists a constant $c>0$ such that
for any $\phi\in \cS$ and $\xx, \yy\in \XX$
\begin{equation}\label{holder-S}
|\phi(\xx)-\phi(\yy)|
\le c\Big(\sum_{i=1,2}\rho_i(x_i, y_i)^{\alpha_i}\Big) \cP_{m,k}(\phi)
\big[\DD^*_{\one,\bsigma}(\xx, \xx_0)+\DD^*_{\one,\bsigma}(\yy, \xx_0)].
\end{equation}
\end{lemma}

\begin{proof}
Let $\varphi^1_0,\varphi^2_0, \varphi^1,\varphi^2 \in \cC^{\infty}(\RR)$ be functions
satisfying $(\ref{phi-phi0})$--$(\ref{1dCalderon})$
and
let $\varphi_\jj:=\varphi^1_{j_1}\otimes\varphi^2_{j_2}$
be just as in Corollary \ref{cor:Calderon}.
Then by Corollary \ref{cor:Calderon}
\begin{equation}\label{cald-1-Hp}
\phi=\sum_{\jj\in\bN_0^2}\varphi_\jj(\sqrt{L})\phi
\quad \mbox{ in }\; \cS.
\end{equation}
We write
\begin{align}
\phi(\xx)-\phi(\yy)&=\sum_{\jj\in\bN_0^2}\big(\varphi_\jj(\sqrt{L})\phi(\xx)-\varphi_\jj(\sqrt{L})\phi(\yy)\big)
=:\sum_{k=1}^{4}\sum_{\jj\in \cJ_k}U_{\jj}(\xx,\yy),
\label{L1Hp1}
\end{align}
where $\cJ_1:=\bN^2$, $\cJ_2:=\{(0,0)\}$, $\cJ_3:=\bN\times\{0\}$ and $\cJ_4:=\{0\}\times\bN$.

We next estimate $\sum_{\jj\in \cJ_1}|U_{\jj}(\xx,\yy)|$.
Let $\mm:=(m,m)$
and set
\begin{equation*}
\omega(\lambda_1, \lambda_2):=|\lambda_1|^{-2m}|\lambda_2|^{-2m}\varphi^1(\lambda_1)\varphi^2(\lambda_2).
\end{equation*}
Clearly, $\omega(\pm\lambda_1,\pm\lambda_2)=\omega(\lambda_1,\lambda_2)$ and $\supp\omega\subset [-2,2]\setminus[-1,1]$.
%
For $\jj\in \cJ_1=\bN^2$ and $\phi\in\cS$
we have
$\varphi_\jj(\sqrt{L})\phi=2^{-2m(j_1+j_2)}\omega(2^{-\jj}\sqrt{L})L^\mm\phi$
and, therefore,
\begin{align}\label{L1Hp2}
\varphi_\jj(\sqrt{L})\phi(\xx)-&\varphi_\jj(\sqrt{L})\phi(\yy)\nonumber
\\
&= 2^{-2m|\jj|}\int_{\XX}\big(\KK_{\omega(2^{-\jj}\sqrt{L})}(\xx,\zz)-\KK_{\omega(2^{-\jj}\sqrt{L})}(\yy,\zz)\big)L^{\mm}\phi(\zz)d\mu(\zz).
\end{align}
By Theorem~\ref{th:boxsupport} it readily follows that
\begin{align}
|\KK_{\omega(2^{-\jj}\sqrt{L})}(\xx,\zz)|&\le c_{\bsigma}\DD_{2^{-\jj},\bsigma+3\dd/2}(\xx,\zz)
\le cV(\xx,2^{-\jj})^{-1}\DD_{2^{-\jj},\bsigma+\dd}^{*}(\xx,\zz)
\nonumber
\\
&\le c2^{\jj\cdot\dd}V(\xx,\one)^{-1}\DD_{\one,\bsigma+\dd}^{*}(\xx,\zz),
\label{L1Hp3}
\end{align}
in light of \eqref{D-D*} and \eqref{rect-doubling}.
Moreover, if $\rho_i(x_i,y_i)\le 2^{-j_i}$, $i=1,2$, then
\begin{align}
|\KK_{\omega(2^{-\jj}\sqrt{L})}(\xx,\zz)-&\KK_{\omega(2^{-\jj}\sqrt{L})}(\yy,\zz)| \nonumber
\\
&\le c\Big(\sum_{i=1,2}(2^{j_i}\rho_i(x_i,y_i))^{\alpha_i}\Big) 2^{\jj\cdot\dd}V(\xx,\one)^{-1}\DD_{\one,\bsigma+\dd}^{*}(\xx,\zz),
\label{L1Hp4}
\end{align}
Note also that, since $\phi\in\cS$,
\begin{equation}
|L^{\mm}\phi(\zz)|\le \cP_{m,k}(\phi)\DD^{*}_{\one,\bsigma+\dd}(\zz,\xx_0).
\label{L1Hp5}
\end{equation}
Combining (\ref{L1Hp2}), (\ref{L1Hp4}) and (\ref{L1Hp5}) we conclude that
if $\rho_i(x_i,y_i)\le 2^{-j_i}$, $i=1,2$, then
\begin{align}\label{L1Hp6}
|\varphi_\jj(\sqrt{L})\phi(\xx)&-\varphi_\jj(\sqrt{L})\phi(\yy)|
\le c 2^{-2m|\jj|}2^{\jj\cdot\dd}\Big(\sum_{i=1,2}(2^{j_i}\rho_i(x_i,y_i))^{\alpha_i}\Big) \cP_{m,k}(\phi)\nonumber
\\
&\times V(\xx,\one)^{-1}\int_{\XX}\DD_{\one,\bsigma+\dd}^{*}(\xx,\zz)\DD^{*}_{\one,\bsigma+\dd}(\zz,\xx_0)d\zz
\\
&\le c \cP_{m,k}(\phi)2^{-2m|\jj|}2^{\jj\cdot\dd}2^{\jj\cdot\balpha}
\Big(\sum_{i=1,2}\rho_i(x_i,y_i)^{\alpha_i}\Big) \DD_{\one,\bsigma}^{*}(\xx,\xx_0), \nonumber
\end{align}
where we used (\ref{tech-2}).

Assume $\rho_{\iota}(x_\iota,y_\iota)>2^{-j_\iota}$ for some $\iota\in\{1,2\}$.
Then by (\ref{L1Hp3}), (\ref{L1Hp5}) and (\ref{tech-2}) we obtain, for $\xx\in \XX$,
\begin{align*}
&|\varphi_\jj(\sqrt{L})\phi(\xx)|\le 2^{-2m|\jj|}\int_{\XX}|\KK_{\omega(2^{-\jj}\sqrt{L})}(\xx,\zz)||L^{\mm}\phi(\zz)|d\zz
\\
&\le c 2^{-2m|\jj|}2^{\jj\cdot\dd}\cP_{m,k}(\phi)(2^{\iota}\rho_{\iota}(x_{\iota},y_{\iota}))^{\alpha_\iota}
V(\xx,\one)^{-1}\int_{\XX}\DD_{\one,\bsigma+\dd}^{*}(\xx,\zz)\DD^{*}_{\one,\bsigma+\dd}(\zz,\xx_0)d\zz
\\
&\le c\cP_{m,k}(\phi)2^{-2m|\jj|}2^{\jj\cdot\dd}2^{\jj\cdot\balpha}
\Big(\sum_{i=1,2}\rho_i(x_i,y_i)^{\alpha_i}\Big) \DD_{\one,\bsigma}^{*}(\xx,\xx_0),
\end{align*}
implying
\begin{align*}
&|\varphi_\jj(\sqrt{L})\phi(\xx)-\varphi_\jj(\sqrt{L})\phi(\yy)|
\le|\varphi_\jj(\sqrt{L})\phi(\xx)|+|\varphi_\jj(\sqrt{L})\phi(\yy)|\nonumber
\\
&\le c 2^{-2m|\jj|}2^{\jj\cdot\dd}2^{\jj\cdot\balpha}\Big(\sum_{i=1,2}\rho_i(x_i,y_i)^{\alpha_i}\Big)
\cP_{m,k}(\phi)\big[\DD^{*}_{\one,\bsigma}(\xx,\xx_0)+\DD_{\one,\bsigma}^{*}(\yy,\xx_0)\big].
\end{align*}

Summing up \eqref{L1Hp6} and the above inequalities over $\jj\in\cJ_1=\bN^2$,
using the assumption that
$m > \frac{d_1+\alpha_1}{2} \vee \frac{d_2+\alpha_2}{2}$,
we get
\begin{align}\label{L1HpS1}
\sum_{\jj\in \cJ_1}&|U_{\jj}(\xx,\yy)| \nonumber
\\
&\le c\cP_{m,k}(\phi)\Big(\sum_{i=1,2}\rho_i(x_i,y_i)^{\alpha_i}\Big)\big[\DD^{*}_{\one,\bsigma}(\xx,\xx_0)+\DD_{\one,\bsigma}^{*}(\yy,\xx_0)\big].
\end{align}

The estimation of the sums $\sum_{\jj\in \cJ_i}|U_{\jj}(\xx,\yy)|$, $i=2,3,4$,
is similar (simpler) and we omit it.
The desired estimate \eqref{holder-S} follows from these estimates, \eqref{L1Hp1}, and \eqref{L1HpS1}.
\end{proof}

\begin{proof}[Proof of Theorem~\ref{thm:converge}]
We will prove \eqref{conv-1}. Then \eqref{conv-2} follows by duality.
We need to show that for any $n,\ell\in \bN_0$ and $\phi\in \cS$
\begin{equation}\label{limit-0}
\lim_{\bdelta\to\zero} \cP_{n,\ell}(\phi-\varphi(\bdelta \sqrt L)\phi) = 0.
\end{equation}
By Theorem~\ref{thm:gen-local} it follows that for any $\bsigma>\zero$ there exists a constant $c_{\bsigma}>0$ such that
\begin{equation}\label{ker-1}
|\KK_{\varphi(\bdelta \sqrt L)}(\xx,\yy)|\le c_\bsigma V(\xx,\bdelta)^{-1}\DD_{\bdelta,\bsigma}^*(\xx, \yy)
\end{equation}
and
\begin{equation}\label{ker-2}
\int_{\XX}\KK_{\varphi(\bdelta \sqrt L)}(\xx,\yy)d\mu(\yy)=\varphi(\zero)=1.
\end{equation}
Fix $n,\ell\in \bN_0$ and let $\bell:=(\ell,\ell)$ and $\bnu\in\bN_0^2$ be so that $0\le \nu_i\le n$, $i=1,2$.
We choose $\bsigma>\dd+\bell+\balpha$.
Also, let $m,k\in \bN_0$ be just as in Lemma~\ref{lem:S},
i.e. $m > \frac{d_1+\alpha_1}{2} \vee \frac{d_2+\alpha_2}{2}$ and $k \ge (\sigma_1+d_1)\vee(\sigma_2+d_2)$.
To prove \eqref{limit-0} we have to deal with
\begin{align}\label{P1Hp3}
\DD_{\one,\bell}^{*}&(\xx,\xx_0)^{-1}|L^{\bnu}(\phi-\varphi(\bdelta\sqrt{L})\phi)(\xx)|
\nonumber
\\
&\le \DD_{\one,\bell}^{*}(\xx,\xx_0)^{-1}\int_{\XX}|\KK_{\varphi(\bdelta \sqrt L)}(\xx,\yy)||L^{\bnu}\phi(\xx)-L^{\bnu}\phi(\yy)|d\mu(\yy)
\nonumber
\\
&\le c\DD_{\one,\bell}^{*}(\xx,\xx_0)^{-1}\int_{\XX}\frac{\DD_{\bdelta,\bsigma}^*(\xx, \yy)}{V(\xx,\bdelta)}|L^{\bnu}\phi(\xx)-L^{\bnu}\phi(\yy)|d\mu(\yy)
\\
&\le c\DD_{\one,\bell}^{*}(\xx,\xx_0)^{-1}\int_{Q(\xx,\one)}\dots + c\DD_{\one,\bell}^{*}(\xx,\xx_0)^{-1}\int_{\XX\setminus Q(\xx,\one)}\dots
\nonumber
\\
&=: W_1(\xx)+W_2(\xx), \nonumber
\end{align}
where we used \eqref{ker-1} and \eqref{ker-2}.
Recall that $Q(\xx,\one):= B_1(x_1, 1)\times B_2(x_2,1)$, see \eqref{rectangles}.

We focus first on the estimation of $W_1(\xx)$. As $\phi\in\cS$ we have $L^{\bnu}\phi\in\cS$,
and in light of \eqref{norm-S}
$\cP_{m,k}(L^{\bnu}\phi) \le \cP_{m+n,k}(\phi)$.
We use this and Lemma~\ref{lem:S} to obtain
\begin{align}\label{P1Hp4}
W_1(\xx)&\le c\cP_{m+n,k}(\phi)\DD_{\one,\bell}^{*}(\xx,\xx_0)^{-1} \nonumber
\\
&\times\int_{Q(\xx,\one)}\Big(\sum_{i=1,2}\rho_i(x_i, y_i)^{\alpha_i}\Big)
\big[\DD^*_{\one,\bsigma}(\xx, \xx_0)+\DD^*_{\one,\bsigma}(\yy, \xx_0)]\frac{\DD_{\bdelta,\bsigma}^*(\xx, \yy)}{V(\xx,\bdelta)}d\mu(\yy)
\nonumber
\\
&=c\cP_{m+n,k}(\phi)\Big(\int_{Q(\xx,\one)}\DD^*_{\one,\bsigma}(\xx, \xx_0) \cdots + \int_{Q(\xx,\one)}\DD^*_{\one,\bsigma}(\yy, \xx_0) \cdots\Big)
\\
&=:c\cP_{m+n,k}(\phi)(J_1(\xx)+J_2(\xx)). \nonumber
\end{align}
To estimate $J_1(\xx)$ we note first that
$\rho_i(x_i, y_i)^{\alpha_i}<\delta_i^{\alpha_i}(1+\delta_i^{-1}\rho_i(x_i, y_i))^{\alpha_i}$,
$i=1,2$,
and hence
\begin{equation}\label{P1Hp5}
\sum_{i=1,2}\rho_i(x_i, y_i)^{\alpha_i}<\frac{\delta_1^{\alpha_1}+\delta_2^{\alpha_2}}{\DD^*_{\bdelta,\balpha}(\xx,\yy)}.
\end{equation}
From the choice of $\bsigma$ it follows that $\bsigma\ge\bell$ and $\bsigma>\dd+\balpha$.
Then using these, \eqref{P1Hp5}, and \eqref{tech-1}) we get
\begin{align}
J_1(\xx)&\le c(\delta_1^{\alpha_1}+\delta_2^{\alpha_2})V(\xx,\bdelta)^{-1}
\int_{\XX}\DD^*_{\bdelta,\bsigma-\balpha}(\xx,\yy)d\mu(\yy)
\le c(\delta_1^{\alpha_1}+\delta_2^{\alpha_2}).\label{P1Hp6}
\end{align}

We next estimate $J_2(\xx)$. We may assume $\bdelta<\one$.
The obvious inequalities
$1+\rho_i(x_i,x_{0i})\le(1+\delta_i^{-1}\rho_i(x_i,y_i))(1+\rho_i(y_i,x_{0i}))$, $i=1,2$,
implies that
\begin{equation*}
\DD_{\one,\bell}^{*}(\xx,\xx_0)^{-1} \le  \DD_{\bdelta,\bell}^{*}(\xx,\yy)^{-1}\DD_{\one,\bell}^{*}(\yy,\xx_0)^{-1}.
\end{equation*}
From this and \eqref{P1Hp5} it follows that
\begin{align}\label{P1Hp7}
&\frac{(\sum_{i=1,2}\rho_i(x_i, y_i)^{\alpha_i})\DD^*_{\one,\bsigma}(\yy, \xx_0)\DD_{\bdelta,\bsigma}^*(\xx, \yy)}{\DD_{\one,\bell}^{*}(\xx,\xx_0)}
\nonumber
\\
&\hspace{1.5in}\le (\delta_1^{\alpha_1}+\delta_2^{\alpha_2})\DD_{\bdelta,\bsigma-\bell-\balpha}^{*}(\xx,\yy)\DD_{\one,\bsigma-\bell}^{*}(\yy,\xx_0)
\\
&\hspace{1.5in}\le (\delta_1^{\alpha_1}+\delta_2^{\alpha_2})\DD_{\bdelta,\bsigma-\bell-\balpha}^{*}(\xx,\yy), \nonumber
\end{align}
where we used that $\bsigma>\bell$.
We use \eqref{P1Hp7} and \eqref{tech-1} and the fact that $\bsigma>\dd+\bell+\balpha$ to derive
\begin{equation*}
J_2(\xx)\le c(\delta_1^{\alpha_1}+\delta_2^{\alpha_2})V(\xx,\bdelta)^{-1}
\int_{\XX}\DD^*_{\bdelta,\bsigma-\balpha}(\xx,\yy)d\mu(\yy)
\le c(\delta_1^{\alpha_1}+\delta_2^{\alpha_2}).
\end{equation*}
From this, \eqref{P1Hp4}, and \eqref{P1Hp6} it follows that
\begin{equation}\label{est-W1}
W_1(x) \le c \cP_{m+n,k}(\phi)(\delta_1^{\alpha_1}+\delta_2^{\alpha_2}).
\end{equation}

We next estimate $W_2(\xx)$. We split it into two:
\begin{align}\label{P1Hp9}
W_2(\xx)&\le c\DD_{\one,\bell}^{*}(\xx,\xx_0)^{-1}
\int_{\XX\setminus Q(\xx,\one)}\frac{\DD_{\bdelta,\bsigma}^*(\xx, \yy)}{V(\xx,\bdelta)}|L^{\bnu}(\phi)(\xx)|d\mu(\yy)
\nonumber
\\
&+c\DD_{\one,\bell}^{*}(\xx,\xx_0)^{-1}
\int_{\XX\setminus Q(\xx,\one)}\frac{\DD_{\bdelta,\bsigma}^*(\xx, \yy)}{V(\xx,\bdelta)}|L^{\bnu}(\phi)(\yy)|d\mu(\yy)
\\
&=: W_{21}(\xx)+W_{22}(\xx).\nonumber
\end{align}
Note that because $0\le \nu_i\le n$ and $k>\sigma_1\vee\sigma_2$ we have
\begin{equation}\label{P1Hp10}
|L^{\bnu}\phi(\zz)|\le \cP_{n,k}(\phi)\DD^{*}_{\one,\bsigma}(\zz,\xx_0),\quad \forall\zz\in \XX,
\end{equation}
and using that $\bsigma>\bell$ we obtain
\begin{equation}
\label{P1Hp11}
W_{21}(\xx)\le \cP_{n,k}(\phi)\int_{\XX\setminus Q(\xx,\one)}\frac{\DD_{\bdelta,\bsigma}^*(\xx, \yy)}{V(\xx,\bdelta)}d\mu(\yy).
\end{equation}
Clearly, for any $\yy\in\XX \setminus Q(\xx,\one)$, there exists $\iota\in\{1,2\}$ such that $\rho_{\iota}(x_\iota,y_\iota)>1$
and hence
\begin{equation}\label{P1Hp12}
1<\delta_{\iota}^{\alpha_\iota}(1+\delta_\iota^{-1}\rho_{\iota}(x_\iota,y_\iota))^{\alpha_{\iota}}
<\frac{\delta_1^{\alpha_1}+\delta_2^{\alpha_2}}{\DD^*_{\bdelta,\balpha}(\xx,\yy)}.
\end{equation}
Combining \eqref{P1Hp11}, \eqref{P1Hp12} and using (\ref{tech-1}) and the fact that $\bsigma>\dd+\balpha$ we get
\begin{align}\label{P1Hp13}
W_{21}(\xx)
&\le c\cP_{n,k}(\phi)(\delta_1^{\alpha_1}+\delta_2^{\alpha_2})
\int_{\XX}\frac{\DD^{*}_{\bdelta,\bsigma-\balpha}(\xx,\yy)}{V(\xx,\bdelta)}d\mu(\yy)
\\
&\le c\cP_{n,k}(\phi)(\delta_1^{\alpha_1}+\delta_2^{\alpha_2}).\nonumber
\end{align}
Similarly
\begin{align}\label{P1Hp14}
W_{22}(\xx)
&\le c\cP_{n,k}(\phi)(\delta_1^{\alpha_1}+\delta_2^{\alpha_2})
\int_{\XX}\frac{\DD^{*}_{\bdelta,\bsigma-\balpha}(\xx,\yy)\DD^*_{\one,\bsigma}(\yy,\xx_0)}{\DD^*_{\one,\mm}(\xx,\xx_0)V(\xx,\bdelta)}d\mu(\yy)
\nonumber
\\
&\le c\cP_{n,k}(\phi)(\delta_1^{\alpha_1}+\delta_2^{\alpha_2})
\int_{\XX}\frac{\DD^{*}_{\bdelta,\bsigma-\bell-\balpha}(\xx,\yy)}{V(\xx,\bdelta)}d\mu(\yy)
\\
&\le c\cP_{n,k}(\phi)(\delta_1^{\alpha_1}+\delta_2^{\alpha_2}), \nonumber
\end{align}
where we used \eqref{tech-1} and the fact that $\bsigma>\dd+\bell+\balpha$.
Combining \eqref{P1Hp9}, \eqref{P1Hp13}, and \eqref{P1Hp14} we get
\begin{equation}\label{est-W2}
W_2(x) \le c \cP_{n,k}(\phi)(\delta_1^{\alpha_1}+\delta_2^{\alpha_2}).
\end{equation}

Finally, \eqref{est-W1} and \eqref{est-W2} yield
\begin{equation*}
\DD_{\one,\bell}^{*}(\xx,\xx_0)^{-1}|L^{\bnu}(\phi-\varphi(\bdelta\sqrt{L})\phi)(\xx)|
\le c \cP_{m+n,k}(\phi)(\delta_1^{\alpha_1}+\delta_2^{\alpha_2}),
\end{equation*}
and hence
\begin{equation*}
\cP_{n,\ell}\big(\phi-\varphi(\bdelta \sqrt L)\phi\big)\le c \cP_{m+n,k}(\phi)(\delta_1^{\alpha_1}+\delta_2^{\alpha_2}).
\end{equation*}
In turn this yields \eqref{limit-0}.

The convergence in \eqref{conv-2} in $L^p$ for $f\in L^p$ follows by a standard argument using \eqref{ker-1} and \eqref{ker-2}.
\end{proof}

\section{Spectral spaces}\label{spectral-spaces}

In this section we deal with spectral spaces,
which can be viewed as a generalization of the band limited functions in the classical setting on $\RR^n$.
Our main goal here is to establish Peetre and Nikolski type inequalities that will be used in what follows.

\begin{definition}\label{def:spec-spaces}
The spectral space $\Sigma_\tt$, $\tt=(t_1,t_2)\in(0,\infty)^2$, is defined by
\begin{equation}\label{spec-sp}
\Sigma_\tt :=\{g\in \cS': \theta (\sqrt L)g = g, \text{ for all } \theta\in \cA_\tt\},
\quad \cS':=\cS'(L_1, L_2),
\end{equation}
where
$$
\cA_\tt:=\{\theta\in \cC^\infty_0(\bR^2): \theta(\pm\lambda_1,\pm\lambda_2)=\theta(\lambda_1,\lambda_2)
\text{ and } \theta\equiv 1 \text{ on } [0, t_1]\times[0,t_2]\}
$$
with $\cC^\infty_0(\bR^2)$ being the set of all compactly supported $\cC^\infty$ functions on $\RR^2$.
\end{definition}

The above definition and Proposition~\ref{prop:slow} imply that each $g\in \Sigma_\tt$
is a slowly growing continuous function.

\begin{proposition}\label{prop:slow-2}
Each $g\in \Sigma_\tt$, $\tt\in(0,\infty)^2$, is a slowly growing continuous function in the sense that
there exist constants $N,c>0$ $($depending on $g$$)$ such that
\begin{align}
|g(\xx)| &\le c V(\xx_0,1)^{-1}\prod_{i=1,2}(t_i^N+t_i^{-N})\big(1+\rho_i(x_i,x_{0i})\big)^{N}
\quad\hbox{and} \label{slow-12}
\\
|g(\xx)-g(\xx')|
&\le c V(\xx_0,1)^{-1}\prod_{i=1,2}(t_i^N+t_i^{-N})\big(1+\rho_i(x_i,x_{0i})\big)^{N}
\sum_{i=1,2}\rho_i(x_i,x'_i)^{\alpha_i}  \label{cont-12}
\end{align}
if $\rho_i(x_i,x'_i)\le t_i^{-1}$, $i=1,2$.
Here $\alpha_1, \alpha_2$ are from \eqref{lip}.
\end{proposition}

The following simple claim will be needed:

\begin{proposition}\label{prop:spectral}
Let $\varphi\in \cC^\infty_0(\bR^2)$ be real-valued,
$\varphi(\pm\lambda_1,\pm\lambda_2)=\varphi(\lambda_1,\lambda_2)$,
and $\supp \varphi \subset [-t_1, t_1]\times[-t_2, t_2]$, $t_1, t_2>0$.
Then for any $f\in \cS'=\cS'(L_1, L_2)$ we have
$\varphi(\sqrt{L})f \in \Sigma_\tt$.
\end{proposition}
\begin{proof}
Let $\theta$ be a function as in Definition~\ref{def:spec-spaces}.
Then $\theta(\blambda)\varphi(\blambda)=\varphi(\blambda)$ and in light of Proposition~\ref{prop:f-phi}
for any $f\in \cS'$ we have
$
\theta (\sqrt L)\varphi(\sqrt L)f = \varphi(\sqrt L)f,
$
which completes the proof.
\end{proof}

\subsection{Peetre type maximal inequality}\label{subsec:peetre}

For reader's convenience we recall the following compact notation from \eqref{rectangles-gamma};
it will be used in what follows:
\begin{equation}\label{V-V12}
V(\xx,\tt^{-1})^{\bgamma}:= V_1(x_1,t_1^{-1})^{\gamma_1} V_2(x_2,t_2^{-1})^{\gamma_2},
\quad \tt\in \bR^2_+, \bgamma\in \bR^2.
\end{equation}

\begin{theorem}\label{thm:Peetre-max}
Let $\bgamma, \btau\in\RR^2$, $\btau> 2\dd$, $r>0$, and $\bnu\in \bN_0^2$.
Then there exists a constant $c>0$ 
such that
for any $g\in\Sigma_\tt$ with $\tt\in [1,\infty)^2$ we have
\begin{equation}
\begin{aligned}\label{Peetre-max}
&\tt^{-2\bnu}\sup_{\yy\in \XX}\frac{V(\yy, \tt^{-1})^{\bgamma}|L^\bnu g(\yy)|}
{\prod_{i=1,2}\big(1+t_i\rho(x_i, y_i)\big)^{\tau_i/r}}
\\
&\quad\le c\sup_{\yy\in \XX}\frac{V(\yy, \tt^{-1})^{\bgamma}|g(\yy)|}
{\prod_{i=1,2}\big(1+t_i\rho(x_i, y_i)\big)^{\tau_i/r}}
\le c\cM_{r}\big(V(\cdot,\tt^{-1})^{\bgamma}g(\cdot)\big)(\xx), \;\; \xx\in \XX.
\end{aligned}
\end{equation}
\end{theorem}

\begin{proof}
We first establish the left-hand side inequality in \eqref{Peetre-max}.

Let $g\in\Sigma_\tt$, $\tt\in[1,\infty)^2$, and set $\bdelta:=\tt^{-1}=(t_1^{-1},t_2^{-1})$.
Let $\theta\in \cC^\infty_0(\bR^2)$,
$\theta(\pm\lambda_1,\pm\lambda_2)=\theta(\lambda_1,\lambda_2)$
and $\theta(\blambda)=1$ for $\blambda\in [0,1]^2$.
Then by \eqref{spec-sp} $g=\theta(\bdelta\sqrt{L})g$.
Let $\omega(\lambda_1, \lambda_2):= \lambda_1^{2\nu_1}\lambda_2^{2\nu_2}\theta(\lambda_1, \lambda_2)$.
Clearly,
$\omega(\bdelta \sqrt{L}) = \bdelta^{2\bnu}L^{\bnu}\theta(\bdelta\sqrt{L})$
and hence
\begin{equation}\label{L-nu-g-p}
L^\bnu g(\yy)
=L^\bnu\theta(\bdelta\sqrt{L})g(\yy) = \bdelta^{-2\bnu}\omega(\bdelta \sqrt{L})g(\yy)
= \bdelta^{-2\bnu}\int_\XX\KK_{\omega(\bdelta \sqrt{L})}(\yy,\zz)g(\zz) d\mu(\zz).
\end{equation}
We use Corollary~\ref{cor:gen-local} and \eqref{D-D*} to conclude that
for any $\bsigma\in(0,\infty)^2$
\begin{equation}\label{KL-theta-p}
|\KK_{\omega(\bdelta\sqrt{L})}(\yy,\zz)|\le c_\bsigma V(\xx,\bdelta)^{-1}\DD^*_{\bdelta,\bsigma}(\yy,\zz).
\end{equation}
Therefore,
\begin{equation}\label{P-1}
\bdelta^{2\bnu}|L^\bnu g(\yy)|\le c V(\yy,\bdelta)^{-1} \int_\XX \DD^*_{\bdelta,\bsigma}(\yy,\zz)|g(\zz)| d\mu(\zz).
\end{equation}
We know from \eqref{rectangles-gamma} that
\begin{equation}\label{P-2}
V(\yy,\bdelta)^{\bgamma}\le cV(\zz,\bdelta)^{\bgamma}\prod_{i=1,2}\big(1+\delta_i^{-1}\rho_i(y_i,z_i)\big)^{d_i|\gamma_i|}.
\end{equation}
and obviously
\begin{equation}\label{triangle}
1+\delta_i^{-1}\rho_i(x_i,z_i) \le (1+ \delta_i^{-1}\rho_i(x_i,y_i))(1+\delta_i^{-1}\rho_i(y_i,z_i)),
\quad i=1,2.
\end{equation}
We use these inequalities and \eqref{P-1} with $\sigma_i:=\tau_i/r+ d_i|\gamma_i| +d_i+1$
to obtain
\begin{align*}
&\frac{\bdelta^{2\bnu}V(\yy, \bdelta)^{\bgamma}|L^\bnu g(\yy)|}
{\prod_{i=1,2}\big(1+\delta_i^{-1}\rho(x_i, y_i)\big)^{\tau_i/r}}
\\
&\le c
\int_\XX \frac{V(\yy,\bdelta)^{-1}V(\zz, \bdelta)^{\bgamma}\prod_{i=1,2}\big(1+\delta_i^{-1}\rho_i(y_i,z_i)\big)^{d_i|\gamma_i|}|g(\zz)|}
{\prod_{i=1,2}\big(1+\delta_i^{-1}\rho(x_i, y_i)\big)^{\tau_i/r}
\big(1+\delta_i^{-1}\rho(y_i, z_i)\big)^{\tau_i/r+ d_i|\gamma_i| +d_i+1}} d\mu(\zz)
\\
&\le c V(\yy,\bdelta)^{-1}
\int_\XX \frac{V(\zz, \bdelta)^{\bgamma}|g(\zz)|}
{\prod_{i=1,2}\big(1+\delta_i^{-1}\rho(x_i, z_i)\big)^{\tau_i/r}
\prod_{i=1,2}\big(1+\delta_i^{-1}\rho(y_i, z_i)\big)^{d_i+1}} d\mu(\zz)
\\
&\le c\sup_{\zz\in \XX}\frac{V(\zz, \bdelta)^{\bgamma}|g(\zz)|}
{\prod_{i=1,2}\big(1+\delta_i^{-1}\rho(x_i, z_i)\big)^{\tau_i/r}}
\int_\XX \frac{V(\yy, \bdelta)^{-1}}
{\prod_{i=1,2}\big(1+\delta_i^{-1}\rho(y_i, z_i)\big)^{d_i+1}} d\mu(\zz)
\\
&\le c\sup_{\zz\in \XX}\frac{V(\zz, \bdelta)^{\bgamma}|g(\zz)|}
{\prod_{i=1,2}\big(1+\delta_i^{-1}\rho(x_i, z_i)\big)^{\tau_i/r}},
\end{align*}
which implies the left-hand side inequality in \eqref{Peetre-max}.
Above for the last inequality we used \eqref{tech-1}.


\smallskip

To prove the right-hand side inequality in \eqref{Peetre-max} we consider three cases.

{\em Case 1:} $1<r<\infty$.
With the notation from above it follows from Corollary~\ref{cor:gen-local} and \eqref{D-D*} that
for any $\bsigma\in(0,\infty)^2$
\begin{equation}\label{Ker-theta}
|\KK_{\theta(\bdelta\sqrt{L})}(\xx,\yy)|\le c_\bsigma V(\xx,\bdelta)^{-1}\DD^*_{\bdelta,\bsigma}(\xx,\yy)
\end{equation}
and hence
\begin{equation}\label{theta-g}
|g(\yy)| = |\theta(\bdelta\sqrt{L})(\yy)|\le c V(\yy,\bdelta)^{-1} \int_\XX \DD^*_{\bdelta,\bsigma}(\yy,\zz)|g(\zz)| d\mu(\zz).
\end{equation}
Choose $\sigma_i:=\tau_i/r+ d_i|\gamma_i| +(d_i+1)/r'$, $i=1,2$, where $1/r+1/r'=1$.
Then just as above we use \eqref{P-2} and \eqref{triangle} to obtain
\begin{align*}
&\frac{V(\yy, \bdelta)^{\bgamma}|g(\yy)|}
{\prod_{i=1,2}\big(1+\delta_i^{-1}\rho(x_i, y_i)\big)^{\tau_i/r}}
\\
&\le c
\int_\XX \frac{V(\yy,\bdelta)^{-1}V(\zz, \bdelta)^{\bgamma}\prod_{i=1,2}\big(1+\delta_i^{-1}\rho_i(y_i,z_i)\big)^{d_i|\gamma_i|}|g(\zz)|}
{\prod_{i=1,2}\big(1+\delta_i^{-1}\rho(x_i, y_i)\big)^{\tau_i/r}
\big(1+\delta_i^{-1}\rho(y_i, z_i)\big)^{\tau_i/r+ d_i|\gamma_i| +(d_i+1)/r'}} d\mu(\zz)
\\
&\le c
\int_\XX \frac{V(\yy,\bdelta)^{-1}V(\zz, \bdelta)^{\bgamma}|g(\zz)|}
{\prod_{i=1,2}\big(1+\delta_i^{-1}\rho(x_i, z_i)\big)^{\tau_i/r}
\prod_{i=1,2}\big(1+\delta_i^{-1}\rho(y_i, z_i)\big)^{(d_i+1)/r'}} d\mu(\zz)
\end{align*}
Applying now H\"{o}lder's inequality we get
\begin{align*}
&\frac{V(\yy, \bdelta)^{\bgamma}|g(\yy)|}
{\prod_{i=1,2}\big(1+\delta_i^{-1}\rho(x_i, y_i)\big)^{\tau_i/r}}
\\
&\qquad \le c\Big(\int_\XX \frac{[V(\zz, \bdelta)^{\bgamma}|g(\zz)|]^r}
{V(\yy,\bdelta)\prod_{i=1,2}\big(1+\delta_i^{-1}\rho(x_i, z_i)\big)^{\tau_i}}d\mu(\zz)\Big)^{1/r}
\\
& \hspace{1.3in}\times
\Big(\int_\XX \frac{1}{V(\yy,\bdelta)\prod_{i=1,2}\big(1+\delta_i^{-1}\rho(y_i, z_i)\big)^{d_i+1}} d\mu(\zz)\Big)^{1/r'}
\\
&\qquad\le c\cM_{r}\big(V(\cdot,\bdelta)^{\bgamma}g(\cdot)\big)(\xx),
\end{align*}
where for the last inequality we used Lemma~\ref{lem:max} and \eqref{tech-1}.
The right-hand side inequality in \eqref{Peetre-max} follows from above.

\smallskip

{\em Case 2:} $r=1$.
The proof in this case is similar to the above proof but is easier; we omit it.

\smallskip

{\em Case 3:} $0<r<1$.
To prove the right-hand side inequality in \eqref{Peetre-max} in this case we first show that
it is valid with a constant $c=c(g)$ depending on $g$.

The function $\theta(\blambda)$ from above obeys the assumptions of Proposition~\ref{prop:slow}
and hence there exist constants $N,c>0$, depending on $g$, such that
\begin{align}\label{slow-22}
|g(\xx)|&= |\theta(\bdelta \sqrt{L})g(\xx)|
\\
&\le c(g)V(\xx_0,\bdelta)^{-1}
\prod_{i=1,2}(\delta_i^N+\delta_i^{-N})\big(1+\rho_i(x_i,x_{0i})\big)^{N}. \nonumber
\end{align}
We may assume $N\ge (\tau_1\vee \tau_2)/r$.
On the other hand, from \eqref{V-gamma-xy} we obtain
\begin{align}\label{V-g-xx-p}
V(\xx,\bdelta)^\bgamma
&\le cV(\xx_0,\bdelta)^\bgamma \prod_{i=1,2}\big(1+\delta_i^{-1}\rho_i(x_i, x_{0i})\big)^{d_i|\gamma_i|}
\\
&\le cV(\xx_0,\bdelta)^\bgamma \prod_{i=1,2}(1+\delta_i^{-1})^{d_i|\gamma_i|}\big(1+\rho_i(x_i, x_{0i})\big)^{d_i|\gamma_i|}. \nonumber
\end{align}
This coupled with \eqref{slow-22} implies
\begin{equation}\label{slow-33}
V(\xx,\bdelta)^\bgamma |g(\xx)| \le c(g,\bdelta)\prod_{i=1,2}\big(1+\rho_i(x_i,x_{0i})\big)^{\tilde{N}},
\end{equation}
where $\tilde{N} := N+|\gamma_1| d_1+ |\gamma_2| d_2$.
In turn, \eqref{slow-33} with $\xx$ replaced by $\zz$ and the fact that $\delta_i=t_i^{-1}\le 1$ lead to the estimates
\begin{align*}
&\frac{V(\zz,\bdelta)^\bgamma|g(\zz)|}{\prod_{i=1,2}\big(1+\delta_i^{-1}\rho_i(x_i,z_i)\big)^{\tilde{N}}}
\le \frac{c\prod_{i=1,2}\big(1+\rho_i(z_i,x_{0i})\big)^{\tilde{N}}}{\prod_{i=1,2}\big(1+\delta_i^{-1}\rho_i(x_i,z_i)\big)^{\tilde{N}}}
\\
&\le \frac{c\prod_{i=1,2}\big(1+\rho_i(x_i,z_i)\big)^{\tilde{N}}\prod_{i=1,2}\big(1+\rho_i(x_i,x_{0i})\big)^{\tilde{N}}}
{\prod_{i=1,2}\big(1+\rho_i(x_i,z_i)\big)^{\tilde{N}}},
\end{align*}
implying
\begin{equation}\label{slow-44}
\sup_{\zz\in \XX}\frac{V(\zz,\bdelta)^\bgamma|g(\zz)|}{\prod_{i=1,2}\big(1+\delta_i^{-1}\rho_i(x_i,z_i)\big)^{\tilde{N}}}
\le c(g,\bdelta)\prod_{i=1,2}\big(1+\rho_i(x_i,x_{0i})\big)^{\tilde{N}}.
\end{equation}

Choose $\sigma_i:= \tilde{N}+|\gamma_1|d_1+|\gamma_2|d_2$ and set $\tilde{\sigma}_i:= \tilde{N}$, $i=1,2$.
As before using \eqref{Ker-theta}, \eqref{P-2}, and \eqref{triangle} we obtain
\begin{align*}
V(\zz,\bdelta)^\bgamma|g(\zz)|
& = V(\zz,\bdelta)^\bgamma \Big|\int_\XX\KK_{\theta(\bdelta\sqrt{L})}(\xx,\yy)g(\yy)d\mu(\yy)\Big|
\\
&\le c\int_\XX V(\zz, \bdelta)^\bgamma V(\yy, \bdelta)^{-1}\DD^*_{\bdelta,\bsigma}(\zz,\yy)|g(\yy)| d\mu(\yy)
\\
&\le c\int_\XX V(\yy, \bdelta)^{\bgamma} V(\yy, \bdelta)^{-1}\DD^*_{\bdelta,\tilde{\bsigma}}(\zz,\yy)|g(\yy)| d\mu(\yy)
\end{align*}
and hence
\begin{align*}
&\frac{V(\zz,\bdelta)^\bgamma|g(\zz)|}{\prod_{i=1,2}(1+\delta_i^{-1}\rho_i(x_i,z_i))^{\tilde{N}}}
\\
&\qquad \le c\int_\XX \frac{V(\yy,\bdelta)^\bgamma V(\yy, \bdelta)^{-1}|g(\yy)|}
{\prod_{i=1,2}(1+\delta_i^{-1}\rho_i(x_i,z_i))^{\tilde{N}}(1+\delta_i^{-1}\rho_i(z_i,y_i))^{\tilde{N}}}d\mu(\yy)
\\
&\qquad \le c\int_\XX \frac{V(\yy,\bdelta)^\bgamma |g(\yy)|}
{V(\yy, \bdelta)\prod_{i=1,2}(1+\delta_i^{-1}\rho_i(x_i,y_i))^{\tilde{N}}}d\mu(\yy)
\\
&\qquad \le c\Big(\sup_{\yy\in \XX}\frac{V(\yy,\bdelta)^\bgamma |g(\yy)|}
{\prod_{i=1,2}(1+\delta_i^{-1}\rho_i(x_i,y_i))^{\tilde{N}}}\Big)^{1-r}
\\
&\hspace{1.6in}\times \int_\XX \frac{[V(\yy,\bdelta)^\bgamma |g(\yy)|]^r}
{V(\yy, \bdelta)\prod_{i=1,2}(1+\delta_i^{-1}\rho_i(x_i,y_i))^{\tilde{N}r}}d\mu(\yy).
\end{align*}
We now take $\sup_{\zz\in \XX}$ above and use the fact that by \eqref{slow-44}
$$
\sup_{\yy\in \XX}\frac{V(\yy,\bdelta)^\bgamma |g(\yy)|}
{\prod_{i=1,2}(1+\delta_i^{-1}\rho_i(x_i,y_i))^{\tilde{N}}} \le c(g,\bdelta,\xx) <\infty,
$$
which allows us to divide and obtain
\begin{multline}\label{Peetre-1}
\Big(\sup_{\zz\in \XX}\frac{V(\zz,\bdelta)^\bgamma |g(\zz)|}
{\prod_{i=1,2}(1+\delta_i^{-1}\rho_i(x_i,z_i))^{\tilde{N}}}\Big)^r
\\
\le c\int_\XX \frac{[V(\yy,\bdelta)^\bgamma |g(\yy)|]^r}
{V(\yy, \bdelta)\prod_{i=1,2}(1+\delta_i^{-1}\rho_i(x_i,y_i))^{\tilde{N}r}}d\mu(\yy).
\end{multline}
Taking $\zz=\xx$ on the left we get
\begin{equation*}
[V(\xx,\bdelta)^\bgamma |g(\xx)|]^r
\le c\int_\XX \frac{[V(\yy,\bdelta)^\bgamma |g(\yy)|]^r}
{V(\yy, \bdelta)\prod_{i=1,2}(1+\delta_i^{-1}\rho_i(x_i,y_i))^{\tilde{N}r}}d\mu(\yy).
\end{equation*}
Using this estimate with $\zz$ in the place of $\xx$ and the fact that $\tilde{N}r \ge \tau_1\vee \tau_2$
we obtain
\begin{align*}
&\frac{[V(\zz,\bdelta)^\bgamma|g(\zz)|]^r}{\prod_{i=1,2}(1+\delta_i^{-1}\rho_i(x_i,z_i))^{\tau_i}}
\\
&\le c\int_\XX \frac{[V(\yy,\bdelta)^\bgamma |g(\yy)|]^r}
{V(\yy, \bdelta)\prod_{i=1,2}(1+\delta_i^{-1}\rho_i(x_i,z_i))^{\tau_i}(1+\delta_i^{-1}\rho_i(z_i,y_i))^{\tau_i}}d\mu(\yy)
\\
&\le c\int_\XX \frac{[V(\yy,\bdelta)^\bgamma |g(\yy)|]^r}
{V(\yy, \bdelta)\prod_{i=1,2}(1+\delta_i^{-1}\rho_i(x_i,y_i))^{\tau_i}}d\mu(\yy).
\\
&\le c\int_\XX \frac{[V(\yy,\bdelta)^\bgamma |g(\yy)|]^r}
{V(\xx, \bdelta)\prod_{i=1,2}(1+\delta_i^{-1}\rho_i(x_i,y_i))^{\tau_i-d_i}}d\mu(\yy),
\end{align*}
where for the first inequality we used \eqref{triangle} and for the last \eqref{V-gamma-xy}.
Applying Lemma~\ref{lem:max} using that $\tau_i > 2d_i$
we obtain
\begin{equation}\label{finite}
\frac{[V(\zz,\bdelta)^\bgamma|g(\zz)|]^r}{\prod_{i=1,2}(1+\delta_i^{-1}\rho_i(x_i,z_i))^{\tau_i}}
\le c(g)\big[\cM_r\big(V(\cdot,\bdelta)^\bgamma g(\cdot)\big)(\xx)\big]^r,
\end{equation}
which implies the right-hand side inequality in \eqref{Peetre-max}.
Here the constant depends on $g$ because $N$ depends on $g$.

\smallskip

We are now prepared to complete the proof of the right-hand inequality in \eqref{Peetre-max}.
Recall the notation (see \eqref{rectangles})
\begin{equation*}
Q(\xx,\bdelta):= B_1(x_1, \delta_1)\times B_2(x_2, \delta_2)
\quad\hbox{and}\quad
\DD^*_{\bdelta,\bsigma}(\xx,\yy):=\prod_{i=1,2}(1+\delta_i^{-1}\rho(x_i,y_i))^{-\sigma_i}
\end{equation*}
and observe that $\mu(Q(\xx,\bdelta))=V(\xx, \bdelta)$, see \eqref{def-V}.

Using the operator $\theta(\bdelta\sqrt{L})$ from above we have
\begin{equation}\label{rep-g}
g(\xx)=\theta(\bdelta\sqrt{L}) g(\xx)=\int_{M}\KK_{\theta(\bdelta\sqrt{L})}(\xx,\yy)g(\yy)d\mu(\yy).
\end{equation}
Let $\yy,\zz\in \XX$ and assume $\uu\in Q(\yy,\bdelta)$.
By Theorem~\ref{th:boxsupport} and \eqref{D-D*} it follows that for any $\bsigma\in(0,\infty)^2$
there exist a constant $c>0$ such that
\begin{equation}\label{pm-1}
\big|\KK_{\theta(\bdelta\sqrt{L})}(\yy,\zz)-\KK_{\theta(\bdelta\sqrt{L})}(\uu,\zz)\big|
\le cV(\yy,\bdelta)^{-1}\sum_{i=1,2}\big(\rho_i(y_i,u_i)/\delta_i\big)^{\alpha_i}\DD^*_{\bdelta,\bsigma}(\yy,\zz).
\end{equation}
Fix $0<\varepsilon<1$.
Then
$$
|g(\yy)|\le\inf_{u\in Q(\yy,\varepsilon \bdelta)}|g(\uu)|+\sup_{\uu\in Q(\yy,\varepsilon \bdelta)}|g(\yy)-g(\uu)|.
$$
Denote
\begin{equation*}
g^*(\xx):= \sup_{\yy\in \XX}\frac{V(\yy, \bdelta)^{\bgamma}|g(\yy)|}
{\prod_{i=1,2}\big(1+\delta_i^{-1}\rho(x_i, y_i)\big)^{\tau_i/r}}
= \sup_{\yy\in \XX}V(\yy,\bdelta)^{\bgamma}\DD^*_{\bdelta,\btau/r}(\xx,\yy)|g(\yy)|.
\end{equation*}
Then
\begin{align}
g^*(\xx)&\le \sup_{\yy\in \XX} V(\yy,\bdelta)^{\bgamma}
\DD^*_{\delta,\btau/r}(\xx,\yy)\inf_{\uu\in Q(\yy,\varepsilon \bdelta)}|g(\uu)|\nonumber
\\
&+\sup_{\yy\in \XX} V(\yy,\bdelta)^{\bgamma}\DD^*_{\bdelta,\btau/r}(\xx,\yy)\sup_{\uu\in Q(\yy,\varepsilon \bdelta)}|g(\yy)-g(\uu)|\nonumber
\\
\label{pm-2}
&=:\sup_{\yy\in \XX}g^*_1(\yy)+\sup_{\yy\in \XX}g^*_2(\yy).
\end{align}

We first estimate $g^*_1(\yy)$, $\yy\in \XX$.
Clearly,
$$
\inf_{\uu\in Q(\yy,\varepsilon \bdelta)}|g(\uu)|
\le\Big(\frac{1}{\mu(Q(\yy,\varepsilon \bdelta))}\int_{Q(\yy,\varepsilon \bdelta)} |g(\uu)|^r d\mu(\uu)\Big)^{1/r},
$$
and hence
\begin{equation}
\begin{aligned}\label{pm-3}
g^*_1(\yy)&\le \mu(Q(\yy,\bdelta))^{\bgamma}\DD^*_{\bdelta,\btau/r}(\xx,\yy)
\Big(\frac{1}{\mu(Q(\yy,\varepsilon \bdelta))}\int_{Q(\yy,\varepsilon \bdelta)} |g(\uu)|^r d\mu(\uu)\Big)^{1/r}
\\
&\le c\DD^*_{\bdelta,\btau/r}(\xx,\yy)
\Big(\frac{1}{\mu(Q(\yy,\varepsilon \bdelta))}
\int_{Q(\yy,\varepsilon \bdelta)} \mu(Q(\uu,\bdelta))^{\bgamma}|g(\uu)|\bigr)^r d\mu(\uu)\Big)^{1/r},
\end{aligned}
\end{equation}
where in the last inequality we used that $\mu(Q(\yy,\bdelta))\sim \mu(Q(\uu,\bdelta))$
since for every $\uu\in Q(\yy,\varepsilon \bdelta)$, $Q(\yy,\bdelta)\subset Q(\uu,2\bdelta)$
and $Q(\uu,\bdelta)\subset Q(\yy,2\bdelta)$.

We introduce the abbreviated notation:
\begin{align*}
Q&:=Q(\yy,\varepsilon \bdelta),
\quad
\tilde{Q}:=B_1\big(x_1,\rho_1(x_1,y_1)+\varepsilon \delta_1\big)\times B_2\big(x_2,\rho_2(x_2,y_2)+\varepsilon \delta_2\big),
\\
Q^{*}&:=B_1\big(y_1,2\rho_1(x_1,y_1)+\varepsilon \delta_1\big)\times B_2\big(y_2,2\rho_2(x_2,y_2)+\varepsilon \delta_2\big).
\end{align*}
From (\ref{pm-3}) taking into account that $Q\subset\tilde{Q}$ we get
 \begin{equation}
\begin{aligned}\label{pm-4}
g^*_1(\yy)&\le c\Big(\frac{\mu(\tilde{Q})}{\mu(Q)}\DD^*_{\bdelta,\btau}(\xx,\yy)\Big)^{1/r}
\Big(\frac{1}{\mu(\tilde{Q})}\int_Q \big(V(\uu,\bdelta)^{\bgamma}|g(\uu)|\big)^r d\mu(\uu)\Big)^{1/r}
\\
&\le c\Big(\frac{\mu(\tilde{Q})}{\mu(Q)}\DD^*_{\bdelta,\btau}(\xx,\yy)\Big)^{1/r}
\Big(\frac{1}{\mu(\tilde{Q})}\int_{\tilde Q} \big(V(\uu,\bdelta)^{\bgamma}|g(\uu)|\big)^r d\mu(\uu)\Big)^{1/r}
\\
&\le c\Big(\frac{\mu(\tilde{Q})}{\mu(Q)}\DD^*_{\bdelta,\btau}(\xx,\yy)\Big)^{1/r}
\cM_{r}\big(V(\cdot,\bdelta)^{\bgamma}g(\cdot)\big)(\xx).
\end{aligned}
\end{equation}
By (\ref{rect-doubling}) and the fact that $0<\varepsilon<1$ we get
$$
\mu(\tilde{Q})\le \mu(Q^*)
\le c\mu(Q)\prod_{i=1,2}\Big(1+\frac{2\rho_i(x_i,y_i)}{\varepsilon\delta_i}\Big)^{d_i}
\le c\varepsilon^{-|\dd|}\mu(Q)\DD^*_{\bdelta,\btau}(\xx,\yy)^{-1}
$$
and hence
\begin{equation}\label{pm-5}
\Big(\frac{\mu(\tilde{Q})}{\mu(Q)}\DD^*_{\bdelta,\btau}(\xx,\yy)\Big)^{1/r}\le c\varepsilon^{-|\dd|/r}.
\end{equation}
This inequality coupled with \eqref{pm-4} yields
\begin{equation}\label{pm-6}
g^*_1(\yy) \le c\varepsilon^{-|\dd|/r}\cM_{r}\big(V(\cdot,\delta)^{\bgamma}g(\cdot)\big)(\xx).
\end{equation}

We now estimate $g^*_2(\yy)$.
Using that $\theta(\delta\sqrt{L})g=g$ and \eqref{pm-1} we obtain
\begin{equation*}
\begin{aligned}
\sup_{\uu\in Q}|g(\yy)-g(\uu)|
&\le \sup_{\uu\in Q}\int_\XX |\KK_{\theta(\delta\sqrt{L})}(\yy,\zz)-\KK_{\theta(\delta\sqrt{L})}(\uu,\zz)||g(\zz)|d\mu(\zz)
\\
&\le c \sup_{\uu\in Q}V(\yy,\bdelta)^{-1}\sum_{i=1,2}\big(\rho_i(y_i,u_i)/\delta_i\big)^{\alpha_i}
\int_\XX \DD^*_{\bdelta,\bsigma}(\yy,\zz)|g(\zz)|d\mu(\zz)
\\
&\le c(\varepsilon^{\alpha_1}+\varepsilon^{\alpha_2})V(\yy,\bdelta)^{-1}
\int_\XX \DD^*_{\bdelta,\bsigma}(\yy,\zz)|g(\zz)|d\mu(\zz).
\end{aligned}
\end{equation*}
Choose $\sigma_i:= \tau_i/r+d_i|\gamma_i|+d_i+1$, $i=1,2$.
From above, the definition of $g^*_2(\yy)$, \eqref{P-2}, and \eqref{triangle} we obtain
\begin{align*}
&g^*_2(\yy)
\\
&\le \frac{c(\varepsilon^{\alpha_1}+\varepsilon^{\alpha_2})}{V(\yy, \bdelta)}
\int_\XX\frac{V(\yy, \bdelta)^\bgamma |g(\zz)|}
{\prod_{i=1,2}(1+\delta_i^{-1}\rho_i(x_i,y_i))^{\tau_i/r}(1+\delta_i^{-1}\rho_i(y_i,z_i))^{\sigma_i}}d\mu(\zz)
\\
&\le \frac{c(\varepsilon^{\alpha_1}+\varepsilon^{\alpha_2})}{V(\yy, \bdelta)}
\int_\XX\frac{V(\zz, \bdelta)^\bgamma |g(\zz)|\prod_{i=1,2}(1+\delta_i^{-1}\rho_i(y_i,z_i))^{d_i|\gamma_i|}}
{\prod_{i=1,2}(1+\delta_i^{-1}\rho_i(x_i,y_i))^{\tau_i/r}(1+\delta_i^{-1}\rho_i(y_i,z_i))^{\sigma_i}}d\mu(\zz)
\\
&\le \frac{c(\varepsilon^{\alpha_1}+\varepsilon^{\alpha_2})}{V(\yy, \bdelta)}
\int_\XX\frac{V(\zz, \bdelta)^\bgamma |g(\zz)|}
{\prod_{i=1,2}(1+\delta_i^{-1}\rho_i(x_i,y_i))^{\tau_i/r}(1+\delta_i^{-1}\rho_i(y_i,z_i))^{\tau_i/r+d_i+1}}d\mu(\zz)
\\
&\le c(\varepsilon^{\alpha_1}+\varepsilon^{\alpha_2})
\sup_{\zz\in \XX}\frac{V(\zz, \bdelta)^\bgamma |g(\zz)|}{\prod_{i=1,2}(1+\delta_i^{-1}\rho_i(x_i,z_i))^{\tau_i/r}}
\\
& \hspace{1.5in}\times \frac{1}{V(\yy, \bdelta)}\int_\XX \frac{1}{\prod_{i=1,2}(1+\delta_i^{-1}\rho_i(y_i,z_i))^{d_i+1}}d\mu(\zz)
\\
&\le c(\varepsilon^{\alpha_1}+\varepsilon^{\alpha_2})g^*(\xx).
\end{align*}
Thus
$$
g^*_2(\yy)\le\tilde{c}(\varepsilon^{\alpha_1}+\varepsilon^{\alpha_2})g^*(\xx),
$$
where the constant $\tilde{c}$ is independent of $\varepsilon$.
We choose $\varepsilon$ so that $\tilde{c}(\varepsilon^{\alpha_1}+\varepsilon^{\alpha_2})=1/2$.
Then $\sup_{\yy\in \XX}g^*_2(\yy)\le g^*(\xx)/2$ which combined with (\ref{pm-6}) leads to
\begin{align*}
g^*(\xx)\le &\sup_{\yy\in \XX}g^*_1(\yy)+\sup_{\yy\in \XX}g^*_2(\yy)\\
&\le c\varepsilon^{-|\dd|/r}\cM_{r}\big(V(\cdot,\bdelta)^{\bgamma}g(\cdot)\big)(\xx)+g^*(\xx)/2.
\end{align*}
This implies \eqref{Peetre-max} using that by \eqref{finite} $g^*(\xx)<\infty$
if $\cM\big(V(\cdot,\bdelta)^\bgamma g(\cdot)\big)(\xx)<\infty$.
\end{proof}

\subsection{Nikolski type inequality}\label{subsec:nikolski}

We will often compare different $L^p$-norms of spectral functions using Nikolski type inequalities.

\begin{theorem}\label{thm:Nik}
$(i)$
Let $0<p\le q\le\infty$, $\bnu\in \bN_0^2$ and $\bgamma\in\bR^2$.
Then there exists a constant $c>0$ such that
for any $g\in\Sigma_\tt$ with $\tt\in[1,\infty)^2$ and $\bnu\in \bN_0^2$ we have
\begin{equation}\label{Band-1}
\big\|V(\cdot,\tt^{-1})^{\bgamma} L^{\bnu}g(\cdot)\big\|_q
\le c \tt^{2\bnu}\big\|V(\cdot,\tt^{-1})^{\bgamma+(\frac{1}{q}-\frac{\one}{p})\one} g(\cdot)\big\|_p,
\quad \tt^{2\bnu}:=t_1^{2\nu_1}t_2^{2\nu_2}.
\end{equation}

$(ii)$
If in addition we assume the non-collapsing condition $\eqref{non-collapsing}$ for the coordinate spaces $\XX_1,\XX_2$,
then for any $g\in\Sigma_\tt$, $\tt\in[1,\infty)^2$, and $\bnu\in \bN_0^2$
\begin{equation}\label{Bl}
\| L^\bnu g\|_q\le c \tt^{2\bnu+(\frac{1}{p}-\frac{1}{q})\dd}\| g\|_p.
\end{equation}
Recall that $\dd=(d_1,d_2)$, $\one:=(1,1)$, and
for the definition of $V(\cdot,\tt^{-1})^{\bgamma}$, $\bgamma\in \RR^2$, see $\eqref{V-V12}$.
\end{theorem}
\begin{proof}
(i)
Let $g\in\Sigma_\tt$, $\tt\in[1,\infty)^2$, and set $\bdelta:=\tt^{-1}=(t_1^{-1},t_2^{-1})$.
Let $\theta\in \cC^\infty_0(\bR^2)$,
$\theta(\pm\lambda_1,\pm\lambda_2)=\theta(\lambda_1,\lambda_2)$
and $\theta(\blambda)=1$ for $\blambda\in [0,1]^2$.
Then $g=\theta(\bdelta\sqrt{L})g$.
Let $\omega(\lambda_1, \lambda_2):= \lambda_1^{2\nu_1}\lambda_2^{2\nu_2}\theta(\lambda_1, \lambda_2)$.
Then
$\omega(\bdelta \sqrt{L}) = \bdelta^{2\bnu}L^{\bnu}\theta(\bdelta\sqrt{L})$
and hence
\begin{equation}\label{L-nu-g}
L^\bnu g(\xx)=L^\bnu\theta(\bdelta\sqrt{L})g(\xx) = \bdelta^{-2\bnu}\omega(\bdelta \sqrt{L})g(\xx)
= \bdelta^{-2\bnu}\int_\XX\KK_{\omega(\bdelta \sqrt{L})}(\xx,\yy)g(\yy) d\mu(\yy).
\end{equation}
We use Corollary~\ref{cor:gen-local} and \eqref{D-D*} to conclude that
for any $\bsigma\in(0,\infty)^2$
\begin{equation}\label{KL-theta}
|\KK_{\omega(\bdelta\sqrt{L})}(\xx,\yy)|\le c_\bsigma V(\xx,\bdelta)^{-1}\DD^*_{\bdelta,\bsigma}(\xx,\yy).
\end{equation}

We next show that inequality \eqref{Band-1} is valid when $q=\infty$, i.e.
\begin{equation}\label{Band-2}
\big\|V(\cdot,\tt^{-1})^{\bgamma} L^{\bnu}g\big\|_\infty
\le c \tt^{2\bnu}\big\|V(\cdot,\tt^{-1})^{\bgamma-\frac{\one}{p}} g\big\|_p.
\end{equation}

{\em Case 1:} $1<p<\infty$.
We use \eqref{L-nu-g} and H\"{o}lder's inequality to obtain
\begin{align}
\bdelta^{2\bnu}|L^\bnu g(\xx)|&\le\big\|V(\cdot,\delta)^{\bgamma-\frac{\one}{p}} g\big\|_p
\Big(\int_\XX \big(|\KK_{\omega(\bdelta\sqrt{L})}(\xx,\yy)
|V(\yy,\bdelta)^{-\bgamma+\frac{\one}{p}}\big)^{p'}d\mu(\yy)\Big)^{1/p'} \nonumber
\\
&=:\big\|V(\cdot,\bdelta)^{\bgamma-\frac{\one}{p}} g\big\|_p \cQ. \label{Band-22}
\end{align}
By \eqref{V-gamma-xy} we have
\begin{equation}\label{V-V}
V(\yy,\bdelta)^{-\bgamma+\frac{\one}{p}}
\le cV(\xx,\bdelta)^{-\bgamma+\frac{\one}{p}}\prod_{i=1,2}\big(1+\delta_i^{-1}\rho_i(x_i,y_i)\big)^{d_i|-\gamma_i+\frac{1}{p}|}.
\end{equation}
Choose $\sigma_i := d_i|-\gamma_i+\frac{1}{p}|+(d_i+1)/p'$
and set
$\tilde{\sigma}_i:= p'(\sigma_i-d_i|-\gamma_i+\frac{1}{p}|)$, $i=1,2$.
Clearly $\tilde{\sigma}_i=d_i+1$.
Using \eqref{KL-theta}, \eqref{V-V}, and \eqref{tech-1} we obtain
\begin{equation*}
\cQ\le V(\xx,\bdelta)^{-\bgamma-(1-\frac{1}{p})\one}
\Big(\int_\XX \DD^*_{\bdelta, \tilde{\bsigma}}(\xx,\yy)d\mu(\yy)\Big)^{1/p'}
\le cV(\xx,\bdelta)^{-\bgamma}.
\end{equation*}
This coupled with \eqref{Band-22} implies \eqref{Band-2}.

\smallskip

{\em Case 2:} $p=1$.
The proof of inequality \eqref{Band-2} in this case is similar to the above proof but easier;
we omit it.

\smallskip

{\em Case 3:} $0< p <1$.
We first show that inequality \eqref{Band-2} holds in this case when $\bnu=(0,0)$, i.e.
\begin{equation}\label{Band-3}
\big\|V(\cdot,\tt^{-1})^{\bgamma}g\big\|_\infty
\le c\big\|V(\cdot,\tt^{-1})^{\bgamma-\frac{\one}{p}} g\big\|_p.
\end{equation}
To this end we first prove that this inequality is valid
with a constant $c=c(g)$ depending on $g$.
The function $\theta(\blambda)$ from above obeys the assumptions of Proposition~\ref{prop:slow}
and hence there exist constants $N,c>0$, depending on $g$, such that
\begin{align}\label{slow-2}
|g(\xx)|&= |\theta(\bdelta \sqrt{L})g(\xx)|
\\
&\le c(g)V(\xx_0,\delta)^{-1}
\prod_{i=1,2}(\delta_i^N+\delta_i^{-N})\big(1+\rho_i(x_i,x_{0i})\big)^{N}. \nonumber
\end{align}
Now, precisely as in the proof of Theorem~\ref{thm:Peetre-max} (cf. \eqref{Peetre-1})
with $r$ replaced by $p$ we obtain
\begin{multline*}
\Big(\sup_{\zz\in \XX}\frac{V(\zz,\bdelta)^\bgamma |g(\zz)|}
{\prod_{i=1,2}(1+\delta_i^{-1}\rho_i(x_i,z_i))^{\tilde{N}}}\Big)^p
\\
\le c\int_\XX \frac{[V(\yy,\bdelta)^\bgamma |g(\yy)|]^p}
{V(\yy, \bdelta)\prod_{i=1,2}(1+\delta_i^{-1}\rho_i(x_i,y_i))^{\tilde{N}p}}d\mu(\yy)
\\
\le c\int_\XX [V(\yy,\bdelta)^{\bgamma-(1/p)\one} |g(\yy)|]^p d\mu(\yy).
\end{multline*}
With $\zz=\xx$ the above implies
\begin{equation}\label{nik-g}
\|V(\cdot,\bdelta)^\bgamma g(\cdot)\|_\infty
\le c(g)\|V(\cdot,\bdelta)^{\bgamma-(1/p)\one} g(\cdot)\|_p.
\end{equation}
Here the constant depends on $g$ because $N$ depends on $g$.

We are now ready to complete the proof of \eqref{Band-3}.
Applying \eqref{Band-2} with $p=2$ and $\bnu=(0,0)$ we get
\begin{align}\label{Vg-Vg}
\big\|V(\cdot,\bdelta)^{\bgamma}g\big\|_{\infty}
&\le c\big\|V(\cdot,\bdelta)^{\bgamma-(1/2)\one} g\big\|_2 \nonumber
\\
&=c\Big(\int_\XX \big(V(\xx,\bdelta)^{\bgamma}|g(\xx)|\big)^{2-p}
\big(V(\xx,\bdelta)^{\bgamma-(1/p)\one}|g(\xx)|\big)^{p}d\mu(\xx) \Big)^{1/2}
\\
&\le c\big\|V(\cdot,\bdelta)^{\bgamma} g(\cdot)\big\|_{\infty}^{1-p/2}
\big\|V(\cdot,\bdelta)^{\bgamma-(1/p)\one} g(\cdot)\big\|_{p}^{p/2}. \nonumber
\end{align}
If $\|V(\cdot,\bdelta)^{\bgamma-\frac{\one}{p}} g(\cdot)\|_p=\infty$, then there is nothing to prove.
If $\|V(\cdot,\bdelta)^{\bgamma-\frac{\one}{p}} g(\cdot)\|_p<\infty$,
then by \eqref{nik-g} $\|V(\cdot,\bdelta)^\bgamma g(\cdot)\|_\infty<\infty$
and we can divide in \eqref{Vg-Vg} by
$\big\|V(\cdot,\bdelta)^{\bgamma} g\big\|_{\infty}^{1-p/2}$
to obtain \eqref{Band-3}. 

Finally, using \eqref{Band-2} with $p=2$, \eqref{Vg-Vg}, and \eqref{Band-3} we get
\begin{align*}
\bdelta^{2\bnu}\big\|V(\cdot,\bdelta)^{\bgamma} L^{\bnu}g(\cdot)\big\|_{\infty}
&\le c\big\|V(\cdot,\bdelta)^{\bgamma-\frac{\one}{2}} g\big\|_2
\\
&\le c\big\|V(\cdot,\bdelta)^{\bgamma} g(\cdot)
\big\|_{\infty}^{1-p/2}\big\|V(\cdot,\bdelta)^{\bgamma-\frac{\one}{p}} g(\cdot)\big\|_{p}^{p/2}
\\
&\le c\big\|V(\cdot,\bdelta)^{\bgamma-\frac{\one}{p}} g\big\|_p,
\end{align*}
which completes the proof of \eqref{Band-2}.

\medskip

We next prove \eqref{Band-1} in the case $0<p<q<\infty$.
From  Theorem~\ref{thm:Peetre-max} it follows that for any $\bgamma\in \bR^2$
\begin{align*}
\bdelta^{2\bnu} V(\xx,\delta)^{\bgamma} |L^\bnu g(\xx)|
\le c \cM_{r}\big(V(\cdot,\bdelta)^{\bgamma}g(\cdot)\big)(\xx).
\end{align*}
We choose $0<r<p$ and apply the maximal inequality \eqref{max} to obtain
\begin{equation}\label{Lpbound}
\delta^{2\bnu} \|V(\cdot,\bdelta)^{\bgamma} L^\bnu g(\cdot)\|_p
\le c\| \cM_{r}\big(V(\cdot,\bdelta)^{\bgamma}g(\cdot)\big)\|_p
\le c \|V(\cdot,\bdelta)^{\bgamma}g(\cdot)\|_p.
\end{equation}

Finally, using \eqref{Band-2} with $\bgamma$ replaced by $\bgamma+\frac{\one}{q}$
and \eqref{Lpbound} with $\bgamma$ replaced by $\bgamma+(\frac{1}{q}-\frac{1}{p})\one$ we get
\begin{align*}
\bdelta^{2\bnu}&\big\|V(\cdot,\bdelta)^{\bgamma} L^\bnu g(\cdot)\big\|_{q}
\\
&\le \big(\delta^{2\bnu} \big\|V(\cdot,\bdelta)^{\bgamma+\frac{\one}{q}} L^\bnu g\big\|_{\infty}\big)^{1-\frac{p}{q}}
\big(\bdelta^{2\bnu}\big\|V(\cdot,\bdelta)^{\bgamma+(\frac{1}{q}-\frac{1}{p})\one} L^\bnu g(\cdot)\big\|_{p}\big)^{\frac pq}
\\
&\le c\big\|V(\cdot,\bdelta)^{\bgamma+(\frac{1}{q}-\frac{1}{p})\one} g(\cdot)\big\|_{p},
\end{align*}
which confirms \eqref{Band-1}.

\smallskip

(ii) From  (i) with $\bgamma=(0,0)$ we derive
\begin{equation*}
\bdelta^{2\bnu}\|L^\bnu g\|_q\le c\big\|V(\cdot,\bdelta)^{(\frac{1}{q}-\frac{1}{p})\one} g(\cdot)\big\|_p
\end{equation*}
and using \eqref{rect-non-coll}, which is a consequence of the non-collapsing condition \eqref{non-collapsing},
we arrive at (\ref{Bl}).
\end{proof}

\section{Product Hardy spaces}\label{sec:Hardy}

Hardy spaces $H^p$, $0<p\le 1$, play a fundamental role in Harmonic analysis, see \cite{FS} and \cite{Stein} and the references therein.
R. Coifman and G. Weiss \cite{CW} introduced and pioneered the theory of Hardy spaces on spaces of homogeneous type.
Later Hardy spaces have been extensively studied on Euclidean \cite{DY} and metric spaces associated with operators \cite{HLMMY}.
Recently, in \cite{DKKP2} the equivalence of Hardy spaces defined by several maximal operators as well as atomic Hardy spaces were obtained in
the setting of a doubling metric measure space $(\XX,\rho,\mu)$ in the presence of a non-negative self-adjoint operator $L$
whose heat kernel has Gaussian localization.
In this setting the Hardy space $H^p$, $0<p\le 1$, is defined as the set of distributions $f\in\cS'$ such that
\begin{equation}\label{Hardyold}
\|f\|_{H^p}:=\big\|\sup_{t>0}|e^{-t^2 L}f(\cdot)|\big\|_p<\infty.
\end{equation}

Hardy spaces $H^p$, $0<p<\infty$, in the two-parameter classical setting have been developed mainly by
Sun-Yung Chang and R. Fefferman, see \cite{CF, CF2, CF3, F1} and the references therein. 
In the anisotropic setting the spaces have been studied in \cite{LBY,BLYZ}. 
Hardy spaces $H^p$, $p\ge1$, have been introduced via area functions on product of spaces of homogeneous type associated with operators \cite{CDLWY}.

In the present section we initiate the development of Hardy spaces $H^p$, $0<p<\infty$, in the two-parameter setting on product domains of this article.
We will introduce the bi-variable version of \eqref{Hardyold} and
establish equivalent characterizations via several maximal operators.

We will operate in the setting of this article (see Section \ref{sec:setting}), assuming in addition that
$\mu_i(\XX_i)=\infty$, $i=1,2$. We start with the following

\begin{definition}
The product Hardy space $H^p=H^p(L_1,L_2)$, $0<p\le 1$, is defined as the set of all $f\in\cS'=\cS'(L_1,L_2)$ such that
\begin{equation}\label{Hardynew}
\|f\|_{H^p}:=\big\|\sup_{\tt>\zero}|e^{-t_1^2 L_1-t_2^2 L_2}f(\cdot)|\big\|_p<\infty.
\end{equation}
\end{definition}

The bi-variate nature of these spaces is exhibited 
in the vector-dilation of the parameter $\tt=(t_1,t_2)$.

\subsection{Maximal operators and the main result}\label{sec:max-operators}

We next introduce several maximal operators that will be used in characterizing the two-parameter (product) Hardy spaces $H^p$.
We first introduce the class of admissible functions on $\bR^2$, suitable for the theory in the product setting.

\begin{definition}\label{def:admissible}
We say that a function $\varphi\in\cS(\bR^2)$ is {\em admissible} if $\varphi$ is real-valued,
$\varphi(0,0)\ne 0$, and
\begin{equation}\label{admissible:equ}
\varphi(\pm\lambda_1,\pm\lambda_2)=\varphi(\lambda_1,\lambda_2),
\quad \forall(\lambda_1,\lambda_2)\in\bR^2.
\end{equation}
\end{definition}

\begin{definition}\label{def:max-op-1}
Let $\varphi\in\cS(\bR^2)$ be an admissible function.
We define the following maximal operators:
Assuming $f\in \cS'$ and $\xx\in \XX$, we set
\begin{equation}\label{def-max-1}
M(f;\varphi)(\xx):= \sup_{\tt>\zero} |\varphi(\tt\sqrt L)f(\xx)|,
\end{equation}
\begin{equation}\label{def-max-2}
M^*_\ba(f;\varphi)(\xx):= \sup_{\tt>\zero}\sup_{\{\yy\in \XX: \rho_i(x_i, y_i)\le a_i t_i\}} |\varphi(\tt\sqrt L)f(\yy)|, \quad \ba\ge \zero,
\end{equation}
and
\begin{equation}\label{def-max-3}
M^{**}_\bgamma(f;\varphi)(\xx):= \sup_{\tt>\zero}\sup_{\yy\in \XX}
\frac{|\varphi(\tt\sqrt L)f(\yy)|}{\prod_{i=1,2}(1+t_i^{-1}\rho_i(x_i,y_i))^{\gamma_i}},\quad \bgamma >\zero.
\end{equation}
Using the compact notation $\DD^*_{\tt,\bgamma}(\xx,\yy):=\prod_{i=1,2}(1+t_i^{-1}\rho_i(x_i,y_i))^{-\gamma_i}$, see \eqref{kernelsD*},
this can be written in the form
\begin{equation}\label{def-max-31}
M^{**}_\bgamma(f;\varphi)(\xx):= \sup_{\tt>\zero}\sup_{\yy\in \XX}
|\varphi(\tt\sqrt L)f(\yy)|\DD^*_{\tt,\bgamma}(\xx,\yy).
\end{equation}
\end{definition}

Obviously,
\begin{equation}\label{relationship}
M(f;\varphi)(\xx) \le M^*_\ba(f;\varphi)(\xx) \le (\one+\ba)^\bgamma M^{**}_\bgamma(f;\varphi)(\xx), \quad \forall f\in\cS',\;\xx\in \XX.
\end{equation}

We next introduce the grand maximal operator.


\begin{definition}\label{def:max-op-2}
Let $\varphi\in \cS(\R^2)$ be admissible and $N\in\bN$. We define
\begin{equation}\label{norm-varphi}
\cN_N(\varphi):= \sup_{\blambda\in\bR^2} (1+|\blambda|)^N \max_{0\le |\bbeta|\le N}|\partial^{\bbeta}\varphi(\blambda)|
\end{equation}
and set
$$
\cF_N:= \{\varphi\in\cS(\bR^2): \varphi\;{\rm is \; admissible \; and } \;\;\cN_N(\varphi)\le 1\}.
$$
The grand maximal operator is defined by
\begin{equation}\label{def-grand-max1}
\sM_N(f)(\xx):= \sup_{\varphi\in\cF_N} M^*_\one(f;\varphi)(\xx), \quad f\in \cS',\;\xx\in \XX,
\end{equation}
that is,
\begin{equation}\label{def-grand-max3}
\sM_{N}(f)(\xx):= \sup_{\varphi\in\cF_N}\sup_{\tt>\zero}\sup_{\{\yy\in \XX: \rho_i(x_i, y_i)\le t_i\}} |\varphi(\tt\sqrt L)f(\yy)|.
\end{equation}
\end{definition}

Observe that for any admissible function $\varphi$, $\ba\ge \one$, and $N\in\bN$,
\begin{equation}\label{est-grand-max}
M^*_\ba(f;\varphi)(\xx) \le |\ba|^N\cN_N(\varphi)\sM_{N}(f)(x),
\quad \forall f\in\cS',\;\xx\in \XX.
\end{equation}
Indeed, consider $\phi(\lambda_1, \lambda_2):= A^{-1}\varphi(a_1^{-1}\lambda_1, a_2^{-1}\lambda_2)$, where $A:=|\ba|^N\cN_N(\varphi)$.
It is readily seen that
$M^*_\ba(f;\varphi)(\xx) = AM^*_\one(f;\phi)(\xx)$
and a simple manipulation shows that
$\cN_N(\phi)\le 1$.
Then \eqref{est-grand-max} follows.

The main result in the section is the following

\begin{theorem}\label{thm:Hp}
Let $0<p\le 1$.
Assume $N >6|\dd|/p+3|\dd|/2+3$, $\ba\ge \one$, $\bgamma>(2/p)\dd$
and let $\varphi\in \cS(\bR^2)$ be admissible. 
Then for any $f\in \cS'$
\begin{equation}\label{equiv-Hp-norms}
\|f\|_{H^p} \sim \|\sM_{N}(f)\|_{L^p}
\sim \|M(f;\varphi)\|_{L^p}
\sim \|M^*_\ba(f;\varphi)\|_{L^p}
\sim \|M^{**}_\bgamma(f;\varphi)\|_{L^p}.
\end{equation}
Here the equivalence constants depend on $\ba$, $\bgamma$, $N$, and $\varphi$, respectively.
\end{theorem}

For the proof of this theorem we need some preparation.
We first establish an useful decomposition identity.

\begin{lemma}\label{lem:Rych}
Let $\varphi\in \cS(\bR^2)$ be admissible with $\varphi(0,0)=1$ and let $N\in \bN_0$.
Then there exist an admissible function $\psi\in \cS(\bR^2)$ with $\psi(0,0)=1$
and a function $\Psi\in \cS(\bR^2)$ with the properties
$\Psi(\pm \lambda_1, \pm \lambda_2)=\Psi(\lambda_1, \lambda_2)$ for $\blambda\in\R^2$ and
$\partial^{\bbeta}\Psi(0,0)=0$ for $|\bbeta|\le N$
such that for any $f\in \cS'$ and $\jj\in \bZ^2$
\begin{equation}\label{dec-Rych}
f = \psi(2^{-\jj}\sqrt L)\varphi(2^{-\jj}\sqrt L)f
+ \sum_{k=1}^\infty \Psi(2^{-\jj-\kk}\sqrt L)\big(\varphi(2^{-\jj-\kk}\sqrt L) - \varphi(2^{-\jj-\kk+\one}\sqrt L)\big)f,
\end{equation}
where the convergence is in $\cS'$ and $\kk:=(k,k)$, $\one:=(1,1)$.
\end{lemma}

\begin{proof}
Let $\varphi$ be from the hypothesis of the lemma and assume $N\ge 2$.
From the binomial formula it readily follows that
\begin{equation}\label{binom}
\sum_{m=1}^{N}{\textstyle\binom{N}{m}}\varphi(\blambda)^{2m}(1-\varphi(\blambda)^2)^{N-m}=1-(1-\varphi(\blambda)^2)^{N}
\end{equation}
and for $k\ge 1$
\begin{align*}
\sum_{m=1}^{N}{\textstyle\binom{N}{m}}&(\varphi(2^{-k}\blambda)^{2}-\varphi(2^{-k+1}\blambda)^{2})^m(1-\varphi(2^{-k}\blambda)^2)^{N-m}
\\
&=(1-\varphi(2^{-k+1}\blambda)^2)^{N}-(1-\varphi(2^{-k}\blambda)^2)^{N}.
\end{align*}
Summing up the above identities and using that $\varphi(0,0)=1$ we get
\begin{align*}
\sum_{k=1}^{\infty}\sum_{m=1}^{N}{\textstyle\binom{N}{m}}&(\varphi(2^{-k}\blambda)^{2}-\varphi(2^{-k+1}\blambda)^{2})^m(1-\varphi(2^{-k}\blambda)^2)^{N-m}
=(1-\varphi(\blambda)^2)^{N}.
\end{align*}
Combining this with \eqref{binom} we find
\begin{align*}
1&=\sum_{m=1}^{N}{\textstyle\binom{N}{m}}\varphi(\blambda)^{2m}(1-\varphi(\blambda)^2)^{N-m}
\\
&+\sum_{k=1}^{\infty}\sum_{m=1}^{N}{\textstyle\binom{N}{m}}(\varphi(2^{-k}\blambda)^{2}-\varphi(2^{-k+1}\blambda)^{2})^m(1-\varphi(2^{-k}\blambda)^2)^{N-m},
\end{align*}
which takes the form
\begin{equation}
\label{L2Hp1}
1=\psi(\blambda)\varphi(\blambda)+\sum_{k=1}^{\infty}\Psi(2^{-k}\blambda)(\varphi(2^{-k}\blambda)-\varphi(2^{-k+1}\blambda)),
\end{equation}
where
$$
\psi(\blambda):=
\sum_{m=1}^{N}{\textstyle\binom{N}{m}}\varphi(\blambda)^{2m-1}(1-\varphi(\blambda)^2)^{N-m}
$$
and
$$
\Psi(\blambda):=(\varphi(\blambda)+\varphi(2\blambda))\sum_{m=1}^{N}{\textstyle\binom{N}{m}}
(\varphi(\blambda)^{2}-\varphi(2\blambda)^{2})^{m-1}(1-\varphi(\blambda)^2)^{N-m}.
$$
Note that since $\varphi$ is admissible, then $\psi$ and $\Psi$ are admissible as well.
Also by the construction $\psi(0,0)=1$ and $\partial^{\beta}\Psi(0,0)=0$ for $|\beta|\le N-2$.
We now replace $\blambda$ by $2\blambda$ in \eqref{L2Hp1} and subtract the resulting identity from \eqref{L2Hp1} to obtain
\begin{equation*}
\Psi(\blambda)(\varphi(\blambda)-\varphi(2\blambda))=\psi(\blambda)\varphi(\blambda)-\psi(2\blambda)\varphi(2\blambda).
\end{equation*}
From this it readily follows that for any $m\in\bN$
\begin{align*}
\sum_{k=1}^{m}\Psi(2^{-\jj-\kk}\blambda)&(\varphi(2^{-\jj-\kk}\blambda)-\varphi(2^{-\jj-\kk+1}\blambda))
\\
&=\psi(2^{-\jj-\mm}\blambda)\varphi(2^{-\jj-\mm}\blambda)-\psi(2^{-\jj}\blambda)\varphi(2^{-\jj}\blambda),
\quad \mm:=(m,m).
\end{align*}
and hence for any $f\in\cS'$
\begin{align}\label{rep-f}
\psi(2^{-\jj}\sqrt{L})\varphi(2^{-\jj}\sqrt{L})f&+\sum_{k=1}^{m}\Psi(2^{-\jj-k}\sqrt{L})
\big(\varphi(2^{-\jj-k}\sqrt{L})-\varphi(2^{-\jj-k+1}\sqrt{L})\big)f \nonumber
\\
&=\psi(2^{-\jj-\mm}\sqrt{L})\varphi(2^{-\jj-\mm}\sqrt{L})f.
\end{align}
From Theorem~\ref{thm:converge} it follows that
$\lim_{m\to\infty} \psi(2^{-\jj-\mm}\sqrt{L})\varphi(2^{-\jj-\mm}\sqrt{L})f =f$ in $\cS'$
and passing to the limit in \eqref{rep-f} as $m\to\infty$ we obtain \eqref{dec-Rych}.
Finally, the lemma follows from the above considerations by replacing $N$ with $N+2$.
\end{proof}

Our next step is to connect different maximal operators from Section \ref{sec:max-operators} by means of the Hardy-Littlewood maximal operator.

\begin{proposition}\label{prop:M-M-1}
Let $\varphi\in\cS(\bR^2)$ be admissible and $0<\theta\le 1$. 
Let also $\bgamma=(\gamma_1,\gamma_2)>\zero$ be so that $\bgamma>(2/\theta)\dd$.
Then there exists a constant $c>0$ such that
\begin{equation}\label{M-M-1}  
M_{\bgamma}^{**}(f;\varphi)(\xx)\le c \cM_{\theta}(M(f;\varphi))(\xx),\quad\text{for}\;f\in\cS',\;\xx\in \XX,
\end{equation}
where $\cM_{\theta}$ is the Hardy-Littlewood maximal operator, see \eqref{Max1}.
\end{proposition}
\begin{proof}
Let $\varphi, \theta$ and $\bgamma$ be from the hypothesis of the proposition.
Without loss of generality we may assume that $\varphi(0,0)=1$; otherwise we use $\varphi(0,0)^{-1}\varphi$ instead.
Fix $f\in\cS'$ and $\xx\in \XX$.

\smallskip

\noindent
{\bf Part 1.}
We first show that inequality \eqref{M-M-1} is valid under the additional assumption
\begin{equation}\label{M-infty}
M_{\bgamma}^{**}(f;\varphi)(\xx) <\infty.
\end{equation}
By Lemma~\ref{lem:Rych} for any $N\in\bN_0$ there exists an admissible function $\psi\in \cS(\bR^2)$ with $\psi(0,0)=1$
and a function $\Psi\in \cS(\bR^2)$ with the properties
$\Psi(\pm \lambda_1, \pm \lambda_2)=\Psi(\lambda_1, \lambda_2)$ for $\blambda\in \R^2$,
and $\partial^{\bbeta}\Psi(0,0)=0$ for $|\bbeta|\le N$
such that
\begin{equation*}
f = \psi(2^{-\jj}\sqrt L)\varphi(2^{-\jj}\sqrt L)f
+ \sum_{k=1}^\infty \Psi(2^{-\jj-\kk}\sqrt L)(\varphi(2^{-\jj-\kk}\sqrt L) - \varphi(2^{-\jj-\kk+\one}\sqrt L))f,
\end{equation*}
where $\kk:=(k,k)$.
We choose $N:=\lfloor 3|\bgamma|+3|\dd|/2+4\rfloor$.

Fix $\tt>\zero$ and let $\jj\in\bZ^2$ be such that $2^{-\jj}\le\tt<2^{-\jj+\one}$.
In light of \eqref{def-max-31} we need to deal with $|\varphi(\tt\sqrt{L})f(\yy)|\DD^{*}_{\tt,\bgamma}(\xx,\yy)$.
From above, using that $t_i\sim 2^{-j_i}$, we get
\begin{align}\label{varphi-0}
&|\varphi(\tt\sqrt{L})f(\yy)|\DD^{*}_{\tt,\bgamma}(\xx,\yy)
\le c|\varphi(\tt\sqrt{L})\psi(2^{-\jj}\sqrt L)\varphi(2^{-\jj}\sqrt L)f(\yy)|\DD^{*}_{2^{-\jj},\bgamma}(\xx,\yy)
\\
&+ c\sum_{k=1}^\infty |\varphi(\tt\sqrt{L})\Psi(2^{-\jj-\kk}\sqrt L)
\big[\varphi(2^{-\jj-\kk}\sqrt L) - \varphi(2^{-\jj-\kk+\one}\sqrt L)\big]f(\yy)|\DD^{*}_{2^{-\jj},\bgamma}(\xx,\yy). \nonumber
\end{align}
Let $\omega(\blambda):=\varphi(\tt 2^{\jj}\blambda)\Psi(2^{-k}\blambda)$.
Then $\varphi(\tt\sqrt{L})\Psi(2^{-\jj-\kk}\sqrt L)=\omega(2^{-\jj}\sqrt{L})$.
Choose $m:=\lfloor \gamma_1+d_1/2+1\rfloor + \lfloor \gamma_2+ d_2/2+1\rfloor$ and $r:= m+|\dd|+1$.
Since $\varphi$ and $\Psi$ are admissible, then $\omega\in \cS(\R^2)$ and $\omega(\pm\lambda_1,\pm\lambda_2)=\omega(\lambda_1,\lambda_2)$.
Also, using that $\tt\sim 2^{-\jj}$ and $k\ge1$ we get
\begin{equation*}
|\partial^{\bbeta}\omega(\blambda)|\le c(1+|\blambda|)^{-N},\quad \blambda\in \R^2.
\end{equation*}
From this it readily follow that
\begin{equation*}
|\partial^{\bbeta}\omega(\blambda)|\le c2^{-k(N-r)/2}(1+|\blambda|)^{-r},
\quad\hbox{if $\; |\blambda|\ge 2^{k/2}$, $|\bbeta|\le N$,}
\end{equation*}
and since $N \ge 3|\bgamma|+3|\dd|/2+3+2\tilde{\eps}$ for some $\tilde{\eps}>0$
it follows that
\begin{equation}\label{omega}
|\partial^{\bbeta}\omega(\blambda)|\le c2^{-k(|\bgamma|+\tilde{\eps})}(1+|\blambda|)^{-r},
\quad|\blambda|\ge 2^{k/2}, |\bbeta|\le N.
\end{equation}
On the other hand,
since $\partial^{\bbeta}\Psi(0,0)=0$ for $|\bbeta|\le N$,
Taylor's formula implies that
$|\partial^{\bbeta}\Psi(\blambda)|\le c|\blambda|^{N-|\bbeta|}$ for $|\blambda|\le 1$ and $|\bbeta|\le N$.
Then for $|\blambda|<2^{k/2}$, we get
\begin{align*}
|\partial^{\bbeta}(\Psi(2^{-k}\blambda))|=2^{-k|\bbeta|}|(\partial^{\bbeta}\Psi)(2^{-k}\blambda)|\le c2^{-k|\bbeta|}(2^{-k}|\blambda|)^{N-|\bbeta|}
\le c2^{-kN/2}.
\end{align*}
From this and the fact that $\varphi\in\cS(\R^2)$ it follows that
\begin{equation*}
|\partial^{\bbeta}\omega(\blambda)|\le c2^{-kN/2}(1+|\blambda|)^{-N}\le c2^{-k(|\bgamma|+\tilde{\eps})}(1+|\blambda|)^{-r},
\quad |\blambda|<2^{k/2}.
\end{equation*}
Combining this estimate with \eqref{omega} leads to
\begin{equation}\label{beta-omega}
|\partial^{\bbeta}\omega(\blambda)|\le c2^{-k(|\bgamma|+\tilde{\eps})}(1+|\blambda|)^{-r},
\quad \blambda\in \R^2, |\bbeta|\le m.
\end{equation}
We use this and apply Theorem~\ref{thm:gen-local} with the roles of $k_1, k_2$ played by $\lfloor \gamma_i+d_i/2+1\rfloor$, $i=1,2$,
to obtain
\begin{align}\label{P2Hp2}
&|\KK_{\varphi(\tt\sqrt{L})\Psi(2^{-\jj-\kk}\sqrt{L})}(\xx,\yy)|
=|\KK_{\omega(2^{-\jj}\sqrt{L})}(\xx,\yy)|
\\
&\le c2^{-k(|\bgamma|+\tilde{\eps})}\DD_{2^{-\jj},\bgamma+\dd/2}(\xx,\yy)
\le c2^{-k(|\bgamma|+\tilde{\eps})}V(\xx,2^{-\jj})^{-1}\DD^*_{2^{-\jj},\bgamma}(\xx,\yy), \nonumber
\end{align}
where for the last inequality we used \eqref{rect-doubling2}.
In turn this estimate implies
\begin{align}\label{varphi-1}
\big|\varphi(\tt\sqrt{L})&\Psi(2^{-\jj-\kk}\sqrt L)
(\varphi(2^{-\jj-\kk}\sqrt L) - \varphi(2^{-\jj-\kk+\one}\sqrt L))f(\yy)\big|\DD^{*}_{2^{-\jj},\bgamma}(\xx,\yy) \nonumber
\\
&\le \DD^{*}_{2^{-\jj},\bgamma}(\xx,\yy)\int_{\XX}|\KK_{\omega(2^{-\jj}\sqrt{L})}(\yy,\zz)|
\\
&\hspace{1in}\;\; \times(|\varphi(2^{-\jj-\kk}\sqrt L)f(\zz)|+|\varphi(2^{-\jj-\kk+\one}\sqrt L)f(\zz)|)d\mu(\zz) \nonumber
\\
&\le c2^{-k(|\bgamma|+\tilde{\eps})}\DD^{*}_{2^{-\jj},\bgamma}(\xx,\yy)
\int_{\XX}\frac{\DD^*_{2^{-\jj},\bgamma}(\yy,\zz)}{V(\zz,2^{-\jj})}(|F_{\jj,k}(\zz)|+|F_{\jj,k-1}(\zz)|)d\mu(\zz), \nonumber
\end{align}
where we used the short-hand notation $F_{\jj,k}(\zz):=\varphi(2^{-\jj-\kk}\sqrt L)f(\zz)$.

On the other hand,
$1+2^{j_i}\rho_i(x_i,z_i)\le(1+2^{j_i}\rho_i(x_i,y_i))(1+2^{j_i}\rho_i(y_i,z_i))$ 
and hence
\begin{equation*}
2^{-k|\bgamma|}\DD^{*}_{2^{-\jj},\bgamma}(\xx,\yy)\DD^*_{2^{-\jj},\bgamma}(\yy,\zz)
\le 2^{-k|\bgamma|}\DD^*_{2^{-\jj},\bgamma}(\xx,\zz)\le\DD^{*}_{2^{-\jj-\kk},\bgamma}(\xx,\zz).
\end{equation*}
With $0<\theta\le 1$ from the hypothesis of the proposition this leads to
\begin{align*}
2^{-k(|\bgamma|+\tilde{\eps})}\DD^{*}_{2^{-\jj},\bgamma}(\xx,\yy)&\DD^*_{2^{-\jj},\bgamma}(\yy,\zz)|F_{\jj,k}(\zz)|
\\
&\le c2^{-k\tilde{\eps}}M_{\bgamma}^{**}(f;\varphi)(\xx)^{1-\theta}\DD^{*}_{2^{-\jj-\kk},\theta\bgamma}(\xx,\zz)|F_{\jj,k}(\zz)|^{\theta}.
\end{align*}
Combining this with \eqref{varphi-1} we get
\begin{align*}
& |\varphi(\tt\sqrt{L})\Psi(2^{-\jj-\kk}\sqrt L)(\varphi(2^{-\jj-\kk}\sqrt L) - \varphi(2^{-\jj-\kk+\one}\sqrt L))f(\yy)|\DD^{*}_{2^{-\jj},\bgamma}(\xx,\yy)
\\
&\le c2^{-k\tilde{\eps}}M_{\bgamma}^{**}(f;\varphi)(\xx)^{1-\theta}
\int_{\XX}\DD^{*}_{2^{-\jj-k},\theta\bgamma}(\xx,\zz)\frac{|F_{\jj,k}(\zz)|^{\theta}+|F_{\jj,k-1}(\zz)|^{\theta}}{V(\zz,2^{-\jj})}d\mu(\zz).
\end{align*}
In the same way we obtain
\begin{align*}
|\varphi(\tt\sqrt{L})\psi(2^{-\jj}\sqrt L)&\varphi(2^{-\jj}\sqrt L)f(\yy)|\DD^{*}_{2^{-\jj},\bgamma}(\xx,\yy)
\\
&\le cM_{\bgamma}^{**}(f;\varphi)(\xx)^{1-\theta}
\int_{\XX}\DD^{*}_{2^{-\jj},\theta\bgamma}(\xx,\zz)\frac{|F_{\jj,0}(\zz)|^{\theta}}{V(\zz,2^{-\jj})}d\mu(\zz).
\end{align*}
Combining the above estimates with \eqref{varphi-0} we obtain
\begin{align}\label{P2Hp3}
|\varphi(\tt\sqrt{L})f(\yy)|&\DD^{*}_{\tt,\bgamma}(\xx,\yy)
\nonumber
\\
&\le cM_{\bgamma}^{**}(f;\varphi)(\xx)^{1-\theta}\sum_{k=0}^{\infty}2^{-k\tilde{\eps}}
\int_{\XX}\DD^{*}_{2^{-\jj-k},\theta\bgamma}(\xx,\zz)\frac{|F_{\jj,k}(\zz)|^{\theta}}{V(\zz,2^{-\jj})}d\mu(\zz).
\end{align}
By \eqref{rect-doubling2} it follows that
\begin{equation}\label{V-VD}
V(\zz,2^{-\jj})^{-1}\le cV(\xx,2^{-\jj})^{-1}\DD^{*}_{2^{-\jj},\dd}(\xx,\zz)^{-1}
\le cV(\xx,2^{-\jj})^{-1}\DD^{*}_{2^{-\jj-\kk},\dd}(\xx,\zz)^{-1},
\end{equation}
which applied in \eqref{P2Hp3} leads to
\begin{align*}
&|\varphi(\tt\sqrt{L})f(\yy)|\DD^{*}_{\tt,\bgamma}(\xx,\yy)
\\
&\le c[M_{\bgamma}^{**}(f;\varphi)(\xx)]^{1-\theta}\sum_{k=0}^{\infty}2^{-k\tilde{\eps}}
\int_{\XX}V(\xx,2^{-\jj})\DD^{*}_{2^{-\jj-\kk},\theta\bgamma-\dd}(\xx,\zz)|F_{\jj,k}(\zz)|^{\theta}d\mu(\zz)
\\
&\le c[M_{\bgamma}^{**}(f;\varphi)(\xx)]^{1-\theta}\sum_{\nu=0}^{\infty}2^{-\nu\epsilon}\cM_{\theta}(F_{\jj,k})(\xx)^{\theta}
\\
&\le c[M_{\bgamma}^{**}(f;\varphi)(\xx)]^{1-\theta}[\cM_{\theta}(M(f;\varphi))(\xx)]^{\theta},
\end{align*}
where for the former inequality we used Lemma~\ref{lem:max}
and the fact that $\theta\bgamma -\dd> \dd$ due to $\bgamma>(2/\theta)\dd$.
Therefore, using the definition of $M_{\bgamma}^{**}(f;\varphi)$ (see \eqref{def-max-31}) we get
\begin{equation*}
M_{\bgamma}^{**}(f;\varphi)(x) \le c[M_{\bgamma}^{**}(f;\varphi)(\xx)]^{1-\theta}[\cM_{\theta}(M(f;\varphi))(\xx)]^{\theta}.
\end{equation*}
At this point we use our assumption that $M_{\bgamma}^{**}(f;\varphi)(\xx)<\infty$ from \eqref{M-infty} to obtain
the desired estimate \eqref{M-M-1}.


\medskip

\noindent
{\bf Part 2.}
Our next step is to show that under the hypothesis of the proposition
if $f\in\cS'$ and $M_\theta(M(f;\varphi))(\xx)<\infty$,
then
$M_\bgamma^{**}(f;\varphi)(\xx) <\infty$,
i.e. \eqref{M-infty} is valid.
As we know from Theorem~\ref{prop:slow} since $f\in \cS'$, then $\varphi(\tt\sqrt{L})f$ is slowly growing (and continuous).
More explicitly, there exist constants $K, c>0$ depending on $f$ such that
\begin{equation}\label{slow}
|\varphi(\tt\sqrt{L})f(\yy)| \le c\prod_{i=1,2}(t_i^{-K}+t_i^K)(1+\rho(y_i, x_{0i}))^K, \quad \forall \yy\in \XX, \tt>\zero.
\end{equation}
With $K$ from above and $0<\eps<1$, $\beps:=(\eps,\eps)$, we introduce the following auxiliary maximal operator
\begin{align*}
 M_\bgamma^{**}&(f;\varphi)^{\eps, K}(\xx) \nonumber
\\
&:= \sup_{\zero<\tt\le \frac{\one}{\beps}}\sup_{\yy\in\XX}
\frac{|\varphi(\tt\sqrt{L})f(\yy)|}
{\prod_{i=1,2}(1+t_i^{-1}\rho(x_i,y_i))^{\gamma_i}(1+\eps\rho(x_i,y_i))^K(1+\eps t_i^{-K})}.
\end{align*}
It is easy to see that
\begin{equation}\label{M-new-l-infty}
M_\bgamma^{**}(f;\varphi)^{\eps, K}(\xx) \le c(\xx,\eps, K)<\infty.
\end{equation}
Indeed, from \eqref{slow} it follows that for $\zero<\tt\le \frac{\one}{\beps}$
\begin{align*}
&\frac{|\varphi(t\sqrt{L})f(\yy)|}
{\prod_{i=1,2}(1+t_i^{-1}\rho(x_i,y_i))^{\gamma_i}(1+\eps\rho(x_i,y_i))^K(1+\eps t_i^{-K})}
\\
&\qquad \le \frac{c\prod_{i=1,2}(t_i^{-K}+t_i^K)(1+\rho(y_i, x_{0i}))^K}
{\prod_{i=1,2}(1+t_i^{-1}\rho(x_i,y_i))^{\gamma_i}(1+\eps\rho(x_i,y_i))^K(1+\eps t_i^{-K})}
\\
& \qquad\le \frac{c\prod_{i=1,2}(1+t_i^{2K})(1+\rho(y_i, x_{i}))^K(1+\rho(x_i, x_{0i}))^K}
{\eps^{2K+2}\prod_{i=1,2}(1+\rho(x_i,y_i))^K}
\\
& \qquad\le c\eps^{-2K-2}(1+\eps^{-2K})^2\prod_{i=1,2}(1+\rho(x_i, x_{0i}))^K,
\end{align*}
which confirms \eqref{M-new-l-infty}.

Our next goal is to establish the inequality
\begin{equation}\label{M-new-M}
M_\bgamma^{**}(f;\varphi)^{\eps, K}(\xx) \le \tilde{c}\cM_\theta(M(f;\varphi))(\xx),
\end{equation}
where the constant $\tilde{c}= \tilde{c}(f)$ is allowed to depend on $f$, but is independent of $\eps$.
Then passing to the limit as $\eps \to 0$ it will follow that
$M_\bgamma^{**}(f;\varphi)(\xx) \le \tilde{c}\cM_\theta(M(f;\varphi))(\xx)$,
which in turn shows that
\eqref{M-infty} holds and the proof will be complete.

The proof of inequality \eqref{M-new-M} will follow very closely in the footsteps of the proof from Part 1.
For this reason we will only indicated the main differences and omit most of the details.

Just as in Part 1 we apply Lemma~\ref{lem:Rych} but this time with
\begin{equation}\label{def-N}
N:=\lfloor 3|\bgamma| + 6K+3|\dd|/2+4\rfloor,
\end{equation}
where $K$ is from \eqref{slow}.
Fix $0<\eps\le 1$.
Let $\zero<\tt\le \one/\beps$ and choose $\jj\in\bZ^2$ so that  $2^{-\jj} \le \tt< 2^{-\jj+\one}$.
Hence, $2^\jj \ge \one/\tt \ge \beps$.
By Lemma~\ref{lem:Rych}
there exist an admissible function $\psi\in \cS(\bR^2)$ with $\psi(0,0)=1$
and a function $\Psi\in \cS(\bR^2)$ with the properties
$\Psi(\pm \lambda_1, \pm \lambda_2)=\Psi(\lambda_1, \lambda_2)$ for $\blambda\in \R^2$ and
$\partial^{\bbeta}\Psi(0,0)=0$ for $|\bbeta|\le N$
such that
for any $f\in \cS'$ and $\jj\in \bZ^2$ identity \eqref{dec-Rych} is valid.
Therefore,
\begin{align*}
&\frac{|\varphi(\tt\sqrt L)f(\yy)|}
{\prod_{i=1,2}(1+t_i\rho(x_i, y_i))^{\gamma_i}(1+\eps\rho(x_i,y_i))^K(1+\eps t_i^{-K})}
\\
&\qquad \le c\frac{|\varphi(\tt\sqrt L)\psi(2^{-\jj}\sqrt L)\varphi(2^{-\jj}\sqrt L)f(\yy)|}
{\prod_{i=1,2}(1+2^{j_i}\rho(x_i, y_i))^{\gamma_i}(1+\eps\rho(x_i,y_i))^K(1+\eps t_i^{-K})}
\\
&\qquad+ c\sum_{k=1}^\infty
\frac{|\varphi(\tt\sqrt L)\Psi(2^{-\jj-\kk}\sqrt L)[\varphi(2^{-\jj-\kk}\sqrt L)- \varphi(2^{-\jj-\kk+\one}\sqrt L)]f(y)|}
{\prod_{i=1,2}(1+2^{j_i}\rho(x_i, y_i))^{\gamma_i}(1+\eps\rho(x_i,y_i))^K(1+\eps t_i^{-K})}.
\end{align*}
Let $\omega(\blambda):= \varphi(\tt 2^\jj\blambda)\Psi(2^{-\kk}\blambda)$.
Then $\varphi(\tt\sqrt{L})\Psi(2^{-\jj-\kk}\sqrt L)=\omega(2^{-\jj}\sqrt{L})$.
Choose $m:=\lfloor \gamma_1+K+d_1/2+1\rfloor + \lfloor \gamma_2+K+d_2/2+1\rfloor$
and $r:= m+|\dd|+1$.
Now just as in the derivation of \eqref{beta-omega} for some constant $\tilde{\eps}>0$
\begin{equation}\label{local-omega-FF}
|\partial^\bbeta\omega(\blambda)|
\le c2^{-k(|\bgamma|+2K+\tilde{\eps})}(1+|\blambda|)^{-r},\;\; \blambda\in\R^2,
\;\;|\bbeta|\le m.
\end{equation}
Applying Theorem~\ref{thm:gen-local} to $\omega(2^{-\jj}\sqrt L)$
with $\lfloor \gamma_i+K+d_i/2+1\rfloor$, $i=1,2$, playing the role of $k_1,k_2$
and using that $2^{j_i} \ge \eps$, $i=1,2$, we get
\begin{align}\label{local-kernel-2}
|\KK_{\varphi(\tt\sqrt L)\Psi(2^{-\jj-\kk}\sqrt L)}(\yy, \zz)|
&\le \frac{c 2^{-k(|\bgamma|+2K+\tilde{\eps})}}{V(\zz, 2^{-\jj})\prod_{i=1,2}\big(1+2^{j_i}\rho(y_i, z_i)\big)^{\gamma_i+K}} \nonumber
\\
&\le \frac{c 2^{-k(|\bgamma|+2K+\tilde{\eps})}}{V(\zz, 2^{-\jj})\prod_{i=1,2}\big(1+2^{j_i}\rho(y_i, z_i)\big)^{\gamma_i}\big(1+\eps\rho(y_i, z_i)\big)^{K}},
\end{align}
This and the inequalities
$1+2^{j_i}\rho(x_i,z_i) \le (1+2^{j_i}\rho(x_i,y_i))(1+2^{j_i}\rho(y_i,z_i))$
and
$1+\eps\rho(x_i,z_i) \le (1+\eps\rho(x_i,y_i))(1+\eps\rho(y_i,z_i))$
lead to
\begin{align*}
&\frac{|\varphi(\tt\sqrt L)\Psi(2^{-\jj-\kk}\sqrt L)\varphi(2^{-\jj-\kk}\sqrt L)f(\yy)|}
{\prod_{i=1,2}(1+2^{j_i}\rho(x_i, y_i))^{\gamma_i}(1+\eps\rho(x_i, y_i))^K(1+\eps t_i^{-K})}
\\
& \le \frac{c2^{-k(|\bgamma|+2K+\tilde{\eps})}}{\prod_{i=1,2}(1+2^{j_i}\rho(x_i, y_i))^{\gamma_i}(1+\eps\rho(x_i, y_i))^K(1+\eps t_i^{-K})}
\\
& \hspace{1.1in} \times \int_\XX \frac{|\varphi(2^{-\jj-\kk}\sqrt L)f(\zz)| d\mu(\zz)}
{V(\zz, 2^{-\jj})\prod_{i=1,2}\big(1+2^{j_i}\rho(y_i, z_i)\big)^{\gamma_i}\big(1+\eps\rho(y_i, z_i)\big)^K}
\\
& \le \frac{c2^{-k\tilde{\eps}}}{\prod_{i=1,2}(1+2^{kK}\eps 2^{j_iK})}
\\
&\hspace{1.1in} \times\int_\XX \frac{2^{-k|\bgamma|}|\varphi(2^{-\jj-\kk}\sqrt L)f(\zz)| d\mu(\zz)}
{V(\zz, 2^{-\jj})\prod_{i=1,2}\big(1+2^{j_i}\rho(x_i, z_i)\big)^{\gamma_i}\big(1+\eps\rho(x_i, z_i)\big)^K}
\\
& \le \frac{c2^{-k\tilde{\eps}}}{\prod_{i=1,2}(1+\eps 2^{(j_i+k)K})}
\\
&\hspace{1.1in} \times\int_\XX \frac{|\varphi(2^{-\jj-\kk}\sqrt L)f(\zz)| d\mu(\zz)}
{V(\zz, 2^{-\jj})\prod_{i=1,2}\big(1+2^{j_i+k}\rho(x_i, z_i)\big)^{\gamma_i}\big(1+\eps\rho(x_i, z_i)\big)^K}
\\
& \le c2^{-k\tilde{\eps}}
\sup_{\zz\in\XX} \left[\frac{|\varphi(2^{-\jj-\kk}\sqrt L)f(\zz)|}
{\prod_{i=1,2}\big(1+2^{j_i+k}\rho(x_i, z_i)\big)^{\gamma_i}\big(1+\eps\rho(x_i, z_i)\big)^K(1+\eps 2^{(j_i+k)K})}\right]^{1-\theta}
\\
&\hspace{1.1in} \times\int_\XX \frac{|\varphi(2^{-\jj-\kk}\sqrt L)f(\zz)|^\theta d\mu(\zz)}
{V(\zz, 2^{-\jj})\prod_{i=1,2}\big(1+2^{j_i}\rho(x_i, z_i)\big)^{\theta\gamma_i}}
\\
&\le c2^{-k\tilde{\eps}} [M_\bgamma^{**}(f;\varphi)^{\eps,K}(x)]^{1-\theta}
[\cM_\theta(|\varphi(2^{-\jj-\kk}\sqrt L)f|)(\xx)]^\theta.
\\
&\le c2^{-k\tilde{\eps}} [M_\bgamma^{**}(f;\varphi)^{\eps,K}(\xx)]^{1-\theta}
[\cM_\theta(M(f;\varphi))(\xx)]^\theta.
\end{align*}
Here for the last two inequalities we proceeded as in Part~1,
i.e. we used \eqref{V-VD}, Lemma~\ref{lem:max} and the fact that $\bgamma>(2/\theta)\dd$.

We carry out the rest of the proof of \eqref{M-new-M} just as in Part~1,
namely, we derive the estimate
\begin{equation*}
M_\bgamma^{**}(f;\varphi)^{\eps,K}(\xx)
\le c[M_\bgamma^{**}(f;\varphi)^{\eps,K}(\xx)]^{1-\theta}
[\cM_\theta(M(f;\varphi))(\xx)]^\theta
\end{equation*}
and using \eqref{M-new-l-infty} we arrive at \eqref{M-new-M}.
We omit the further details.

Observe that the constant $\tilde{c}$ in \eqref{M-new-M} depends on $f$ because it depends on $K$,
which in turn depends on $f$.


As was already alluded to above passing to the limit as $\eps \to 0$ in inequality \eqref{M-new-M}
leads to the inequality
\begin{equation}\label{M0-M}
M_\bgamma^{**}(f;\varphi)(\xx) \le \tilde{c}\cM_\theta(M(f;\varphi))(\xx).
\end{equation}
This step, however, needs some clarification.
We set
\begin{equation*}
F(\tt,\yy,\eps):= \frac{|\varphi(\tt\sqrt{L})f(\yy)|}
{\prod_{i=1,2}(1+t_i^{-1}\rho(x_i,y_i))^{\gamma_i}(1+\eps\rho(x_i,y_i))^K(1+\eps t_i^{-K})}
\end{equation*}
and
$A:= \tilde{c} \cM_\theta(M(f;\varphi))(\xx)$.
Then
\begin{equation*}
M_\bgamma^{**}(f;\varphi)^{\eps,K}(\xx)=\sup_{\zero<\tt\le\frac{\one}{\beps}} \sup_{\yy\in \XX} F(\tt,\yy,\eps)
\;\hbox{and}\;
M_\bgamma^{**}(f;\varphi)(\xx)=\sup_{\tt>\zero} \sup_{\yy\in \XX} F(\tt,\yy,0).
\end{equation*}
Now, inequality \eqref{M-new-M} can be written as
\begin{equation}\label{sup-F-A}
\sup_{\zero<\tt\le\frac{\one}{\beps}} \sup_{\yy\in \XX} F(\tt,\yy,\eps) \le A.
\end{equation}
Denote $\Omega_n:=\{(\tt,\yy): \tt\in [1/n, n]^2, \rho(\xx,\yy)\le n\}$, $n\in \bN$.
Then from \eqref{sup-F-A} it follows that for $\eps\le n$ we have
\begin{equation}\label{sup-Omega-A}
\sup_{(\tt,\yy)\in\Omega_n} F(\tt,\yy,\eps) \le A.
\end{equation}
Clearly, $\lim_{\eps\to 0} F(\tt, \yy,\eps) = F(\tt,\yy,0)$.
Moreover, in light of \eqref{slow} this convergence is uniform on $\Omega_n$, that is,
\begin{equation*}
\lim_{\eps\to 0}\sup_{(\tt,\yy)\in\Omega_n}|F(\tt,\yy,\eps) - F(\tt,\yy,0)| =0.
\end{equation*}
From this and \eqref{sup-Omega-A} it follows that
$ 
\sup_{(\tt,\yy)\in\Omega_n}F(\tt,\yy,0) \le A.
$ 
Using this and the fact that $(0,\infty)^2\times\XX = \cup_{n\in\bN} \Omega_n$
we arrive at \eqref{M0-M}.
The proof of Proposition~\ref{prop:M-M-1} is complete.
\end{proof}


A variation of the proof of Proposition~\ref{prop:M-M-1} will give us another important ingredient for the proof of Theorem~\ref{thm:Hp}.

\begin{proposition}\label{prop:grant-M}
Let $\varphi\in\cS(\bR^2)$ be admissible and $0<\theta\le 1$. 
Let also $N\in\bN$ be so that $N>6|\dd|/\theta+3|\dd|/2+3$.
Then there exists a constant $c>0$ such that
\begin{equation}\label{grant-M}
\sM_{N}(f)(\xx)\le c \cM_{\theta}(M(f;\varphi))(\xx),\quad\text{for}\;f\in\cS',\;\xx\in \XX.
\end{equation}
\end{proposition}

\begin{proof}
Let $\varphi\in\cS(\bR^2)$ be admissible.
Fix $\phi\in\cF_{N}$, where as in the hypothesis of the proposition
$N>6|\dd|/\theta+3|\dd|/2+3$, $0<\theta \le 1$.
It is readily seen that $M_{\one}^{*}(f;\phi)(\xx)\le 2^{|\bgamma|}M_{\bgamma}^{**}(f;\phi)(\xx)$.
Therefore, in order to establish \eqref{grant-M} it suffices to prove that for some $\bgamma >\zero$
\begin{equation}\label{P3Hpeqsuf}
M_{\bgamma}^{**}(f;\phi)(\xx)\le c\cM_{\theta}(M(f;\varphi))(\xx),
\end{equation}
where $c>0$ is a constant independent of $\phi$.
We choose $\bgamma$ so that $\bgamma >(2/\theta)\dd$ and $N>3|\bgamma|+3|\dd|/2+3$.

Fix $\tt>\zero$ and let $\jj\in\bZ^2$, be  such that $2^{-\jj}\le\tt<2^{-\jj+1}$.
Using \eqref{varphi-0} we get
\begin{align*}
&|\phi(\tt\sqrt{L})f(\yy)|\DD^{*}_{\tt,\bgamma}(\xx,\yy)
\le c|\phi(\tt\sqrt{L})\psi(2^{-\jj}\sqrt L)\varphi(2^{-\jj}\sqrt L)f(\yy)|\DD^{*}_{2^{-\jj},\bgamma}(\xx,\yy)
\\
&+ c\sum_{k=1}^\infty |\phi(\tt\sqrt{L})\Psi(2^{-\jj-k}\sqrt L)
(\varphi(2^{-\jj-\kk}\sqrt L) - \varphi(2^{-\jj-\kk+\one}\sqrt L))f(\yy)|\DD^{*}_{2^{-\jj},\bgamma}(\xx,\yy).
\end{align*}
Just as in the proof of Proposition~\ref{prop:M-M-1} (see \eqref{P2Hp2}) we get
\begin{align*}
|\KK_{\phi(\tt\sqrt{L})\Psi(2^{-\jj-k}\sqrt{L})}(\xx,\yy)|
&\le c2^{-k(|\bgamma|+\varepsilon)}V(\xx,2^{-\jj})^{-1}\DD^*_{2^{-\jj},\bgamma}(\xx,\yy),
\end{align*}
but with the constant $c>0$ independent of $\phi$ due to $\cN_{N}(\phi)\le1$.

Just as in the proof of Proposition~\ref{prop:M-M-1} and using \eqref{M-M-1} we get
\begin{align*}
|\phi(\tt\sqrt{L})f(\yy)|\DD^{*}_{\tt,\bgamma}(\xx,\yy)
&\le cM_{\bgamma}^{**}(f;\varphi)(\xx)^{1-\theta}\cM_{\theta}(M(f;\varphi))(\xx)^{\theta}
\\
&\le c\cM_{\theta}(M(f;\varphi))(\xx),
\end{align*}
which implies \eqref{P3Hpeqsuf} and completes the proof.
\end{proof}

We are now prepared to prove Theorem~\ref{thm:Hp}.

\begin{proof}[Proof of Theorem~\ref{thm:Hp}]
Let $\Phi_0(\blambda):=e^{-|\blambda|^2}$. Clearly, $\Phi_0$ is admissible and $\Phi(\zero)=1$.
Let $0<p\le 1$ and assume $N>6|\dd|/p+3|\dd|/2+3$.
We choose $\theta$ so that $0<\theta<p$ and $N>6|\dd|/\theta+3|\dd|/2+3$.
Now, we apply Proposition~\ref{prop:grant-M} and use the maximal inequality \eqref{max} to obtain
\begin{equation*}
\|\sM_{N}f\|_{L^p}\le c\|\cM_{\theta}(M(f;\Phi_0))\|_{L^p}\le c\|M(f;\Phi_0)\|_{L^p}=c\|f\|_{H^p}.
\end{equation*}
In the other direction we use inequalities \eqref{relationship} and \eqref{est-grand-max} and get
\begin{equation*}
\|f\|_{H^p}=\|M(f;\Phi_0)\|_{L^p}\le \|M_{\one}^{*}(f;\Phi_0)\|_{L^p}\le c\|\sM_{N}f\|_{L^p}.
\end{equation*}
Similarly as above we obtain $\|\sM_{N}f\|_p\sim \|M(f;\varphi)\|_p$,
with the equivalence constants depending in addition on the admissible function $\varphi$.

Let $\bgamma > (2/p)\dd$. This time we choose $\theta$ so that $0<\theta<p$ and $\bgamma > (2/\theta)\dd$.
Applying Proposition~\ref{prop:M-M-1} and the maximal inequality \eqref{max} we obtain
\begin{equation*}
\|M_{\bgamma}^{**}(f;\varphi)\|_p\le c\|\cM_{\theta}(M(f;\varphi))\|_p \le c\|M(f;\varphi)\|_p.
\end{equation*}
The remaining estimates in \eqref{equiv-Hp-norms} follow using \eqref{relationship} and \eqref{est-grand-max}.
\end{proof}

\subsection{Product Hardy spaces $H^p$ when $p>1$}

We close this section by establishing the well expected coincidence between product Hardy and Lebesgue spaces when $p>1$.

\begin{proposition}\label{prop:Hp-p-big}
Let $1 <p<\infty$. Then $H^p=L^p$ with equivalent norms.
\end{proposition}

\begin{proof}
We carry out this proof similarly as in the classical case on $\R^n$.

(a) We first show that $L^p\subset H^p$.
Let $f\in L^p$, $p>1$.
By Lemma~\ref{lem:max}, applied with $\varphi(\blambda) = \Phi_0(\blambda):=e^{-|\blambda|^2}$, we get
\begin{equation*}
|\Phi_0(\tt\sqrt{L})f(\xx)|\le c\cM_1(f)(\xx)
\;\; \Rightarrow \;\;
M(f;\Phi_0)(\xx)\le c\cM_1 f(x), \;\;\xx\in\XX.
\end{equation*}
Now, applying the maximal inequality \eqref{max} we obtain
\begin{equation*}
\|f\|_{H^p}=\|M(f;\Phi_0)\|_{L^p}\le c\|\cM_1 f\|_{L^p}\le c\|f\|_{L^p}.
\end{equation*}
Hence, $f\in H^p$.

(b) For the other direction, assume $f\in H^p$, $p>1$.
Then by Definition~\ref{Hardynew}
\begin{equation*}
\big\|\sup_{\tt>\zero}|\Phi_0(\tt\sqrt{L})f(\cdot)|\big\|_{L^p} = \|f\|_{H^p} <\infty.
\end{equation*}
Hence, the sequence $\{\Phi_0((\frac{1}{n},\frac{1}{n})\sqrt{L})f\}_{n\ge 1}$ is bounded on $L^p$.
By the Banach-Alaoglu theorem there exists a subsequence $\{n_k\}_{k\ge 1}$ and a function $f_0\in L^p$ such that
\begin{equation*}
\Phi_0\Big(\Big(\frac{1}{n_k},\frac{1}{n_k}\Big)\sqrt{L}\Big)f\rightarrow f_0
\end{equation*}
in the weak$^*$ topology of $L^p$.
But by Theorem~\ref{thm:converge}
$$
\Phi_0\Big(\Big(\frac{1}{n_k},\frac{1}{n_k}\Big)\sqrt{L}\Big)f\rightarrow f,\quad\text{in}\;\;\cS'.
$$
Therefore, $f=f_0\in L^p$.
Further, by Theorem~\ref{thm:converge} it follows that
\begin{equation*}
\|\Phi_0(\tt\sqrt{L})f - f\|_{L^p} \to 0
\quad\hbox{as}\quad \tt\to \zero
\end{equation*}
and hence
\begin{equation*}
\|f\|_{L^p} \le \big\|\sup_{\tt>\zero}|\Phi_0(\tt\sqrt{L})f|\big\|_{L^p}  = \|f\|_{H^p}.
\end{equation*}
The proof is complete.
\end{proof}

\subsection{Open problem}

The atomic decomposition of $H^p$ plays an important role in the theory of Hardy spaces.
The atomic decomposition of $H^p$ in the classical setting on $\RR^n$ in dimension $n=1$ was first established by R. Coifman \cite{Coifman}
and by R. Latter \cite{Latter} in dimensions $n>1$.
For more details, see \cite{Stein}.
Atomic decompositions of Hardy spaces are also available in nonclassical settings, see e.g. \cite{Bow}.
Closely related to our development in the present section is the atomic decomposition of the Hardy spaces $H^p$
in the general setting of a metric measure space with the doubling property
and in the presence of a non-negative self-adjoint operator whose heat kernel has Gaussian localization
and the Markov property established in \cite{DKKP2}.

The situation is more complicated in the two-parameter case on product spaces.
An atomic decomposition of the $H^p$ spaces in the two-parameter classical setting has been obtained by
Sun-Yung Chang and R. Fefferman, see \cite{CF, CF2, CF3}.
It is an open problem to obtain an atomic decomposition of the product Hardy spaces $H^p$ considered in this section.

\section{Besov and Triebel-Lizorkin spaces with dominating mixed smoothness}\label{sec:B-F-spaces}

The purpose of this section is to develop the basic theory
of Besov and Triebel-Lizorkin spaces with dominating mixed smoothness on product domains $\XX_1\times \XX_2$
associated with operators $L_1,L_2$ in the setting of this article.
The theory of Besov and Triebel-Lizorkin spaces with dominating mixed smoothness in the classical case
on $\RR^m\times\RR^n$ has been mainly developed by  H.-J. Schmeisser and H. Triebel \cite{ST},
see also \cite{Schmeisser, V} and the references therein.
In \cite{Schmeisser, ST, V} these spaces are denoted by
$S_{pq}^{r_1,r_2}B(\RR^m\times\RR^n)$ and $S_{pq}^{r_1,r_2}F(\RR^m\times\RR^n)$
and the respective norms by
$\|f\,|\,S_{pq}^{r_1,r_2}B(\RR^m\times\RR^n)\|$ and
$\|f\,|\,S_{pq}^{r_1,r_2}F(\RR^m\times\RR^n)\|$.
We will use the following more compact notation for the analogue of these spaces
on product domains in the setting of this article:
$\BB_{pq}^\ss$, $\tBB_{pq}^\ss$, $\FF_{pq}^\ss$, and $\tFF_{pq}^\ss$.
For short we will call them {\em mixed-smoothness B and F-spaces}.
In the previous sections we have developed all necessary tools for this theory.

The basics of the theory of ordinary Besov and Triebel-Lizorkin spaces associated
with the operators $L_1,L_2$ in the setting of this article
will be rapidly developed in the next section.
It is useful for better understanding
of the Besov and Triebel-Lizorkin spaces with dominating mixed smoothness.

\subsection{Definition of mixed-smoothness Besov and Triebel-Lizorkin spaces}\label{subsec:B-F-spaces}

To deal with possible anisotropic geometries of the coordinate spaces $\XX_1, \XX_2$
we introduce two types of
mixed-smoothness Besov and Triebel-Lizorkin spaces in the setting of this article:
(i) {\em Classical mixed-smoothness} B-spaces $\BB_{pq}^{s}= \BB_{pq}^{s}(L)$ and F-spaces $\FF_{pq}^{s}=\FF_{pq}^{s}(L)$,
and
(ii) {\em Nonclassical mixed-smoothness} B-spaces $\tBB_{pq}^{s}=\tBB_{pq}^{s}(L)$ and F-spaces $\tFF_{pq}^{s}=\tFF_{pq}^{s}(L)$.
To define these spaces we introduce two pairs of functions
$\varphi^i_0,\varphi^i \in \cC^\infty(\RR)$, $i=1,2$, satisfying the conditions:
\begin{equation}\label{cond-1}
\begin{aligned}
&(i)\ \hbox{$\varphi^i_0,\varphi^i$ are real-valued and even,}
\\
&(ii)\ \supp\varphi^i_0\subset[-2,2], \
|\varphi^i_0(t)| \ge \hat{c}_i>0 \hbox{  for  } t\in [-5/3, 5/3],
\\
&(iii)\ \supp\varphi^i\subset[-2,2]\setminus [-1/2, 1/2],\
|\varphi^i(t)| \ge \hat{c}_i>0 \;\hbox{ for } t\in [3/5, 5/3].
\end{aligned}
\end{equation}
As before we will use the notation $\varphi^i_n(t):=\varphi^i(2^{-n}t), n\in \bN$.
Then from \eqref{cond-1} it follows that
\begin{equation}\label{cond-2}
\sum_{n\in \bN_0} |\varphi^i_n(t)| \ge \hat{c}_i >0, \quad t\in \RR.
\end{equation}
Further, for each $\jj=(j_1,j_2)\in \bN^2_0$ we set $\varphi_\jj:=\varphi^1_{j_1}\otimes\varphi^2_{j_2}$.
Then from above it follow that 
$$
\sum_{\jj\in\bN_0^2}|\varphi_j(\blambda)|=\prod_{i=1,2} \sum_{j_i\in\bN_0}|\varphi^i_{j_i}(\lambda_i)|
\ge \hat{c}_1 \hat{c}_2>0,
\quad \blambda=(\lambda_1,\lambda_2)\in \RR^2.
$$


\begin{definition}\label{def-B-spaces}
Let $\ss \in \R^2$ and $0<p,q \le \infty$.

$(i)$ The classical mixed-smoothness Besov space  $\BB_{pq}^{\ss}=\BB_{pq}^{\ss}(L)$
is defined as the set of all $f \in \cS'$ such that
\begin{equation}\label{def-Besov-space1}
\|f\|_{\BB_{pq}^{\ss}} :=
\Big(\sum_{\jj\in\bN_0^2} \Big(2^{\jj\cdot\ss}
\|\varphi_\jj(\sqrt{L}) f\|_{p}\Big)^q\Big)^{1/q} <\infty.
\end{equation}

$(ii)$ The non-classical mixed-smoothness Besov space  $\tBB_{pq}^{\ss}= \tBB_{pq}^{\ss}(L)$ is defined as the set
of all $f \in \cS'$ such that
\begin{equation}\label{def-Besov-space2}
\|f\|_{\tBB_{pq}^{\ss}}
:= \Big(\sum_{\jj\in\bN_0^2} \Big( \big\|V(\cdot, 2^{-\jj})^{-\ss/\dd}
\varphi_\jj(\sqrt{L}) f(\cdot)\big\|_{p}\Big)^q\Big)^{1/q} <\infty.
\end{equation}
Above the $\ell^q$-norm is replaced by the sup-norm if $q=\infty$.
\end{definition}

\begin{definition}\label{def-F-spaces}
Let $\ss \in \RR^2$, $0<p< \infty$ and $0<q \le \infty$.

$(i)$ The classical mixed-smoothness Triebel-Lizorkin space  $\FF_{pq}^{\ss}=\FF_{pq}^{\ss}(L)$
is defined as the set of all $f \in \cS'$ such that
\begin{equation}\label{def-F-space1}
\|f\|_{\FF_{pq}^\ss} :=\Big\|\Big(\sum_{\jj\in \bN_0^2}
\Big(2^{\jj\cdot\ss}|\varphi_\jj(\sqrt L) f(\cdot)|\Big)^q\Big)^{1/q}\Big\|_{p} <\infty.
\end{equation}

$(ii)$  The non-classical mixed-smoothness Triebel-Lizorkin space
$\tFF_{pq}^{s}= \tFF_{pq}^{s}(L)$ is defined as the set of all $f \in \cS'$
such that
\begin{equation}\label{def-F-space2}
\|f\|_{\tFF_{pq}^\ss} :=
\Big\|\Big(\sum_{\jj\in \bN_0^2} \Big(V(\cdot, 2^{-\jj})^{-\ss/\dd}
|\varphi_\jj(\sqrt L) f(\cdot)|\Big)^q\Big)^{1/q}\Big\|_{p} <\infty.
\end{equation}
As before above the $\ell^q$-norm is replaced by the sup-norm if $q=\infty$.

\end{definition}

%



Note that the mixed-smoothness B and F-spaces are quasi-Banach spaces (Banach spaces if $p,q\ge1$). Moreover, as is expected, the above spaces 
are independent of the selection of the pair of functions $\varphi^i_0,\varphi^i, i=1,2$,
used in their definition. We establish this in the following proposition.

\begin{proposition}\label{prop:independent}
Let the functions $\{\varphi^i_0,\varphi^i\}$ and $\{\Phi^i_0,\Phi^i\}$, $i=1,2$,
be just as in Definitions~\ref{def-B-spaces}, \ref{def-F-spaces}, see \eqref{cond-1}.
Then the corresponding spaces $\BB^\ss_{pq}(\varphi)$ and $\BB^\ss_{pq}(\Phi)$ are the same
with equivalent quasi-norms: $\|\cdot\|_{\BB^\ss_{pq}(\varphi)} \sim\|\cdot\|_{\BB^\ss_{pq}(\Phi)}$ .
The same holds true for the spaces $\tBB^\ss_{pq}, \FF^\ss_{pq}$, and $\tFF^\ss_{pq}$.
\end{proposition}
\begin{proof}
We will only prove the above statement for the spaces $\tFF^\ss_{pq}$ and $\tBB^\ss_{pq}$.

{\em Claim 1:}
$\tFF^\ss_{pq}(\varphi)= \tFF^\ss_{pq}(\Phi)$ and
$\|\cdot\|_{\tFF^\ss_{pq}(\varphi)} \sim\|\cdot\|_{\tFF^\ss_{pq}(\Phi)}$.
As is well known (e.g. \cite{FJW}) there exist functions
$\Psi^i_0,\Psi^i \in \cC^\infty(\RR)$, $i=1,2$, satisfying \eqref{cond-1} such that
$$
\Phi^i_0(t)\Psi^i_0(t)+\sum_{n=1}^{\infty}\Phi^i(2^{-n}t)\Psi^i(2^{-n}t)=1, \quad t\in \RR.
$$
As before for $n\in \bN$ we denote
$\Phi^i_n(t):=\Phi_i(2^{-n}t)$ and $\Psi^i_n(t):=\Psi^i(2^{-n}t)$,  $i=1,2$.
Also, for $\jj=(j_1,j_2)\in \bN^2_0$ we set
$\Phi_\jj:=\Phi^1_{j_1}\otimes\Phi^2_{j_2}$ and $\Psi_\jj:=\Psi^1_{j_1}\otimes\Psi^2_{j_2}$.

Let $f\in\cS'$. By Corollary~\ref{cor:Calderon} we have the decomposition
$$
f=\sum_{\jj\in\bN_0^2}\Psi_\jj(\sqrt{L})\Phi_\jj(\sqrt{L})f \quad\text{in }\cS',
$$
where we used that for any $\jj=(j_1,j_2)\in \bN_0^2$,
$$(\Phi^1_{j_1}(\sqrt {L_1})\Psi^1_{j_1}(\sqrt {L_1}))\otimes(\Phi^2_{j_2}(\sqrt {L_2})\Psi^2_{j_2}(\sqrt {L_2}))
=\Phi_\jj(\sqrt{L})\Psi_\jj(\sqrt{L}).
$$
Let $\jj=(j_1,j_2)\in\bN_0^2$. From the support conditions of these spectral multipliers
it follows that
\begin{equation}\label{indep-0}
\varphi_\jj(\sqrt{L})f=\sum_{k_2=j_2-1}^{j_2+1}\sum_{k_1=j_1-1}^{j_1+1}\varphi_\jj(\sqrt{L})\Psi_\kk(\sqrt{L})\Phi_\kk(\sqrt{L})f,
\end{equation}
where by definition $\Phi^i_{-1}=\Psi^i_{-1}\equiv 0$.

Fix $0<r<\min\{p,q\}$ and choose
$\sigma_i:=2d_i+|s_i|+ (2d_i+1)/r$, $i=1,2$.
Denote by $\KK(\xx,\yy)$ the kernel of the operator $\varphi_\jj(\sqrt{L})\Psi_\kk(\sqrt{L})$.
From Theorem~\ref{thm:gen-local} we know that for any $\bsigma>0$
there exists a constant $c=c_\bsigma>0$ such that
$$
|\KK_{\varphi_\jj(\sqrt{L})}(\xx,\zz)|\le c\DD_{2^{-\jj},\bsigma}(\xx,\zz)
\quad\text{and}\quad
|\KK_{\Psi_\kk(\sqrt{L})}(\zz,\yy)|\le c\DD_{2^{-\kk},\bsigma}(\zz,\yy).
$$
Using \eqref{rect-doubling} (taking into account that $j_i\sim k_i$), \eqref{tech-3} ($\bsigma >2\dd$), and \eqref{D-D*} we get
\begin{align*}
|\KK(\xx,\yy)|&\le \int_\XX |\KK_{\varphi_\jj(\sqrt{L})}(\xx,\zz)||\KK_{\Psi_\kk(\sqrt{L})}(\zz,\yy)|d\mu(\zz)
\\
&\le c\int_\XX \DD_{2^{-\kk},\bsigma}(\xx,\zz)\DD_{2^{-\kk},\bsigma}(\zz,\yy)d\mu(\zz)
\\
&\le c\DD_{2^{-\kk},\bsigma}(\xx,\yy)\le cV(\xx,2^{-\kk})^{-1}\DD_{2^{-\kk},\bsigma-\dd/2}^*(\xx,\yy).
\end{align*}
Hence,
\begin{align*}
|\varphi_\jj(\sqrt{L})\Psi_\kk(\sqrt{L})&\Phi_\kk(\sqrt{L})f(\xx)|
\le\int_\XX |\KK(\xx,\yy)||\Phi_\kk(\sqrt{L})f(\yy)|d\mu(\yy)
\\
&\le cV(\xx,2^{-\kk})^{-1}\int_\XX \DD_{2^{-\kk},\bsigma-\dd/2}^*(\xx,\yy)|\Phi_\kk(\sqrt{L})f(\yy)|d\mu(\yy).
\end{align*}
We now use \eqref{V-gamma-xy}, \eqref{rect-doubling}, and the fact that $j_i\sim k_i$ to obtain
\begin{align*}
V(\xx,2^{-\jj})^{-\ss/\dd} \DD_{2^{-\kk},\bsigma-\dd/2}^*(\xx,\yy)
&\le cV(\yy,2^{-\kk})^{-\ss/\dd} \DD_{2^{-\kk},\tilde{\bsigma}}^*(\xx,\yy)
\\
& = cV(\yy,2^{-\kk})^{-\ss/\dd} \DD_{2^{-\kk},\btau/r}^*(\xx,\yy)\DD_{2^{-\kk},\bet}^*(\xx,\yy),
\end{align*}
where $\tilde{\sigma}_i := \sigma_i-d_i/2 - |s_i| = (2d_i+1)/r+ 3d_i/2$
and $\tau_i:=2d_i+1$, $\eta_i:= 3d_i/2$.
Therefore,
\begin{align}\label{indep-3}
&V(\xx,2^{-\jj})^{-\ss/\dd}|\varphi_\jj(\sqrt{L})\Psi_\kk(\sqrt{L})\Phi_\kk(\sqrt{L})f(\xx)|
\\
&\le cV(\xx,2^{-\kk})^{-1}\int_\XX \DD_{2^{-\kk},\btau/r}^*(\xx,\yy)\DD_{2^{-\kk},\btau}^*(\xx,\yy)
V(\yy,2^{-\kk})^{-\ss/\dd}|\Phi_\kk(\sqrt{L})f(\yy)|d\mu(\yy) \nonumber
\\
&\le c\sup_{\yy\in \XX} V(\yy,2^{-\kk})^{-\ss/\dd}|\Phi_\kk(\sqrt{L})f(\yy)|\DD_{2^{-\kk},\btau/r}^*(\xx,\yy)
\int_\XX \frac{\DD_{2^{-\kk},\bet}^*(\xx,\yy)}{V(\xx,2^{-\kk})}d\mu(\yy) \nonumber
\\
&\le c\sup_{\yy\in \XX} V(\yy,2^{-\kk})^{-\ss/\dd}|\Phi_\kk(\sqrt{L})f(\yy)|\DD_{2^{-\kk},\btau/r}^*(\xx,\yy), \nonumber
\end{align}
where for the last inequality we used \eqref{tech-1} ($\bet>\dd$).
Because $\supp\Phi_k\subset[-2^{k-1},2^{k+1}]$ Proposition~\ref{prop:spectral} implies that
$\Phi_k(\sqrt{L})f\in\Sigma_{2^{\kk+\one}}$ and
hence by Theorem~\ref{thm:Peetre-max} we get
\begin{equation}\label{indep-2}
\begin{aligned}
\sup_{\yy\in \XX}V(\yy,2^{-\kk})^{-\ss/\dd}&|\Phi_\kk(\sqrt{L})f(\yy)|\DD_{2^{-\kk},\btau/r}^*(\xx,\yy)
\\
&\le c\cM_{r}\big(V(\cdot,2^{-\kk})^{-\ss/\dd}\Phi_\kk(\sqrt{L})f(\cdot)\big)(\xx).
\end{aligned}
\end{equation}
Combining \eqref{indep-0}--\eqref{indep-2} we arrive at
\begin{equation}\label{indep-4}
\begin{aligned}
&V(\xx,2^{-j})^{-\ss/\dd}|\varphi_\jj(\sqrt{L})f(\xx)|
\\
&\le c\sum_{k_2=j_2-1}^{j_2+1}\sum_{k_1=j_1-1}^{j_1+1}
\cM_{r}\big(V(\cdot,2^{-\kk})^{-\ss/\dd}\Phi_\kk(\sqrt{L})f(\cdot)\big)(\xx).
\end{aligned}
\end{equation}
We use now the definition of the $\tFF$-norm and the maximal inequality \eqref{max} to obtain
\begin{align*}
\|f\|_{\tFF^\ss_{pq}(\varphi)}
&=\Big\|\Big(\sum_{\jj\in \bN_0^2} \Big(V(\cdot, 2^{-\jj})^{-\ss/\dd}
|\varphi_\jj(\sqrt L) f(\cdot)|\Big)^q\Big)^{1/q}\Big\|_{p}
\\
&\le c\Big\|\Big(\sum_{\jj\in \bN_0^2} \Big(\sum_{k_2=j_2-1}^{j_2+1}\sum_{k_1=j_1-1}^{j_1+1}
\cM_{r}\big(V(\cdot,2^{-k})^{-\ss/\dd}\Phi_\kk(\sqrt{L})f(\cdot)|\big)\Big)^q\Big)^{1/q}\Big\|_{p}
\\
&\le c\Big\|\Big(\sum_{\jj\in \bN_0^2}\sum_{k_2=j_2-1}^{j_2+1}\sum_{k_1=j_1-1}^{j_1+1} \Big[\cM_{r}\big(V(\cdot,2^{-\kk})^{-\ss/\dd}\Phi_\kk(\sqrt{L})f(\cdot)|\big)\Big]^q\Big)^{1/q}\Big\|_{p}
\\
&\le c\Big\|\Big(\sum_{\kk\in \bN_0^2}
\Big[\cM_{r}\big(V(\cdot,2^{-\kk})^{-\ss/\dd}\Phi_\kk(\sqrt{L})f(\cdot)|\big)\Big]^q\Big)^{1/q}\Big\|_{p}
\\
&\le c\Big\|\Big(\sum_{\kk\in \bN_0^2}
\Big(V(\cdot,2^{-\kk})^{-\ss/\dd}\Phi_\kk(\sqrt{L})f(\cdot)|\Big)^q\Big)^{1/q}\Big\|_{p}
\\
&= c\|f\|_{\tFF^s_{pq}(\Phi)}.
\end{align*}
To prove the inequality in the opposite direction we simply interchange the roles of
the $\varphi$'s and $\Phi$'s above. This completes the proof.

\smallskip

{\em Claim 2:} $\tBB^\ss_{pq}(\varphi)= \tBB^\ss_{pq}(\Phi)$ and
$\|\cdot\|_{\tBB^\ss_{pq}(\varphi)} \sim\|\cdot\|_{\tBB^\ss_{pq}(\Phi)}$.
The proof follows in the footsteps of the proof of the previous claim.
Assuming that $0<r<p$ and applying the maximal inequality \eqref{max}
(for a single function) to (\ref{indep-4}) we get that
\begin{align*}
&\|V(\cdot,2^{-\jj})^{-\ss/\dd}\varphi_\jj(\sqrt{L})f(\cdot)\|_p
\\
&\le c\sum_{k_2=j_2-1}^{j_2+1}\sum_{k_1=j_1-1}^{j_1+1}
\big\|\cM_{r}\big(V(\cdot,2^{-\kk})^{-\ss/\dd}\Phi_\kk(\sqrt{L})f(\cdot)\big)\big\|_p
\\
&\le c\sum_{k_2=j_2-1}^{j_2+1}\sum_{k_1=j_1-1}^{j_1+1}
\big\|V(\cdot,2^{-\kk})^{-\ss/\dd}\Phi_\kk(\sqrt{L})f(\cdot)\big)\big\|_p.
\end{align*}
Finally, taking the $\ell_q$-norm we get
$
\|f\|_{\tBB^\ss_{pq}(\varphi)}\le c\|f\|_{\tBB^\ss_{pq}(\Phi)}.
$
For the inequality in the other direction we switch the roles of the $\varphi$'s and $\Phi$'s.
\end{proof}

\subsection{Embeddings}

The purpose of this subsection is to establish embedding results that involve
mixed-smoothness Besov and Triebel-Lizorkin spaces and other important classes.
Recall that a quasi-normed space $X$ is continuously embedded in the quasi-normed space $Y$ ($X\hookrightarrow Y$)
if $X\subset Y$ and there exists a constant $c>0$ such that
$\|f\|_Y\le c\|f\|_X$ for all $f\in X$.

\subsubsection{Embeddings between B and F-spaces and test functions or distributions}
We start with the following:

\begin{theorem}\label{thm:embed}
Let $\ss\in\bR^2$ and $0<q\le \infty$. If $0<p\le \infty$ then
\begin{equation*}
\cS\hookrightarrow \BB^\ss_{pq},\;\tBB^\ss_{pq}\hookrightarrow\cS',
\end{equation*}
and, if $0<p< \infty$, then
\begin{equation*}
\cS\hookrightarrow \;\FF^\ss_{pq},\;\tFF^\ss_{pq}\hookrightarrow\cS'.
\end{equation*}
\end{theorem}

\begin{proof}
We will establish only the embeddings involving the $\tBB^\ss_{pq}$ spaces.
The proofs of the other embeddings are similar; we omit them.

\smallskip

{\em Claim 1:} $\cS\hookrightarrow \tBB^\ss_{pq}$.
We need to show that there exist constants $m,k\in \bN_0$ and $c>0$ (depending on $\ss,p,q$) such that
\begin{equation}\label{S->B}
\|\phi\|_{\tBB^\ss_{pq}} \le c\cP_{m,k}(\phi),\quad \forall \phi\in\cS.
\end{equation}
In light of the definition of the mixed-smoothness Besov spaces it suffices to prove that
there exist $m,k\in \bN_0$, $c>0$, and $\beps\in \bR_+^2$ such that for all $\jj\in\bN_0^2$
\begin{equation}\label{suff}
\big\|V(\cdot,2^{-\jj})^{-\ss/\dd}\varphi_\jj(\sqrt{L})\phi(\cdot)\big\|_p
\le c\cP_{m,k}(\phi)2^{-\beps\cdot\jj},
\quad \forall \phi\in\cS.
\end{equation}
Assume $\phi\in\cS$ and let $\jj=(j_1,j_2)\in\bN^2$.
The proof below still works with minor changes in the case when $j_1=0$ or $j_2=0$; we omit it.
Choose
\begin{equation}\label{def-km}
k>\max\{d_1(1+1/p)+|s_1|, d_2(1+1/p)+|s_2|\},
\quad
m>\max\{|s_1|+d_1, |s_2|+d_2\}/2
\end{equation}
and set $\mm:=(m,m)$.
We define
$\omega(\lambda_1, \lambda_2):= \lambda_1^{-2m}\lambda_2^{-2m}\varphi^1(\lambda_1)\varphi^2(\lambda_2)$.
Clearly,
$\varphi_\jj(\sqrt{L}) = 2^{-2\jj\cdot\mm}\omega(2^{-\jj}\sqrt{L})L^{\mm}$.
From Theorem~\ref{thm:gen-local} and \eqref{D-D*} it follows that for any $\bsigma>0$
there exists a constant $c_\bsigma>0$ such that
\begin{equation}\label{local-omega}
|\KK_{\omega(2^{-\jj}\sqrt{L})}(\xx,\yy)| \le c_\bsigma V(\xx, 2^{-\jj})^{-1} \DD^*_{2^{-\jj},\bsigma}(\xx,\yy).
\end{equation}
We choose $\bsigma=\kk:=(k,k)$.
On the other hand, because $\phi\in\cS$ we have
\begin{equation}\label{L-phi}
|L^{\mm}\phi(\yy)| \le \cP_{m,k}(\phi)\prod_{i=1,2}(1+\rho_i(y_i,x_{0i}))^{-k}
= \cP_{m,k}(\phi)\DD^*_{\one,\kk}(\yy,\xx_0).
\end{equation}
Also, from \eqref{rect-doubling} it follows that
$V(\xx, 2^{-\jj})^{-1} \le c2^{\jj\cdot\dd}V(\xx, \one)^{-1}$.
This along with \eqref{local-omega} and \eqref{L-phi} yield
\begin{align}\label{varphi-phi}
|\varphi_\jj(\sqrt{L})\phi(\xx)|
&\le c2^{-2\jj\cdot\mm}\int_\XX |\KK_{\omega(2^{-\jj}\sqrt{L})}(\xx,\yy)||L^{\mm}\phi(\yy)| d\mu(\yy) \nonumber
\\
&\le c2^{-2\jj\cdot\mm}\cP_{m,k}(\phi)V(\xx, 2^{-\jj})^{-1}
\int_\XX \DD^*_{2^{-\jj},\kk}(\xx,\yy)\DD^*_{\one,\kk}(\yy,\xx_0)d\mu(\yy)
\\
&\le c2^{-2\jj\cdot\mm+\jj\cdot\dd}\cP_{m,k}(\phi)V(\xx, \one)^{-1}
\int_\XX \DD^*_{\one,\kk}(\xx,\yy)\DD^*_{\one,\kk}(\yy,\xx_0)d\mu(\yy) \nonumber
\\
&\le c2^{-2\jj\cdot\mm+\jj\cdot\dd}\cP_{m,k}(\phi)\DD^*_{\one,\kk-\dd}(\xx,\xx_0), \nonumber
\end{align}
where for the last inequality we used \eqref{tech-2} ($\kk>\dd$).

We next estimate $V(\xx,2^{-j})^{-\ss/\dd}$ using \eqref{rect-doubling} and \eqref{V-gamma-xy}.
We obtain
\begin{align*}
V(\xx,2^{-j})^{-\ss/\dd}
&\le c2^{j_1|s_1|+j_2|s_2|}V(\xx,\one)^{-\ss/\dd}
\\
&\le c2^{j_1|s_1|+j_2|s_2|}V(\xx_0,\one)^{-\ss/\dd}\prod_{i=1,2}\big(1+\rho_i(x_i,x_{0i})\big)^{|s_i|}.
\end{align*}
Combining this with \eqref{varphi-phi} we arrive at
\begin{align*}
&V(\xx,2^{-j})^{-\ss/\dd}|\varphi_\jj(\sqrt{L})\phi(\xx)|
\\
& \le c2^{-j_1\eps_1-j_2\eps_2}
V(\xx_0,\one)^{-\ss/\dd}\cP_{m,k}(\phi)\DD^*_{\one,\kk-\dd}(\xx,\xx_0)
\prod_{i=1,2}\big(1+\rho_i(x_i,x_{0i})\big)^{|s_i|}
\\
& = c2^{-j_1\eps_1-j_2\eps_2}
V(\xx_0,\one)^{-\ss/\dd}\cP_{m,k}(\phi)
\prod_{i=1,2}\big(1+\rho_i(x_i,x_{0i})\big)^{-k+d_i+|s_i|},
\end{align*}
where $\eps_i= 2m-|s_i|-d_i >0$, $i=1,2$ (cf. \eqref{def-km}).
We now take the $L^p$ norm and use \eqref{tech-1} and the fact that
$k> d_i+|s_i|+ d_i/p$ (cf. \eqref{def-km})
to obtain
\begin{align*}
\|V(\cdot,2^{-j})^{-\ss/\dd}|\varphi_\jj(\sqrt{L})\phi(\cdot)|\|_p
\le c2^{-j_1\eps_1-j_2\eps_2}V(\xx_0,\one)^{-\ss/\dd+\one/p}\cP_{m,k}(\phi),
\end{align*}
which confirms \eqref{suff} and completes the proof of Claim 1.

\smallskip

{\em Claim 2:} $\tBB^s_{pq}\hookrightarrow \cS'$.
This claim essentially will follow by duality
employing the embedding $\cS \hookrightarrow \tBB^s_{pq}$ just proved.
We need to show that there exist constants $m,k\in\bN$ and $c>0$ (depending on $\ss,p,q$)
such that
\begin{equation}\label{embed}
|\langle f,\phi\rangle|\le c\|f\|_{\tBB^s_{pq}}\cP_{m,k}(\phi),
\quad \forall f\in\tBB^s_{pq}, \; \phi\in\cS.
\end{equation}
Assume $f\in \tBB^s_{pq}$.
Let $\varphi^i_0,\varphi^i\in\cC^{\infty}(\bR)$, $i=1,2$,
be functions as in Definitions~\ref{def-B-spaces},~\ref{def-F-spaces},
i.e. satisfying \eqref{cond-1}, with the additional assumption that they satisfy the condition
\begin{equation*}
(\varphi_0^i)^2 (t)+\sum_{n=1}^{\infty}(\varphi^i)^2(2^{-n}t)=1,
\quad \forall t\in\RR, \; i=1,2.
\end{equation*}
For each $\jj=(j_1,j_2)\in \bN^2_0$ we set $\varphi_\jj:=\varphi^1_{j_1}\otimes\varphi^2_{j_2}$.
From Corollary~\ref{cor:Calderon} it follows that
\begin{equation}\label{emb55}
f=\sum_{\jj\in\bN_0^2}\varphi_\jj^2(\sqrt{L})f\quad\text{in }\cS',
\end{equation}
and hence for any $\phi\in \cS$
\begin{equation}
\label{emb6}
\langle f,\phi\rangle=\sum_{\jj\in\bN_0^2}\langle\varphi_\jj^2(\sqrt{L})f,\phi\rangle
=\sum_{\jj\in\bN_0^2}\langle\varphi_\jj(\sqrt{L})f,\varphi_\jj(\sqrt{L})\phi\rangle.
\end{equation}
Therefore, in order to prove \eqref{embed}
it suffices to show that there exist constants $m,k\in\bN$, $c>0$, end $\beps=(\eps_1,\eps_2)\in \RR_+^2$ such that
\begin{equation}\label{suff2}
|\langle\varphi_\jj(\sqrt{L})f,\varphi_\jj(\sqrt{L})\phi\rangle|
\le c\|f\|_{\tBB^\ss_{pq}}\cP_{m,k}(\phi)2^{-\beps\cdot\jj},
\quad \forall \phi\in \cS, \; \jj\in \bN_0^2.
\end{equation}

Without loss of generality we may assume that $j\in\bN^2$. We further distinguish two cases.

{\em Case 1:} $1\le p\le\infty$.
We will employ \eqref{suff} with $\ss$ replaced by $-\ss$
and $p$ replaced by $p'$ ($1/p+1/p'=1$).
In light of \eqref{def-km} we choose
\begin{equation*}
k>\max\{d_1(2-1/p)+|s_1|, d_2(2-1/p)+|s_2|\},
\quad
m>\max\{|s_1|+d_1, |s_2|+d_2\}/2.
\end{equation*}
Then using H\"{o}lder's inequality, Definition~\ref{def-B-spaces}, and \eqref{suff} we get
\begin{align*}
|\langle\varphi_\jj(\sqrt{L})f,\varphi_\jj(\sqrt{L})\phi\rangle|
&\le\big\|V(\cdot,2^{-\jj})^{-\ss/\dd}\varphi_\jj(\sqrt{L})f(\cdot)\big\|_p
\big\|V(\cdot,2^{-\jj})^{\ss/\dd}\varphi_\jj(\sqrt{L})\phi(\cdot)\big\|_{p'}
\\
&\le c\|f\|_{\tBB^\ss_{pq}}\cP_{m,k} (\phi)2^{-\beps\cdot\jj},
\end{align*}
where $\eps_i:= 2m-|s_i|-d_i >0$, $i=1,2$.
This confirms \eqref{suff2}.

\smallskip

{\em Case 2:} $0< p<1$.
Let $\gamma_i:=s_i/d_i-1/p+1$, $i=1,2$.
By Proposition~\ref{prop:spectral} it readily follows that
$\varphi_j(\sqrt{L})f\in\Sigma_{2^{\jj+\one}}$
and hence using Theorem~\ref{thm:Nik} and Definition~\ref{def-B-spaces} we get
\begin{equation}\label{nik-embed}
\big\|V(\cdot,2^{-\jj})^{-\bgamma}\varphi_\jj(\sqrt{L})f(\cdot)\big\|_1
\le c\big\|V(\cdot,2^{-\jj})^{-\ss/\dd}\varphi_\jj(\sqrt{L})f(\cdot)\big\|_p
\le c\|f\|_{\tBB^\ss_{pq}}.
\end{equation}
On the other hand, we have already shown that \eqref{suff}
is valid with $k,m$ from \eqref{def-km} and $\eps_i=2m-|s_i|-d_i$, $i=1,2$.
Hence, applying \eqref{suff} with
$\tilde{s}_i:= -s_i+ (1/p-1)d_i$ in the place if $s_i$, $i=1,2$, and $p=\infty$
we obtain
\begin{equation}\label{embed-2}
\big\|V(\cdot,2^{-\jj})^{\bgamma}\varphi_\jj(\sqrt{L})f(\cdot)\big\|_\infty
\le \cP_{m,k}(\phi)2^{-\tilde{\beps}\cdot\jj},
\end{equation}
where
\begin{equation*}
k>\max\{d_1+|\tilde{s}_1|, d_2+|\tilde{s}_2|\},
\;
m>\max\{|\tilde{s}_1|+d_1, |\tilde{s}_2|+d_2\}/2,
\;
\tilde{\eps}_i:= 2m-|\tilde{s}_i|-d_i >0.
\end{equation*}
The above selection of $\tilde{s}_i$ is justified by the identity
$-\tilde{s}/\dd = \ss/\dd + (1-1/p)\one$.
Finally, we use \eqref{nik-embed} and \eqref{embed-2} to obtain for any $\phi\in\cS$
\begin{align*}
|\langle\varphi_\jj(\sqrt{L})f,\varphi_\jj(\sqrt{L})\phi\rangle|
&\le\big\|V(\cdot,2^{-\jj})^{-\bgamma}\varphi_\jj(\sqrt{L})f(\cdot)\big\|_1
\big\|V(\cdot,2^{-\jj})^{\bgamma}\varphi_\jj(\sqrt{L})\phi(\cdot)\big\|_{\infty}
\\
&\le c\|f\|_{\tBB^\ss_{pq}}\cP_{m,k}(\phi)2^{-\tilde{\beps}\cdot\jj},
\end{align*}
which confirms \eqref{suff2}.
\end{proof}

\subsubsection{Embeddings between mixed-smoothness B-spaces with different parameters}

From Definition~\ref{def-B-spaces} it follows that
if $\ss,\ss'\in\bR^2$ and $s_i\ge s'_i$, $i=1,2$, $0<p\le\infty$ and $0<q\le r\le\infty$,
then $\|\cdot\|_{\BB^{\ss'}_{pr}}\le \|\cdot\|_{\BB^{\ss}_{pq}}$,
i.e. $\BB^{\ss}_{pq}\hookrightarrow \BB^{\ss'}_{pr}$.

As a consequence of Theorem~\ref{thm:Nik} we get:

\begin{proposition}\label{prop:B-embed}
Let $\ss,\ss'\in\bR^2$, $0<p\le r\le\infty$, $0<q\le\tau\le\infty$, and
\begin{equation}\label{saemb}
\frac{\ss}{\dd}-\frac{\one}{p}=\frac{\ss'}{\dd}-\frac{\one}{r}
\quad\Longleftrightarrow\quad
\frac{s_i}{d_i}-\frac{1}{p}=\frac{s'_i}{d_i}-\frac{1}{r}, \; i=1,2.
\end{equation}
Then  $\tBB^{s}_{pq}\hookrightarrow \tBB^{\ss'}_{r\tau}$.
If in addition the non-collapsing condition $\eqref{non-collapsing}$ holds for the spaces $\XX_1$ and $\XX_2$,
then $\BB^{\ss}_{pq}\hookrightarrow \BB^{\ss'}_{r\tau}$.
\end{proposition}

\begin{proof}
Let the family $\{\varphi_\jj\}_{\jj\in\bN_0^2}$ be as in Definition~\ref{def-B-spaces}.

$(i)$
Let $f\in \tBB^{s}_{pq}$ and $j\in\bN_0^2$.
Then $\varphi_j(\sqrt{L})f\in\Sigma_{2^{\jj+\one}}$
and by \eqref{Band-1} there exists a constant $c>0$ independent of $f$ and $\jj$ such that
\begin{align*}
\big\|V(\cdot, 2^{-\jj})^{-\ss'/\dd}\varphi_\jj(\sqrt{L}) f(\cdot)\big\|_{r}
&\le c\big\|V(\cdot, 2^{-\jj})^{-\frac{\ss'}{\dd}+(\frac{1}{r}-\frac{1}{p})\one}\varphi_\jj(\sqrt{L}) f(\cdot)\big\|_{p}
\\
&=c\big\|V(\cdot, 2^{-\jj})^{-\ss/\dd}\varphi_\jj(\sqrt{L}) f(\cdot)\big\|_{p},
\end{align*}
using (\ref{saemb}).
Now the embedding follows from Definition~\ref{def-B-spaces}.

$(ii)$ Let $f\in \BB^{s}_{pq}$ and $\jj\in\bN_0^2$.
Then $\varphi_j(\sqrt{L})f\in\Sigma_{2^{\jj+\one}}$.
By \eqref{Bl} there exists a constant $c>0$ independent of $f$ and $\jj$ such that
$$
2^{\jj\cdot\ss'}\big\|\varphi_\jj(\sqrt{L}) f\big\|_{r}
\le c2^{\jj\cdot\ss'}2^{\jj\cdot\dd(1/p-1/r)}\big\|\varphi_\jj(\sqrt{L}) f\big\|_{p}
=c2^{\jj\cdot\ss}\big\|\varphi_\jj(\sqrt{L}) f\big\|_{p}
$$
and the embedding again follows in light of Definition~\ref{def-B-spaces}.
\end{proof}

\begin{remark}
Proposition~\ref{prop:B-embed} shows that the nonclassical mixed-smoothness Besov spaces $\tBB^{s}_{pq}$
are embedded into each other correctly,
while the Besov spaces $\BB^{s}_{pq}$ are embedded into each other correctly
under the non-collapsing condition.
\end{remark}

\subsubsection{Embedding of mixed-smoothness Besov spaces into Lebesgue spaces}

\begin{proposition}\label{prop:embed-B-L}
Let $0<p,q\le\infty$.

$(i)$ If $\ss\in \bR^2_+$, then $\BB^{\ss}_{pq}\hookrightarrow L^p$.

$(ii)$ If the non-collapsing condition $(\ref{non-collapsing})$ is valid
for the coordinate spaces $\XX_1,\XX_2$ and $\ss\in\bR^2$ is such that $s_i>d_i/p$, $i=1,2$,
then $\BB^{\ss}_{pq}\hookrightarrow L^{\infty}$.
\end{proposition}
\begin{proof}
(i) Let $r:=\min\{1,p,q\}$ and $\tau:=\frac{q}{r}\ge1$.
Assuming $f\in \BB^\ss_{pq}$
we use the decomposition \eqref{cald-2}, the $r$-triangle inequality and H\"{o}lder's inequality to obtain
\begin{align*}
\|f\|_p^r&=\Big\|\sum_{\jj\in\bN_0^2}\varphi_\jj (\sqrt{L})f\Big\|_p^r
\le\sum_{\jj\in\bN_0^2}\big\|\varphi_\jj (\sqrt{L})f\big\|_p^r
\\
&\le \Big(\sum_{\jj\in\bN_0^2}\big(2^{\jj\cdot\ss}\big\|\varphi_\jj (\sqrt{L})f\big\|_p\big)^q\Big)^{r/q}
\Big(\sum_{\jj\in\bN_0^2} 2^{-r\tau'\jj\cdot\ss}\Big)^{1/\tau'}
\le c\|f\|_{\BB^\ss_{pq}}^r,
\end{align*}
taking into account that $s_i>0$, $i=1,2$.

(ii) Let $f\in \BB^{s}_{pq}$.
By \eqref{cald-2} and \eqref{Bl} we get
\begin{align*}
\|f\|_{\infty}&\le\sum_{\jj\in\bN_0^2}\big\|\varphi_\jj (\sqrt{L})f\big\|_{\infty}
\le c\sum_{\jj\in\bN_0^2}2^{\jj\cdot\dd/p}\big\|\varphi_\jj (\sqrt{L})f\big\|_{p}
\\
&\le c\sum_{\jj\in\bN_0^2}2^{-\jj(\ss-\dd/p)}\|f\|_{\BB^\ss_{pq}}
\le c\|f\|_{\BB^\ss_{pq}},
\end{align*}
using that $s_i>d_i/p$, $i=1,2$.
\end{proof}

\section{Spectral multipliers}\label{sec:multipliers}

An important class of operators in the classical (Euclidean) harmonic analysis
is the class of Fourier multipliers.
Let $\m$ be a bounded function on $\bR^d$. The Fourier multiplier $T_\m$ associated to $\m$ is defined by
$$
T_\m f(\xx):=\cF^{-1}\big(\m\cF f\big)(\xx),\quad f\in L^2,
$$
where $\cF$ and $\cF^{-1}$ stand for the Fourier and inverse Fourier transforms.
Plancerel's inequality guarantees the boundedness of $T_\m$ on $L^2$.
The fundamental problem here is to find conditions on the symbol function $\m$
under which $T_\m$ is bounded in certain function classes such as distributions and smoothness spaces.

In the setting of this article a natural substitute for the Fourier theory is the Spectral theory.
Let $\m:[0, \infty)^2 \to\bR$ be a measurable function.
The spectral multiplier $\m(\sqrt{L})= \m(\sqrt{L_1}, \sqrt{J_2})$ is defined by
\begin{equation}\label{eq:specmult}
\m(\sqrt L)=\int_{[0,\infty)^2}\m(\sqrt {\lambda_1},\sqrt{\lambda_2})dE_{(\lambda_1,\lambda_2)}.
\end{equation}
where $dE_{(\lambda_1,\lambda_2)}$ is the spectral measure defined in \S\ref{subsec:spectr-func}.
Clearly, if $\m$ is bounded on $[0,\infty)^2$, then $\m(\sqrt{L})$ is bounded on $L^2(\XX)$.

Our goal is to study the boundedness of spectral multipliers on the classes of
test functions $\cS=\cS(L_1, L_2)$ and distributions $\cS'=\cS'(L_1, L_2)$, see Section~\ref{sec:distributions},
as well as the mixed-smoothness Besov and Triebel-Lizoskin spaces, introduced in Section~\ref{sec:B-F-spaces}.
Single-variable spectral multiplier results are obtained in \cite{BD,GN,GN2,GKKP2}.

\subsection{Bounded spectral multipliers on distributions}\label{subsec:multiplier-S}

We begin by introducing spectral multipliers on test functions
and distributions associated to the operators $L_1, L_2$.

\begin{definition}\label{def:mS}
A function $\m \in \cC^{\infty}(\R^2)$ is called {\em admissible}
if it is real-valued,
$\m(\pm\lambda_1,\pm\lambda_2)=\m(\lambda_1,\lambda_2)$ for $(\lambda_1,\lambda_2)\in \bR^2$,
and all its partial derivatives have at most polynomial growth,
i.e. for any $\bbeta\in\bN_0^2$ there exist constants $c_{\bbeta}>0$ and $N_{\bbeta}\in\bN_0$ such that
\begin{equation}\label{pg}
\big|\partial^{\bbeta}\m(\blambda)\big|\le c_{\bbeta}(1+|\blambda|)^{N_{\bbeta}}.
\end{equation}

Assume that $\m \in \cC^{\infty}(\R^2)$ is admissible and let
$\varphi^i_0,\varphi^i \in \cC^{\infty}(\RR)$, $i=1,2$, be two pairs of functions
satisfying conditions \eqref{phi-phi0}--\eqref{1dCalderon}.
For reader's convenience we recall them here:
\begin{equation}\label{cond-11}
\begin{aligned}
&(i)\ \hbox{$\varphi^i_0,\varphi^i$ are real-valued and even,}
\\
&(ii)\ \supp\varphi^i_0\subset[-2,2], \ \supp\varphi^i\subset[-2,2]\setminus[-1/2, 1/2],
\end{aligned}
\end{equation}
and
\begin{equation}\label{Cald-1d}
\varphi^i_0(t)+\sum_{n\ge 1}\varphi^i(2^{-n} t)=1, \quad \forall t\in\RR,
\; i=1,2.
\end{equation}
Using the standard notation $\varphi^i_n(t):=\varphi^i(2^{-n}t)$, 
we define $\varphi_\jj:=\varphi^1_{j_1}\otimes\varphi^2_{j_2}$ for $\jj=(j_1,j_2)\in \bN^2_0$.
Then from \eqref{Cald-1d} it readily follows that
\begin{equation*}
\sum_{\jj\in\bN_0^2} \varphi_\jj(\blambda) = 1, \quad \blambda\in \RR^2.
\end{equation*}
Further, set $\m_\jj(\blambda):=\m(\blambda)\varphi_\jj(\blambda)$.
Then obviously
\begin{equation}\label{m-lam}
\m(\blambda)=\sum_{\jj\in \bN_0^2}\m(\blambda)\varphi_\jj(\blambda)=\sum_{\jj\in \bN_0^2}\m_\jj(\blambda).
\end{equation}

We define the multiplier operator
$\m(\sqrt{L})=\m(\sqrt{L_1},\sqrt{L_2})$ on $\cS=\cS(L_1,L_2)$
by
\begin{equation}\label{def:mL}
\m(\sqrt{L})\phi
:=\sum_{\jj\in \bN_0^2}\m_\jj(\sqrt{L})\phi,
\quad \phi\in \cS,
\end{equation}
and on $\cS'=\cS'(L_1,L_2)$ by
\begin{equation}\label{def:mL-1}
\m(\sqrt{L})f
:=\sum_{\jj\in \bN_0^2}\m_\jj(\sqrt{L})f,
\quad f\in \cS',
\end{equation}
where the convergence is in $\cS$ and $\cS'$, respectively.

\end{definition}

\begin{theorem}\label{thm:mS}
Let $\m \in \cC^{\infty}(\R^2)$ be admissible $($Defintion~\ref{def:mS}$)$.
Then the multiplier operator $\m(\sqrt{L})$ is well defined and continuous on $\cS$ and $\cS'$.
\end{theorem}
\begin{proof}
Let $\m \in \cC^{\infty}(\R^2)$ be admissible in the sense of Definition~\ref{def:mS}.
We first show that the operator $\m(\sqrt{L})$ is well defined by \eqref{def:mL} and \eqref{def:mL-1}.
To this end we will show that
(a) the series in \eqref{def:mL} and \eqref{def:mL-1} converge
and
(b) the definition of $\m(L)$ 
is independent of the specific selection of the functions $\varphi^i_0,\varphi^i$, $i=1,2$.

(a)
To prove that the series in \eqref{def:mL} converges it suffices to
show that the sequence
$
\Big\{\sum_{\zero\le\jj\le \nn}\m_\jj(\sqrt{L})\phi\Big\}
$
is a Cauchy sequence in $\cS$ for any $\phi\in \cS$, i.e. for $m, k\in \bN_0$
\begin{equation}\label{cauchy}
\cP_{m,k}\Big(\sum_{\zero\le\jj\le \nn}\m_\jj(\sqrt{L})\phi -\sum_{\zero\le\jj\le \bell}\m_\jj(\sqrt{L})\phi \Big) \to 0
\quad\hbox{as}\quad |\nn|, |\bell|\to \infty,\;\; \forall \phi\in\cS.
\end{equation}
We denote $\Phi^i(t):= \varphi_0^i(2t)+ \varphi^i(t)$, $i=1,2$.
From \eqref{cond-11}--\eqref{Cald-1d} it follows that
$\supp \Phi^i \subset [-2,2]$, $\Phi^i(t)=1$ for $t\in[-1,1]$,
and most importantly
\begin{equation}\label{sum-1}
\sum_{j=0}^{n}\varphi^i_{j}(t)=\Phi^i(2^{-n}t),\quad t\in\RR, \; n\in\bN.
\end{equation}
Hence, for $\nn\in \bN^2$
\begin{equation}\label{sum-2}
\sum_{j_1=0}^{n_1}\sum_{j_2=0}^{n_2} \m(\lambda_1,\lambda_2)\varphi_{j_1}^1(\lambda_1)\varphi_{j_2}^2(\lambda_2)
= \m(\lambda_1,\lambda_2)\Phi^1(2^{-n_1}\lambda_1)\Phi^2(2^{-n_2}\lambda_2),
\end{equation}
implying
\begin{equation*}
S_\nn:=\sum_{\zero\le\jj\le \nn}\m_\jj(\sqrt{L}) =\m(\sqrt{L})\Phi(2^{-\nn}\sqrt{L}),
\end{equation*}
where
$\Phi:=\Phi^1\otimes\Phi^2$.
If $\bell,\nn\in \bN^2$, then from above
\begin{equation*}
S_\nn - S_\bell
=\sum_{\zero\le\jj\le \nn}\m_\jj(\sqrt{L}) - \sum_{\zero\le\jj\le \bell}\m_\jj(\sqrt{L})
= \m(L)\big(\Phi(2^{-\nn}\sqrt{L})-\Phi(2^{-\bell}\sqrt{L})\big).
\end{equation*}
Consider the functions
\begin{equation*}
\omega_{\bell,\nn}(\lambda_1, \lambda_2)
:= \Phi^1(2^{-n_1}\lambda_1)\Phi^2(2^{-n_2}\lambda_2) - \Phi^1(2^{-\ell_1}\lambda_1)\Phi^2(2^{-\ell_2}\lambda_2)
\end{equation*}
and
\begin{equation}\label{def-theta-ln}
\theta_{\bell,\nn}(\lambda_1, \lambda_2)
:= \min\{2^{\ell_1\wedge n_1}, 2^{\ell_2\wedge n_2}\}^{r}
\m(\lambda_1,\lambda_2)\omega_{\bell,\nn}(\lambda_1, \lambda_2)(\lambda_1^2+\lambda_2^2)^{-r}.
\end{equation}
Clearly,
\begin{equation*}
S_\nn - S_\bell = \min\{2^{\ell_1\wedge n_1}, 2^{\ell_2\wedge n_2}\}^{-r}\theta_{\bell,\nn}(\sqrt{L_1}, \sqrt{L_2})(L_1+L_2)^r
\end{equation*}
and hence
\begin{equation}\label{S-S}
L^\bnu(S_\nn - S_\bell)
= \min\{2^{\ell_1\wedge n_1}, 2^{\ell_2\wedge n_2}\}^{-r} \theta_{\bell,\nn}(\sqrt{L_1}, \sqrt{L_2})
\sum_{\kappa=0}^{r}\binom{r}{\kappa}L^{(\nu_1+\kappa, \nu_2+r-\kappa)}.
\end{equation}

Fix $m,k\in\bN_0$. 
Choose $\tilde{k}_i > k+5d_i/2$, $i=1,2$,
and introduce the number $N:= \max\{N_\bbeta: |\bbeta|\le \tilde{k}_1+\tilde{k}_2\}$.
Also, we choose $r > N+ d_1+d_2+\tilde{k}_1+\tilde{k}_2$.
We claim that  $\theta_{\bell,\nn}(\sqrt{L})$ is an integral operator whose kernel
satisfies
\begin{equation}\label{K-theta}
|\KK_{\theta_{\bell,\nn}(\sqrt{L})}(\xx,\yy)|
\le c\DD_{\one,\tilde{\kk}}(\xx,\yy).
\end{equation}
To prove this, in light of Theorem~\ref{thm:gen-local},
it  suffices to show that for all $|\bbeta|\le \tilde{k}_1+\tilde{k}_2$
\begin{equation}\label{partial-theta}
|\partial^{\bbeta}\theta_{\bell,\nn}(\blambda)|
\le c(1+|\blambda|)^{-q},
\quad q > d_1+d_2+\tilde{k}_1+\tilde{k}_2.
\end{equation}
Clearly,
\begin{equation*}
|\partial^{\bbeta}\omega_{\bell,\nn}(\blambda)|\le c(\bbeta)
\quad\hbox{and}\quad
|\partial^{\bbeta}(|\blambda|^{-2r})|\le c(\bbeta,r)|\blambda|^{-2r}, \quad |\blambda|\ge 1.
\end{equation*}
In addition to this
$\omega_{\bell,\nn}(\blambda)=0$ for $\blambda \in [0,2^{\ell_1\wedge n_1}]\times[0,2^{\ell_2\wedge n_2}]$.
Therefore, using the Libniz rule:
$
\partial^\bbeta(fg)=\sum_{\bgamma\le \bbeta} \binom{\bbeta}{\bgamma}\partial^\bgamma f \partial^{\bbeta-\bgamma}g
$
we get
\begin{align*}
|\partial^{\bbeta}(\omega_{\bell,\nn}(\blambda)|\blambda|^{-r})|
\le c(1+|\blambda|)^{-2r}
\le c\min\{2^{\ell_1\wedge n_1}, 2^{\ell_2\wedge n_2}\}^{-r}(1+|\blambda|)^{-r},
\end{align*}
for $\blambda \in [0, \infty)^2\setminus [0,2^{\ell_1\wedge n_1}]\times[0,2^{\ell_2\wedge n_2}]$
and
$\partial^{\bbeta}(\omega_{\bell,\nn}(\blambda)|\blambda|^{-r})=0$ otherwise.
Consequently,
\begin{equation}\label{partial-omega}
|\partial^{\bbeta}(\omega_{\bell,\nn}(\blambda)|\blambda|^{-r})|
\le c\min\{2^{\ell_1\wedge n_1}, 2^{\ell_2\wedge n_2}\}^{-r}(1+|\blambda|)^{-r},
\quad \forall \blambda\in [0,\infty)^2.
\end{equation}
On the other hand, from the assumption that the function $\m$ is admissible it follows that
\begin{equation*} 
|\partial^{\bbeta}\m(\blambda)|\le c_{\bbeta}(1+|\blambda|)^{N_{\bbeta}}\le c (1+|\blambda|)^N
\quad\hbox{if}\;\; |\bbeta|\le \tilde{k}_1+\tilde{k}_2.
\end{equation*}
Using this and \eqref{partial-omega} along with the Libniz rule, and \eqref{def-theta-ln} we obtain
\begin{equation*}
|\partial^{\bbeta}\theta_{\bell,\nn}(\blambda)|
\le c(1+|\blambda|)^{-r+N}
\quad\hbox{if}\quad |\bbeta|\le \tilde{k}_1+\tilde{k}_2,
\end{equation*}
which confirms \eqref{partial-theta} because $r > N+ d_1+d_2+\tilde{k}_1+\tilde{k}_2$.

We introduce the abbreviated notation
$\eps(\bell, \nn,r):= \min\{2^{\ell_1\wedge n_1}, 2^{\ell_2\wedge n_2}\}^{-r}$.
Let $\phi\in\cS$ and fix $K\ge (\tilde{k}_1-d_1/2)\vee(\tilde{k}_2-d_2/2)$.
Then using \eqref{S-S} and \eqref{K-theta} we obtain
\begin{align}\label{L-SS}
|L^\bnu(S_\nn &- S_\bell)\phi(\xx)| \nonumber
\\
&\le c\eps(\bell, \nn,r)
\sum_{\kappa=0}^{r}\int_\XX|\KK_{\theta_{\bell,\nn}(\sqrt{L})}(\xx,\yy)|
|L^{(\nu_1+\kappa, \nu_2+r-\kappa)}\phi(\yy)| d\mu(\yy) \nonumber
\\
&\le c\eps(\bell, \nn,r)
\cP_{m+r,K}(\phi)\int_\XX \DD_{\one, \tilde{\kk}}(\xx,\yy)\DD^*_{\one, \tilde{\kk}-\dd/2}(\yy,\xx_0) d\mu(\yy)
\\
&\le c\eps(\bell, \nn,r)
\cP_{m+r,K}(\phi)\int_\XX V(\yy, \one)^{-1} \DD^*_{\one, \tilde{\kk}-\dd/2}(\xx,\yy)
\DD^*_{\one, \tilde{\kk}-\dd/2}(\yy,\xx_0) d\mu(\yy) \nonumber
\\
&\le c\eps(\bell, \nn,r)\cP_{m+r,K}(\phi) \DD^*_{\one, \tilde{\kk}-\dd/2}(\xx,\xx_0). \nonumber
\end{align}
Here for the third inequality we used \eqref{D-D*} and
for the last we use \eqref{tech-4} and the fact that $\tilde{k}_i -d_i/2>2d_i$.

Now, using \eqref{L-SS} we get
\begin{align*}
\cP_{m,k}(S_\nn\phi-S_\bell\phi)
\le c\eps(\bell, \nn,r)\cP_{m+r,K}(\phi)
\sup_{\xx\in \XX}\prod_{i=1,2}\frac{(1+\rho_i(x_i,x_{0i}))^k}{(1+\rho_i(x_i,x_{0i}))^{\tilde{k}_i-d_i/2}}.
\end{align*}
Finally, taking into account that $\tilde{k}_i > k+5d_i/2$ we conclude that
\begin{equation*}
\cP_{m,k}\Big(\sum_{\zero\le\jj\le \nn}\m_\jj(\sqrt{L})\phi -\sum_{\zero\le\jj\le \bell}\m_\jj(\sqrt{L})\phi \Big)
\le c\min\{2^{\ell_1\wedge n_1}, 2^{\ell_2\wedge n_2}\}^{-r}\cP_{m+r,K}(\phi),
\end{equation*}
which implies \eqref{cauchy}.

\smallskip

The convergence of the series in \eqref{def:mL-1} follows by duality from the convergence in \eqref{def:mL}.

\smallskip

(b) To show that the definition of the multiplier $\m(\sqrt{L})$
is independent of the specific selection of the functions $\varphi_0^i, \varphi^i$, $i=1,2$,
we assume that $\psi_0^i, \psi^i\in \cC^\infty(\RR)$, $i=1,2$, are other two pairs of functions
satisfying conditions \eqref{cond-11} and \eqref{Cald-1d}.
Denote $\psi_\jj:=\psi^1_{j_1}\otimes\psi^2_{j_2}$ for $\jj=(j_1,j_2)\in \bN^2_0$
and set $\widetilde{\m}_\jj(\blambda):=\m(\blambda)\psi_\jj(\blambda)$.
Then
\begin{equation*} 
\m(\blambda)=\sum_{\jj\in \bN_0^2}\m(\blambda)\psi_\jj(\blambda)=\sum_{\jj\in \bN_0^2}\widetilde{\m}_j(\blambda).
\end{equation*}
We define the multiplier operator
$\widetilde{\m}(\sqrt{L})$ on $\cS$
by
\begin{equation*} 
\widetilde{\m}(\sqrt{L})\phi
:=\sum_{\jj\in \bN_0^2}\widetilde{\m}_\jj(\sqrt{L})\phi,
\quad \phi\in \cS,
\end{equation*}
and will show that $\widetilde{\m}(\sqrt{L})\phi=\m(\sqrt{L})\phi$ for $\phi\in \cS$.

Indeed set $\Psi^i(t):= \psi_0^i(2t)+ \psi^i(t)$, $i=1,2$.
From \eqref{cond-11}--\eqref{Cald-1d} it follows that
$\supp \Psi^i \subset [-2,2]$, $\Psi^i(t)=1$ for $t\in[-1,1]$,
and
\begin{equation}\label{sum-11}
\sum_{j=0}^{n}\psi^i_{j}(t)=\Psi^i(2^{-n}t),\quad t\in\RR, \; n\in\bN.
\end{equation}
Denote
\begin{equation*}
\widetilde{S}_\nn:=\sum_{\zero\le\jj\le \nn}\widetilde{\m}_\jj(\sqrt{L}) =\m(\sqrt{L})\Psi(2^{-\nn}\sqrt{L}),
\end{equation*}
where $\Psi:=\Psi^1\otimes \Psi^2$,
and observe that
for $\bell,\nn\in \bN^2$ we have
\begin{equation*}
S_\nn - \widetilde{S}_\bell
=\sum_{\zero\le\jj\le \nn}\m_\jj(\sqrt{L}) - \sum_{\zero\le\jj\le \bell}\widetilde{\m}_\jj(\sqrt{L})
= \m(L)\big(\Phi(2^{-\nn}\sqrt{L})-\Psi(2^{-\bell}\sqrt{L})\big).
\end{equation*}
We introduce the function
\begin{equation*}
\widetilde{\omega}_{\bell,\nn}(\lambda_1, \lambda_2)
:= \Phi^1(2^{-n_1}\lambda_1)\Phi^2(2^{-n_2}\lambda_2) - \Psi^1(2^{-\ell_1}\lambda_1)\Psi^2(2^{-\ell_2}\lambda_2).
\end{equation*}
From this point on we repeat the proof from part (a) above
with the role of $\omega_{\bell,\nn}$ played by $\widetilde{\omega}_{\bell,\nn}$
to conclude that for any $m,k\in\bN_0$
\begin{equation*}
\cP_{m,k}(S_\nn\phi - \widetilde{S}_\bell\phi) \to 0
\quad\hbox{as}\quad |\nn|, |\bell|\to \infty,
\quad \forall \phi\in\cS,
\end{equation*}
which yields
$\widetilde{\m}(\sqrt{L})\phi=\m(\sqrt{L})\phi$ for $\phi\in \cS$.
This readily implies that
$\widetilde{\m}(\sqrt{L})f=\m(\sqrt{L})f$ for $f\in \cS'$.


\medskip

\noindent
{\bf The continuity of the operator $\m(\sqrt{L})$.}
We now focus on the proof of the continuity of the operator $\m(\sqrt{L})$ on $\cS$ and $\cS'$.
To a large extent we will proceed similarly as in the above proof of the first part of the theorem.

By duality it suffices to prove the continuity of $\m(\sqrt{L})$ on $\cS$ only.
Let $m,k\in\bN_0$.
We need to show that there exist constants $m^*, k^*\in\bN$ and $c>0$ such that
\begin{equation}\label{P-P}
\cP_{m,k}(\m(\sqrt{L})\phi)\le c\cP_{m^*,k^*}(\phi), \quad \forall\phi\in\cS.
\end{equation}
By \eqref{norm-S} it suffices to prove that there exist constants $m^*, k^*\in\bN$ and $c>0$ such that
for any $\phi\in\cS$, $\bnu=(\nu_1,\nu_2)\in\bN_0^2$ with $|\nu_i|\le m$, and  $\xx\in \XX$
\begin{equation}\label{S-bound-1}
\prod_{i=1,2}\big(1+\rho_i(x_i,x_{0i})\big)^{k}|L^{\bnu}\m(\sqrt{L})\phi(\xx)|
\le c\cP_{m^*,k^*}(\phi).
\end{equation}
Choose $\tilde{k}_i > k+5d_i/2$, $i=1,2$,
and set $N:= \max\{N_\bbeta: |\bbeta|\le \tilde{k}_1+\tilde{k}_2\}$, see \eqref{pg}.
Also, choose $r > N+ d_1+d_2+\tilde{k}_1+\tilde{k}_2$.

From the proof of the first part of the theorem
$
\m(\sqrt{L})\phi=\sum_{\jj\in \bN_0^2}\m_\jj(\sqrt{L})\phi
$
in $\cS$,
where $\m_\jj(\blambda) := \m(\blambda)\varphi_\jj(\blambda)$
and hence
\begin{equation}
\label{L-m}
|L^{\bnu}\m(\sqrt{L})\phi(\xx)|\le\sum_{\jj\in\bN_0^2}|L^{\bnu}\m_\jj(\sqrt{L})\phi(\xx)|.
\end{equation}

We next estimate $|L^{\bnu}\m_\jj(\sqrt{L})\phi(\xx)|$ in the case when $\jj=(j_1, j_2)\in \bN^2$.
The estimation of $|L^{\bnu}\m_\jj(\sqrt{L})\phi(\xx)|$ when $j_1=0$ or $j_2=0$ is similar; we omit it.
Recall that
\begin{equation*}
\m_\jj(\blambda)
= \m(\blambda)\varphi_\jj(\blambda)
= \m(\lambda_1, \lambda_2)\varphi^1(2^{-j_1}\lambda_1)\varphi^2(2^{-j_2}\lambda_2),
\quad \jj\in \bN^2.
\end{equation*}
Consider the function
\begin{align}\label{def-thet-j}
\theta_\jj(\blambda)
&= (2^{j_1}+2^{j_2})^r\m_\jj(\blambda)|\blambda|^{-r}
\\
&= (2^{j_1}+2^{j_2})^r\m(\lambda_1, \lambda_2)
\varphi^1(2^{-j_1}\lambda_1)\varphi^2(2^{-j_2}\lambda_2)(\lambda_1^2+\lambda_2^2)^{-r}. \nonumber
\end{align}
Clearly,
\begin{equation*}
\m_\jj(\sqrt{L}) = (2^{j_1}+2^{j_2})^{-r}\theta_\jj(\sqrt{L})(L_1+L_2)^r
\end{equation*}
and hence
\begin{equation}\label{L-m-1}
L^{\bnu}\m_\jj(\sqrt{L})= (2^{j_1}+2^{j_2})^{-r}\theta_\jj(\sqrt{L})
\sum_{\kappa=0}^r\binom{r}{\kappa}L^{(\nu_1+\kappa, \nu_2+r-\kappa)}.
\end{equation}

We next show that  $\theta_\jj(\sqrt{L})$ is an integral operator
with kernel satisfying
\begin{equation}\label{K-theta-2}
|\KK_{\theta_\jj(\sqrt{L})}(\xx,\yy)|
\le c\DD_{\one,\tilde{\kk}}(\xx,\yy).
\end{equation}
In light of Theorem~\ref{thm:gen-local}
it  suffices to show that for all $|\bbeta|\le \tilde{k}_1+\tilde{k}_2$
\begin{equation}\label{partial-theta-2}
|\partial^{\bbeta}\theta_\jj(\blambda)|
\le c(1+|\blambda|)^{-q},
\quad q > d_1+d_2+\tilde{k}_1+\tilde{k}_2.
\end{equation}
It is easy to see that
\begin{equation*}
|\partial^{\bbeta}\varphi_\jj(\blambda)|\le c(\bbeta)
\quad\hbox{and}\quad
|\partial^{\bbeta}(|\blambda|^{-2r})|\le c(\bbeta,r)|\blambda|^{-2r}, \;\; |\blambda|\ge 1.
\end{equation*}
Observe that
$\supp \varphi_\jj \cap [0, \infty)^2 \subset [2^{j_1-1}, 2^{j_1+1}]\times[2^{j_2-1}, 2^{j_2+1}]$.
Hence, using the Libniz rule
\begin{align*}
|\partial^{\bbeta}(\varphi_\jj(\blambda)|\blambda|^{-r})|
\le c(1+|\blambda|)^{-2r}
\le c(2^{j_1}+2^{j_2})^{-r}(1+|\blambda|)^{-r},
\end{align*}
for $\blambda \in [2^{j_1-1}, 2^{j_1+1}]\times[2^{j_2-1}, 2^{j_2+1}]$
and
$\partial^{\bbeta}(\varphi_\jj(\blambda)|\blambda|^{-r})=0$ otherwise.
Therefore,
\begin{equation}\label{partial-omega-2}
|\partial^{\bbeta}(\varphi_\jj(\blambda)|\blambda|^{-r})|
\le c(2^{j_1}+2^{j_2})^{-r}(1+|\blambda|)^{-r},
\quad \forall \blambda\in [0,\infty)^2.
\end{equation}
On the other hand, from \eqref{pg} we have
\begin{equation*} 
|\partial^{\bbeta}\m(\blambda)|\le c_{\bbeta}(1+|\blambda|)^{N_{\bbeta}}\le c (1+|\blambda|)^N
\quad\hbox{if}\;\; |\bbeta|\le \tilde{k}_1+\tilde{k}_2.
\end{equation*}
Using this and \eqref{partial-omega-2} together with the Libniz rule, and \eqref{def-thet-j} we get
\begin{equation*}
|\partial^{\bbeta}\theta_\jj(\blambda)|
\le c(1+|\blambda|)^{-r+N}
\quad\hbox{if}\quad |\bbeta|\le \tilde{k}_1+\tilde{k}_2,
\end{equation*}
which confirms \eqref{partial-theta-2} because $r > N+ d_1+d_2+\tilde{k}_1+\tilde{k}_2$.

Let $\phi\in\cS$ and choose $K\ge (\tilde{k}_1-d_1/2)\vee(\tilde{k}_2-d_2/2)$.
We use \eqref{L-m-1} and \eqref{K-theta-2} to obtain
\begin{align*}
&|L^\bnu\m_\jj(\sqrt{L})\phi(\xx)|
\\
&\qquad\le c(2^{j_1}+2^{j_2})^{-r}
\sum_{\kappa=0}^{r}\int_\XX|\KK_{\theta_\jj(\sqrt{L})}(\xx,\yy)|
|L^{(\nu_1+\kappa, \nu_2+r-\kappa)}\phi(\yy)| d\mu(\yy)
\\
&\qquad\le c(2^{j_1}+2^{j_2})^{-r}
\cP_{m+r,K}(\phi)\int_\XX \DD_{\one, \tilde{\kk}}(\xx,\yy)\DD^*_{\one, \tilde{\kk}-\dd/2}(\yy,\xx_0) d\mu(\yy)
\\
&\qquad\le c(2^{j_1}+2^{j_2})^{-r}
\cP_{m+r,K}(\phi)\int_\XX V(\yy, \one)^{-1} \DD^*_{\one, \tilde{\kk}-\dd/2}(\xx,\yy)\DD^*_{\one, \tilde{\kk}-\dd/2}(\yy,\xx_0) d\mu(\yy)
\\
&\qquad\le c(2^{j_1}+2^{j_2})^{-r}\cP_{m+r,K}(\phi) \DD^*_{\one, \tilde{\kk}-\dd/2}(\xx,\xx_0),
\end{align*}
where for the third inequality we used \eqref{D-D*} and
for the last we use \eqref{tech-4} and the fact that $\tilde{k}_i -d_i/2>2d_i$.
Using the above in \eqref{L-m} and summing up we obtain
\begin{align*}
|L^\bnu\m(\sqrt{L})\phi(\xx)|
&\le c\cP_{m+r,K}(\phi) \DD^*_{\one, \tilde{\kk}-\dd/2}(\xx,\xx_0)\sum_{\jj\in\bN_0^2}(2^{j_1}+2^{j_2})^{-r}
\\
&\le c\cP_{m+r,K}(\phi) \DD^*_{\one, \tilde{\kk}-\dd/2}(\xx,\xx_0),
\end{align*}
implying
\begin{align*}
\prod_{i=1,2}\big(1+\rho_i(x_i,x_{0i})\big)^{k}\big|L^\bnu \m(\sqrt{L})\phi(\xx)\big|
&\le \cP_{m+r,K}(\phi)\prod_{i=1,2}
\frac{\big(1+\rho_i(x_i,x_{0i})\big)^{k}}{\big(1+\rho_i(x_i,x_{0i})\big)^{\tilde{k}-d_i/2}}
\\
&\le c \cP_{m+r,K}(\phi),
\end{align*}
which yields \eqref{S-bound-1}.
Above we used that $\tilde{k}-d_i/2\ge k$.
The proof is complete.
\end{proof}

\subsection{Spectral multipliers acting on mixed-smoothness B and F-spaces}\label{subsec:multipl-BF}

Here we consider classes of spectral multipliers
that can be extended to bounded operators between mixed-smoothness Besov or Triebel-Lizorkin spaces.

\begin{definition}\label{def:muldipl-BF}
Assume $\btau=(\tau_1, \tau_2)\in\bR^2$, $\bkappa=(\kappa_1,\kappa_2)\in\bN^2$
and let the real-valued function $\m\in \cC^{\kappa_1+\kappa_2}(\RR^2)$ obey
$\m(\pm \lambda_1, \pm\lambda_2)=\m(\lambda_1,\lambda_2)$ for $(\lambda_1,\lambda_2)\in \bR^2$.
We say that $\m$ belongs to the class $\bM(\btau,\bkappa)$ if
for any $\bbeta=(\beta_1,\beta_2)$ with $\bbeta\le\bkappa$
there exists a constant $c_{\bbeta}>0$ such that
\begin{equation}\label{multibounds}
|\partial^\bbeta\m(\blambda)|
\le c_{\bbeta}\big(1+\lambda_1\big)^{\tau_1-\beta_1}(1+\lambda_2\big)^{\tau_2-\beta_2},
\quad \forall \; \blambda\ge \zero. 
\end{equation}
\end{definition}

\begin{remark} \label{ex:ma}
$(1)$
The multiplier
$\m_{\btau}(\blambda):=\big(1+\lambda_1^2\big)^{\tau_1/2}\big(1+\lambda_2^2\big)^{\tau_2/2}$
belongs to the class $\bM(\btau,\bkappa)$ for each $\bkappa\in\bN^2$.

$(2)$
The assumption $(\ref{multibounds})$ can be replaced by a H\"{o}rmander-type one,
giving a broader class of multipliers. We do not elaborate on this issue here.


$(3)$
We will use the notation $\bM(\bkappa):=\bM((0,0),\bkappa)$ whenever $\btau=(0,0)$.
\end{remark}

We now present our principle multiplier result.

\begin{theorem}
\label{th:multiBF}
Let $\ss, \btau\in\bR^2$, $0<q\le\infty$, $\bkappa\in\bN^2$ and $\m\in\mathbb{M}(\btau,\bkappa)$.

$(i)$ If $0<p\le\infty$ and $\bkappa>\frac{2\dd}{p}+\frac{3\dd}{2}$,
then the spectral multiplier $\m(\sqrt{L})$ is bounded from $\BB^{\ss+\btau}_{pq}$ into $\BB^\ss_{pq}$.

$(ii)$ If $0<p<\infty$ and $\bkappa>\frac{2\dd}{\min\{p,q\}}+\frac{3\dd}{2}$,
then the spectral multiplier $\m(\sqrt{L})$ is bounded from $\FF^{\ss+\btau}_{pq}$ into $\FF^\ss_{pq}$.
\end{theorem}

\begin{proof}
We will prove only part (ii) of the theorem.
The proof of part (i) is similar; we omit it.

Let $\ss, \btau\in\RR^2$, $0<q\leq \infty$, $0<p<\infty$ and
$\m\in \bM(\btau,\bkappa)$, where $\bkappa>\frac{2\dd}{\min\{p,q\}}+\frac{3\dd}{2}$.
Assume $f\in \FF^{\ss+\btau}_{pq}$ and
let $\varphi^i_0,\varphi^i, i=1,2$, be two pairs of real-valued functions
satisfying \eqref{cond-11}--\eqref{Cald-1d}.
From Corollary~\ref{cor:Calderon} we have the decomposition
\begin{equation*}
f= \sum\limits_{\jj\in\bN_0^2} \varphi_\jj(\sqrt{L})f
\quad(\text{convergence in}\;\mathcal{\cS'}),
\end{equation*}
and using Theorem~\ref{thm:mS} we get
\begin{equation}\label{eq:Fbounded1}
\m(\sqrt{L})f=\sum\limits_{\jj\in\bN_0^2} \m(\sqrt{L})\varphi_\jj(\sqrt{L})f
\quad(\text{convergence in}\;\mathcal{\cS'}).
\end{equation}
We will use (\ref{eq:Fbounded1}) to estimate $\|\m(\sqrt{L})f\|_{\FF^{\ss}_{pq}}$.
Let $\kk\in\bN_0^2$.
Using what we know about the supports of the functions $\varphi_\jj$ we obtain
\begin{equation}\label{eq:Fbounded2}
\begin{aligned}
\varphi_\kk(\sqrt{L})\m(\sqrt{L})f
&=\sum\limits_{j_2=k_2-1}^{k_2+1}\sum\limits_{j_1=k_1-1}^{k_1+1}\varphi_\kk(\sqrt{L})\m(\sqrt{L})\varphi_\jj(\sqrt{L})f
\\
&=\sum\limits_{j_2=k_2-1}^{k_2+1}\sum\limits_{j_1=k_1-1}^{k_1+1} \m_\kk(2^{-\kk}\sqrt{L})\varphi_\jj(\sqrt{L})f,
\end{aligned}
\end{equation}
where $\m_\kk(\blambda):=\m(2^\kk\blambda)\varphi_\kk(2^\kk\blambda)$ and $\varphi_{-1}\equiv0$.
Hence, for $\xx\in \XX$ we have
\begin{equation}\label{eq:Fbounded3}
\begin{aligned}
|\varphi_\kk(\sqrt{L})&\m(\sqrt{L})f(\xx)|
\le \sum\limits_{j_2=k_2-1}^{k_2+1}\sum\limits_{j_1=k_1-1}^{k_1+1}
|\m_\kk\big(2^{-\kk}\sqrt{L}\big)\varphi_\jj(\sqrt{L})f(\xx)|
\\
&\le \sum\limits_{j_2=k_2-1}^{k_2+1}\sum\limits_{j_1=k_1-1}^{k_1+1}
\int_\XX |\KK_{\m_\kk\big(2^{-\kk}\sqrt{L}\big)}(\xx,\yy)||\varphi_\jj(\sqrt{L})f(\yy)|d\mu(\yy)
\\
&=:\sum\limits_{j_2=k_2-1}^{k_2+1}\sum\limits_{j_1=k_1-1}^{k_1+1} I_{\jj,\kk}(\xx).
\end{aligned}
\end{equation}

We need to estimate the decay of the kernels of the operators $\m_\kk(2^{-\kk}\sqrt{L})$.
We claim that for any $\kk\in \bN_0^2$ and
$\bkappa=(\kappa_1,\kappa_2)\in \bR^2_+$ such that $\bkappa > 3\dd/2$
\begin{equation}\label{K-mk}
\big|\KK_{\m_\kk(2^{-\kk}\sqrt{L})}(\xx,\yy)\big|
\le c2^{\kk\cdot\btau}V(\xx,2^{-\kk})^{-1}\DD^*_{2^{-\kk},\bkappa-\frac{\dd}{2}}(\xx,\yy).
\end{equation}
To establish this estimate we will apply Theorem~\ref{th:boxsupport} and inequality \eqref{D-D*}.
Clearly, $\supp \m_\kk\subset [-2,2]^2$
and hence we only need to show that
\begin{equation}\label{partial-mk}
\|\partial^{\bbeta}\m_\kk\|_{\infty}\le c_{\bkappa} 2^{\kk\cdot\btau}
\quad\text{for}\;\; |\bbeta|\le \kappa_1+\kappa_2.
\end{equation}
Observe first that $\varphi_\kk(2^\kk \blambda)$ is one of the the four functions
$\varphi^1_0(\lambda_1)\varphi^2_0(\lambda_2)$, $\varphi^1_0(\lambda_1)\varphi^2(\lambda_2)$,
$\varphi^1(\lambda_1)\varphi^2_0(\lambda_2)$, or $\varphi^1(\lambda_1)\varphi^2(\lambda_2)$,
depending on whether  $k_i=0$ or $k_i\ge 1$, $i=1,2$.
Hence
$\|\partial^{\bbeta}(\varphi_\kk(2^\kk \blambda))\|_{\infty}\le c(\bbeta)$.
On the other hand, from the hypothesis of the theorem
$\m\in \bM(\btau,\bkappa)$ and hence $\m(\blambda)$ obeys \eqref{multibounds}
leading to the conclusion that
\begin{align*}
|\partial^\bbeta (\m(2^\kk\blambda))|
&= |\partial^\bbeta (\m(2^{k_1}\lambda_1, 2^{k_2}\lambda_2))|
= |2^{k_1\beta_1+k_2\beta_2}\partial^\bbeta \m(2^{k_1}\lambda_1, 2^{k_2}\lambda_2)|
\\
&\le c2^{k_1\beta_1+k_2\beta_2}(1+2^{k_1}\lambda_1)^{\tau_1-\beta_1}(1+2^{k_2}\lambda_2)^{\tau_2-\beta_2}
\\
&\le c2^{k_1\tau_1+k_2\tau_2}(1+\lambda_1)^{\tau_1-\beta_1}(1+\lambda_2)^{\tau_2-\beta_2}
\le c2^{k_1\tau_1+k_2\tau_2},
\end{align*}
whenever $\beta_i\le \kappa_i$, $i=1,2$, and $\blambda\in [0,2]^2$.
The above estimates along with the Libniz rule imply \eqref{partial-mk}.
Then applying Theorem~\ref{th:boxsupport} shows that estimate \eqref{K-mk} is valid.

Choose $0<r<\min \{p,q\}$ and $\bet\in\RR^2$ with $\eta_i>2d_i$ so that $\kappa_i>\eta_i/r+3d_i/2$, $i=1,2$.
Assume $\kk,\jj\in\bN_0^2$ with $|j_i-k_i|\le 1$, $i=1,2$.
From this and \eqref{K-mk} it follows that
\begin{equation*}
\begin{aligned}
\big|\KK_{\m_\kk(2^{-\kk}\sqrt{L})}(\xx,\yy)\big|
&\le c2^{\kk\cdot\btau}V(\xx,2^{-\kk})^{-1}\DD^*_{2^{-\kk},\bkappa-\frac{\dd}{2}}(\xx,\yy)
\\
&\le c2^{\jj\cdot\btau} V(\xx,2^{-\jj})^{-1}\DD^*_{2^{-\jj},\bkappa-\frac{\dd}{2}}(\xx,\yy).
\end{aligned}
\end{equation*}
We use the above to estimate $I_{\jj,\kk}(\xx)$.
We obtain
\begin{equation}\label{eq:Fbounded6}
\begin{aligned}
I_{\jj,\kk}(\xx)
&\le c2^{\jj\cdot\btau} V(\xx,2^{-\jj})^{-1}
\int_\XX \DD_{2^{-\jj},\bkappa-\frac{\dd}{2}}^*(\xx,\yy)|\varphi_\jj(\sqrt{L})f(\yy)|d\mu(\yy)
\\
&\le c2^{\jj\cdot\btau} \sup_{\yy\in \XX} |\varphi_\jj(\sqrt{L})f(\yy)|\DD^*_{2^{-\jj},\frac{\bet}{r}}(\xx,\yy)
\\
&\hspace{1in}\times V(\xx,2^{-\jj})^{-1}\int_\XX \DD^*_{2^{-\jj},\bkappa-\frac{\dd}{2}-\frac{\bet}{r}}(\xx,\yy)d\mu(\yy)
\\
&\le c2^{\jj\cdot\btau}\cM_r\big(\varphi_\jj(\sqrt{L})f\big)(\xx).
\end{aligned}
\end{equation}
Here for the last inequality we used:
(a) the maximal inequality \eqref{Peetre-max} and
the fact that $\varphi_\jj(\sqrt{L})f \in \Sigma_{2^{\jj+\one}}$ and $\eta_i>2d_i$;
and
(b) \eqref{tech-1} and the fact that $\kappa_i>3d_i/2+\eta_i/r$.

Now, combining \eqref{eq:Fbounded3}) and (\ref{eq:Fbounded6}) we obtain
\begin{equation*}
|\varphi_\kk(\sqrt{L})\m(\sqrt{L})f(\xx)|
\le c\sum_{j_2=k_2-1}^{k_2+1}\sum_{j_1=k_1-1}^{k_1+1}\cM_r\big(2^{\jj\cdot\btau}\varphi_\jj(\sqrt{L})f\big)(\xx).
\end{equation*}
In turn, this and the maximal inequality \eqref{max} lead to
\begin{align*}
\|\m(\sqrt{L})f\|_{\FF^\ss_{pq}}
&=\Big\|\Big(\sum_{\kk\in\bN_0^2} \Big(2^{\kk\cdot\ss}|\varphi_\kk(\sqrt L)\m(\sqrt{L}) f(\cdot)|\Big)^q\Big)^{1/q}\Big\|_p
\\
&\le c\Big\|\Big(\sum_{\kk\in\bN_0^2} \Big(2^{\kk\cdot\ss}
\sum_{j_2=k_2-1}^{k_2+1}\sum_{j_1=k_1-1}^{k_1+1}
\cM_r\big(2^{\jj\cdot\btau}\varphi_\jj(\sqrt{L})f\big)(\cdot)\Big)^q\Big)^{1/q}\Big\|_p
\\
&\le c\Big\|\Big(\sum_{\kk\in\bN_0^2}\sum_{j_2=k_2-1}^{k_2+1}\sum_{j_1=k_1-1}^{k_1+1}
\big[\cM_r\big(2^{\jj\cdot(\ss+\btau)}\varphi_\jj(\sqrt{L})f\big)(\cdot)\big]^q\Big)^{1/q}\Big\|_p
\\
&\le c\Big\|\Big(\sum_{\jj\in\bN_0^2}
\big[\cM_r\big(2^{\jj\cdot(\ss+\btau)}\varphi_\jj(\sqrt{L})f\big)(\cdot)\big]^q\Big)^{1/q}\Big\|_p
\\
&\le c\Big\|\Big(\sum_{\jj\in\bN_0^2}
\big(2^{\jj\cdot(\ss+\btau)}|\varphi_\jj(\sqrt{L})f(\cdot)|\big)^q\Big)^{1/q}\Big\|_p
\\
&\le c\|f\|_{\FF^{\ss+\btau}_{pq}}
\end{align*}
and the proof is complete.
\end{proof}

\subsection{Lifting property}

As we saw in Remark \ref{ex:ma} the multiplier $\m_\btau$ belongs to
the class $\bM(\btau,\bkappa)$ for each $\bkappa\in\bN^2$.
This means that we can apply Theorem~\ref{th:multiBF} to both $\m_{\btau}$ and $\m_{-\btau}$.
On the other hand the spectral multiplier rising from $\m_{\btau}$ is the tensor product $\big(I_1+L_1\big)^{\tau_1/2}\otimes\big(I_2+L_2\big)^{\tau_2/2}$.
As a consequence of Theorem~\ref{th:multiBF} we infer the following lifting property:

\begin{corollary}\label{col:lifting}
Let $\ss, \btau\in\bR^2$ and $0<q\le\infty$. Then

$(i)$
If $0<p\le\infty$, then the operator
$\big(I_1+L_1\big)^{\tau_1/2}\otimes\big(I_2+L_2\big)^{\tau_2/2}$
is bounded from $\BB^{\ss}_{pq}$ into $\BB^{\ss-\btau}_{pq}$ and $\big\|\big(I_1+L_1\big)^{\tau_1/2}\otimes\big(I_2+L_2\big)^{\tau_2/2}f\big\|_{\BB^{\ss-\btau}_{pq}}$
is an equivalent quasi-norm for $\BB^\ss_{pq}$.

$(ii)$
If $0<p<\infty$, then the operator
$\big(I_1+L_1\big)^{\tau_1/2}\otimes\big(I_2+L_2\big)^{\tau_2/2}$
is bounded from $\FF^{\ss}_{pq}$ into $\FF^{\ss-\btau}_{pq}$ and $\big\|\big(I_1+L_1\big)^{\tau_1/2}\otimes\big(I_2+L_2\big)^{\tau_2/2}f\big\|_{\FF^{\ss-\btau}_{pq}}$
is an equivalent quasi-norm for $\FF^\ss_{pq}$.
\end{corollary}

This corollary is in the spirit of the classical smoothness reduction on $\bR^d$
by using as multipliers $\m_{\btau}(\bxi):=\big(1+|\bxi|^2\big)^{\btau/2}$, $\btau\in\bR$.

\subsection{Spectral multipliers on non-classical mixed-smoothness B-F-spaces}

It is natural to question whether Theorem~\ref{th:multiBF} holds for the $\tBB$ and $\tFF$ spaces.
We next establish the following result:

\begin{theorem}\label{th:multitBtF}
Let $\ss\in\bR^2$, $0<q\le\infty$, $\bkappa\in\bN^2$ and $\m\in\bM(\bkappa)$.

$(i)$ If $0<p\le\infty$ and $\kappa_i>|s_i|+\frac{2d_i}{p}+\frac{3d_i}{2}, i=1,2$,
then the spectral multiplier $\m(\sqrt{L})$ is bounded on $\tBB^{\ss}_{pq}$.

$(ii)$ If $0<p<\infty$ and $\kappa_i>|s_i|+\frac{2d_i}{\min(p,q)}+\frac{3d_i}{2}$,
then the spectral multiplier $\m(\sqrt{L})$ is bounded on $\tFF^\ss_{pq}$.
\end{theorem}

\begin{proof}
This proof is quite similar to the proof of Theorem~\ref{th:multiBF}.
So, we will use the notation and elements of the proof of Theorem~\ref{th:multiBF} and only indicate the differences.
Let $\kk\in\bN_0^2$.
Just as in \eqref{eq:Fbounded3} we get
\begin{equation}\label{eq:tB1}
V(\xx,2^{-\kk})^{-\ss/\dd}\big|\varphi_\kk(\sqrt{L})\m(\sqrt{L})f(\xx)\big|
\le \sum_{j_1=k_1-1}^{k_1+1}\sum_{j_2=k_2-1}^{k_2+1}V(\xx,2^{-\kk})^{-\ss/\dd}I_{\jj,\kk}(\xx).
\end{equation}
By \eqref{V-gamma-xy} and since $j_i\sim k_i$ we have
\begin{equation}\label{eq:tB2}
V(\xx,2^{-\kk})^{-\ss/\dd}\le c\DD^*_{2^\jj,\tilde{\ss}}(\xx,\yy)^{-1}V(\yy,2^{-\jj})^{-\ss/\dd},
\end{equation}
where $\tilde{\ss}:=(|s_1|,|s_2|)$.

Just as in \eqref{eq:Fbounded6} (with $\btau=(0,0)$)
using the maximal inequality \eqref{Peetre-max},
\eqref{eq:tB1}, \eqref{eq:tB2},
and estimate \eqref{tech-1},
which holds due to the assumption $\kappa_i>|s_i|+\frac{\eta_i}{r}+\frac{3d_i}{2}$,
we obtain
\begin{align*}
&V(\xx,2^{-\kk})^{-\ss/\dd}I_{\jj,\kk}(\xx)
\\
&\qquad \le cV(\xx,2^{-\jj})^{-1}
\int_\XX V(\yy,2^{-\jj})^{-\ss/\dd}\DD^*_{2^\jj,\bkappa-\tilde{\ss}-\frac{\dd}{2}}(\xx,\yy)\big|\varphi_\jj(\sqrt{L})f(\yy)\big|d\mu(\yy)
\\
&\qquad \le c\sup_{\yy\in \XX}V(\yy,2^{-\jj})^{-\ss/\dd}\big|\varphi_\jj(\sqrt{L})f(\yy)\big|\DD^*_{2^\jj,\frac{\bet}{r}}(\xx,\yy)
\\
&\hspace{1.5in}\times V(\xx,2^{-\jj})^{-1}\int_\XX \DD^*_{2^\jj,\bkappa-\tilde{\ss}-\frac{\bet}{r}-\frac{\dd}{2}}(\xx,\yy)d\mu(\yy)
\\
&\qquad \le c\cM_r\Big(V(\cdot,2^{-\jj})^{-\ss/\dd}\big|\varphi_\jj(\sqrt{L})f(\cdot)\big|\Big)(\xx).
\end{align*}
The rest of  the proof is identical to the one of Theorem \ref{th:multiBF}; we skip further details.
\end{proof}




\section{Ordinary Besov and Triebel-Lizorkin spaces on product domains}\label{sec:reg-B-F-spaces}

In this section we develop the basics of ordinary Besov and Triebel-Lizorkin spaces (B and F-spaces)
on the product domain $\XX=\XX_1\times \XX_2$ associated with the operators $L_1, L_2$.
As in the case of mixed-smoothness B and F-spaces, developed in Section~\ref{sec:B-F-spaces},
we will introduce {\em classical and nonclassical} ordinary B and F-spaces,
which reflects the possible anisotropic geometries of the spaces $\XX_1, \XX_2$.
The ordinary Besov and Triebel-Lizorkin spaces will be denoted by
$B_{pq}^s$, $\tB_{pq}^s$, $F_{pq}^s$, and $\tF_{pq}^s$
and developed with full set of parameters $s,p,q$.
Note that here the smoothness parameter $s\in \RR$, i.e. $s$ is a scalar,
while for the B and F-spaces with dominating mixed smoothness the smoothness parameter $\ss$ is a vector,
i.e. $\ss\in \RR^2$.
The theory of Besov and Triebel-Lizorkin spaces on product domains is grossly underdeveloped
because B and F-spaces on domains other than $\RR^n$ are much more recent vintage.
We refer the reader to \cite{IPX} for examples of ordinary B and F-spaces on specific product domains
such as $[-1,1]^d$ associated with the Jacobi operator
or $B^{d_1}\times [-1,1]^{d_2}$ associated with operators with polynomial eigenfunctions
on the unit ball $B^{d_1}$ in $\RR^{d_1}$ and $[-1,1]^{d_2}$.
Ordinary B and F-spaces on products of other domains such as the sphere, simplex, torus, and $\RR^d$
associated with operators are also discussed in \cite{IPX}.
Our development here will cover all these specific examples of ordinary B and F-spaces
as well as other examples that involve e.g. Riemannian manifolds, Lee groups and more general Dirichlet spaces,
see Section~\ref{sec:examples} below.

The theory of ordinary B and F-spaces on product domains is somewhat simpler
than the theory of the mixed-smoothness B and F-spaces, developed in Section~\ref{sec:B-F-spaces}.
The proofs of all results that we will present here either follow with minor modifications
from the proofs of the corresponding results for mixed-smoothness B and F-spaces
or can be carried out along the lines of the proofs in Section~\ref{sec:B-F-spaces}.
Therefore, we will put the emphasis only on the differences between the two theories
and omit the routine proofs.
Overall our aim here is to exhibit the main differences between
mixed-smoothness B and F-spaces and ordinary B and F-spaces on product domains
associated with operators and put in perspective our development from the previous sections.

\smallskip

\noindent
{\bf Properties of product spaces and notation.}
Here we will be operating in the setting described in \S\ref{sec:genset}.
In particular, we consider the product space $\XX=\XX_1\times \XX_2$
equipped with the metric
\begin{equation}\label{metric-r}
\rho(\xx,\yy):=\max\big\{\rho_1(x_1,y_1),\rho_2(x_2,y_2)\big)\},
\quad\xx,\;\yy\in \XX,
\end{equation}
and the product measure $\mu:=\mu_1\times\mu_2$.

The ball $B(\xx,r) \subset \XX$ centered at $\xx=(x_1,x_2)\in \XX$ of radius $r>0$ is obviously given by
\begin{equation}\label{ball-r}
B(\xx,r)=B_1(x_1,r)\times B_2(x_2,r)
\end{equation}
and
\begin{equation}\label{ballvolume-r}
V(\xx,r):=\mu(B(\xx,r))=\mu_1(B_1(x_1,r))\mu_2(B_2(x_2,r)),
\end{equation}
where $B_i(x_i,r)$ is the respective ball in $\XX_i$, $i=1,2$.
From \eqref{eq:doubling-0} it follows that the metric measure space $(\XX,\rho,\mu)$
has the doubling property and \eqref{doubling-d} yields
\begin{equation}\label{doubl-r}
V(\xx,\lambda\delta) \le c\lambda^d V(\xx,\delta)),\quad \delta\ge 0, \;\lambda\ge 1,
\end{equation}
where $d:= d_1+d_2$.

We define, for $\xx=(x_1, x_2)$, $\yy=(y_1, y_2)$ in $\XX$ and $\delta, \sigma >0$,
\begin{equation}\label{def-D}
\DD_{\delta,\sigma}(\xx,\yy)
:= \frac{\prod_{i=1,2}\big(1+\delta^{-1}\rho_i (x_i,y_i)\big)^{-\sigma}}{[V(\xx, \delta)V(\yy, \delta)]^{1/2}}.
\end{equation}

\begin{remark}\label{rem:regular}
Observe that unlike in previous sections in this section the parameters $d$, $\delta$, $\sigma$,
and later $\gamma$, $\tau$, $j$, etc. are scalars, which reflects the nature of
the ordinary Besov and Triebel-Lizorkin spaces that will be introduced below
as a single parameter smoothness spaces.
Also the meaning of $V(\xx, \delta)^\gamma$ is somewhat different compared to
$V(\xx, \bdelta)^\bgamma$, compare \eqref{ballvolume-r} to \eqref{V-gamma-xy}; 
compare $\DD_{\delta,\sigma}(\xx,\yy)$ from \eqref{def-D}
to $\DD_{\bdelta,\bsigma}(\xx,\yy)$ from \eqref{kernelsD} as well.
\end{remark}

\noindent
{\bf Integral operators.} 
Operators of the form
$F(\delta\sqrt{L}):= F(\delta\sqrt{L_1}, \delta\sqrt{L_2})$,
where $F(\lambda_1, \lambda_2)$ is smooth and satisfies the condition
$F(\pm \lambda_1, \pm\lambda_2)=F(\lambda_1,\lambda_2)$,
will play an important role.
The next theorem is an immediate consequence of Theorem~\ref{thm:gen-local}.

\begin{theorem}\label{thm:gen-local-r}
Let $F\in \cC^{2k}(\bR^2)$, $k>3(d_1\vee d_2)/2$, $k\ge 2$, 
be real-valued and
satisfy the conditions:
$F(\pm \lambda_1, \pm\lambda_2)=F(\lambda_1,\lambda_2)$ for $(\lambda_1,\lambda_2)\in \bR^2$
and
\begin{equation*}
|\partial^{\bbeta}F(\blambda)|\le C_k(1+|\blambda|)^{-r},\quad r>d+2k,\; |\bbeta| \le 2k.
\end{equation*}
Then $F(\delta \sqrt L):=F(\delta\sqrt{L_1}, \delta\sqrt{L_2})$, $\delta>0$, is an integral operator
whose kernel $\KK_{F(\delta \sqrt L)}(\xx,\yy)$ 
satisfies
\begin{equation}\label{K-local-r}
|\KK_{F(\delta \sqrt L)}(\xx,\yy)|\le c_k \DD_{\delta,k}(\xx, \yy) \quad \text{and}
\end{equation}
\begin{equation}\label{K-lip-r}
|\KK_{F(\delta \sqrt L)}(\xx,\yy)- \KK_{F(\delta \sqrt L)}(\xx,\yy')|
\le c_k' \Big(\frac{\rho(\yy, \yy')}{\delta}\Big)^{\alpha_1\wedge \alpha_2}\DD_{\delta,k}(\xx,\yy)
\end{equation}
if $\rho(\yy,\yy')\le \delta$, where
the constants $c_k,c_{k'}$ depend on $k$, $d$, and $C_k$.
\end{theorem}

This theorem makes it possible to develop the theory of ordinary Besov and Triebel-Lizorkin spaces
in the setting of this article.

\smallskip

\noindent
{\bf Single parameter functional calculus.}
Clearly, (see Proposition~\ref{prop:L1xL2}) the operator
$L_1\otimes I_2+I_1\otimes L_2$ is essentially self-adjoint and non-negative.
We denote its closure again by $L_1\otimes I_2+I_1\otimes L_2$.
Denote by $\sE_\lambda$, $\lambda\ge 0$, the spectral resolution associated to
$L_1\otimes I_2+I_1\otimes L_2$ on $L^2(\XX, d\mu)$.
Then for any bounded measurable function $F$ on $[0, \infty)$ the operator $F(L_1\otimes I_2+I_1\otimes L_2)$
is defined by
\begin{equation*}
F(L_1\otimes I_2+I_1\otimes L_2)=\int_0^\infty F(\lambda) d\sE_\lambda.
\end{equation*}
The semigroup associated to $L_1\otimes I_2+I_1\otimes L_2$ takes the form
$$
e^{-t(L_1\otimes I_2+I_1\otimes L_2)}= e^{-tL_1}\otimes e^{-tL_2}, \quad t> 0,
$$
and using the assumptions from \S\ref{sec:genset} it follows that
$e^{-t(L_1\otimes I_2+I_1\otimes L_2)}$ has Gaussian localization, H\"{o}lder continuity,
and the Markov property, see \S\ref{sec:genset} and Proposition~\ref{prop:Gauss-local}.
Also, as was alluded to above the space $(\XX, \rho, \mu)$, $\XX=\XX_1\times \XX_2$, has the doubling property.
Therefore, a smooth (one parameter) functional calculus
as in \S\ref{subsec:local-kern} or in \cite[Section 3]{KP} can be developed in full.

We next make the connection between the single parameter functional calculus
induced by $L_1\otimes I_2+I_1\otimes L_2$
and the two-parameter functional calculus associated to $L_1,L_2$ from Section~\ref{sec:func-calc}.
As above assume that $F$ is a bounded measurable function on $[0, \infty)$.
Consider the function $G(\lambda_1, \lambda_2):= F(\lambda_1^2+\lambda_2^2)$ on $[0, \infty)^2$.
Then we have
\begin{equation*}
F(L_1\otimes I_2+I_1\otimes L_2)=G(\sqrt{L_1}, \sqrt{L_2}).
\end{equation*}
As will be shown in \S\ref{subsec:B-F-spaces-r} below this leads to the conclusion that
the ordinary Besov and Triebel-Lizorkin spaces on the product domain $\XX_1\times \XX_2$
can be developed based on the functional calculus associated to $L_1\otimes I_2+I_1\otimes L_2$.

\smallskip

\noindent
{\bf Maximal operators.} 
In the current setting the maximal operator $\cM_r$, $r>0$, is defined as usual by
\begin{equation}\label{Max1-r}
\cM_r f(\xx):= \sup_{B \ni \xx} \Big(\frac{1}{\mu(B)}\int_{B}|f(\yy)|^r d\mu(\yy)\Big)^{1/r},
\quad \xx\in \XX,
\end{equation}
where the $\sup$ is taken over all balls $B\subset \XX$
such that $\xx\in B$.

With this notation the Fefferman-Stein vector-valued maximal inequality takes
the usual form \eqref{max}.

\smallskip

\noindent
{\bf Test functions and distributions} 
Naturally, in the current setting the class of test functions
$\cS:= \cS(L_1, L_2)$ is defined as the set of all functions
\begin{equation*}
\phi\in \cap_{\bnu\in\bN_0^2} D(L_1^{\nu_1}\otimes L_2^{\nu_2}),
\end{equation*}
where $D(L_1^{\nu_1}\otimes L_2^{\nu_2})$ is the domain of the operator $L_1^{\nu_1}\otimes L_2^{\nu_2}$, such that
\begin{equation}\label{norm-S-r}
\cP_{m,k}(\phi)
:= \sup_{\xx\in \XX}\big(1+\rho(\xx,\xx_{0})\big)^{k} \max_{0\le \nu_i\le m} \big|(L_1^{\nu_1}\otimes L_2^{\nu_2})\phi(\xx)\big| <\infty,
\quad \forall\, k,m\ge 0.
\end{equation}
As before here the point $\xx_0\in \XX$ is selected arbitrarily and fixed once and for all.
It~is easy to see that this class of test functions coincides with the class $\cS$
from \S\ref{subsec:test-dist} and has the same topology.
Also, as in \S\ref{subsec:test-dist} the space $\cS'=\cS'(L_1, L_2)$ of distributions associated to $L_1, L_2$
is defined as the set of all
continuous linear functionals on $\cS=\cS(L_1, L_2)$.

\smallskip

\noindent
{\bf Spectral spaces.}  
The spectral space $\Sigma_t$, $t>0$, is defined by
\begin{equation}\label{spec-sp-r}
\Sigma_t :=\big\{g\in \cS': \theta (\sqrt L_1, \sqrt{L_2})g = g, \text{ for all } \theta\in \cA_t\big\},
\quad \cS':=\cS'(L_1, L_2),
\end{equation}
where
$$
\cA_t:=\big\{\theta\in \cC^\infty_0(\bR^2): \theta(\pm\lambda_1,\pm\lambda_2)=\theta(\lambda_1,\lambda_2)
\text{ and } \theta\equiv 1 \text{ on } [0, t]^2\big\}.
$$
Observe that $\Sigma_t=\Sigma_{(t,t)}$, where $\Sigma_{(t,t)}$ is the spectral space
introduced in Definition~\ref{def:spec-spaces}.

The Peetre maximal inequality takes this form:
If $\gamma >0$, $\tau > 2(d_1\vee d_2)$, and $\bnu\in \bN_0^2$,
then there exists a constant $c>0$
such that
for any $g\in\Sigma_t$ and $t\ge 1$
\begin{equation}
\begin{aligned}\label{Peetre-max-r}
&t^{-2|\bnu|}\sup_{\yy\in \XX}\frac{V(\yy, t^{-1})^{\gamma}|(L_1^{\nu_1}\otimes L_2^{\nu_2}) g(\yy)|}
{\prod_{i=1,2}\big(1+t\rho(x_i, y_i)\big)^{\tau/r}}
\\
&\quad\le c\sup_{\yy\in \XX}\frac{V(\yy,t^{-1})^{\gamma}|g(\yy)|}
{\prod_{i=1,2}\big(1+t\rho(x_i, y_i)\big)^{\tau/r}}
\le c\cM_{r}\big(V(\cdot,t^{-1})^{\gamma}g(\cdot)\big)(\xx), \;\; \xx\in \XX.
\end{aligned}
\end{equation}

\smallskip

The Nikolski type inequality takes the form:
If $0<p\le q\le\infty$, $\nu\in \bN_0$, and $\gamma\in\bR$,
then there exists a constant $c>0$ such that
for any $g\in\Sigma_t$, $t\ge 1$, and $\bnu\in \bN_0^2$, we have
\begin{equation}\label{Band-1-r}
\big\|V(\cdot,t^{-1})^{\gamma} (L^{\nu_1}\otimes L^{\nu_2})g(\cdot)\big\|_q
\le c t^{2|\bnu|}\big\|V(\cdot,t^{-1})^{\gamma+\frac{1}{q}-\frac{1}{p}} g(\cdot)\big\|_p.
\end{equation}

Inequalities \eqref{Peetre-max-r}, \eqref{Band-1-r} follow readily by Theorems~\ref{thm:Peetre-max}, \ref{thm:Nik}.

\subsection{Definition of ordinary Besov and Triebel-Lizorkin spaces}\label{subsec:B-F-spaces-r}

We now introduce ordinary Besov and Triebel-Lizorkin spaces on $\XX=\XX_1\times \XX_2$
associated with the operators $L_1, L_2$.

We introduce two types of ordinary B and F-spaces:
(i) Classical ordinary B-spaces  $B_{pq}^{s}= B_{pq}^{s}(L_1, L_2)$
and F-spaces $F_{pq}^{s}=F_{pq}^{s}(L_1,L_2)$,
and
(ii) Nonclassical ordinary B-spaces $\tB_{pq}^{s}=\tB_{pq}^{s}(L_1, L_2)$
and F-spaces $\tF_{pq}^{s}=\tF_{pq}^{s}(L_1, L_2)$.
To define these spaces we introduce a pair of real-valued functions
$\varphi_0,\varphi \in \cC^\infty(\RR^2)$ satisfying the conditions:
\begin{equation}\label{cond-1-r}
\begin{aligned}
&(i)\ \varphi_0(\pm \lambda_1,\pm\lambda_2)= \varphi_0(\lambda_1,\lambda_2),
\ \varphi(\pm \lambda_1,\pm\lambda_2)= \varphi(\lambda_1,\lambda_2), \; \blambda\in\RR^2,
\\
&(ii)\ \supp\varphi_0\subset B(0,2), \
|\varphi_0(\blambda)|\ge\hat{c}>0 \;\hbox{  for  } |\blambda|\le 5/3,
\\
&(iii)\ \supp\varphi\subset B(0,2)\setminus B(0,1/2),\
|\varphi(\blambda)| \ge \hat{c}>0 \hbox{ for } 3/5\le|\blambda|\le 5/3,
\end{aligned}
\end{equation}
where $B(0,r):= \{\blambda\in\RR^2: |\blambda|<r\}$, $|\blambda|:=(\lambda_1^2+\lambda_2^2)^{1/2}$.

Denote $\varphi_j(\blambda):=\varphi(2^{-j}\blambda)$, $j\in \bN$.
Then from above it follows that
\begin{equation}\label{cond-2-r}
\sum_{j\in \bN_0} |\varphi_j(\blambda)| \ge \hat{c} >0, \quad \blambda\in\RR^2.
\end{equation}
In what follows we will use the short-hand notation
$\varphi_j(\sqrt{L}):= \varphi(2^{-j}\sqrt{L_1}, 2^{-j}\sqrt{L_2})$
and
$\sqrt{L}:=(\sqrt{L_1}, \sqrt{L_2})$.
Also, recall that here $d:=d_1+d_2$.


\begin{definition}\label{def-B-spaces-r}
Let $s \in \R$ and $0<p,q \le \infty$.

$(i)$ The ordinary Besov space  $B_{pq}^s=B_{pq}^s(L)$
is defined as the set of all $f \in \cS'$ such that
\begin{equation}\label{def-Besov-space1-r}
\|f\|_{B_{pq}^s} :=
\Big(\sum_{j\in\bN_0} \Big(2^{js}
\|\varphi_j(\sqrt{L}) f\|_{p}\Big)^q\Big)^{1/q} <\infty.
\end{equation}

$(ii)$ The ordinary Besov space  $\tB_{pq}^s= \tB_{pq}^s(L)$ is defined as the set
of all $f \in \cS'$ such that
\begin{equation}\label{def-Besov-space2-r}
\|f\|_{\tB_{pq}^s}
:= \Big(\sum_{j\in\bN_0} \Big( \big\|V(\cdot, 2^{-j})^{-s/d}
\varphi_j(\sqrt{L}) f(\cdot)\big\|_{p}\Big)^q\Big)^{1/q} <\infty.
\end{equation}
Above the $\ell^q$-norm is replaced by the sup-norm if $q=\infty$.
\end{definition}

\begin{definition}\label{def-F-spaces-r}
Let $s \in \RR$, $0<p< \infty$ and $0<q \le \infty$.

$(i)$ The ordinary Triebel-Lizorkin space  $F_{pq}^s=F_{pq}^s(L)$
is defined as the set of all $f \in \cS'$ such that
\begin{equation}\label{def-F-space1-r}
\|f\|_{F_{pq}^s} :=\Big\|\Big(\sum_{j\in \bN_0}
\Big(2^{js}|\varphi_j(\sqrt L) f(\cdot)|\Big)^q\Big)^{1/q}\Big\|_{p} <\infty.
\end{equation}

$(ii)$  The ordinary Triebel-Lizorkin space  $\tF_{pq}^s= \tF_{pq}^s(L)$ is defined as the set of all $f \in \cS'$
such that
\begin{equation}\label{def-F-space2-r}
\|f\|_{\tF_{pq}^s} :=
\Big\|\Big(\sum_{j\in \bN_0} \Big(V(\cdot, 2^{-j})^{-s/d}
|\varphi_j(\sqrt L) f(\cdot)|\Big)^q\Big)^{1/q}\Big\|_{p} <\infty.
\end{equation}

\end{definition}

A straightforward adaptation of the proof of Proposition~\ref{prop:independent}
shows that the definition of the ordinary Besov and Triebel-Lizorkin spaces from above
is independent of the specific selection of the functions $\varphi_0,\varphi\in\cC^\infty(\RR^2)$
satisfying conditions \eqref{cond-1-r}.

We next show that the ordinary Besov and Triebel-Lizorkin spaces defined above
essentially depend only on the operator $L_1\otimes I_2+I_1\otimes L_2$.
Indeed, let
$\psi_0(\blambda) = \Phi_0(|\blambda|^2)$ and $\psi(\blambda) = \Phi(|\blambda|^2)$,
where $\Phi_0, \Phi\in \cC^\infty(\RR)$ are real-valued even functions such that
\begin{equation}\label{cond-3-r}
\begin{aligned}
&(i)\ \supp\Phi_0\subset [-4,4], \
|\Phi_0(t)|\ge\hat{c}>0 \;\hbox{  for  } |t|\le 25/9,
\\
&(ii)\ \supp\Phi\cap[0,\infty)\subset [1/4,4],\
|\Phi(t)| \ge \hat{c}>0 \;\hbox{ for } 9/25\le|t|\le 25/9.
\end{aligned}
\end{equation}

It is easily seen that the functions $\psi_0,\psi\in \cC^\infty(\RR^2)$ and satisfy conditions \eqref{cond-1-r}.
Therefore, in the definition of the spaces $B_{pq}^s$, $\tB_{pq}^s$, $F_{pq}^s$, and $\tF_{pq}^s$ above
the functions $\varphi_0, \varphi$ can be replaced by $\psi_0, \psi$.
Observe now that for $j\ge 1$
\begin{equation*}
\psi_j(\sqrt{L})= \Phi(2^{-2j}(L_1\otimes I_2+I_1\otimes L_2))
\end{equation*}
and similarly
$\psi_0(\sqrt{L})= \Phi_0(L_1\otimes I_2+I_1\otimes L_2)$.
Furthermore, as was explained above. 
the smooth functional calculus associated to the operator $L_1\otimes I_2+I_1\otimes L_2$ can be developed
as a single parameter functional calculus.
Therefore, the theory of ordinary Besov and Triebel-Lizorkin spaces defined above
can be developed based on the single parameter calculus associated to the operator $L_1\otimes I_2+I_1\otimes L_2$.

\smallskip

\noindent
{\bf Embeddings.}
The embeddings of the ordinary Besov and Triebel-Lizorkin spaces in the setting of this article
that we present next are quite similar to the embedings of the mixed-smoothness Besov and Triebel-Lizorkin spaces.

The embeddings between ordinary B and F-spaces and test functions or distributions
take the form:
Let $s\in\bR$ and $0<q\le \infty$. If $0<p\le \infty$ then
\begin{equation*}
\cS\hookrightarrow B^s_{pq},\;\tB^s_{pq}\hookrightarrow\cS',
\end{equation*}
and, if $0<p< \infty$, then
\begin{equation*}
\cS\hookrightarrow \; F^s_{pq},\;\tF^s_{pq}\hookrightarrow\cS'.
\end{equation*}
The proof of this result is carried out along the lines of the proof of Theorem~\ref{thm:embed}.

\smallskip

The following embedding result between ordinary B-spaces with different parameters
follows readily by the Nikolski inequality \eqref{Band-1-r}:
If $s, s'\in\bR$, $0<p\le r\le\infty$, $0<q\le\tau\le\infty$, and
\begin{equation*}
\frac{s}{d}-\frac{1}{p}=\frac{s'}{d}-\frac{1}{r},
\end{equation*}
then $\tB^{s}_{pq}\hookrightarrow \tB^{s'}_{r\tau}$.
If in addition the non-collapsing condition $\eqref{non-collapsing}$ holds for the spaces $\XX_1$ and $\XX_2$,
then $B^{s}_{pq}\hookrightarrow B^{s'}_{r\tau}$.

Also, just as in Theorem~\ref{prop:embed-B-L} we have:
If $s>0$ and $0<p,q\le\infty$, then $B^{s}_{pq}\hookrightarrow L^p$.

\subsection{Spectral multipliers}\label{subsec:multipl-r}

Here we consider spectral multipliers that are compatible with the ordinary
Besov and Triebel-Lizorkin defined in \S\ref{subsec:B-F-spaces-r}.

\begin{definition}\label{def:mS-r}
A function $\m \in \cC^{\infty}(\R^2)$ is called {\em admissible}
if it is real-valued,
$\m(\pm\lambda_1,\pm\lambda_2)=\m(\lambda_1,\lambda_2)$ for $(\lambda_1,\lambda_2)\in \bR^2$,
and all its partial derivatives have at most polynomial growth,
i.e. for any $\bbeta\in\bN_0^2$ there exist constants $c_{\bbeta}>0$ and $N_{\bbeta}\in\bN_0$ such that
\begin{equation*} 
\big|\partial^{\bbeta}\m(\blambda)\big|\le c_{\bbeta}(1+|\blambda|)^{N_{\bbeta}}.
\end{equation*}

Let $\m \in \cC^{\infty}(\R^2)$ be admissible and
consider a pair of real-valued functions
$\varphi_0,\varphi \in \cC^\infty(\RR^2)$ satisfying the conditions:
\begin{equation*} 
\begin{aligned}
&(i)\ \varphi_0(\pm \lambda_1,\pm\lambda_2)= \varphi_0(\lambda_1,\lambda_2),
\ \varphi(\pm \lambda_1,\pm\lambda_2)= \varphi(\lambda_1,\lambda_2), \; \blambda\in\RR^2,
\\
&(ii)\ \supp\varphi_0\subset B(0,2), \
\supp\varphi\subset B(0,2)\setminus B(0,1/2),
\end{aligned}
\end{equation*}
where $B(0,r):= \{\blambda\in\RR^2: |\blambda|<r\}$,
and
\begin{equation*} 
\varphi_0(\blambda)+\sum_{j\in \bN}\varphi(2^{-j}\blambda)=1, \quad \forall \blambda \in \RR^2.
\end{equation*}
Denote $\varphi_j(\blambda):=\varphi(2^{-j}\blambda)$, $j\in \bN$.
Then from above it follows that
$\sum_{j\in\bN_0} \varphi_j(\blambda) = 1$ for $\blambda \in \RR^2$.
Set $\m_j(\blambda):=\m(\blambda)\varphi_j(\blambda)$.
Then obviously
\begin{equation*} 
\m(\blambda)=\sum_{j\in \bN_0}\m(\blambda)\varphi_j(\blambda)=\sum_{j\in \bN_0}\m_j(\blambda).
\end{equation*}

We define the multiplier operator
$\m(\sqrt{L})=\m(\sqrt{L_1},\sqrt{L_2})$ on $\cS=\cS(L_1,L_2)$
by
\begin{equation*} 
\m(\sqrt{L})\phi
:=\sum_{j\in \bN_0}\m_j(\sqrt{L})\phi,
\quad \phi\in \cS,
\end{equation*}
and on $\cS'=\cS'(L_1,L_2)$ by
\begin{equation*} 
\m(\sqrt{L})f
:=\sum_{j\in \bN_0}\m_j(\sqrt{L})f,
\quad f\in \cS',
\end{equation*}
where the convergence is in $\cS$ and $\cS'$.

\end{definition}

An easy adaptation of the proof of Theorem~\ref{thm:mS}
leads to the following claim:
If~$\m \in \cC^{\infty}(\R^2)$ is admissible $($Defintion~\ref{def:mS-r}$)$,
then the multiplier operator $\m(\sqrt{L})$ is well defined and continuous on $\cS$ and $\cS'$.

\smallskip

With the next definition we introduce a class of multipliers that
will act on ordinary B and F-spaces.

\begin{definition}\label{def:multipl-BF-r}
Assume $\tau\in\bR$, $\kappa\in\bN$
and let the real-valued function $\m\in \cC^{\kappa}(\RR^2)$ satisfy
$\m(\pm \lambda_1, \pm\lambda_2)=\m(\lambda_1,\lambda_2)$ for $(\lambda_1,\lambda_2)\in \bR^2$.
We say that $\m$ belongs to the class $\bM(\tau,\kappa)$ if
for any $\bbeta=(\beta_1,\beta_2)$ with $\beta_i\le\kappa$
there exists a constant $c_{\bbeta}>0$ such that
\begin{equation}\label{multibounds-r}
|\partial^\bbeta\m(\blambda)|
\le c_{\bbeta}\big(1+\lambda_1\big)^{\tau-\beta_1}(1+\lambda_2\big)^{\tau-\beta_2},
\quad \forall \; \blambda\ge \zero. 
\end{equation}
\end{definition}

The following two theorems are analogues of Theorems~\ref{th:multiBF}, \ref{th:multitBtF}.

\begin{theorem}
\label{th:multiBF-r}
Let $s, \tau\in\bR$, $0<q\le\infty$, $\kappa\in\bN$ and $\m\in\bM(\tau,\kappa)$.

$(i)$ If $0<p\le\infty$ and $\kappa>\frac{2d}{p}+\frac{3d}{2}$,
then the spectral multiplier $\m(\sqrt{L})$ is bounded from $B^{s+2\tau}_{pq}$ into $B^s_{pq}$.

$(ii)$ If $0<p<\infty$ and $\kappa>\frac{2d}{\min\{p,q\}}+\frac{3d}{2}$,
then the spectral multiplier $\m(\sqrt{L})$ is bounded from $F^{s+2\tau}_{pq}$ into $F^s_{pq}$.
\end{theorem}

\begin{theorem}\label{th:multitBtF-r}
Let $s\in\bR$, $0<q\le\infty$, $\kappa\in\bN$ and $\m\in\bM(0,\kappa)$.

$(i)$ If $0<p\le\infty$ and $\kappa>|s|+\frac{2d}{p}+\frac{3d}{2}$,
then the spectral multiplier $\m(\sqrt{L})$ is bounded on $\tB^{s}_{pq}$.

$(ii)$ If $0<p<\infty$ and $\kappa>|s|+\frac{2d}{\min(p,q)}+\frac{3d}{2}$,
then the spectral multiplier $\m(\sqrt{L})$ is bounded on $\tF^s_{pq}$.
\end{theorem}

The proofs of these two theorems are carried out along the lines of the proofs of Theorems~\ref{th:multiBF}, \ref{th:multitBtF}.
We omit them.

\section{Appendix}\label{sec:appendix}

\subsection{Examples of coordinate spaces covered by our theory}\label{sec:examples}

In this section we provide several examples
of metric measure spaces $(\XX,\rho,\mu)$ associated with operators $L$
satisfying the conditions on the coordinate spaces from our general setting in \S\ref{sec:genset}.

\begin{example}
The classical setup of $\XX=\RR^n$ equipped with the Euclidean distance and
Lebesgue measure, combined with the Laplace operator $L=-\Delta$
is our principle example of a coordinate space covered by our general setting in \S\ref{sec:genset}.

Another classical example is the torus
$\XX= \RR^n/2\pi\ZZ^n$
equipped with the Euclidean distance and
Lebesgue measure, combined with the Laplace operator $L=-\Delta$.
\end{example}


\begin{example}
The unit sphere $\XX=\bS^{n-1}$ in $\RR^n$ equipped with the standard (geodesic) distance and surface measure
associated with the Laplace-Beltrami operator
is also a typical example of a coordinate space in the sense of our general setting from \S\ref{sec:genset}.
\end{example}


\begin{example}
Consider $\XX =[-1, 1]$ equipped with the weighted measure
\begin{equation*}
d\mu(x) := w(x) dx = (1-x)^{\alpha}(1+x)^{\beta} dx, \quad \alpha, \beta>-1,
\end{equation*}
and the distance
$\rho(x,y) := |\arccos x - \arccos y|$.
It is easy to see that $($e.g. \cite{CKP}$)$
\begin{equation*}
V(x, r)\sim r(1-x+r^2)^{\alpha+1/2}(1+x+r^2)^{\beta+1/2}, \quad x\in [-1, 1], \; 0<r\le \pi.
\end{equation*}
Therefore, we have a doubling metric measure space. 
Let $L$ be the classical Jacobi operator defined by
\begin{equation*}
Lf(x):=-\frac{\big[w(x)(1-x^2)f'(x)\big]'}{w(x)}.
\end{equation*}
As is shown in \cite{CKP} $($see also \cite{KPX}$)$
the associated heat kernel has Gaussian localization, H\"{o}lder continuity, and the Markov property.
\end{example}


\begin{example}
Consider the \textit{unit ball} $\bB^n:=\big\{x\in\mathbb{R}^n: \|x\|<1\big\}$
equipped with measure
\begin{equation*}
d\mu := (1-\| x\|^2)^{\gamma-1/2} dx, \quad \gamma >-1,
\end{equation*}
and metric
\begin{equation*}
\rho(x,y) := \arccos \big(\langle x, y\rangle + \sqrt{1-\| x\|^2}\sqrt{1-\| y\|^2}\big),
\end{equation*}
where $\langle x, y\rangle$ is the inner product of $x, y\in \R^n$
and $\|x\|:= \sqrt{\langle x, x\rangle}$.
It is easy to see that $(\bB^n,\rho,\mu)$ is a doubling metric measure space, since for the ball centered at $x\in \bB^n$ of radius $r$ one has
\begin{equation*}
V(x, r) \sim r^n(1-\|x\|^2+r^2)^\gamma.
\end{equation*}
%
As is well known $($e.g. \cite{KPX}$)$ the operator
\begin{equation}
L:= -\sum_{i=1}^n (1-x_i^2)\partial^2_i + 2\sum_{1\le i < j
\le n}x_i x_j\partial_i\partial_j + (n+2 \gamma)\sum_{i=1}^n x_i \partial_i,
\end{equation}
acting on smooth functions on $\bB^n$
is non-negative and self-adjoint with polynomial eigenfunctions.
More importantly, its heat kernel
has Gaussian localization, H\"{o}lder continuity, and the Markov property
$($see \cite{KPX, KPX2}$)$.

\end{example}


\begin{example}
Consider the simplex
\begin{equation*}
\bT^n:=\Big\{x \in \RR^n: x_1 > 0,\dots, x_n>0,\; |x| < 1 \Big\},
\quad |x|:= x_1+\cdots+x_n,
\end{equation*}
with measure
\begin{equation*}
d\mu(x) = \prod_{i=1}^{n} x_{i}^{\kappa_i-1/2}(1-|x|)^{\kappa_{n+1}-1/2}dx,
\quad  \kappa_i >-1/2,
\end{equation*}
and distance
\begin{equation*}
\rho(x,y) = \arccos \Big(\sum_{i=1}^n \sqrt{x_i y_i} + \sqrt{1-|x|}\sqrt{1-|y|}\Big).
\end{equation*}
The volume of the balls here $($see e.g. \cite{KPX}$)$ behave as
\begin{equation*}
V(x, r) \sim r^n (1-|x|+r^2)^{\kappa_{n+1}}\prod_{i=1}^n(x_i+r^2)^{\kappa_i}
\end{equation*}
and hence
$(\bT^n,\rho,\mu)$ is a doubling metric measure space.
The associated operator $($see \cite{KPX}$)$
\begin{equation*}
L := -\sum_{i=1}^n x_i\partial_i^2 + \sum_{i=1}^n\sum_{j=1}^n x_ix_j \partial_i\partial_j
- \sum_{i=1}^n \big(\kappa_i + \tfrac12 -(|\kappa|+ \tfrac{n+1}{2}) x_i\big) \partial_i
\end{equation*}
with $|\kappa|:=\kappa_1+\dots+\kappa_{n+1}$
is non-negative and self-adjoint with polynomial eigenfunctions.
Furthermore, its heat kernel
has Gaussian localization, H\"{o}lder continuity, and the Markov property
$($see \cite{KPX, KPX2}$)$.

\end{example}

\begin{example} As three generic examples of coordinate spaces in the sense of the setting from \S\ref{sec:genset} we list the followings:

\begin{itemize}
\item The Euclidean space $\R^n$ associated with uniformly elliptic divergence form operators.

\item Lie groups or homogeneous spaces with polynomial volume growth, associated with sub-Laplacians.

\item Complete Riemannian manifolds with Ricci curvature bounded from below associated with the Laplace-Beltrami operator.
\end{itemize}

\end{example}

There are various other examples of spaces that can serve as coordinate spaces and obey the conditions from \S\ref{sec:genset}.
For more details see the examples presented in \cite{CKP, KP}.

\subsection{Proof of Proposition~\ref{prop:prealgebra}}

To show that $\overline{E}_0(R)$ is well defined we first extend the operator $E_0$ to $\BB_0(\R^2)$.
For any $R\in \BB_0(\R^2)$ represented as in \eqref{R-JJ} we define
\begin{equation}\label{def:E0ext}
E_0(R) := \sum_{j=1}^N E_0(J_{1j}\times J_{2j})
\quad\hbox{on}\quad \sL^2(\XX,d\mu).
\end{equation}
We next show that the operator $E_0(R)$ is well defined.
To this end we need to show that the definition of $E_0(R)$ in \eqref{def:E0ext}
is independent of the selection of $\{J_{1j}\times J_{2j}\}$.


We first show that if
\begin{equation*}
R = (a_1, b_1]\times (a_2, b_2]=: \tilde{J}_1 \times \tilde{J}_2
\end{equation*}
is represented as in \eqref{R-JJ}, i.e.
$
R = \tilde{J}_1 \times \tilde{J}_2=\cup_{j=1}^N J_{1j}\times J_{2j},
$
then
\begin{equation}\label{E0-E0}
E_0(\tilde{J}_1 \times \tilde{J}_2) := \sum_{j=1}^N E_0(J_{1j}\times J_{2j}).
\end{equation}
Indeed, let us project the end points of the intervals
$\{J_{1j}\}$ onto $(a_1, b_1]$ and the end points of the intervals $\{J_{2j}\}$ onto $(a_2, b_2]$.
Denote these points by $\{\xi_j\}$ and $\{\eta_k\}$ and
assume that they are ordered as follows
\begin{equation}\label{xi-eta}
a_1=\xi_0 < \xi_1 < \cdots <\xi_n=b_1
\quad\hbox{and}\quad
a_2=\eta_0 < \eta_1 < \cdots <\eta_m=b_2.
\end{equation}
Denote $I_{1j}:=(\xi_{j-1}, \xi_j]$ and $I_{2k}:=(\eta_{k-1}, \eta_k]$.
Then
$
\tilde{J}_1 = \cup_{j=1}^n I_{1j}
$
and
$
\tilde{J}_2 = \cup_{k=1}^n I_{2k}.
$
From \eqref{additive} it follows that
\begin{equation*}
E_1(\tilde{J}_1) = \sum_{j=1}^n E_1(I_{1j}),
\quad
E_2(\tilde{J}_2) = \sum_{k=1}^m E_2(I_{2k})
\end{equation*}
and using \eqref{E0-J1-J2} we obtain
\begin{equation}\label{E0-lk}
E_0(\tilde{J}_1 \times \tilde{J}_2) = \sum_{\ell=1}^n\sum_{k=1}^m E_0(I_{1\ell}\times I_{2k}).
\end{equation}
On the other hand, clearly each rectangle $J_{1j}\times J_{2j}$ either coincides with one of the rectangles
$I_{1\ell}\times I_{2k}$
or
\begin{equation*}
J_{1j}=\cup_{\ell=\ell_1}^{\ell_2} I_{1\ell}
\quad\hbox{and}\quad
J_{2j}=\cup_{k=k_1}^{k_2} I_{2k}
\end{equation*}
and hence
$
J_{1j}\times J_{2j} = \cup_{\ell=\ell_1}^{\ell_2}\cup_{k=k_1}^{k_2} I_{1\ell}\times I_{2k}.
$
Hence, just as above
\begin{equation*}
E_0(J_{1j}\times J_{2j}) = \sum_{\ell=\ell_1}^{\ell_2}\sum_{k=k_1}^{k_2} E_0(I_{1\ell}\times I_{2k}).
\end{equation*}
Combining this with \eqref{E0-lk} implies \eqref{E0-E0}.


\smallskip

Assume now that for $R\in \BB_0(\R^2)$ we have
\begin{equation}\label{R-JJ-II}
R = \cup_{j=1}^N J_{1j}\times J_{2j} = \cup_{k=1}^K I_{1k}\times I_{2k},
\quad J_{ij}, I_{ik}\in \cJ,
\end{equation}
where the rectangles $\{J_{1j}\times J_{2j}\}$ are disjoint and
$\{I_{1k}\times I_{2k}\}$ are also disjoint.
To find a common refinement of these to sets of rectangles
we proceed similarly as above.
We project the end points of the intervals
$\{J_{1j}\}$ onto the $x_1$-axis and the end points of the intervals $\{J_{2j}\}$ onto the $x_2$-axis.
Denote these points by $\{\xi_j\}$ and $\{\eta_k\}$ and
assume that they are ordered just as in \eqref{xi-eta}, i.e.
\begin{equation*}
\xi_0 < \xi_1 < \cdots <\xi_n
\quad\hbox{and}\quad
\eta_0 < \eta_1 < \cdots <\eta_m.
\end{equation*}
Denote
$Q_{1\ell}:=(\xi_{\ell-1}, \xi_\ell]$ and $Q_{2\nu}:=(\eta_{\nu-1}, \eta_\nu]$.
Also, denote by $\cI_1$ the set of all indices $\ell$ such that $Q_{1\ell}$ is contained
in some intervals $J_{1j}$ and $I_{1k}$,
and by $\cI_2$ the set of all indices $\nu$ such that $Q_{2\nu}$ is contained
in some intervals $J_{2j}$ and $I_{2k}$.

Clearly, each rectangle $J_{1j}\times J_{2j}$ either coincides with one of the rectangles
$Q_{1\ell}\times Q_{2\nu}$
or
\begin{equation*}
J_{1j}=\cup_{\ell=\ell_1}^{\ell_2} Q_{1\ell}
\quad\hbox{and}\quad
J_{2j}=\cup_{\nu=\nu_1}^{\nu_2} Q_{2\nu},
\end{equation*}
and hence
$
J_{1j}\times J_{2j} = \cup_{\ell=\ell_1}^{\ell_2}\cup_{\nu=\nu_1}^{\nu_2} Q_{1\ell}\times Q_{2\nu}.
$
Similarly, each rectangle $I_{1k}\times I_{2k}$ is covered by some of the rectangles $\{Q_{1\ell}\times Q_{2\nu}\}$.
Thus,
$\{Q_{1\ell}\times Q_{2\nu}\}_{\ell\in\cI_1, \nu\in\cI_2}$
is a common refinement of
the rectangles $\{J_{1j}\times J_{2j}\}$ and $\{I_{1k}\times I_{2k}\}$.

Using \eqref{E0-E0} we obtain
\begin{equation*}
\cup_{j=1}^N E_0(J_{1j}\times J_{2j})
= \sum_{\ell\in\cI_1} \sum_{\nu\in\cI_2}E(Q_{1\ell}\times Q_{2\nu})
=\cup_{k=1}^K E_0(I_{1k}\times I_{2k}).
\end{equation*}
Therefore, the operator $E_0(R)$ is well defined by \eqref{def:E0ext}.
Since $\overline{E}_0(R)$ is the closure of $E_0(R)$, then $\overline{E}_0(R)$ is also well defined by \eqref{def:E0bar}.

The identity $\overline{E}_0(\R^2)=I$ follows readily by the obvious fact that $E_0(\R^2)=I$ on $\sL^2(\XX, d\mu)$.

Our next step is to show that if
\begin{equation}\label{JJ}
\tilde{J}_1 \times \tilde{J}_2
= (a_1, b_1]\times (a_2, b_2]
= \cup_{j=1}^\infty J_{1j}\times J_{2j},
\quad J_{ij}\in \cJ(\R),
\end{equation}
where
$\{J_{1j}\times J_{2j}\}$ are disjoint, then
\begin{equation}\label{E0-bar-infty}
\overline{E}_0(\tilde{J}_1 \times \tilde{J}_2) = \sum_{j=1}^\infty \overline{E}_0(J_{1j}\times J_{2j}),
\end{equation}
where the convergence is unconditional in the strong sense.

{\em Case 1.}
Consider first the case when there exist two countable collections of disjoint intervals
$\{I_{1k}\}_{k=1}^\infty$,
$\{I_{1\ell}\}_{\ell=1}^\infty$ in $\cJ(\R)$
such that
\begin{equation*}
\tilde{J_1} = \cup_{k=1}^\infty I_{1k},
\quad
\tilde{J_2} = \cup_{\ell=1}^\infty I_{2\ell},
\quad\hbox{and}\quad
\{J_{1j}\times J_{2j}\}_{j=1}^\infty = \{I_{1k}\times I_{2\ell}\}_{k,\ell\ge 1}.
\end{equation*}
In this particular case \eqref{E0-bar-infty} takes the form
\begin{equation}\label{E0-prod-infty}
\overline{E}_0(\tilde{J}_1 \times \tilde{J}_2) = \sum_{k=1}^\infty \sum_{\ell=1}^\infty \overline{E}_0(I_{1k}\times I_{2\ell}).
\end{equation}
To prove this we first show that
\begin{equation}\label{E0-infty}
E_0(\tilde{J}_1 \times \tilde{J}_2) = \sum_{k=1}^\infty \sum_{\ell=1}^\infty E_0(I_{1k}\times I_{2\ell})
\quad\hbox{on}\quad \sL^2(\XX, d\mu).
\end{equation}
Indeed, since $E_1$, $E_2$ are spectral measures we have for any $f\in L^2(\XX_1, d\mu_1)$, $g\in L^2(\XX_2, d\mu_2)$
\begin{equation}\label{E1-E2-conv}
E_1(\tilde{J}_1)f=\lim_{K\to\infty} \sum_{k=1}^K E_1(I_{1k})f
\quad\hbox{and}\quad
E_2(\tilde{J}_2)f=\lim_{N\to\infty} \sum_{\ell=1}^N E_2(I_{2\ell})g,
\end{equation}
where the convergence is unconditional in the respective $L^2$-norms.
Therefore,
\begin{align*}
E_0(\tilde{J}_1 \times \tilde{J}_2)f\otimes g
&= \lim_{K, N\to\infty}\Big(\sum_{k=1}^K E_1(I_{1k})f\Big)\otimes \Big(\sum_{\ell=1}^N E_2(I_{2\ell})g\Big)
\\
& = \lim_{K, N\to\infty} \sum_{k=1}^K \sum_{\ell=1}^N E_1(I_{1k})f\otimes E_2(I_{2\ell})g
\\
& = \lim_{K, N\to\infty} \sum_{k=1}^K \sum_{\ell=1}^N E_0(I_{1k}\times I_{2\ell})f\otimes g
\\
& = \sum_{k=1}^\infty \sum_{\ell=1}^\infty E_0(I_{1k}\times I_{2\ell})f\otimes g,
\end{align*}
where the convergence is in the $L^2(\XX, d\mu)$-norm.
Here we used the obvious fact that
if $\lim_{n\to\infty}\|F_n-F\|_{L^2(\XX_1)}=0$ and $\lim_{m\to\infty}\|G_m-G\|_{L^2(\XX_2)}=0$,
then
$\lim_{n,m\to\infty}\|F_n\otimes G_m-F\otimes G\|_{L^2(\XX)}=0$.
Clearly, \eqref{E0-infty} follows from the above.

By Proposition~\ref{prop:E0} we know that for any $I_1, I_2\in \cJ(\R)$
the operator $\overline{E}_0(I_1 \times I_2)$ is a projector
and using \eqref{E0E0} we conclude that for any $K, N\ge 1$
the operator
$\sum_{k=1}^K \sum_{\ell=1}^N \overline{E}_0(I_{1k}\times I_{2\ell})$
is a projector in $L^2(\XX, d\mu)$ (see e.g. \cite[Theorem 3.7]{Prugov}).
Therefore, for any $h \in L^2(\XX, d\mu)$ we have
\begin{align*}
\sum_{k=1}^K \sum_{\ell=1}^N \|\overline{E}_0(I_{1k}\times I_{2\ell})h\|^2
= \Big\|\Big(\sum_{k=1}^K \sum_{\ell=1}^N \overline{E}_0(I_{1k}\times I_{2\ell})\Big)h\Big\|^2
\le \|h\|^2
\end{align*}
and hence
\begin{equation*}
\sum_{k=1}^\infty \sum_{\ell=1}^\infty \|\overline{E}_0(I_{1k}\times I_{2\ell})h\|^2 \le \|h\|^2.
\end{equation*}
Consequently, the series
$\sum_{k=1}^\infty \sum_{\ell=1}^\infty \overline{E}_0(I_{1k}\times I_{2\ell})h$
converges unconditionally in the $L^2(\XX, d\mu)$-norm for each $h\in L^2(\XX, d\mu)$.
To show that
\begin{equation}\label{E0-h}
\overline{E}_0(\tilde{J}_1 \times \tilde{J}_2)h=\sum_{k=1}^\infty \sum_{\ell=1}^\infty \overline{E}_0(I_{1k}\times I_{2\ell})h
\end{equation}
we take any function $w\in \sL^2(\XX, d\mu)$.
Then
\begin{align*}
\Big\langle \sum_{k=1}^\infty \sum_{\ell=1}^\infty& \overline{E}_0(I_{1k}\times I_{2\ell})h, w\Big\rangle
= \sum_{k=1}^\infty \sum_{\ell=1}^\infty \langle \overline{E}_0(I_{1k}\times I_{2\ell})h, w\rangle
\\
&= \sum_{k=1}^\infty \sum_{\ell=1}^\infty \langle h, \overline{E}_0(I_{1k}\times I_{2\ell})w\rangle
= \sum_{k=1}^\infty \sum_{\ell=1}^\infty \langle h, E_0(I_{1k}\times I_{2\ell})w\rangle
\\
&= \Big\langle h, \sum_{k=1}^\infty \sum_{\ell=1}^\infty E_0(I_{1k}\times I_{2\ell})w\Big\rangle
= \langle h, E_0(\tilde{J}_1 \times \tilde{J}_2)w\rangle
\\
&= \langle h, \overline{E}_0(\tilde{J}_1 \times \tilde{J}_2)w\rangle
= \langle\overline{E}_0(\tilde{J}_1 \times \tilde{J}_2)h, w\rangle.
\end{align*}
Above we used \eqref{E0-infty} and that the operator $\overline{E}_0(I_{1k}\times I_{2\ell})$ is self-adjoint
and is the unique extension of the operator $E_0(I_{1k}\times I_{2\ell})$ from $\sL^2(\XX, d\mu)$ to $L^2(\XX, d\mu)$.
Since the above identities are valid for each $w\in \sL^2(\XX, d\mu)$ and $\sL^2(\XX, d\mu)$ is dense in $L^2(\XX, d\mu)$
they imply \eqref{E0-prod-infty}.

{\em Case 2.} We now establish \eqref{E0-bar-infty} in the general case.
Assume that the cover \eqref{JJ} of $\tilde{J}_{1}\times \tilde{J}_{2}$ is valid,
where $\{J_{1j}\times J_{2j}\}$ are disjoint.
We project the end points of all intervals $\{J_{1j}\}$ onto $(a_1, b_1]$
to obtain a countable set of points $\{\xi_j\}_{j\ge 1}$.
Let $\{I_{1k}\}_{k\ge 1}$ be the disjoint collection of all interval with end points $\{\xi_j\}_{j\ge 1}$,
where neither interval $I_{1k}$ contains a point from $\{\xi_j\}_{j\ge 1}$ in its interior.
Thus we have the representation
$(a_1, b_1]= \cup_{k\ge 1}I_{1k}$.
We similarly construct a disjoint collection of intervals $\{I_{2\ell}\}_{\ell\ge 1}$
so that $(a_2, b_2]= \cup_{\ell\ge 1}I_{2\ell}$.
As a result we have
\begin{equation}\label{J1-J2}
\tilde{J}_1 \times \tilde{J}_2
= (a_1, b_1]\times (a_2, b_2]
= \cup_{k=1}^\infty\cup_{\ell=1}^\infty I_{1k}\times I_{2\ell}
\end{equation}
and each rectangle $J_{1j}\times J_{2j}$ is a finite or countable disjoint union of rectangles $\{I_{1k}\times I_{2\ell}\}$
of the form
\begin{equation*}
J_{1j}\times J_{2j} = \cup_{k\in \cK_j}\cup_{\ell\in\cL_j} I_{1k}\times I_{2\ell},
\end{equation*}
where
$J_{1j}=\cup_{k\in \cK_j} I_{1k}$ and $J_{2j}=\cup_{\ell\in \cL_j} I_{2\ell}$.
Clearly, we have the disjoint unions
\begin{equation*}
\tilde{J}_1 \times \tilde{J}_2
= \cup_{k=1}^\infty\cup_{\ell=1}^\infty I_{1k}\times I_{2\ell}
= \cup_{j=1}^\infty \cup_{k\in \cK_j}\cup_{\ell\in\cL_j} I_{1k}\times I_{2\ell}.
\end{equation*}

By the result from Case 1 (see \eqref{E0-prod-infty}) for any $h\in L^2(\XX, d\mu)$
\begin{equation*}
\overline{E}_0(\tilde{J}_1 \times \tilde{J}_2)h = \sum_{k=1}^\infty \sum_{\ell=1}^\infty \overline{E}_0(I_{1k}\times I_{2\ell})h,
\end{equation*}
where the convergence is unconditional in the $L^2(\XX, d\mu)$-norm.
Since the convergence above is unconditional we can rearrange the terms above as we wish
and hence
\begin{equation}\label{E0-mix}
\overline{E}_0(\tilde{J}_1 \times \tilde{J}_2)h
= \sum_{k=1}^\infty \sum_{\ell=1}^\infty \overline{E}_0(I_{1k}\times I_{2\ell})h
= \sum_{j=1}^\infty \sum_{k\in \cK_j}\sum_{\ell\in\cL_j} \overline{E}_0(I_{1k}\times I_{2\ell})h.
\end{equation}
Applying again the result from Case 1 and \eqref{J1-J2} we obtain for $j=1, 2, \dots$
\begin{equation*}
\overline{E}_0(J_{1j}\times J_{2j})h = \sum_{k\in \cK_j}\sum_{\ell\in\cL_j} \overline{E}_0((I_{1k}\times I_{2\ell})h.
\end{equation*}
This coupled with \eqref{E0-mix} leads to
\begin{equation*}
\overline{E}_0(\tilde{J}_1 \times \tilde{J}_2)h = \sum_{j=1}^\infty \overline{E}_0(J_{1j}\times J_{2j})h,
\quad \forall h\in L^2(\XX, d\mu),
\end{equation*}
which implies \eqref{E0-bar-infty}.

From the definition in \eqref{def:E0bar} it readily follows that
to prove \eqref{additive-B0} it suffices to show that
\begin{equation}\label{rep-ER}
\overline{E}_0(R) = \sum_{j=1}^\infty \overline{E}_0(J_{1j}\times J_{2j})
\end{equation}
for any $R\in \sB_0(\R^2)$ that is a disjoint union of the form
\begin{equation}\label{rep-R}
R = \cup_{j=1}^\infty J_{1j}\times J_{2j},
\quad J_{ij} \in \cJ(\R).
\end{equation}
Assume that $R\in \sB_0(\R^2)$ has this representation.
But, because $R\in \sB_0(\R^2)$ from \eqref{R-JJ} it follows that
$R$ can also be represented as a disjoint union of finitely many rectangles, say,
\begin{equation}\label{rep-R-2}
R = \cup_{k=1}^K I_{1k}\times I_{2k},
\quad I_{ik} \in \cJ(\R).
\end{equation}
Then by the definition in \eqref{def:E0bar} we have
\begin{equation}\label{E0-R-I}
\overline{E}_0(R) = \cup_{k=1}^K \overline{E}_0(I_{1k}\times I_{2k}).
\end{equation}
Evidently, each rectangle $J_{1j}\times J_{2j}$ from \eqref{rep-R} is
contained either in a single rectangle $I_{1k}\times I_{2k}$ from \eqref{rep-R-2}
or in the union of several rectangles $\{I_{1k}\times I_{2k}\}$.
Denote by
\begin{equation*}
\{Q_{1j}^k\times Q_{2j}^k\}_{k\in\cK_j}, \quad Q_{ij}^k\in \cJ(\R),\; \cK_j \le K,
\end{equation*}
the set of all rectangles obtained by intersecting
$J_{1j}\times J_{2j}$ with rectangles
$\{I_{1k}\times I_{2k}\}$.
We order the rectangles
$\{Q_{1j}^k\times Q_{2j}^k: k\in\cK_j, j\in \bN\}$
in a sequence that we denote by
$\{\tilde{Q}_{1\nu}\times \tilde{Q}_{2\nu}\}_{\nu\ge 1}$.

Clearly,
\begin{equation*}
J_{1j}\times J_{2j} = \cup_{k\in\cK_j} Q_{1j}^k\times Q_{2j}^k
\end{equation*}
and hence
\begin{equation*}
\overline{E}_0(J_{1j}\times J_{2j}) = \sum_{k\in\cK_j} \overline{E}_0(Q_{1j}^k\times Q_{2j}^k).
\end{equation*}
Summing up we get
\begin{equation}\label{sum-E0}
\sum_{j=1}^\infty\overline{E}_0(J_{1j}\times J_{2j})
= \sum_{j=1}^\infty\sum_{k\in\cK_j} \overline{E}_0(Q_{1j}^k\times Q_{2j}^k)
=\sum_{\nu=1}^\infty \overline{E}_0(\tilde{Q}_{1\nu}\times \tilde{Q}_{2\nu}).
\end{equation}

On the other hand,
from \eqref{rep-R}-\eqref{rep-R-2} and the construction of the rectangles $\{\tilde{Q}_{1\nu}\times \tilde{Q}_{2\nu}\}$
it readily follows that each rectangle $I_{1k}\times I_{2k}$  is covered by the rectangles
$\{\tilde{Q}_{1\nu}\times \tilde{Q}_{2\nu}\}_{\nu\ge 1}$.
More precisely, denoting by $\cN_k$ the indices of the rectangles $\tilde{Q}_{1\nu}\times \tilde{Q}_{2\nu}$
contained in $I_{1k}\times I_{2k}$ we have
\begin{equation*}
I_{1k}\times I_{2k} = \cup_{\nu\in\cN_k} \tilde{Q}_{1\nu}\times \tilde{Q}_{2\nu}.
\end{equation*}
Applying now identity \eqref{E0-bar-infty} we obtain
\begin{equation*}
\overline{E}_0(I_{1k}\times I_{2k}) = \sum_{\nu\in\cN_k} \overline{E}_0(\tilde{Q}_{1\nu}\times \tilde{Q}_{2\nu})
\end{equation*}
and summing up we arrive at
\begin{equation}\label{sum-E0-2}
\sum_{k=1}^K\overline{E}_0(I_{1k}\times I_{2k})
= \sum_{k=1}^K\sum_{\nu\in\cN_k} \overline{E}_0(\tilde{Q}_{1\nu}\times \tilde{Q}_{2\nu})
=\sum_{\nu=1}^\infty \overline{E}_0(\tilde{Q}_{1\nu}\times \tilde{Q}_{2\nu}).
\end{equation}
Clearly, identity \eqref{rep-ER} follows from \eqref{E0-R-I}, \eqref{sum-E0}, and \eqref{sum-E0-2}.
This completes the proof of Proposition~\ref{prop:prealgebra}.
\qed

\end{document}